\documentclass[11pt,reqno]{amsart}
\usepackage{amssymb,amsmath,amsfonts,amsthm,enumerate,enumitem,stmaryrd,tensor,mathtools,dsfont,upgreek,bbm,mathrsfs,nameref,xcolor,bm,tikz,stmaryrd,todonotes,graphicx,thmtools}
\usepackage[framemethod=TikZ]{mdframed}
\usetikzlibrary{arrows}
\usepackage[left=0.6 in, right=0.6 in, top=0.6 in, bottom=0.6 in]{geometry}

\usepackage{hyperref,cleveref}
\usepackage{caption}
\usepackage{subcaption}

\numberwithin{equation}{subsection}
\setlist[enumerate]{leftmargin=0.3in}
\setlist[itemize]{leftmargin=0.3in}
\definecolor{popblue}{RGB}{55,115,255}
\definecolor{lightbl}{RGB}{155,205,255}
\definecolor{depthbl}{RGB}{145,215,255}
\definecolor{fancyre}{RGB}{225,55,115}
\definecolor{lightgr}{RGB}{230,255,230}
\definecolor{darkgre}{RGB}{25,105,25}
\definecolor{darkblu}{RGB}{15,75,185}
\definecolor{mellowy}{RGB}{225,225,35}

\renewcommand{\tilde}[1]{\widetilde{#1}}
\renewcommand{\Bar}{\overline}

\newcommand{\R}{\mathbb{R}}
\newcommand{\N}{\mathbb{N}}
\newcommand{\Z}{\mathbb{Z}}

\newcommand{\C}{\mathbb{C}}

\newcommand{\m}{\mathrm}

\newcommand{\lv}{\lVert}
\newcommand{\rv}{\rVert}

\newcommand{\al}{\alpha}
\newcommand{\be}{\beta}
\newcommand{\es}{\varnothing}
\newcommand{\lra}{\;\Leftrightarrow\;}
\newcommand{\ep}{\varepsilon}
\newcommand{\f}{\frac}
\newcommand{\sig}{\sigma}
\newcommand{\gam}{\gamma}
\newcommand{\del}{\delta}

\newcommand{\pd}{\partial}

\newcommand{\grad}{\nabla}
\newcommand{\bpm}{\begin{pmatrix}}
\newcommand{\epm}{\end{pmatrix}}

\newcommand{\loc}{\m{loc}}
\renewcommand{\bar}{\overline}

\newcommand{\emb}{\hookrightarrow}
\newcommand{\res}{\restriction}

\renewcommand{\bigstar}{%
{
  \tikz[baseline=(center), scale=0.2] {
    \coordinate (center) at (0,-0.5);
    \filldraw (0,0.4) -- (0.1,0) -- (0,-0.4) -- (-0.1,0) -- cycle;
    \filldraw (0.4,0) -- (0,0.1) -- (-0.4,0) -- (0,-0.1) -- cycle;
  }
}
}

\renewcommand{\le}{\leqslant}
\renewcommand{\ge}{\geqslant}

\newcommand{\norm}[1]{\left\lv#1\right\rv}
\newcommand{\bnorm}[1]{\Big\lv#1\Big\rv}
\newcommand{\snorm}[1]{\big\lv#1\big\rv}
\newcommand{\tnorm}[1]{\lv#1\rv}

\newcommand{\bp}[1]{\Big(#1\Big)}
\renewcommand{\sp}[1]{\big(#1\big)}
\newcommand{\tp}[1]{(#1)}

\newcommand{\abs}[1]{\left|#1\right|}
\newcommand{\babs}[1]{\Big|#1\Big|}

\newcommand{\tabs}[1]{|#1|}

\newcommand{\bsb}[1]{\Big[{#1}\Big]}
\newcommand{\ssb}[1]{\big[{#1}\big]}
\newcommand{\tsb}[1]{[{#1}]}

\newcommand{\tcb}[1]{\{{#1}\}}

\providecommand{\tbr}[1]{\langle #1 \rangle}

\renewcommand{\bf}[1]{\mathbf{#1}}

\declaretheoremstyle[
    headfont=\bfseries\rmfamily\color{popblue}, bodyfont=\normalfont, 
    mdframed={
        linewidth=0.8pt,
        linecolor=lightbl, backgroundcolor=lightbl!5,
        skipabove=4pt,
    }
]{blueStyle}

\declaretheoremstyle[
    headfont=\bfseries\rmfamily\color{popblue}, 
    bodyfont=\itshape,
]{boringStyle}
\declaretheorem[style=boringStyle, numberwithin=section, name=Proposition]{propC}
\declaretheorem[style=boringStyle, sibling=propC, name=Lemma]{lemC}
\declaretheorem[style=boringStyle, sibling=propC, name=Theorem]{thmC}
\declaretheorem[style=boringStyle, sibling=propC, name=Corollary]{coroC}
\declaretheorem[style=boringStyle, sibling=propC, name=Remark]{rmkC}

\declaretheoremstyle[
    headfont=\bfseries\rmfamily\color{darkgre}, bodyfont=\normalfont, 
    mdframed={
        linewidth=0.8pt,
        linecolor=darkgre!35, backgroundcolor=lightgr!7,
        skipabove=4pt,
    }
]{greenStyle}

\declaretheoremstyle[
    headfont=\bfseries\rmfamily\color{fancyre}, bodyfont=\normalfont, 
    mdframed={
        linewidth=0.8pt,
        linecolor=fancyre!35, backgroundcolor=mellowy!6,
        skipabove=4pt,
    }
]{redStyle}
\declaretheorem[style=boringStyle, sibling=propC, name=Definition]{defnC}

\declaretheorem[style=boringStyle, name=Theorem]{introthm}
\declaretheorem[style=boringStyle, name=Definition, sibling=introthm]{introdefn}
\declaretheorem[style=boringStyle, name=Corollary, sibling=introthm]{introcoro}
\author{Noah Stevenson}
\address{
	Department of Mathematics\\
	Princeton University\\
	Princeton, NJ 08544, USA
}
\email[N. Stevenson]{stevenson@princeton.edu}
\thanks{N. Stevenson was supported by an NSF Graduate Research Fellowship}
\author{Ian Tice}
\address{
	Department of Mathematical Sciences\\
	Carnegie Mellon University\\
	Pittsburgh, PA 15213, USA
}
\email[I. Tice]{iantice@andrew.cmu.edu}
\thanks{I. Tice was supported by an NSF Grant (DMS \#2204912). }

\DeclareFontFamily{U}{cbgreek}{}
\DeclareFontShape{U}{cbgreek}{m}{n}{
  <-6>    grmn0500
  <6-7>   grmn0600
  <7-8>   grmn0700
  <8-9>   grmn0800
  <9-10>  grmn0900
  <10-12> grmn1000
  <12-17> grmn1200
  <17->   grmn1728
}{}
\DeclareFontShape{U}{cbgreek}{bx}{n}{
  <-6>    grxn0500
  <6-7>   grxn0600
  <7-8>   grxn0700
  <8-9>   grxn0800
  <9-10>  grxn0900
  <10-12> grxn1000
  <12-17> grxn1200
  <17->   grxn1728
}{}

\DeclareRobustCommand{\qoppa}{%
  \text{\usefont{U}{cbgreek}{\normalorbold}{n}\symbol{19}}%
}

\DeclareRobustCommand{\Sampi}{%
  \text{\usefont{U}{cbgreek}{\normalorbold}{n}\symbol{23}}%
}

\DeclareRobustCommand{\kkappa}{%
  \text{\usefont{U}{cbgreek}{\normalorbold}{n}\symbol{107}}%
}
\makeatletter
\newcommand{\normalorbold}{%
  \ifnum\pdf@strcmp{\math@version}{bold}=\z@ bx\else m\fi
}
\makeatother

\title[Traveling bore waves]{Gravity driven traveling bore wave solutions to the free boundary incompressible Navier-Stokes equations}

\begin{document}
\begin{abstract}    
    We give the first mathematical construction of two-dimensional traveling bore wave solutions to the free boundary incompressible Navier-Stokes equations for a single finite depth layer of constant density fluid. Our construction is based on a rigorous justification of the formal shallow water limit, which postulates that in a certain scaling regime the full free boundary traveling Navier-Stokes system of PDEs reduces to a governing system of ODEs. We find heteroclinic orbits solving these ODEs and, through a delicate fixed point argument employing the Stokes problem in thin domains and a nonautonomous orbital perturbation theory, use these ODE solutions as the germs from which we build bore PDE solutions for sufficiently shallow layers.
\end{abstract}
\maketitle
\section{Introduction}\label{section on introduction}

A fluid bore is a special type of traveling wave, observed on the surfaces of canals and rivers, that possesses distinct asymptotic heights in the upstream and downstream limits. It has been understood since at least the work of Rayleigh~\cite{rayleigh1914theory}, over a century ago, that a single layer of inviscid homogeneous incompressible fluid cannot sustain bores. Rayleigh's original argument, which has since been reconfirmed and generalized in modern proofs, is that inviscid conservation laws are incompatible with a bore-like change of states, and he deduces from this that including viscous effects is essential to understand the phenomenon. Therefore, in a mathematical study of the existence of bores it is reasonable to account for more robust fluid mechanical effects, such as viscosity and other sources of friction.

The mathematical search for bores then leads to the realm of traveling \emph{viscous} surface waves, a topic relatively unexplored by rigorous analysis,  in stark contrast with the classical, centuries-spanning study of traveling waves on the free surface of an \emph{inviscid} fluid. The vast majority of the latter consists of investigations of irrotational water waves, with important effects such as vorticity only recently receiving rigorous attention. The typical inviscid surface wave handled by this theory is a traveling profile sustained purely by a constant gravitational force. On the other hand, all known traveling waves with viscosity are generated by gravity acting in conjunction with a spatially varying force or surface stress (e.g. localized external pressure) that perturbs the equilibrium. Thus, a construction of viscous bore solutions would not only stand as the first mathematical verification that single layers can support a profile with differing asymptotic heights, but would also show for the first time that gravity alone can sustain nontrivial viscous surface waves.

In this paper we prove that within the realm of shallow fluids, modeled by the free boundary incompressible Navier-Stokes equations over an inclined plane, bores not only exist but are, in a sense, common steady states. Due to the shallow assumption, the dynamics are described by a viscous Saint-Venant system of ordinary differential equations, which we show admit special heteroclinic orbits that then give an effective approximation of our Navier-Stokes solutions. Moreover, we produce both small and large bores, as measured by the size of the jump in velocity across the bore's transition front.  Furthermore, our construction, which is general and modular, opens the door to further searches for other types of gravity driven viscous surface waves. 

\subsection{Traveling bore wave solutions to the free boundary Navier-Stokes system}\label{subsection on the free boundary navier stokes system}

In this paper we study two-dimensional finite-depth layers of incompressible, viscous fluid with free boundaries.  We posit that the fluid sits atop a flat rigid surface and meets an atmosphere of constant pressure along its free boundary. The dynamics of the fluid are governed by the incompressible Navier-Stokes equations, but we focus our attention on traveling wave solutions, which are a special type of solution that is time-independent when viewed in a coordinate frame moving at a constant speed parallel to the fluid bottom.  Within the traveling frame the fluid occupies a static, but unknown, set of the form
\begin{equation}\label{notation for the fluid domain}
    \Omega_\zeta = \tcb{\tp{x,z}\in\R\times\R\;:\;0<z<\zeta(x)}
\end{equation}
where $\zeta:\R\to \R^+$ is an unknown function, which we assume is continuous in order to ensure $\Omega_\zeta$ is open and connected.   The graph of $\zeta$, 
\begin{equation}\label{notation for the free surface set}
    \Sigma_\zeta = \tcb{\tp{x,z}\in\R\times\R\;:\;z = \zeta(x)},
\end{equation}
describes the free boundary along the upper part of the fluid, while the set $\Sigma_0 = \R\times\tcb{0}$ describes the rigid bottom boundary of the fluid.

We assume that the only bulk force acting on the fluid is gravity, which has both a standard vertical component as well as a horizontal component.  The latter can be thought of as corresponding to the fluid being situated on an inclined plane, after a coordinate transformation.  The fluid configuration is described by the aforementioned $\zeta$ in addition to the unknowns $(v,q) : \Omega_\zeta \to \R^2 \times \R$, the velocity field and pressure profiles in the traveling frame.  These must satisfy the traveling free boundary Navier-Stokes equations:
\begin{equation}\label{parent free boundary navier-stokes system}
    \begin{cases}
        \tp{v-\Bar{\gam} e_1}\cdot\grad v + \grad q - \mu\Delta v = \kappa e_1,\quad \grad\cdot v=0&\text{in }\Omega_\zeta,\\
        -\tp{q - \mu\mathbb{D}v}\mathcal{N}_{\zeta} + \tp{\Bar{g}\zeta - \Bar{\sig}\mathcal{H}\tp{\zeta}}\mathcal{N}_\zeta = 0,\quad \Bar{\gam} \partial_1 \zeta + v\cdot\mathcal{N}_\zeta = 0&\text{on }\Sigma_\zeta,\\
        v_2=0,\quad \mu\pd_2v_1 = \Bar{a}v_1&\text{on }\Sigma_0,
    \end{cases}
\end{equation}
where $\mathcal{N}_\zeta = (-\zeta',1)$ is a (non-unit) outward normal and $\mathcal{H}\tp{\zeta} = \zeta''/\tp{1+\tabs{\zeta'}^2 }^{3/2}$  is the mean curvature of the free surface.  The roles of the parameters are as follows: $\mu>0$ is the fluid viscosity, $\Bar{\gamma}>0$ is the speed of the traveling wave frame in the $e_1$ direction, $\Bar{g}\ge0$ is the vertical gravitational acceleration, $\kappa >0$ is the horizontal gravitational acceleration,  $\Bar{\sig}\ge 0$ is the surface tension coefficient,  and $\Bar{a}>0$ is the slip parameter (with units of speed so that $\mu/\Bar{a}$ is the slip length).  For the sake of brevity, we have chosen to state the system~\eqref{parent free boundary navier-stokes system} directly in traveling form rather than record its derivation from the time-dependent Navier-Stokes system; we refer to~\cite{MR4630597,MR4609068} for details on how to perform the derivation. The first equation in~\eqref{parent free boundary navier-stokes system} corresponds to the balance of momentum, the second to the conservation of mass, and the third to the balance of surface stresses on the free boundary, while the fourth is the kinematic equation for the free surface. The final two equations are known as the Navier slip boundary conditions, as they forbid fluid transport through the rigid interface but allow horizontal slip at the cost of an oppositional tangential stress.  First proposed by Navier in~\cite{Navier}, this boundary condition generalizes the usual no-slip condition (recovered formally by sending $\Bar{a} \to \infty$) to allow the modeling of slip along the boundary, a phenomenon that has since been experimentally observed: see~\cite{Neto} and references therein. 

For any continuously differentiable map $w : \Omega_\zeta \to \R^2$ satisfying $\nabla \cdot w =0$ in $\Omega_\zeta$ and $w_2=0$ on $\Sigma_0$ we may compute 
\begin{equation}
  \partial_1 \int_0^{\zeta(x)} w_1(x,z)\;\m{d}z =  w_1(x,\zeta(x)) \zeta'(x) - \int_0^{\zeta(x)} \partial_2 w_2(x,z)\;\m{d}z  
     =  -w(x,\zeta(x)) \cdot \mathcal{N}_\zeta(x).  
\end{equation}
From this we readily deduce that the kinematic boundary condition in~\eqref{parent free boundary navier-stokes system} is equivalent to 
\begin{equation}\label{intro_flux_cond}
    \Phi =  -\Bar{\gamma} \zeta(x)  + \int_0^{\zeta(x)} v_1(x,z)\;\m{d}z = \int_0^{\zeta(x)} (v_1(x,z) - \Bar{\gamma})\;\m{d}z  \text{ for } x \in \R 
\end{equation}
for a constant $\Phi \in \R$ measuring the relative flux across all vertical line segments $(0,\zeta(x))$.  We will use this reformulation of the kinematic boundary condition, as the flux constant will play a crucial role in our analysis.

The horizontal gravitational acceleration term $\kappa e_1$ in~\eqref{parent free boundary navier-stokes system} allows the system to support equilibrium shear flow solutions.  These are special solutions of constant height and pressure with velocities that depend only on the vertical variable:  
\begin{equation}\label{intro_shear_form}
\zeta(x) = h \in \R^+,\;q(x,z) = \Bar{g} h, \text{ and }  v(x,z) = s(z) e_1 \text{ with } s(z) = -\tp{\kappa/2\mu} z^2 + \tp{\kappa h/\mu} z + \tp{\kappa h/\Bar{a}}.
\end{equation}
Note that this constitutes a one-parameter family of solutions indexed by the constant height variable, $h \in \R^+$.  However, if we assume that a given fluid configuration is subject to the relative flux condition~\eqref{intro_flux_cond} for a fixed constant $\Phi\in \R$, then we can significantly reduce the parameter space.  Indeed, plugging the shear flows from~\eqref{intro_shear_form} into~\eqref{intro_flux_cond} yields the identity 
\begin{equation}\label{intro_bore_equation}
    \Phi = \tp{\kappa/3\mu}h^3 + \tp{\kappa/\Bar{a}} h^2 - \Bar{\gamma} h, 
\end{equation}
and we learn from trivial algebra that there are at most three possible values of $h \in \R^+$ compatible with the constraint~\eqref{intro_bore_equation}.  In fact, we can say quite a bit more.  First, if $\Bar{\gamma} <0$ then there exists a solution with $h \in \R^+$ if and only if $\Phi >0$, and in this case the solution is uniquely determined by the flux.  Second, if $\Bar{\gamma} > 0$ and $\Phi \ge 0$ then again there is a unique solution $h \in \R^+$.  Third, if $\Bar{\gamma} >0$ there exists $\Phi_{\m{min}} <0$ such that if $\Phi \in (\Phi_{\m{min}},0)$ then there exist two distinct solutions $0 < h_- < h_+$, which then give rise to two distinct shear flows $(\zeta_\pm,v_\pm,q_\pm)$ determined by $h_\pm$ as in~\eqref{intro_shear_form}; see Figure \ref{fig_two_shears_like_this} for a graphical depiction.

\begin{figure}[!h]
		\centering
		\begin{subfigure}{0.32\textwidth}
			\centering
			\includegraphics[width=0.9\textwidth]{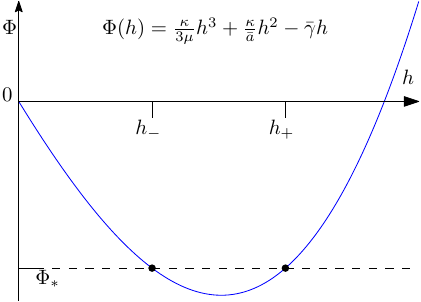}
		\end{subfigure}
		\begin{subfigure}{0.65\textwidth}
			\centering
			\includegraphics[width=0.9\textwidth]{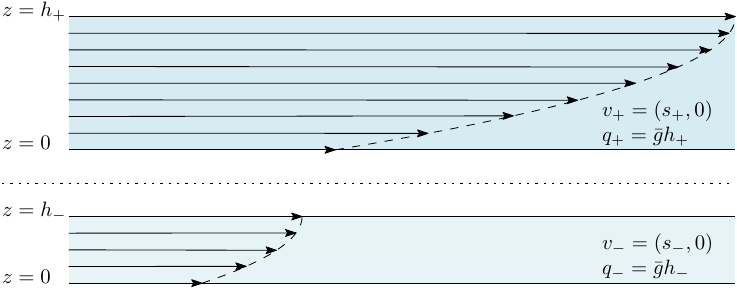}
		\end{subfigure}
		\caption{
	       On the left we have the plot of $\Phi(h)$ from~\eqref{intro_bore_equation} when $\Bar{\gamma} >0$, in which case choosing a negative flux constant, $\Phi_\ast  \in (\Phi_{\m{min}},0)$, allows for two positive solutions, $0 < h_- < h_+$, that give rise to distinct shear flow solutions $\zeta_\pm = h_\pm,$ $q_\pm = \Bar{g} h_\pm$, and $v_\pm = s_\pm e_1$.  On the right we have cartoons of these two shear flow solutions, with the $h_+$ solution above the dotted line and $h_-$ below.  The arrows plot the velocities with the dashed line indicating a plot of the shear profiles $s_\pm$.  Evidently, these have very different profiles but the same relative flux $\Phi_\ast$.   
		}
		\label{fig_two_shears_like_this}
\end{figure}

The above discussion raises an intriguing question: if the flux is chosen with $\Phi \in (\Phi_{\m{min}},0)$ do there exist solutions to~\eqref{parent free boundary navier-stokes system} that converge to one of the shear flow states $(\zeta_\pm,v_\pm,q_\pm)$ as $x \to -\infty$ and the other state as $x \to \infty$?  Our principal task in this paper is to answer this question in the affirmative.  We will define a solution of this form to be a  \emph{bore}, a term borrowed from the natural phenomenon of tidal bores, which are fluid configurations found in channels and rivers that behave in a similar manner, with different fluid heights upstream and downstream.  As we will see, bore solutions can also be thought of in more mathematical terms as traveling fronts or viscous shocks and rarefaction waves.  

A precise statement of our main result can be found in Section \ref{subsection on statement of the main theorems and discussion}, but an informal summary is as follows.  We show that bore solutions to~\eqref{parent free boundary navier-stokes system} exist, provided that we assume that the fluid is sufficiently shallow.  This shallow condition is quantified in terms of a dimensionless depth parameter $0 < \ep \ll 1$ and the various physical parameters appearing in~\eqref{parent free boundary navier-stokes system}, some of which we need to assume scale in particular ways with $\ep$.  Nevertheless, we prove that bores are ubiquitous within this shallow regime of physical parameter space. The solutions we construct fall into two broad qualitative categories, for which we continue to borrow tidal terminology: solutions that transition from $h_+$ at $-\infty$ to $h_-$ at $+\infty$, which we refer to as \emph{surging bores} (or viscous shocks); and solutions that transition from $h_-$ at $-\infty$ to $h_+$ at $+\infty$, which we refer to as \emph{ebbing bores} (or viscous rarefactions).  The eager reader may jump ahead to Figure~\ref{the 4 bores plots} to see numerical simulations of the surging and ebbing bores, though the plots will make more sense after some further notation is developed. All of the bores we construct are surging when the traveling wave speed $\Bar{\gamma}$ is sufficiently small and ebbing otherwise; see Corollary~\ref{coro on dimensional and eulerian bores} for a more precise statement. 

\subsection{Survey of previous work}\label{subsection on survey of previous work}

The mathematical study of traveling surface waves on free boundary fluids has spanned centuries, and as such, we cannot hope to give an extensive review here.  Instead, we content ourselves with considering only the results most closely related to those of this paper.  In particular, unless stated otherwise, we will restrict our attention to two-dimensional fluid configurations of a single layer.

The study of the inviscid, non-tilted analog of system~\eqref{parent free boundary navier-stokes system} (the traveling free boundary Euler equations) is also known as the traveling water wave problem. In this situation, the existence of various species of traveling waves was established long ago, starting with the work of Stokes~\cite{Stokes}. The first mathematical construction of two-dimensional periodic water waves is due to Nekrasov~\cite{nekrasov1921steady} and Levi-Civita~\cite{levi1924determinazione}. These are small amplitude and irrotational waves. The existence of large amplitude, irrotational, and periodic water waves of various types, including those that solve the Stokes conjecture, was later established by Krasovski\u{\i}~\cite{MR138284}, Keady and Norbury~\cite{MR502787}, Toland~\cite{MR513927}, Amick, Fraenkel, and Toland~\cite{MR666110}, Plotnikov~\cite{MR1883094}, McLeod~\cite{MR1446239}, and Plotnikov and Toland~\cite{MR2038344}.  The incorporation of the effects of vorticity into the periodic  water wave problem is a more recent development, beginning with the work of Constantin and Strauss~\cite{MR2027299}. Wahl\'en~\cite{MR2262949,MR2292728} added the effects of surface tension to the rotational problem, while Walsh~\cite{MR2529956,MR3177721,MR3177722} considered density stratification.

We now turn our attention to the case of solitary water waves. The first rigorous proof of small amplitude, solitary, and irrotational water waves is due to Lavrentiev~\cite{lavrentiev1947theory}, and this theory was then extended by Friedrichs and Hyers~\cite{MR65317},  Beale~\cite{MR445136}, Mielke~\cite{MR929221}, and Pego and Sun~\cite{MR3544325}.  Large amplitude solitary waves came with the work of Amick and Toland~\cite{MR647410,MR629699}. For generalizations with vorticity and other effects, we refer to the work of Hur~\cite{MR2385741}, Groves and Wahl\'en~\cite{MR2454604}, Wheeler~\cite{MR3106477}, Chen, Walsh, and Wheeler~\cite{MR3765551,MR3995048,chen2023globalbifurcationmonotonefronts}, Haziot and Wheeler~\cite{MR4565802}, and Ehrnstr\"om, Walsh, and Zeng~\cite{MR4577959}.  There have also been recent works that study the effects of wind or localized pressure sources in the classical water wave equations: see the papers of B\"uhler, Shatah, and Walsh~\cite{MR3073648}, B\"uhler, Shatah, Walsh, and Zeng~\cite{MR3544318}, and Wheeler~\cite{MR3375547}.

The proof of the \emph{nonexistence} of traveling bore wave solutions to the single-layer inviscid and irrotational water wave equations goes back to Rayleigh~\cite{rayleigh1914theory}. More modern proofs, which incorporate vorticity, can be found in the work of Wheeler~\cite{MR3375547} and Haziot and Wheeler~\cite{MR4565802}. Nevertheless, traveling inviscid internal fronts are possible within the context of a two-fluid system, as first proved in a small amplitude context by Amick and Turner~\cite{MR1011554}.  See also the work of  Mielke~\cite{MR1685879}, Makarenko~\cite{MR1229538}, and Chen, Walsh, and Wheeler~\cite{MR4399497}. The first construction of large bores for a two-fluid Euler systems is due to Chen, Walsh, and Wheeler~\cite{MR4196778}.

This review of the inviscid water wave problem has barely scratched the surface. We refer the reader looking for a deeper dive to the survey articles of Toland~\cite{Toland_1996}, Dias and Iooss~\cite{MR1984157}, Groves~\cite{Groves_2004}, Strauss~\cite{Strauss_2010}, and Haziot, Hur, Strauss, Toland, Wahl\'en, Walsh, and Wheeler~\cite{MR4406719}.

Juxtaposed with the extended history of inviscid water waves are the more recent developments on the traveling viscous surface wave problem. The objects of study in this realm are free boundary fluids in which the effects of viscosity are taken into account and are modeled by  Navier-Stokes systems similar to~\eqref{parent free boundary navier-stokes system}. Due to the dissipative nature of viscosity, one might expect that the generation of nontrivial traveling waves in viscous free boundary fluids necessitates a balancing source of power; to date, all proofs of the existence of such waves have required such a power source in the form of a spatially localized and moving source of force or stress acting on the fluid in addition to gravitational or capillary effects.  For example, this could take the form of a moving, localized pressure disturbance.  

Leoni and Tice~\cite{MR4630597} provided the first construction of traveling solutions to the free boundary incompressible Navier-Stokes equations.  This result was generalized by Stevenson and Tice~\cite{MR4337506,stevenson2023wellposedness} to incompressible fluids with multiple layers and to compressible free boundary flows, and by Koganemaru and Tice~\cite{MR4609068,MR4785303} to periodic or inclined fluids and flows obeying the Navier slip boundary condition. Nguyen and Tice~\cite{MR4690615} constructed traveling wave solutions to free boundary Darcy flow and, in the periodic case, performed a stability analysis. For technical reasons, these results all required a nontrivial traveling wave speed and did not address the stationary problem.  The existence of solitary stationary waves was then proved by Stevenson and Tice~\cite{MR4787851}, though only in three spatial dimensions.  All of the previously mentioned papers on stationary and traveling viscous surface waves produce small amplitude solutions.  Large amplitude traveling wave solutions have yet to be found for the free boundary Navier-Stokes equations; however, they were produced for Darcy flow by Nguyen~\cite{nguyen2023largetravelingcapillarygravitywaves} and Brownfield and Nguyen~\cite{MR4797733} and for the shallow water equations by Stevenson and Tice~\cite{stev_tice_big_waves}.

While the constructions mentioned in the previous paragraph are mathematical, solitary viscous traveling waves generated by moving, localized disturbances have also been observed in several experimental studies. We refer to the papers of Akylas, Cho, Diorio, and Duncan~\cite{CDAD_2011,DCDA_2011}, Duncan and Masnadi~\cite{MD_2017}, and Cho and Park~\cite{PC_2016,PC_2018} and the references therein for details. On the other hand, the experimental results of Russel~\cite{russell}, Cornish~\cite{cornish}, and Favre~\cite{favre} suggest that gravitational forces alone may suffice for sustaining viscous surface waves, which is more akin to the typical setup in the classical theory of inviscid surface waves.

To the best of the authors' knowledge, this work stands as the first construction of nontrivial traveling viscous surface waves driven solely by gravity as well as the first verification that viscous fluids can support single-layer traveling front solutions, which is in stark contrast with the inviscid case.

\subsection{Shallow nondimensionalization and flattening}\label{subsection on flattening, etc}

We now endeavor to quantify the shallowness of the fluid layer, nondimensionalize the fluid equations, and reformulate them in a form more convenient for analysis.  To begin, we introduce a small dimensionless parameter $0<\ep<1$ and make the following five ans\"atze, which are motivated by the theory of the shallow water equations and place~\eqref{parent free boundary navier-stokes system} into a certain $\ep$-shallow regime.  First, we posit that the height of the free surface $\zeta$ is proportional to $\ep$.  Second, the slip parameter is proportional to $\ep$: there exists $a>0$ for which $\Bar{a} = \ep a$.  Third, the surface tension is at most of order $\ep$: for some $\sig\ge 0$ we have $\Bar{\sig} = \ep\sig$.  Fourth, the normal component of gravity is of order at most $1/\ep$:  for some $g\ge 0$ we have $\Bar{g} = g/\ep$.  Fifth, we assume that the traveling frame speed decomposes according to $\Bar{\gam} = \gam + \ep^2\tilde{\gam}$ with $\gam>0$ and $\tilde{\gam}\in\R$;  here the parameter $\gam$ is arbitrary, whereas $\tilde{\gam}$ is a correction factor that will later be determined as a function of the other parameters.  To summarize the parameter scaling:
\begin{equation}
    \Bar{a} = \ep a,\; \Bar{\sig} = \ep \sigma,\; \Bar{g} = g/\ep,\;\text{and } \Bar{\gam} =\gam + \ep^2\tilde{\gam}.
\end{equation}

We now embark on the task of nondimensionalizing system~\eqref{parent free boundary navier-stokes system}. For reasons that will only become clear later, we select units with the effective normalization $\gam\mapsto 4$ and $\kappa/a\mapsto 4$.  Define the length and velocity scales $\m{L} = a\gam/\kappa$ and $\m{U} =\gam/4$; these determine the nondimensional free surface $\pmb{\zeta}:\R\to\R$, velocity $\pmb{v}:\Omega_{\ep\pmb{\zeta}}\to\R^2$, and pressure $\pmb{q}:\Omega_{\ep\pmb{\zeta}}\to\R$ via 
\begin{equation}\label{nondimensional unknowns}
    \zeta\tp{x} = \ep\m{L}\pmb{\zeta}\tp{x/\m{L}},\quad v(x,z) = \m{U}\pmb{v}\tp{x/\m{L},z/\m{L}}, \text{ and } q(x,z) = \m{U}^2\pmb{q}\tp{x/\m{L},z/\m{L}} \text{ for } \tp{x,z}\in\Omega_\zeta.
\end{equation}
We also introduce the following non-dimensional physical parameters:
\begin{equation}\label{nondimensional parameters}
    \upmu = 4\kappa\mu/a\gamma^2,\quad \m{a} = 4a/\gam,\quad \m{g} = 16ga/\kappa\gamma,\quad \upsigma = 16\sig/\gam^2, \text{ and } \Bar{\upgamma} =  4\tilde{\gam}/\gam.
\end{equation}
In terms of~\eqref{nondimensional unknowns} and~\eqref{nondimensional parameters}, the system~\eqref{parent free boundary navier-stokes system} is rewritten as
\begin{equation}\label{parent Navier-Stokes system nondimensionalized}
    \begin{cases}
        \tp{\pmb{v} - \tp{4 + \ep^2\Bar{\upgamma}} e_1}\cdot\grad\pmb{v} + \grad\pmb{q} - \upmu\Delta\pmb{v} = 4\m{a} e_1,\quad \grad\cdot\pmb{v} = 0&\text{in }\Omega_{\ep\pmb{\zeta}}\\
        -\tp{\pmb{q} - \upmu\mathbb{D}\pmb{v}}\mathcal{N}_{\ep\pmb{\zeta}} + \tp{\m{g}\pmb{\zeta} - \ep\upsigma\mathcal{H}\tp{\ep\pmb{\zeta}}}\mathcal{N}_{\ep\pmb{\zeta}} = 0,\quad \ep\tp{4 + \ep^2\Bar{\upgamma}}\pd_1\pmb{\zeta} + \pmb{v}\cdot\mathcal{N}_{\ep\pmb{\zeta}} = 0&\text{on }\Sigma_{\ep\pmb{\zeta}},\\
        \pmb{v}_2 = 0,\quad \upmu\pd_2\pmb{v}_1 = \ep\m{a}\pmb{v}_1&\text{on }\Sigma_0.
    \end{cases}
\end{equation}

The first difficulty encountered when analyzing~\eqref{parent Navier-Stokes system nondimensionalized} is that the domain $\Omega_{\ep\pmb{\zeta}}$ itself is one of the unknowns of the problem. This is inconvenient, and our task now is to reformulate system~\eqref{parent Navier-Stokes system nondimensionalized} in a fixed slab domain of height $\ep$. The tool that allows us to do so is the diffeomorphism $\mathfrak{F}_{\pmb{\zeta}}:\Bar{\Omega_{\ep}}\to\Bar{\Omega_{\ep\pmb{\zeta}}}$ defined  by $\mathfrak{F}_{\pmb{\zeta}}\tp{x,y} = \tp{x,y\pmb{\zeta}\tp{x}}$ for $\tp{x,y}\in\Bar{\Omega_\ep}$. Associated with $\mathfrak{F}_{\pmb{\zeta}}$ are the matrix quantities $\mathcal{A}_{\pmb{\zeta}},\mathcal{M}_{\pmb{\zeta}}:\Omega_\ep\to\R^{2\times 2}$ defined by
\begin{equation}\label{geometry matrices}
    \mathcal{A}_{\pmb{\zeta}}\tp{x,y} = \grad\mathfrak{F}_{\pmb{\zeta}}\tp{x,y}^{-\m{t}} = \bpm1&-y\pmb{\zeta}'\tp{x}/\pmb{\zeta}\tp{x}\\0&1/\pmb{\zeta}\tp{x}\epm\text{ and }\mathcal{M}_{\pmb{\zeta}}\tp{x,y} = \pmb{\zeta}\tp{x}\mathcal{A}_{\pmb{\zeta}}\tp{x,y}^{\m{t}} = \bpm\pmb{\zeta}\tp{x}&0\\-y\pmb{\zeta}'\tp{x}&1\epm,
\end{equation}
along with the flattened velocity $\pmb{u} = \pmb{v}\circ\mathfrak{F}_{\pmb{\zeta}}:\Omega_\ep\to\R^2$ and pressure $\pmb{p} = \pmb{q}\circ\mathfrak{F}_{\pmb{\zeta}}:\Omega_\ep\to\R$.

Next, we work out the analog of~\eqref{intro_flux_cond} in this flattened reformulation.  The flattened velocity $\pmb{u}$ satisfies $\grad\cdot\tp{\mathcal{M}_{\pmb{\zeta}}\pmb{u}}=0$, and upon integrating this  over  $y \in \tp{0,\ep}$ and substituting  the flattened kinematic boundary condition $\ep\tp{4 + \ep^2\Bar{\upgamma}}\pd_1\pmb{\zeta} + \pmb{u}\cdot\mathcal{M}_{\pmb{\zeta}}^{\m{t}}e_2=0$, we find that
\begin{equation}\label{this computation is kinda important}
    0 = \pd_1\bp{\pmb{\zeta}\f{1}{\ep}\int_0^{\ep}\pmb{u}_1\tp{\cdot,y}\;\m{d}y} + \f{1}{\ep}\mathcal{M}_{\pmb{\zeta}}\pmb{u}\cdot e_2 = \pd_1\bp{\pmb{\zeta}\bp{\f{1}{\ep}\int_0^{\ep}\pmb{u}_1\tp{\cdot,y}\;\m{d}y-\tp{4 + \ep^2\Bar{\upgamma}}}}.
\end{equation}
Therefore, we deduce the existence of a constant $\hat{\m{A}}\in\R$ such that the right hand argument of $\pd_1$ above is identically $-\hat{\m{A}}$, a sign convention more convenient for constructing bores.  We will think of $\hat{\m{A}}$ as another given physical parameter in the problem. We thus arrive at the equivalent flattened reformulation of system~\eqref{parent Navier-Stokes system nondimensionalized}:
\begin{equation}\label{flattened free boundary Navier-Stokes system}
    \begin{cases}
        \tp{\pmb{u} - \tp{4 + \ep^2\Bar{\upgamma}} e_1}\cdot\grad^{\mathcal{A}_{\pmb{\zeta}}}\pmb{u} + \grad^{\mathcal{A}_{\pmb{\zeta}}}\pmb{p} - \upmu\Delta^{\mathcal{A}_{\pmb{\zeta}}}\pmb{u} = 4\m{a} e_1,\quad \grad^{\mathcal{A}_{\pmb{\zeta}}}\cdot\pmb{u} = 0,&\text{in }\Omega_{\ep},\\
        -\tp{\pmb{p} - \upmu\mathbb{D}^{\mathcal{A}_{\pmb{\zeta}}}\pmb{u}}\mathcal{N}_{\ep\pmb{\zeta}} + \tp{\m{g}\pmb{\zeta} - \ep\upsigma\mathcal{H}\tp{\ep\pmb{\zeta}}}\mathcal{N}_{\ep\pmb{\zeta}} = 0&\text{on }\Sigma_{\ep},\\
        \pmb{u}_2=0,\quad \upmu\pd_2^{\mathcal{A}_{\pmb{\zeta}}}\pmb{u}_1 = \ep\m{a}\pmb{u}_1&\text{on }\Sigma_0,\\
        \sp{4 + \ep^2\Bar{\upgamma} - \f{1}{\ep}\int_0^\ep\pmb{u}_1\tp{\cdot,y}\;\m{d}y}\pmb{\zeta} = \hat{\m{A}}&\text{on }\Sigma_\ep,
    \end{cases}
\end{equation}
where $\grad^{\mathcal{A}_{\pmb{\zeta}}} = \mathcal{A}_{\pmb{\zeta}}\grad = \tp{\pd_1^{\mathcal{A}_{\pmb{\zeta}}},\pd_2^{\mathcal{A}_{\pmb{\zeta}}}}$, $\Delta^{\mathcal{A}_{\pmb{\zeta}}} = \grad^{\mathcal{A}_{\pmb{\zeta}}}\cdot\grad^{\mathcal{A}_{\pmb{\zeta}}}$, and $\mathbb{D}^{\mathcal{A}_{\pmb{\zeta}}}\pmb{u} = \grad\pmb{u}\mathcal{A}_{\pmb{\zeta}}^{\m{t}} + \mathcal{A}_{\pmb{\zeta}}\grad\pmb{u}^{\m{t}}$.  

\subsection{Bores, parameter tuning, and a final reformulation}\label{subsection on eq and bores}

As we have now reformulated the original equations~\eqref{parent free boundary navier-stokes system} into the form~\eqref{flattened free boundary Navier-Stokes system}, it is important to recompute the equilibrium shear solutions and give a mathematically precise definition of bores.  We do this now and then give one final technically convenient reformulation of the equations.

The shear solutions to ~\eqref{flattened free boundary Navier-Stokes system} are $\pmb{\zeta} = H \in\R^+$, $\pmb{p} = \m{g}H$,  and 
\begin{equation}
    \pmb{u}\tp{x,y} = b(y)e_1 \text{ for } b(y) = 4H + 4\tp{\m{a}/\upmu}H^2\tp{\ep y - y^2/2},
\end{equation}
and the relative flux identity~\eqref{intro_bore_equation} then becomes  
\begin{equation}\label{the end state consistency equation}
  \tp{4 + \ep^2\Bar{\upgamma}}H - 4H^2 - \ep^2\tp{4\m{a}/3\upmu}H^3 = \hat{\m{A}}.
\end{equation}
With these calculations in hand, we can now give a mathematically precise definition of bores as well as their surging and ebbing forms.

\begin{introdefn}[Bore waves]\label{defn of bore waves}
    Given $\ep\in\tp{0,1}$, $\upmu,\m{a}>0$, $\m{g},\upsigma\ge0$, $\Bar{\upgamma}\in\R$, and $\hat{\m{A}}>0$, we say that a tuple $\tp{\pmb{\zeta},\pmb{u},\pmb{p}}\in H^{5/2}_\loc\tp{\Sigma_\ep}\times H^2_\loc\tp{\Omega_\ep;\R^2}\times H^1_\loc\tp{\Omega_\ep}$ is a bore wave solution system~\eqref{flattened free boundary Navier-Stokes system} if the following are satisfied.
    \begin{enumerate}
        \item We have $\pmb{\zeta}(x)>0$ for all $x\in\R$ and $(\pmb{\zeta},\pmb{u},\pmb{p})$ is a strong solution to the free boundary Navier-Stokes system~\eqref{flattened free boundary Navier-Stokes system}.
        \item There exists $0 < \m{H}_- < \m{H}_+$ such that $\tp{4 + \ep^2\Bar{\upgamma}}\m{H}_\pm - 4\m{H}^2_\pm - \ep^2\f{4\m{a}}{3\upmu}\m{H}^3_\pm = \hat{\m{A}}$.
        \item There exists $\iota\in\tcb{-1,1}$ such that for every $y\in(0,\ep)$ we have that
        \begin{equation}\label{end states of the bore solutions}
            \lim_{x\to\pm\infty}\pmb{\zeta}(\iota x) = \m{H}_\pm,\quad\lim_{x\to\pm\infty}\pmb{u}_1\tp{\iota x,y} = 4\m{H}_\pm + \f{4\m{a}}{\upmu}\m{H}^2_\pm\bp{\ep y - \f12y^2},
            \text{ and } 
            \lim_{x\to\pm\infty}\pmb{u}_2\tp{\iota x,y}=0.
        \end{equation}
    \end{enumerate}
    If $\tp{\pmb{\zeta},\pmb{u},\pmb{p}}$ is a bore wave, then we say that it is surging if $\iota = -1$ and ebbing if $\iota = 1$.
\end{introdefn}

We show that both surging and ebbing bore solutions to system~\eqref{flattened free boundary Navier-Stokes system} exist.  The roots of the polynomial in~\eqref{the end state consistency equation} determine the asymptotic heights of the bore waves, and in general, these roots are nontrivially $\ep$-dependent.  However, for technical reasons in our construction it is crucial that the roots $\m{H}_\pm >0$ from the second item of the definition are actually independent of $\ep$; this can be achieved via the judicious selection of $\Bar{\upgamma}$ and $\hat{\m{A}}$, as we now show.  We first make the ansatz that $\hat{\m{A}} = \m{A} + \ep^2\Bar{\m{A}}$ with the parameter $\m{A}\in\tp{0,1}$ freely chosen but the parameters $\Bar{\m{A}},\Bar{\upgamma}\in\R$ determined in terms of the other physical parameters as follows.  When $\ep=0$, the cubic in~\eqref{the end state consistency equation} reduces to a quadratic equation with roots $\m{H}_{\pm} = \m{H}_{\pm}\tp{\m{A}}>0$ satisfying
\begin{equation}\label{sweq}
    \m{H}_{\pm} - \m{H}^2_{\pm} = \m{A}/4, \text{ and hence } \m{H}_{\pm} = \tp{1 \pm \tp{1 - \m{A}}^{1/2}}/2 .
\end{equation}
Observe that $0<\m{H}_-<1/2<\m{H}_+<1$ and that $\m{H}_+ + \m{H}_- = 1$. We then define $\Bar{\upgamma}$ and $\Bar{\m{A}}$ in such a way that both $\m{H}_+$ and $\m{H}_-$ are solutions to the cubic~\eqref{the end state consistency equation} for all $\ep>0$, which in light of~\eqref{sweq} is equivalent to  $\Bar{\upgamma}$ and $\Bar{\m{A}}$ solving 
\begin{equation}\label{the parameter tuning identities 440 hz}
    \Bar{\upgamma}\m{H}_\pm - \tp{4\m{a}/3\upmu}\m{H}_\pm^3 = \Bar{\m{A}}\text{ and hence }\Bar{\upgamma} = \tp{4\m{a}/3\upmu}\tp{\m{H}_+^2 + \m{H}_+\m{H}_- + \m{H}^2_-}\text{ and }\Bar{\m{A}} = \tp{4\m{a}/3\upmu}\tp{\m{H}_+^2\m{H}_- + \m{H}_+\m{H}_-^2}.
\end{equation}
We refer to the above as the parameter tuning identities; they play an essential role in our analysis. Notice that $\m{H}_{\pm}$ are only functions of the parameter $\m{A}\in\tp{0,1}$ (thanks to our choice of units in~\eqref{nondimensional unknowns}). As $\m{A}\to 0$, we have $\tp{\m{H}_-,\m{H}_+}\to\tp{0,1}$ whereas when $\m{A}\to 1$ we have $\tp{\m{H}_-,\m{H}_+}\to\tp{1/2,1/2}$. Therefore, $\m{A}$ is a measurement of the separation of the asymptotic bore heights, and we will occasionally refer to $\m{A}\sim 1$ as the small bore regime and $\m{A}\sim 0$ as the large bore regime.

We conclude this subsection by making one further reformulation of the Navier-Stokes system~\eqref{flattened free boundary Navier-Stokes system}. The current form of the final equation is somewhat unacceptable in that the  regularity count of the free surface function relative to that of the velocity is suboptimal. This can be fixed by employing the following trick. The final equation is written
\begin{equation}
    \pmb{d} = \m{A} + \ep^2\Bar{\m{A}} 
    \text{ for }
    \pmb{d} = \bp{4 + \ep^2\Bar{\upgamma} - \f{1}{\ep}\int_0^\ep\pmb{u}_1\tp{\cdot,y}\;\m{d}y}\pmb{\zeta}.
\end{equation}
In our functional framework the quantity $\pmb{d}$ will always be a tempered distribution (in fact, a bounded function), so the above equation is equivalent to $\pmb{d} - \pd_1^2\pmb{d} = \m{A} + \ep^2\Bar{\m{A}}$. A computation similar to~\eqref{this computation is kinda important} shows that 
\begin{equation}
    \pd_1\pmb{d} = \tp{4 + \ep^2\Bar{\upgamma}}\pd_1\pmb{\zeta} + \tp{\pmb{u}\cdot\mathcal{N}_{\ep\pmb{\zeta}}}/\ep,
\end{equation}
and hence $\pmb{d} - \pd_1^2\pmb{d} = \m{A} + \ep^2\Bar{\m{A}}$ is equivalent to the final identity in our next, and final, reformulation of the free boundary Navier-Stokes system:
\begin{equation}\label{FBINSE, param. tuned and funny flux}
    \begin{cases}
        \tp{\pmb{u} - \tp{4 + \ep^2\Bar{\upgamma}}e_1}\cdot\grad^{\mathcal{A}_{\pmb{\zeta}}}\pmb{u} + \grad^{\mathcal{A}_{\pmb{\zeta}}}\pmb{p} - \upmu\Delta^{\mathcal{A}_{\pmb{\zeta}}}\pmb{u} = 4\m{a} e_1,\quad \grad^{\mathcal{A}_{\pmb{\zeta}}}\cdot\pmb{u} = 0&\text{in }\Omega_{\ep},\\
        -\tp{\pmb{p} - \upmu\mathbb{D}^{\mathcal{A}_{\pmb{\zeta}}}\pmb{u}}\mathcal{N}_{\ep\pmb{\zeta}} + \tp{\m{g}\pmb{\zeta} - \ep\upsigma\mathcal{H}\tp{\ep\pmb{\zeta}}}\mathcal{N}_{\ep\pmb{\zeta}} = 0&\text{on }\Sigma_{\ep},\\
        \pmb{u}_2=0,\quad \upmu\pd_2^{\mathcal{A}_{\pmb{\zeta}}}\pmb{u}_1 = \ep\m{a}\pmb{u}_1&\text{on }\Sigma_0,\\
        \sp{4+ \ep^2\Bar{\upgamma} - \f{1}{\ep}\int_0^\ep\pmb{u}_1\tp{\cdot,y}\;\m{d}y}\pmb{\zeta} - \tp{4 + \ep^2\Bar{\upgamma}}\pd_1^2\pmb{\zeta} - \f{1}{\ep}\pd_1\tp{\pmb{u}\cdot\mathcal{N}_{\ep\pmb{\zeta}}} = \m{A} + \ep^2\Bar{\m{A}}&\text{on }\Sigma_\ep.
    \end{cases}
\end{equation}

\subsection{Statement of the main results and discussion}\label{subsection on statement of the main theorems and discussion}

In terms of the nondimensional parameters~\eqref{nondimensional parameters}, there are no restrictions on $\upmu,\m{a}>0$ or $\upsigma\ge0$ for our bore construction, but we must constrain the parameters $\m{g}\ge 0$ and $\m{A}\in(0,1)$ to belong to certain sets, which we need to describe now in order to precisely state our main results.  Define the function $\uprho_-:(0,1)\to\R$ via $\uprho_-\tp{\m{A}}=\log\m{H}_-\tp{\m{A}}$, where $\m{H}_-$ is given in~\eqref{sweq}. We then define the function $\uprho_\star:(0,1)\to\R$ via the condition that $\uprho_\star\tp{\m{A}}>\uprho_-\tp{\m{A}}$ is the unique number such that
\begin{equation}\label{equi-energy point}
    \int_{\uprho_-(\m{A})}^{\uprho_{\star}\tp{\m{A}}}\bp{1 - e^{x}-\f{\m{A}}{4}e^{-x}}\;\m{d}x=0.
\end{equation}
The rigorous construction of $\uprho_\star$ is handled in Lemma~\ref{lem on preliminary facts}. Using the functions $\uprho_-$ and $\uprho_\star$, we define the sets
\begin{equation}\label{regions of left and right handed bore waves}
    \mathfrak{C}_1 = \tcb{\tp{\m{g},\m{A}}\in[0,\infty)\times\tp{0,1}\;:\;\m{g}<\m{A}^2e^{-3\uprho_\star\tp{\m{A}}}} 
    \text{ and }
    \mathfrak{C}_{-1}=\tcb{\tp{\m{g},\m{A}}\in[0,\infty)\times\tp{0,1}\;:\;\m{g}>\m{A}^2e^{-3\uprho_{-}\tp{\m{A}}}},
\end{equation}
each of which is nonempty (see Lemma~\ref{lem  on sufficient conditions for a sign} for further details).   A graphical depiction may be found in Figure~\ref{the left and right handed bore regions}.

\begin{figure}[!h]
    \centering
    \scalebox{0.16}{\includegraphics{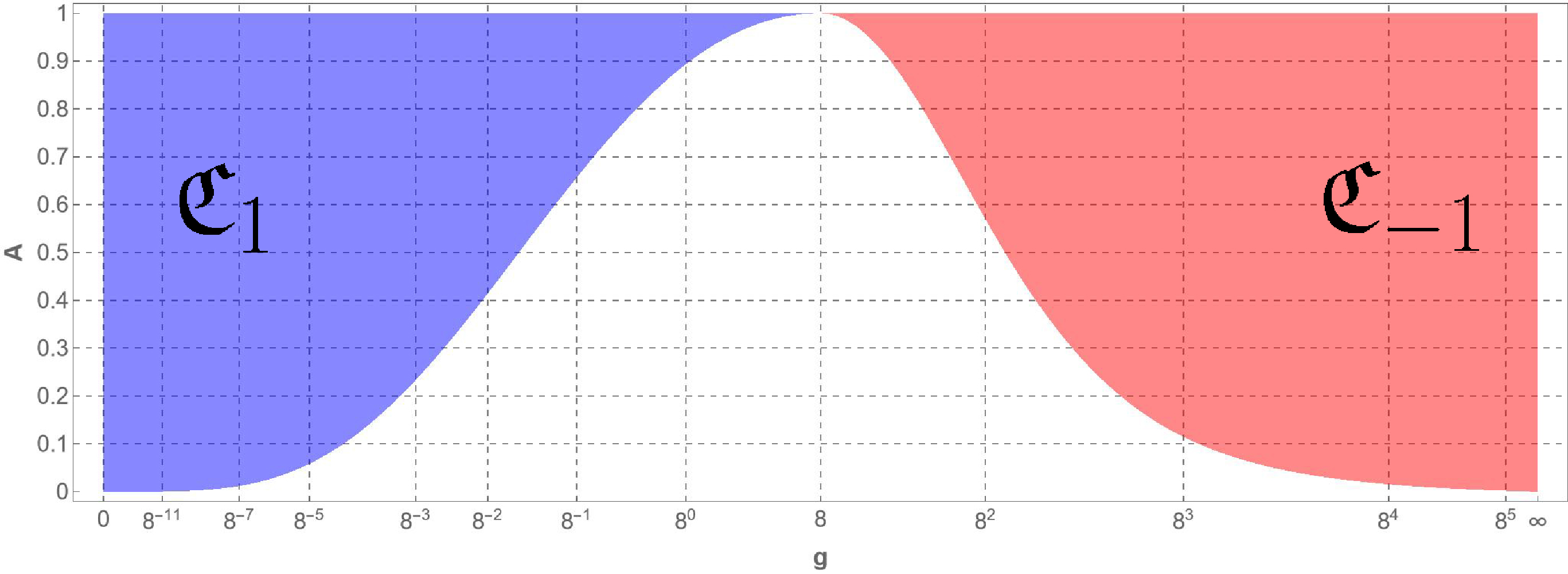}}
    \caption{
    The blue and red shaded regions in $\tp{\m{g},\m{A}}$-space depicted above, labeled $\mathfrak{C}_{\pm1}$ and defined in~\eqref{regions of left and right handed bore waves}, correspond to the parameter values for which we are able to produce traveling bore waves in Theorem~\ref{first main theorem}. The waves coming from the blue $\mathfrak{C}_1$ region are ebbing bores, whereas the waves coming from the red $\mathfrak{C}_{-1}$ region are surging bores.  The parameter $\m{A}\in\tp{0,1}$ on the vertical axis measures the separation of the equilibrium heights, with $\m{A}$ close to $1$ being the `small bore' regime and $\m{A}$ close to $0$ being the `large bore' regime. Notice that the only value of $\m{g}$ not covered by these regions is $\m{g} = 8$.  }
    \label{the left and right handed bore regions}
\end{figure}

We can now state our main theorem, the proof of which is found in Section~\ref{subsection on further properties}.  In what follows, the space $W^{\infty,\infty}$ consists of smooth functions with all derivatives bounded, while $H^\infty\subset W^{\infty,\infty}$ enforces that the function itself and all of its derivatives are square summable. See Section~\ref{subsection on notation conventions} for the precise conventions of notation.
\begin{introthm}[Proved in Section~\ref{subsection on further properties}]\label{first main theorem}
    Fix $\upmu,\m{a}>0$, $\iota\in\tcb{-1,1}$, $\tp{\m{g},\m{A}}\in\mathfrak{C}_\iota$, and $\upsigma\ge0$. There exist  $\mathsf{H}\in W^{\infty,\infty}\tp{\R}$, $C,\del\in\R^+$, and $\epsilon_\star\in\tp{0,1}$ with the property that for all $\ep\in\tp{0,\epsilon_\star}$ there exists $\Bar{\eta}\in H^{\infty}\tp{\Sigma_\ep}$, $\Bar{u}\in H^\infty\tp{\Omega_\ep;\R^2}$, and $\Bar{p}\in H^\infty\tp{\Omega_\ep}$ such that the following hold.
    \begin{enumerate}
        \item We have the estimate $\tnorm{\Bar{\eta}}_{H^{5/2}\tp{\Sigma_\ep}} + \tnorm{\Bar{u}}_{\tp{L^\infty\cap H^2}\tp{\Omega_\ep}} + \tnorm{\Bar{p}}_{H^1\tp{\Omega_\ep}}\le C\ep$.
        \item The function $\mathsf{H}$ is everywhere positive, solves the ODE
        \begin{equation}\label{shallow water just H version}
            4\upmu\m{A}\bp{\frac{\mathsf{H}'}{\mathsf{H}}}' = 4\m{a}\bp{1 - \mathsf{H} - \f{\m{A}}{4\mathsf{H}}} - \bp{\f{\m{A}^2}{\mathsf{H}^2} - \m{g}\mathsf{H}}\mathsf{H}' \text{ in }\R, 
        \end{equation}
        and obeys the limits
        \begin{equation}\label{the following limits are satisfied}
            \lim_{x\to\pm\infty}\mathsf{H}\tp{\iota x}  = \m{H}_{\pm},
        \end{equation}
        where $\m{H}_\pm$ are defined in~\eqref{sweq}.  Moreover, $\mathsf{H}' \in H^\infty(\R)$.
        \item Upon defining $\pmb{\zeta}\in  W^{\infty,\infty}\tp{\Sigma_\ep}$, $\pmb{u}\in W^{\infty,\infty}\tp{\Omega_\ep;\R^2}$, and $\pmb{p}\in W^{\infty,\infty}\tp{\Omega_\ep}$ via
        \begin{equation}\label{isolation of just the distinguished part of the solution}
            \pmb{\zeta} = \mathsf{H} + \Bar{\eta},\quad\pmb{u} = \bp{4\mathsf{H} + \f{4\m{a}}{\upmu}\mathsf{H}^2\bp{\ep y - \f12 y^2}}e_1 + \bp{4 - \f{\m{A}}{\mathsf{H}} - 4\mathsf{H}}e_1 + \Bar{u},
            \text{ and }
            \pmb{p} = \m{g}\mathsf{H} - 2\upmu\m{A}\f{\mathsf{H}'}{\mathsf{H}^2} + \Bar{p},
        \end{equation}
        the triple $\tp{\pmb{\zeta},\pmb{u},\pmb{p}}$ is a classical bore wave solution to system~\eqref{flattened free boundary Navier-Stokes system} in the sense of Definition~\ref{defn of bore waves}  with $\hat{\m{A}} = \m{A} + \ep^2\Bar{\m{A}}$ and $\Bar{\upgamma}$, $\Bar{\m{A}}$  determined by the parameter tuning identities~\eqref{the parameter tuning identities 440 hz}. Moreover, $\inf\tp{4 - \pmb{u}_1}\ge\del$.

        \item If $\iota = -1$ then $\tp{\pmb{\zeta},\pmb{u},\pmb{p}}$ is a surging bore, and if $\iota = 1$ then $\tp{\pmb{\zeta},\pmb{u},\pmb{p}}$ is an ebbing bore.
    \end{enumerate}
\end{introthm}

Several remarks about the theorem are in order.  First, we refer to Figure~\ref{the 4 bores plots} for numerical simulations of the bore wave solutions produced by our main result. These highlight the striking differences between surging and ebbing bores and also demonstrate that each appears to come in both monotone and oscillatory forms. Second, while the set $\mathfrak{C}_{1}$ is bounded, the set $\mathfrak{C}_{-1}$ is unbounded and of infinite measure, which gives a sense in which surging bores are more common than their ebbing counterparts. Third, we are only able to guarantee that bore waves exist with $\tp{\m{g},\m{A}} \in \mathfrak{C}_{-1} \cup \mathfrak{C}_{1}$, and we make no claim as  to whether or not bores, or other types of traveling waves for that matter, exist for parameter values outside this set; our results simply do not address this region (the white part of Figure \ref{the left and right handed bore regions}). Fourth, the entirety of the result applies in the case $\upsigma =0$, the regime in which surface tension is entirely neglected. In particular, this confirms that, indeed, gravitational forces alone can generate bores.  

\begin{figure}[!h]
    \centering
    \scalebox{0.157}{\includegraphics{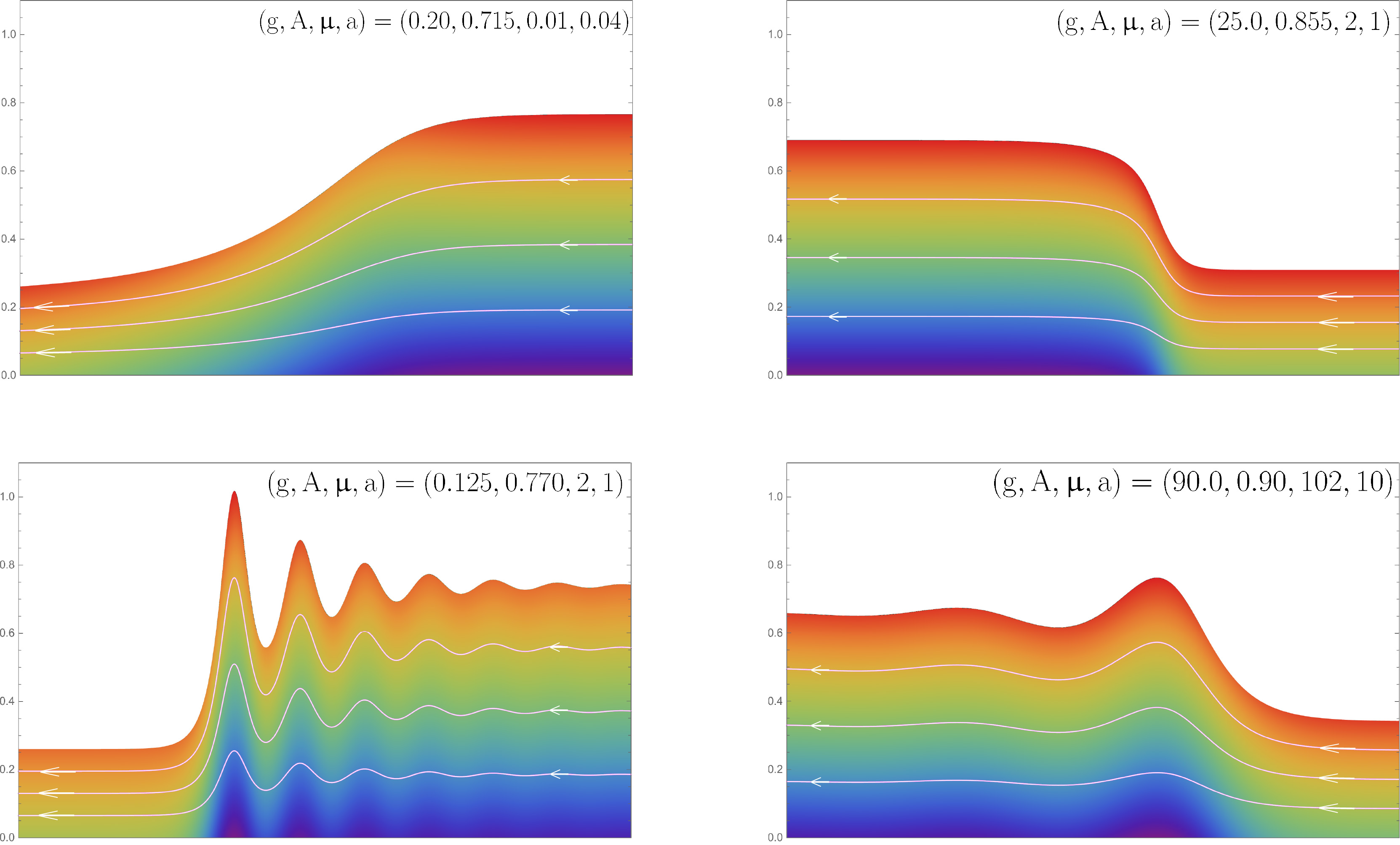}}
    \caption{Here we present four numerical simulations of the graphs of $\mathsf{H}$, the leading order parts of the free surface components corresponding to our rightward traveling bores constructed by Theorem~\ref{first main theorem} (with $\upsigma = 0$). Aspects of the fluid velocity are depicted in the subgraph. The streamlines shown are that of the flow seen by an observer \emph{moving with the wave}. The color indicates the variation in the curl, which is negative throughout, with violet and red indicating the regions of largest and smallest magnitudes, respectively. The top row has solutions that appear to be monotone, while the bottom row has solutions that appear oscillatory. On the left we show ebbing bores, corresponding to $\tp{\m{g},\m{A}}\in\mathfrak{C}_1$, whereas on the right we show surging bores, corresponding to $\tp{\m{g},\m{A}}\in\mathfrak{C}_{-1}$. The sets $\mathfrak{C}_{\pm 1}$ are defined in~\eqref{regions of left and right handed bore waves} and sketched in Figure~\ref{the left and right handed bore regions}.}
    \label{the 4 bores plots}
\end{figure}

The following  corollaries unpack specific consequences of our main result that are worth emphasizing.  To contextualize the first we recall that the parameter tuning~\eqref{the parameter tuning identities 440 hz} forces the $\ep$-independent end state heights $\m{H}_{\pm}$ (given by~\eqref{sweq}) of any bore wave to satisfy $0<\m{H}_-<\m{H}_+<1$, but these states are very close for large portions of the parameter space $\mathfrak{C}_{-1} \cup \mathfrak{C}_{1}$.  Our first corollary underlines the fact that Theorem \ref{first main theorem} yields bores for which $\m{H}_+-\m{H}_-$ is arbitrarily close to $1$ and the jumps in the velocity and pressure are also of unit order.
\begin{introcoro}[Existence of large shallow bore waves]\label{coro on the existence of large bores}
     Let $0<\al<1$, $\iota\in\tcb{-1,1}$, $\upmu,\m{a}>0$, and $\upsigma\ge0$. There exists $\tp{\m{g},\m{A}}\in\mathfrak{C}_\iota$ and a smooth traveling bore wave solution $\tp{\pmb{\zeta},\pmb{u},\pmb{\eta}}$ to~\eqref{flattened free boundary Navier-Stokes system} produced by Theorem~\ref{first main theorem} with end state heights satisfying $\lim_{x\to\pm\infty}\pmb{\zeta}\tp{\iota x} = \m{H}_\pm$ with $\m{H}_+-\m{H}_{-}>1-\al$, and end state velocity and pressure satisfying  
        \begin{equation}\label{the dimensionless hydraulic jumps}
            \iota\cdot\lim_{x\to\infty}\tsb{\pmb{u}(x,y) - \pmb{u}(-x,y)} = 4\tp{\m{H}_+ - \m{H}_-}\bp{1 + \f{\m{a}}{\upmu}\bp{\ep y - \f12y^2}}e_1 
            \text{ and }
            \iota\cdot\lim_{x\to\infty}\tsb{\pmb{p}\tp{x,y} - \pmb{p}\tp{-x,y}} = \m{g}\tp{\m{H}_+ - \m{H}_-},
        \end{equation}
        for $0<y<\ep$.
\end{introcoro}

Our second corollary emphasizes that the bore waves produced by Theorem~\ref{first main theorem} become increasingly well-approximated by solutions to one-dimensional shallow water equations as $\ep \to 0$.

\begin{introcoro}[The shallow water limit for bore waves]\label{coro on justification of the shallow water limit}
    Under the hypotheses of Theorem~\ref{first main theorem}, let $0<\ep<\epsilon_\star$ and 
    \begin{equation}
        \mathsf{H}\in W^{\infty,\infty}\tp{\R},\quad\pmb{\zeta}\in W^{\infty,\infty}\tp{\Sigma_\ep},
        \text{ and }
        \pmb{u}\in W^{\infty,\infty}\tp{\Omega_\ep;\R^2}
    \end{equation}
    be components of any bore wave solution to system~\eqref{flattened free boundary Navier-Stokes system} produced by Theorem~\ref{first main theorem}. Define $\mathsf{U} = 4 - \m{A}/\mathsf{H}\in W^{\infty,\infty}\tp{\R}$. Then we have the estimates
    \begin{equation}\label{justification of the shallow water limit}
        \tnorm{\pmb{\zeta} - \mathsf{H}}_{L^\infty\tp{\Sigma_\ep}}+\tnorm{\pmb{u} - \mathsf{U}e_1}_{L^\infty\tp{\Omega_\ep}}\lesssim \ep
    \end{equation}
    for an implicit constant independent of $\ep$. Moreover, $\mathsf{H}$ and $\mathsf{U}$ are a traveling wave solution to the one-dimensional inclined viscous shallow water equations with laminar drag:
    \begin{equation}\label{the one dimensional shallow water equations UH form}
        \tp{\tp{4 - \mathsf{U}}\mathsf{H}}' = 0,\quad \mathsf{H}\tp{\mathsf{U} - 4}\mathsf{U}' + \m{a}\mathsf{U} - 4\upmu\tp{\mathsf{H}\mathsf{U}'}' + \m{g}\mathsf{H}\mathsf{H}' = 4\m{a}\mathsf{H}.
    \end{equation}
\end{introcoro}

The final corollary demonstrates that we can undo the flattening and nondimensionalization procedure of Sections~\ref{subsection on flattening, etc} and~\ref{subsection on eq and bores} to obtain bore wave solutions to the Eulerian traveling Navier-Stokes system~\eqref{parent free boundary navier-stokes system}.

\begin{introcoro}[Dimensional, Eulerian bores]\label{coro on dimensional and eulerian bores}
    Fix the following dimensional parameters: viscosity $\mu>0$, horizontal gravitational acceleration $\kappa>0$, slip parameter $a>0$,  vertical gravitational acceleration $g\ge0$, and surface tension $\sigma\ge0$.  Let $0<\gam<\infty$ be a traveling frame speed  such that $\gam\neq2ga/\kappa$, and let $\iota\in\tcb{-1,1}$ be such that $\iota\tp{\gam - 2ga/\kappa}>0$.

    The set of $\m{A}\in\tp{0,1}$ such that $\tp{16ga/\kappa\gam,\m{A}}\in\mathfrak{C}_{\iota}$ is nonempty; for each such $\m{A}$ there exists $\epsilon_\star\in\tp{0,1}$ with the property that for all $\ep\in\tp{0,\epsilon_\star}$ there exists $\zeta\in W^{\infty,\infty}\tp{\R}$ with $\zeta>0$, $v\in W^{\infty,\infty}\tp{\Omega_\zeta;\R^2}$, and $q\in W^{\infty,\infty}\tp{\Omega_\zeta}$ such that upon setting
        \begin{equation}\label{dimensional parameters}
            \Bar{\gamma} = \gam + \ep^2\f{a^2\gam^2}{3\kappa\mu}\tp{\m{H}_+\tp{\m{A}}^2 + \m{H}_+\tp{\m{A}}\m{H}_-\tp{\m{A}} + \m{H}_-\tp{\m{A}}^2},\quad \Bar{g} = g/\ep,\quad \Bar{\sig} = \ep\sig, \text{ and } \Bar{a} = \ep a
        \end{equation}
        we have that $\tp{\zeta,v,q}$ is a smooth bore wave solution to the free boundary Navier-Stokes system~\eqref{parent free boundary navier-stokes system}  obeying the limits
        \begin{equation}\label{the dimensional limits}
            \lim_{x\to\pm\infty}\zeta\tp{\iota x} = \ep\f{a\gam}{\kappa}\m{H}_\pm\tp{\m{A}}=\zeta_\pm,\quad\lim_{\substack{x\to\pm\infty\\0<z<\ep\tp{a\gam/\kappa}\m{H}_\pm\tp{\m{A}}}}q\tp{\iota x,z} = \f{ga\gam}{\kappa}\m{H}_\pm\tp{\m{A}} = \Bar{g}\zeta_{\pm},
        \end{equation}
        and
        \begin{equation}\label{the dimensional limits 2}
            \lim_{\substack{x\to\pm\infty\\0<z<\ep \tp{a\gam/\kappa}\m{H}_\pm\tp{\m{A}}}}v(\iota x,z) = \bp{\gam\m{H}_\pm\tp{\m{A}} + \ep\f{a\gam}{\mu}\m{H}_{\pm}\tp{\m{A}}z - \f{\kappa}{2\mu}z^2}e_1 = \bp{\f{\kappa\zeta_\pm}{\Bar{a}} + \f{\kappa\zeta_{\pm}}{\mu}z - \f{\kappa}{2\mu}z^2}e_1.
        \end{equation}
        In particular, if  $\iota=-1$ then $0<\gam<2ga/\kappa$ and the solution is a surging bore, while if $\iota=1$ then $\gam>2ga/\kappa$ and the solution is an ebbing bore.    Moreover, the relative velocity flux $\Phi<0$ (see~\eqref{intro_flux_cond} and~\eqref{intro_bore_equation}) of these solutions takes the form
        \begin{equation}\label{full relative velocity flux of the bore wave solutions}
            \Phi = -\ep\f{a\gam^2}{4\kappa}\m{A} - \ep^3\f{a^3\gam^3}{3\kappa^2\mu}\tp{\m{H}_+\tp{\m{A}}^2\m{H}_-\tp{\m{A}} + \m{H}_+\tp{\m{A}}\m{H}_{-}\tp{\m{A}}^2}.
        \end{equation}
\end{introcoro}

We emphasize that Corollary~\ref{coro on dimensional and eulerian bores} establishes our previous ubiquity claim for shallow bores.  Indeed,  for any given set of fixed physical parameters $\mu,\kappa,a>0$ and $g,\sig\ge0$ the system~\eqref{parent free boundary navier-stokes system} is guaranteed to admit traveling bore wave solutions for every leading order choice of wave speed $\gam>0$, excepting the sole choice $\gam=2ga/\kappa$, and every $\ep>0$ sufficiently small. An important consequence of the corollary is that we produce surging bores if and only if $\gamma \kappa / 2g a < 1$, and this dimensionless quantity can be understood as a sort of averaged Froude number.  Indeed, if we average the end state heights and bottom speeds, we get $\zeta_{\m{avg}} = \Bar{a}\gam/2\kappa$ and $v_{\m{avg}} = \gam/2$, so if we define a Froude number, $\m{Fr}$, from these we obtain
\begin{equation}\label{Sigmund Freud}
    \m{Fr}^2 = \f{v_{\m{avg}}^2}{\Bar{g}\zeta_{\m{avg}}} = \f{\gam\kappa}{2ga} = \f{8}{\m{g}}.
\end{equation}
Thus, our surging bores all satisfy the subcritical Froude number condition $\m{Fr} <1$, while the ebbing bores satisfy the supercritical condition $\m{Fr} >1$, and we fail to construct any bores at the critical value $\m{Fr}=1$.

Corollary~\ref{coro on dimensional and eulerian bores} also shows that, while the bore waves we construct are small amplitude in physical units (as seen by the first limit in~\eqref{the dimensional limits}), the pressure and velocity differentials are made large in physical units when the parameter $\m{A}\in(0,1)$ is near zero.  In the dimensional form of Corollary~\ref{coro on dimensional and eulerian bores} it is also easier than in Theorem \ref{first main theorem} to single out the most important physical parameters (and thus effects).  Viscosity, slip, and horizontal gravity are essential to our result (i.e. $\mu,a,\kappa >0$), as it is the `competition' between these three that leads to bore generation.  On the other hand, capillarity plays no essential role and may be neglected $(\sigma =0)$, and while vertical gravity can also be neglected $(g=0)$ this results only in ebbing bores.

Before turning to the discussion of the difficulties and techniques used to prove Theorem \ref{first main theorem}, we pause to highlight several related lines of inquiry that are beyond the scope of this work but that we believe are important and interesting.  First, is it necessary to have $\kappa >0$ in order to produce bores, or can they somehow be sustained purely with vertical gravity?  Second, do bore-like solutions exist with $\kappa>0$ but $\Bar{\gamma} < 0$,  i.e. waves traveling upstream?  Third, what happens in other scaling regimes when capillarity is no longer perturbative?

We now turn to a discussion of our techniques for constructing bore waves, which will also serve as an outline of the paper's structure. As is clear in the statement of Theorem~\ref{first main theorem}, our strategy hinges on the assumption that the fluid is shallow, as this allows us to make connections to the viscous shallow water equations (sometimes called the viscous Saint-Venant equations). This is a system of simplified equations and variables that emerge as the effective limit of the Navier-Stokes system~\eqref{parent free boundary navier-stokes system} (and its dynamic form, as well) in both the vanishing depth and long wavelength limit. Taking this limit necessitates introducing the parameter $\ep >0$ and making a number of assumptions about the $\ep$-scaling of various parameters and unknowns. Our selection of assumptions in this paper is made as in~\eqref{parent Navier-Stokes system nondimensionalized} to ensure that viscous, slip, and horizontal gravitational effects survive in the limit, but there are, in fact, many other possibilities.  Indeed, the literature on the shallow water and related equations is varied and vast; we refer, for instance, to the  works~\cite{MR1324142,RevModPhys.69.931,MR1867089,MR1821555,MR1975092,MR2118849,MR2397996,MR2281291,MR2562163,mascia_2010,MR4105349,stevenson2023shallow,MR4873829} and the references therein for further exploration.

The vanishing depth limit of~\eqref{parent free boundary navier-stokes system} is known (see, for instance, Appendix B in Stevenson and Tice~\cite{stev_tice_big_waves}) to be described by two simple unknowns,  $(\mathsf{H}, \mathsf{U}) : \R \to \R^+ \times \R$, the effective fluid height and horizontal velocity, respectively. These satisfy the (nondimensionalized) traveling wave viscous shallow water system ~\eqref{the one dimensional shallow water equations UH form}
in which the Navier-Stokes forcing terms manifest as $4\upmu\tp{\mathsf{H}\mathsf{U}'}'$ for shear viscosity, $\m{a}\mathsf{U}$ for slip,  $\m{g}\mathsf{H}\mathsf{H}'$ for vertical gravity, and $4\m{a}\mathsf{H}$ for horizontal gravity.   Mirroring what we did for Navier-Stokes, we integrate the first equation in~\eqref{the one dimensional shallow water equations UH form} and introduce a shallow water flux parameter: $\tp{4 - \mathsf{U}}\mathsf{H} = \m{A} \in \R$.  This allows us to solve for $\mathsf{U}$ in terms of $\mathsf{H}$ and thereby reduce to a second order ODE  for $\mathsf{H}$ alone, namely~\eqref{shallow water just H version}.  The appearance of the quantity $\tp{\mathsf{H}'/\mathsf{H}}'$ in~\eqref{shallow water just H version} and the importance of $\mathsf{H}>0$ both suggest the swap to the new unknown $\rho:\R\to\R$ determined via $e^\rho = \mathsf{H}$.  This yields a reformulation of~\eqref{shallow water just H version} in terms of $\rho$ as the Li\'enard-type equation
\begin{equation}\label{The shallow water ODEs v3}
    \rho'' = F(\rho) - G(\rho)\rho',
\end{equation}
where $F,G:\R\to\R$ are the smooth functions given by
\begin{equation}\label{the hamiltonian and dissipation}
    F(x) = \f{\m{a}}{\upmu\m{A}}\bp{1 - e^{x} - \f{\m{A}}{4}e^{-x}} 
    \text{ and }
    G(x) = \f{1}{4\upmu}\bp{\m{A}e^{-x} - \f{\m{g}}{\m{A}}e^{2x}}.
\end{equation}

There are two equilibrium states for~\eqref{The shallow water ODEs v3}, corresponding to the roots of $F$ and given by $\uprho_\pm = \log\m{H}_\pm$ with $\m{H}_\pm\in\tp{0,1}$ as defined in~\eqref{sweq}.  This fact reveals a key point about how we have selected the parameter tuning in~\eqref{the parameter tuning identities 440 hz}: this has been done specifically to make the positive roots of the cubic relative flux equation~\eqref{the end state consistency equation} match up with the equilibrium states of $F$ as above, and to allow the parameter $\m{A}$ to serve as a relative flux measurement in both the Navier-Stokes and shallow water systems.  In turn, this provides a tantalizing first indication of the existence of bore waves since the ODE~\eqref{The shallow water ODEs v3} can in principle support heteroclinic orbits between $\uprho_\pm$, which upon swapping back to the $\mathsf{H}$ and $ \mathsf{U}$ variables would yield a shallow water approximation of a bore wave.  

To pursue this line of reasoning, we need more refined information about the equation~\eqref{The shallow water ODEs v3}, and in particular the functions $F$ and $G$.  The function $G$ is strictly decreasing and converges to $\mp \infty$ as $x \to \pm \infty$, while  $F$ is strictly concave and satisfies $F(x)>0$ if and only if $\uprho_-<x<\uprho_+$.  The structure of~\eqref{The shallow water ODEs v3} suggests introducing the potential for $F$, the smooth function $V : \R \to \R$,  which we will  zero at the $\uprho_-$ state:
\begin{equation}\label{the potential energy of the system}
  V(x) = -\int_{\uprho_-}^xF(s)\;\m{d}s.
\end{equation}
Then $V$ has exactly two critical points, namely $\uprho_\pm$, but there is a third important state, $\uprho_\star>\uprho_+$, which we have already seen appear in~\eqref{equi-energy point} and is selected such that $V(\uprho_\star) = 0$. This means that $\uprho_-$ and $\uprho_\star$ are equipotential states bookending the lower energy equilibrium $\uprho_+$:  see Figure \ref{the potential energy plot} for a graphical depiction.

\begin{figure}[!h]
    \centering
    \scalebox{0.4}{\includegraphics{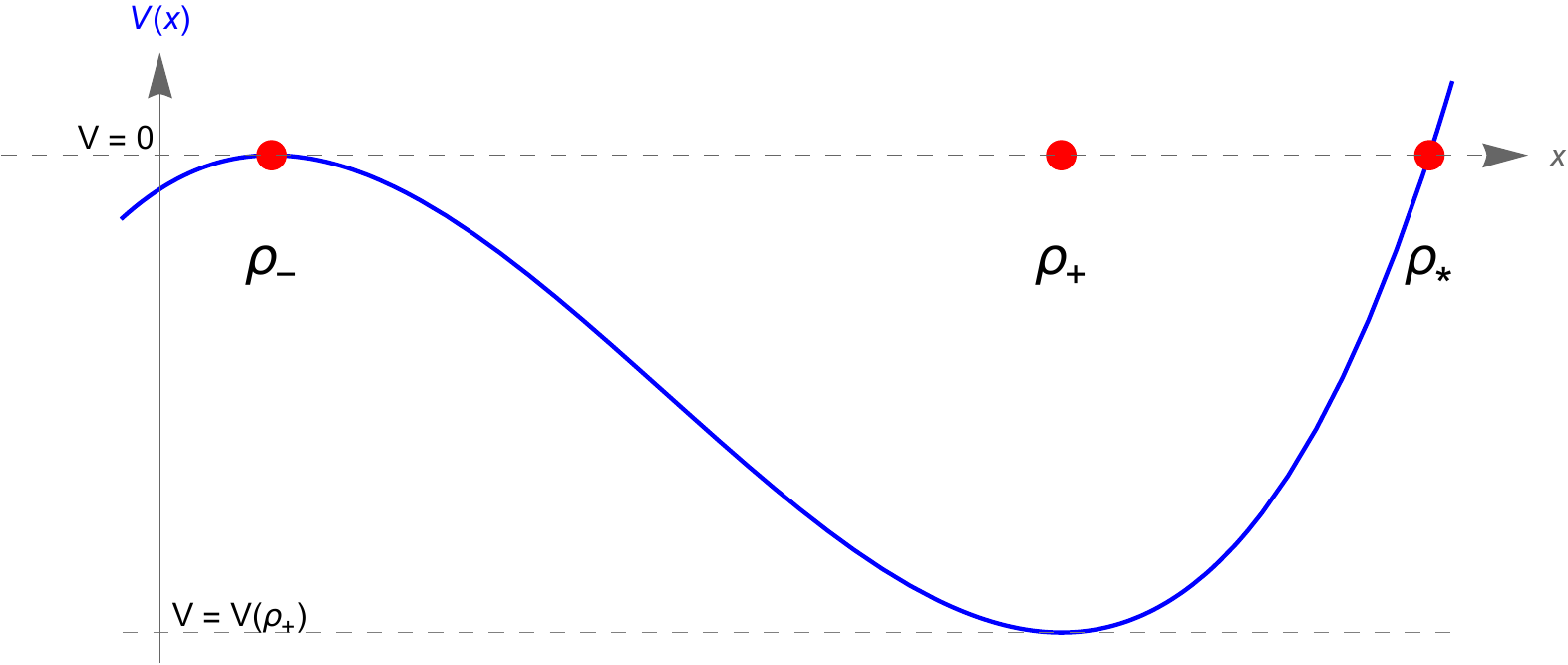}}
    \caption{We plot the graph of the potential energy $V$~\eqref{the potential energy of the system} of the Li\'enard system~\eqref{The shallow water ODEs v3}. The states $\uprho_-$ and $\uprho_+$ are the critical points of $V$. The top dashed line indicates $V=0$ and intersects the equipotential states $\uprho_-$ and $\uprho_\star$. The lower dashed line shows $V = V\tp{\uprho_+}<0$. See also Figure~\ref{comparison of the trapping regions} for associated phase space portraits.}
    \label{the potential energy plot}
\end{figure}

This conjunction of special properties of $V$ and $G$ then provides a mechanism for constructing heteroclinic orbits  between $\uprho_\pm$.  Indeed, as we prove in  Section~\ref{subsection on distinguished shallow water bore solutions}, when $G$ has a sign between the equipotential states, meaning $|G(x)|>0$ for all $x\in[\uprho_-,\uprho_\star]$, we can combine energy bounds and  topological arguments to prove that the ODE~\eqref{The shallow water ODEs v3} admits a global solution over $\R$ whose far field limits are the equilibria $\uprho_-$ and $\uprho_+$.  Since $G$ is decreasing there are exactly two ways it can have a sign in $[\uprho_-,\uprho_\star]$: $G(\uprho_\star)>0$, in which case the $G$ term in \ref{subsection on distinguished shallow water bore solutions} acts dissipatively, driving the orbit from $\uprho_-$ to the lower potential state $\uprho_+$ (an ebbing bore); and $G(\uprho_-)<0$, in which case $G$ acts as an amplifier, driving the orbit up the potential from $\uprho_+$ to $\uprho_-$ (a surging bore).  The sets $\mathfrak{C}_{\pm1}$, which are defined in~\eqref{regions of left and right handed bore waves} and depicted in Figure~\ref{the left and right handed bore regions}, precisely encode the values of the parameters $\tp{\m{g},\m{A}}$ for which these sign conditions hold, with $\mathfrak{C}_{1}$ corresponding to $G>0$ and $\mathfrak{C}_{-1}$ corresponding to $G<0$.  A crucial feature of this construction is that the state $\uprho_-$ is always hyperbolic, while $\uprho_+$ is a sink when $G>0$ and a source when $G <0$.  This means that while the actual solutions to~\eqref{The shallow water ODEs v3} are not unique due to translation invariance, the orbit itself (viewed as a curve in phase space) is uniquely determined.

With the heteroclinic orbit solutions to~\eqref{The shallow water ODEs v3} in hand, the aim is to return to the Navier-Stokes system and attempt to construct bore wave solutions as perturbations of the shallow water bores.  It turns out that there is not a unique way to attempt this, but our strategy (and most other plausible ones) for doing so involves introducing a model Stokes problem in the $\ep$-thin domain $\Omega_\ep$ with non-constant coefficients determined by a given function $h: \R \to \R^+$: see~\eqref{linear stokes in a thin domain} for the precise statement.  As we will discuss in greater detail later, the goal is to show that this $\ep$-Stokes problem induces an isomorphism between certain Banach spaces and build the bore waves via a fixed point argument that crucially employs this isomorphism. The $\ep$-Stokes system is elliptic (in the sense of Agmon, Douglis, and Nirenberg~\cite{adn2}), so it is perhaps unsurprising that it induces an isomorphism, but the key point here is that in order for this to be useful in a fixed point argument we need to understand how the operator norms of this map and its inverse depend on $\ep$ and the function $h$.  This, it turns out, is a nontrivial task.     

The principal technical tool for obtaining estimates to Stokes problems with applied stress conditions such as~\eqref{linear stokes in a thin domain} is the Korn inequality. Thus, in some sense, the main task then reduces to having a good understanding Korn's inequality in $\ep$-thin domains, and here fortune favors us somewhat. Indeed, Chapter 3 of Lewicka~\cite{MR4592573} provides a detailed treatment of the Korn inequality, including an analysis in thin domains based on earlier work of Lewicka and M\"uller~\cite{MR2795715}. Unfortunately, the exact result we require does not appear in either of these works. Nevertheless, the ideas and techniques developed there are adaptable, and we use them as the foundation on which we build our specialized thin domain Korn inequality for the $\ep$-Stokes problem~\eqref{linear stokes in a thin domain}. This, other $\ep$-thin functional inequalities, and the $\ep$-Stokes problem are our focus of study in Section~\ref{section on pdes in thin domains}.

Finally, we turn to a discussion of the principal part of our construction. This is quite technical if presented in full detail, so our aim here is to give a schematic, abstract overview and emphasize the connections to the aforementioned tools.  The precise details are recorded in Sections~\ref{section on the shallow water and residual equations} and~\ref{section on analysis of bore waves}.  In broad strokes, the idea is to mimic the derivation of the shallow water system from Navier-Stokes by positing that the full solution is decomposed into a leading order part that is close to the shallow water bore in some sense, and a small residual term that can be constructed as a fixed point, at least when $\ep$ is sufficiently small.  However, actually executing this strategy turns out to be significantly more delicate than what is suggested by the shallow water derivation.

In the following discussion, the solution to the Navier-Stokes system~\eqref{FBINSE, param. tuned and funny flux} we seek will be denoted $\pmb{X} = \tp{\pmb{\zeta},\pmb{u},\pmb{p}}$.  The first task is to specify the structural relationship between $\pmb{X}$ and our shallow water heteroclinic orbit, $\tp{\mathsf{H},\mathsf{U}}$, which we do by making the schematic ansatz 
\begin{equation}\label{schematic ansatz}
    \pmb{X} = \mathscr{P}\tp{X,\ep^2x} + \ep^2x.
\end{equation}
In the above $X = (H,U,\dotsc)$ is a tuple of smooth, but not necessarily square summable, scalar functions over $\R$ referred to as the \emph{ODE unknowns}; these should be thought of as approximations $H\approx\mathsf{H}$ and $U\approx\mathsf{U}$, with the other terms in the tuple playing auxiliary roles. The tuple $x = \tp{\eta,u,p}$ constitutes the \emph{residual unknowns}; these encode the difference between the Navier-Stokes and shallow water solutions, and we will assume these belong to $L^2$-based Sobolev spaces. The \emph{ansatz prefix} $\mathscr{P}(\cdot,\cdot)$ dictates the leading order contribution to the Navier-Stokes solution $\pmb{X}$ as a function of the ODE unknowns and a scaled down version of the residual unknowns.  The precise form of $\mathscr{P}$ is chosen to satisfy three mandates:  first, it should be close to the shallow water bore solutions; second, it should yield the correct shear flow equilibria in the far field; and third, the dependence on $\ep^2 x$ should act to minimize the complexity of the resulting equations for $X$ and $x$.

Once we select the prefix operator $\mathscr{P}$,  we plug the ansatz into the Navier-Stokes system~\eqref{FBINSE, param. tuned and funny flux}.  Due to the increased number of unknown terms in the ansatz, this obviously results in an underdetermined system for $x$ and $X$, so we must impose enough auxiliary equations to make the system determined.  In our case, all of the new equations will be ODEs, which explains the naming convention for the new $X$ variable.  A good guiding principle for finding these ODEs is to sort the terms in the system according to their size relative to $\ep$ and impose equations to make the highest order terms vanish.  We emphasize that there is no unique way to do this and that it can be advantageous to shift terms in the sorting hierarchy to make analysis easier. 
 
The most obvious choice of auxiliary ODEs decouples the $X$ unknown from $x$ and yields a nonlinear autonomous system
\begin{equation}\label{first instance of NODE}
    \mathscr{N}_{\m{ODE}}\tp{X} = 0,
\end{equation}
which is precisely solved by our distinguished shallow water solution $X = \mathsf{X} = \tp{\mathsf{H},\mathsf{U},\dots}$.  Plugging $X = \mathsf{X}$ into the ansatz then provides a nonlinear PDE system that must be satisfied by the residual unknown $x$ in order for $\pmb{X}$ to solve ~\eqref{FBINSE, param. tuned and funny flux}; we write this schematically as 
\begin{equation}\label{first instance of NPDE}
    0= \mathscr{N}_{\m{PDE}}^\ep\tp{\mathsf{X},x} = \mathscr{L}^\ep\tp{\mathsf{X},\ep^2x}x + \mathscr{A}^\ep\tp{\mathsf{X}}x + \mathscr{F}^\ep\tp{\mathsf{X},\ep x} ,
\end{equation}
with the nonlinear operator $\mathscr{N}^\ep_{\m{PDE}}$ decomposed into three distinct pieces. The first and most important of these, $\mathscr{L}^\ep\tp{\mathsf{X},\ep^2x}x$, is based on a version of the $\ep$-Stokes operator discussed above that is quasilinear because its coefficients depend on both $\mathsf{X}$ and $\ep^2 x$. The second term consists of the \emph{adversarial linear operator} $\mathscr{A}^\ep\tp{\mathsf{X}}$  applied to $x$; the moniker adversarial is to emphasize that this map generically spoils the invertibility of the sum of linear maps $\mathscr{L}^\ep\tp{\mathsf{X},\ep^2 x} + \mathscr{A}^\ep\tp{\mathsf{X}}$ when we know the first summand is invertible. The third term is the \emph{residual source term}, which encompasses a large number of nonlinear and remainder terms. Crucially, its structure allows us to guarantee it is contractive in $x$ for $\ep$ small.  To execute our fixed point strategy for solving~\eqref{first instance of NPDE} we would first establish the invertibility of $\mathscr{L}^\ep\tp{\mathsf{X},\ep^2 x}$ for all $x$ in some complete metric space and then study the map
\begin{equation}\label{the operator which does not work}
    x\mapsto -\tsb{\mathscr{L}^\ep\tp{\mathsf{X},\ep^2x}}^{-1}\tp{\mathscr{A}^\ep\tp{\mathsf{X}}x + \mathscr{F}^\ep\tp{\mathsf{X},\ep x}}.
\end{equation}
Unfortunately, $\mathscr{A}^\ep\tp{\mathsf{X}}$ remains adversarial in this role as well, as it prevents this map from being contractive.  We thus arrive at the conclusion that the obvious choice of auxiliary equations~\eqref{first instance of NODE} does not work due to the obstructions caused by the adversarial linear operator.

The decoupling approach above is an obvious choice only insofar as it presents minimal technical complexity. If we are to salvage our construction, we must employ a more intricate approach that embraces a coupling of $X$ and $x$ in the ODEs for the sake of weakening the adversarial linear operator.   To understand this approach schematically, we introduce the notion of a \emph{coupling operator} $\mathscr{C}^\ep\tp{\cdot,\cdot}$ as well as the \emph{coupling response operator} $F[\mathscr{C}^\ep]\tp{\cdot,\cdot}$, which together encode how the unknowns $X$ and $x$ couple in the systems of PDEs and ODEs:
\begin{equation}\label{NODE and NPDE, second version}
    \mathscr{N}^\ep_{\m{PDE}}\tp{X,x} + \mathscr{C}^\ep\tp{X,x} = 0\text{ and }\mathscr{N}_{\m{ODE}}\tp{X} + F[\mathscr{C}^\ep]\tp{X,\ep^2 x} = 0.
\end{equation}
These operators can thought of as an abstract expression of the sorting hierarchy resulting from plugging the ansatz~\eqref{schematic ansatz} into~\eqref{FBINSE, param. tuned and funny flux}.    The ansatz~\eqref{schematic ansatz} is chosen to make the residual contribution to $\pmb{X}$ negligible relative to that of the ODE unknown, and this scale difference is inherited by $\mathscr{C}^\ep$ and $F[\mathscr{C}^\ep]$, as can be seen through size comparisons: if $|\mathscr{C}^\ep\tp{X,x}|\lesssim 1$ then $|F[\mathscr{C}^\ep]\tp{X,\ep^2x}|\lesssim\ep^2$. The upshot is that we have the freedom to select $\mathscr{C}^\ep$ to influence the PDEs at unit scale, paying only the price of a small perturbation to the ODEs.

The above observations suggest a mechanism to pacify the adversarial operator: we choose $\mathscr{C}^\ep(X,x) \approx -\mathscr{A}^\ep\tp{X}x$ to cancel the harmful contribution in~\eqref{the operator which does not work}.  We cannot take equality here due to the constraint that $\mathscr{C}^\ep\tp{X,x}$ must be smooth and independent of the vertical coordinate.  Nevertheless, we identify a choice that, in a sense, converges to the adversarial operator in the limit of vanishing $\ep$ (see the latter half of Section~\ref{subsection on properties of residual forcing and nonlinearities} for the precise estimates on this convergence).   

At first glance, it might seem that introducing the coupling response term to the ODE in~\eqref{NODE and NPDE, second version} is a disastrous choice, as it breaks the autonomous structure, which plays a key role in the construction of the heteroclinic orbits.  However, the same structure in $\mathscr{N}_{\m{ODE}}$ that we exploit to show the existence of heteroclinic orbits offers a resilient defense to this threat: the configuration of its equilibria renders the system stable under small nonautonomous perturbations.  In Section~\ref{subsection on perturbation theory for heteroclinic orbits} we isolate this stability mechanism and prove a general result showing the persistence of heteroclinic orbits under small and nonautonomous perturbations, akin to an implicit function theorem for such orbits.  This  grants us the existence of a perturbation operator $\tau^\ep(\cdot)$, which can be thought of as the implicit function for solutions near $\mathsf{X}$,  such that for all $x$ in an appropriate space 
\begin{equation}\label{parametrizing the solutions to the ODEs}
    \mathscr{N}_{\m{ODE}}\tp{\mathsf{X} + \tau^\ep\tp{\ep^2 x}} + F[\mathscr{C}^\ep]\tp{\mathsf{X} + \tau^\ep\tp{\ep^2 x},\ep^2 x} = 0.
\end{equation}

Parameterizing solutions to the ODE in~\eqref{NODE and NPDE, second version} as in~\eqref{parametrizing the solutions to the ODEs} provides the principal tool to defeat the adversarial linear operator.  Indeed, we solve for the ODE unknown in terms of $\mathsf{X}$ and $x$ via $X = \mathsf{X} + \tau^\ep\tp{\ep^2 x}$, check that $\mathscr{L}^\ep\tp{\mathsf{X} + \tau^\ep\tp{\ep^2 x},\ep^2x}$ is invertible for all $x$ in an appropriate complete metric space, and reformulate the PDE in~\eqref{NODE and NPDE, second version} as the fixed point problem
\begin{equation}\label{the real fixed point identity}
    x = - \tsb{\mathscr{L}^\ep\tp{\mathsf{X} + \tau^\ep\tp{\ep^2 x},\ep^2x}}^{-1}\tp{\mathscr{A}^\ep\tp{\mathsf{X} + \tau^\ep\tp{\ep^2 x}}x + \mathscr{C}^\ep\tp{\mathsf{X} + \tau^\ep\tp{\ep^2 x},x} +  \mathscr{F}^\ep\tp{\mathsf{X} + \tau^\ep\tp{\ep^2 x},\ep x}}.
\end{equation}
In contrast with what we saw for~\eqref{the operator which does not work}, the operator on the right hand side of~\eqref{the real fixed point identity} is contractive for small enough $\ep>0$, precisely because we can take $\mathscr{C}^\ep$ to approximate $-\mathscr{A}^\ep$. The contraction mapping theorem then applies, and we obtain the existence of a residual $x$ that solves~\eqref{the real fixed point identity}, which in turn defines the ODE solution $X = \mathsf{X} + \tau^\ep\tp{\ep^2 x}$.  Together these identities imply that $\pmb{X}$ given by~\eqref{schematic ansatz} solves the free boundary Navier-Stokes system~\eqref{FBINSE, param. tuned and funny flux}.  Our results guarantee that $x$ and $\tau^\ep\tp{\ep^2 x}$ vanish at infinity and that $\ep^2 x$ is small, so the solution $\pmb{X} = \mathscr{P}\tp{X,\ep^2x} + \ep^2x$ is close to the shallow water heteroclinic orbit $\tp{\mathsf{H},\mathsf{U}}$ and has far field limits of distinct shear flow equilibria.  We have therefore found the desired bore wave solutions.  The final component of our argument is an a posteriori elliptic regularity promotion to show that the bores are, in fact, smooth.

\subsection{Notation conventions}\label{subsection on notation conventions}

The naturals are the set $\N = \tcb{0,1,2,\dots}$ while $\N^+ = \N\setminus\tcb{0}$. The positive real numbers are $\R^+ = \tp{0,\infty}$. For $d\in\N$ the bracket $\tbr{\cdot}:\C^d\to\R^+$ is the function with action defined by $\tbr{x} = \tp{1 + \tabs{x_1}^2+\cdots+\tabs{x_d}^2}^{1/2}$.
The characteristic function of a a set $E_0\subseteq E_1$ is denoted by $\mathds{1}_{E_0}:E_1\to\tcb{0,1}$. The notation $\ell_0\lesssim\ell_1$ will mean that there exists a constant $C\in\R^+$, depending only on the parameters mentioned in context, such that $\ell_0\le C\ell_1$; we add emphasis to the dependence on $C$ to one or more particular parameters $p_0,\dots,p_1$ by writing $\ell_0\lesssim_{p_0,\dots,p_1}\ell_1$. If $\tcb{X_r}_{r=1}^\ell$ are normed spaces with product space $X$ endowed with any product norm $\tnorm{\cdot}_X$, then we write $\tnorm{x_1,\dots,x_\ell}_{X} = \tnorm{\tp{x_1,\dots,x_\ell}}_{X}$. If $E$ is a subset of any topological space, we denote its closure by $\Bar{E}$ and its interior by $\m{int}E$.

The Fourier transform and its inverse on the space of tempered distributions $\mathscr{S}^\ast\tp{\R^d;\C}$, which are normalized to be unitary on $L^2\tp{\R^d;\C}$, are denoted by $\mathscr{F}$ and $\mathscr{F}^{-1}$, respectively. For measurable functions $\chi:\R^d\to\C$ we write $\chi(D)$ to be the Fourier multiplication operator with symbol $\xi\mapsto\chi(\xi)$.

The $\R^d$ gradient is denoted by $\grad = \tp{\pd_1,\dots,\pd_d}$ and the divergence of a vector field $\Psi:\R^d\to\R^d$ is $\grad\cdot\Psi=\pd_1\Psi_1+\cdots\pd_d\Psi_d$.    If $E\subseteq\R^d$ is any open subset, $k,\ell\in\N$, $t\in[1,\infty]$, we write $W^{k,t}\tp{E;\R^\ell}$ to denote the usual Sobolev space of $\R^\ell$-valued functions with weak derivatives of order at most $k$ in $L^t\tp{E;\R^\ell}$; these are endowed with a standard choice of norm $\tnorm{f}_{W^{k,t}\tp{E}} = \tp{\tnorm{f}_{L^t\tp{E}}^2 + \cdots + \tnorm{\grad^r f}_{L^t\tp{E}}^2}^{1/2}$. When $t=2$ we write $H^k\tp{E;\R^\ell}$ in place of $W^{k,t}\tp{E;\R^\ell}$ to emphasize that the space is Hilbert. We also denote $W^{\infty,t}\tp{E;\R^\ell} = \bigcap_{k=0}^\infty W^{k,t}\tp{E;\R^\ell}$. The set of $k$-times continuously differentiable $\R^\ell$-valued 
(compactly supported) functions on $E$ is $C^k\tp{E;\R^\ell}$ (resp. $C^k_c\tp{E;\R^\ell}$). Finally, we denote $C^\infty\tp{E;\R^\ell}=\bigcap_{k=0}^\infty C^k\tp{E;\R^\ell}$ and $C^\infty_c\tp{E;\R^\ell} = \bigcap_{k=0}^\infty C^k_c\tp{E;\R^\ell}$.

\section{Analysis and perturbations of shallow water heteroclinic orbits}\label{section on analysis and perturbations of shallow water heteroclinic orbits}

The method by which we produce bore wave solutions to system~\eqref{FBINSE, param. tuned and funny flux} passes through an analysis of the shallow water equations. In a sense to be made explicit later in Sections~\ref{section on the shallow water and residual equations} and~\ref{section on analysis of bore waves}, solutions to the $\ep$-shallow free boundary Navier-Stokes equations~\eqref{FBINSE, param. tuned and funny flux} are, up to a remainder of order $\ep$, approximated by solutions to a corresponding system of ODEs that emerge in the limit as $\ep\to0$ (see also Corollary~\ref{coro on justification of the shallow water limit}). These are what we refer to as the distinguished shallow water equations. In Section~\ref{subsection on distinguished shallow water bore solutions}, via the combination of energy estimates and a topological argument, we prove the existence of heteroclinic orbit solutions to these shallow water equations.

As was discussed in Section~\ref{subsection on statement of the main theorems and discussion}, the overarching strategy is to find a Navier-Stokes solution `nearby' these shallow water solutions; however doing so turns out to be rather subtle and we end up requiring persistence and stability of our shallow water heteroclinic orbits under a general class of nonautonomous perturbations. This perturbation theory is developed in a rather general setting in Section~\ref{subsection on perturbation theory for heteroclinic orbits}.

\subsection{Distinguished shallow water bore solutions}\label{subsection on distinguished shallow water bore solutions}

The one-dimensional viscous shallow water equations with gravity, laminar drag, and prescribed relative velocity flux, which in a certain sense are the limit of~\eqref{FBINSE, param. tuned and funny flux} as $\ep\to0$, are a coupled system of ODEs in the unknown height and (depth-averaged) velocity $\mathsf{H},\mathsf{U}:\R\to\R$ given by
\begin{equation}\label{The shallow water ODEs}
    \tp{4 - \mathsf{U}}\mathsf{H} = \m{A},\quad \mathsf{H}\tp{\mathsf{U} - 4}\mathsf{U}' + \m{a}\mathsf{U} - 4\upmu\tp{\mathsf{H}\mathsf{U}'}' + \m{g}\mathsf{H}\mathsf{H}' = 4\m{a}\mathsf{H}.
\end{equation}
The parameters are taken to satisfy: $\upmu,\m{a}>0$, $\m{g}\ge0$, and $\m{A}\in\tp{0,1}$. We are only interested in solutions to~\eqref{The shallow water ODEs} that satisfy $\mathsf{H}>0$ and $\mathsf{U}<4$. Under these assumptions, we can use the first identity in~\eqref{The shallow water ODEs} to solve for $\mathsf{U}$ as a function of $\mathsf{H}$, namely $\mathsf{U} = 4 - \m{A}/\mathsf{H}$. Then this expression can be substituted into the second identity of~\eqref{The shallow water ODEs} to derive the closed second order ordinary differential equation for the surface height $\mathsf{H}$ given by~\eqref{shallow water just H version} from the second item of Theorem~\ref{first main theorem}.

We make two more reformulations of these shallow water equations. Since we require that $\mathsf{H}>0$, it is convenient to swap unknowns to the function $\rho:\R\to\R$ such that $e^\rho =\mathsf{H}$. Then we calculate that $\rho$ obeys the Li\'enard equation~\eqref{The shallow water ODEs v3} for $F$ and $G$ defined in~\eqref{the hamiltonian and dissipation}.  Finally, our methods in this subsection rely on phase space analysis and thus~\eqref{The shallow water ODEs v3} is written as the first order system of ODEs in the unknown vector $X:\R\to\R^2$ with vector field $\Phi : \R^2 \to \R^2$, 
\begin{equation}\label{The shallow water ODEs v4}
    X' = \Phi(X) \text{ with } \quad\Phi(x) = \tp{x_2,F(x_1) - x_2G(x_1)}.
\end{equation}
Note the correspondence between~\eqref{The shallow water ODEs v3} and~\eqref{The shallow water ODEs v4} is given by $X = \tp{\rho,\rho'}$.

The equilibrium solutions to~\eqref{The shallow water ODEs v4} correspond to the zeros of $F$ which, in turn correspond to the states $\m{H}_{\pm}$ of~\eqref{sweq}. We use the notation
\begin{equation}\label{equilibria for the rho dude}
    \uprho_{\pm} = \log\m{H}_{\pm}\text{ so that }\Phi(\uprho_\pm,0) = 0.
\end{equation}
The points $\uprho_{\pm}$ are functions of $\m{A}$ alone and we will find it occasionally useful to write $\uprho_{\pm}\tp{\m{A}}$ to emphasize this dependence.

The following lemma builds a special region related to the potential energy~\eqref{the potential energy of the system} associated with $F$ from~\eqref{the hamiltonian and dissipation} and the equilibria~\eqref{equilibria for the rho dude}, which we dub the trapping region.
\begin{lemC}[Preliminary facts]\label{lem on preliminary facts}
    The following hold.
    \begin{enumerate}
        \item $F>0$ on $(\uprho_-,\uprho_+)$ with $\pm F'(\uprho_\pm)>0$.
        \item There exists a unique $\uprho_\star>\uprho_+$ such that $\int_{\uprho_-}^{\uprho_\star}F(s)\;\m{d}s = 0$.
        \item The function $v:[\uprho_-,\uprho_\star]\to[0,\infty)$ given by
        \begin{equation}\label{the maximal velocity function}
             v(x_1) = \bp{2\int_{\uprho_-}^{x_1}F(s)\;\m{d}s }^{1/2} = \tp{-2V(x_1)}^{1/2}
        \end{equation}
        is well-defined and continuously differentiable and satisfies $v(x_1)>0$ if $\uprho_-<x_1<\uprho_\star$ and $v(\uprho_-) = v(\uprho_\star) = 0$.
        \item The region $R = \tcb{x\in\R^2\;:\;\uprho_-\le x_1\le\uprho_\star,\quad|x_2|\le v(x_1)}$ is nonempty and compact.
        \item The map $G:\R\to\R$ is strictly decreasing and, if $\m{g}>0$, there exists a unique $\uprho_0\in\R$, such that $G(\uprho_0)=0$.
    \end{enumerate}
\end{lemC}
\begin{proof}
    The function $F$ from~\eqref{the hamiltonian and dissipation} is strictly concave, $F(x_1)\to-\infty$ as $|x_1|\to\infty$, and $F(x_1)=0$ if and only if $x_1=\uprho_\pm$; therefore, the first item must hold.     The existence and uniqueness of the second item follows by the intermediate value theorem and the fact that $x_1\mapsto \int_{\uprho_+}^{x_1}F(s)\;\m{d}s$ is strictly deceasing to $-\infty$ as $x_1\to\infty$.    The third and fourth items now follow from simple calculations.     The fact that $G$ is strictly deceasing following from the observation that $G'(x_1) = -\tp{\m{A}e^{-x_1} + \tp{\m{g}/\m{A}}e^{2x_1}}/4\upmu <0$. If $\m{g}>0$ we compute directly that $\uprho_0 = \log\tp{\m{A}^2/\m{g}}/3$. This gives the fifth item.
\end{proof}

\begin{figure}[!h]
    \centering
    \scalebox{0.145}{\includegraphics{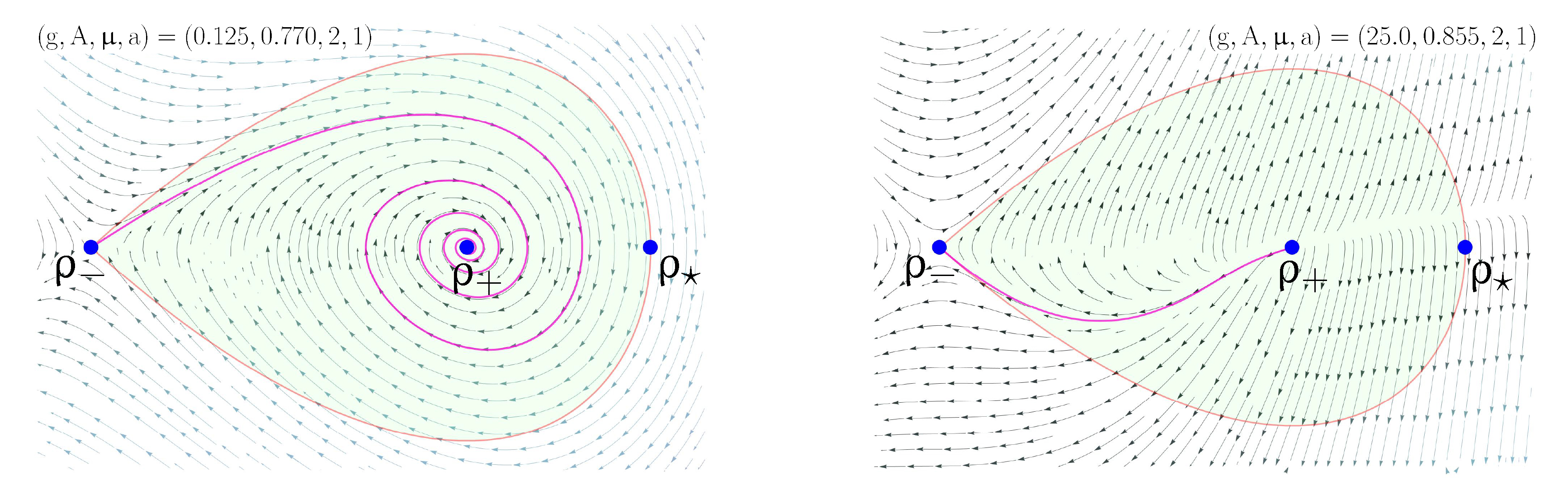}}
    \caption{Shown here are a streamline plots of the vector field $\Phi$ from~\eqref{The shallow water ODEs v4} superimposed on the trapping region $R$ (from the fourth item of Lemma~\ref{lem on preliminary facts}) in green, with the boundary of $R$ shown in red. Both plots obey $\tp{\upmu,\m{a}} = (2,1)$. The plot on the left corresponds to parameter values $\tp{\m{g},\m{A}} = (0.125,0.770)\in\mathfrak{C}_1$ whereas the plot on the right has $(\m{g},\m{A}) = (25.0,0.855)\in\mathfrak{C}_{-1}$. The sets $\mathfrak{C}_{\pm 1}$ are defined in~\eqref{the sets of different chirality}. The curves in pink are the global unstable/stable manifolds for the left equilibrium intersected with $R$; we construct these trajectories in the second and fourth items of Proposition~\ref{prop on trapping region}.}
    \label{comparison of the trapping regions}
\end{figure}

We note that like the points $\uprho_{\pm}$, the point $\uprho_\star$ is a function of $\m{A}$ alone and so we will also write $\uprho_\star\tp{\m{A}}$, when necessary. The next lemma studies the sign of the function $G$ on the interval $[\uprho_-,\uprho_\star]$. Notice that from the fifth item of Lemma~\ref{lem on preliminary facts}, we learn that $G$ is signed on $[\uprho_-,\uprho_\star]$ if and only if $G\tp{\uprho_\star}>0$ or $G(\uprho_-)<0$. This is what motivates the forthcoming definition~\eqref{the sets of different chirality}.
\begin{lemC}[Conditions for a sign]\label{lem  on sufficient conditions for a sign}
    Let
    \begin{equation}\label{the sets of different chirality}
        \mathfrak{C}_1 = \tcb{\tp{\m{g},\m{A}}\in[0,\infty)\times(0,1)\;:\;G(\uprho_\star\tp{\m{A}})>0}
        \text{ and }
        \mathfrak{C}_{-1} = \tcb{\tp{\m{g},\m{A}}\in[0,\infty)\times(0,1)\;:\;G(\uprho_-\tp{\m{A}})<0}.
    \end{equation} 
    The following hold for $\tp{\m{g},\m{A}}\in[0,\infty)\times\tp{0,1}$.
    \begin{enumerate}
        \item We have $(\m{g},\m{A})\in\mathfrak{C}_1$ if and only if $\m{g}<\m{A}^2e^{-3\uprho_\star\tp{\m{A}}}$ and $(\m{g},\m{A})\in\mathfrak{C}_{-1}$ of and only if $\m{g}>\m{A}^2e^{-3\uprho_-\tp{\m{A}}}$. In particular, for $\iota\in\tcb{-1,1}$ we have that
        \begin{equation}\label{non-degenerate I}
            \tcb{\tp{\tilde{g},\m{A}}\;:\;\tilde{g}\in[0,\infty)}  \cap\mathfrak{C}_\iota \neq\es      .
        \end{equation}
        \item For $\iota\in\tcb{-1,1}$, we have the inclusion $\mathfrak{C}_{\iota}\subset\tcb{\tp{\tilde{g},\tilde{A}}\in\R^2\;:\; \iota\tp{\tilde{g}-8}<0}$.
        \item If $\m{g}<8$, there exists $\tilde{\m{A}}\in(0,1)$ such that $\tp{\m{g},\tilde{\m{A}}}\in\mathfrak{C}_1$. If $\m{g}>8$, there exists $\tilde{\m{A}}\in(0,1)$ such that $\tp{\m{g},\tilde{\m{A}}}\in\mathfrak{C}_{-1}$. If $\m{g} = 8$, then no matter the choice of $\tilde{\m{A}}\in\tp{0,1}$, we have $\tp{\m{g},\tilde{\m{A}}}\not\in\tp{\mathfrak{C}_{-1}\cup\mathfrak{C}_1}$.
    \end{enumerate}
\end{lemC}
\begin{proof}
    If $\uprho\in\tcb{\uprho_-,\uprho_+}$ then a direct computation from the definition of $G$ shows that $\pm G(\uprho)>0$ if and only if $0 < \pm\tp{\m{A}^2e^{-3\uprho} - \m{g}}$. The asserted equivalence of the first item now follows. The assertion~\eqref{non-degenerate I} follows by taking $\m{g}$ sufficiently small in the case $\iota = 1$ and by taking $\m{g}$ sufficiently large in the case $\iota = -1$.

    For the second item, consider first the case that $\iota=1$. If $\m{g}\ge8$, then $\uprho_0\le\f13\log\tp{\m{A}^2/8}<\log(1/2)<\uprho_\star$ and hence $G(\uprho_\star)<0$ and so $\tp{\m{g},\m{A}}\not\in\mathfrak{C}_1$. On the other hand, if $\iota = -1$ and $\tp{\m{g},\m{A}}\in\mathfrak{C}_{-1}$, then
    \begin{equation}
        \m{g}>\m{A}^2e^{-3\uprho_-\tp{\m{A}}} = 8\tp{1 + \tp{1 - \m{A}}^{1/2}}^3/\m{A}\ge 8.
    \end{equation}

    The third item's positive inclusions follow by taking $\m{A}\in(0,1)$ sufficiently close to $1$ and using that 
    \begin{equation}
        \lim_{\m{A}\to1}\m{A}^2e^{-3\uprho_\star\tp{\m{A}}} = \lim_{\m{A}\to1}\m{A}^2e^{-3\uprho_-\tp{\m{A}}} = 8.
    \end{equation}
    The claimed negative inclusion follows from $\m{g}=8$ implying that $\uprho_0 = \log\tp{\tilde{\m{A}}^{2/3}/2}$ (from the fifth item of Lemma~\ref{lem on preliminary facts}). Then $\uprho_+\tp{\tilde{\m{A}}}=\log\m{H}_+\tp{\tilde{\m{A}}}>\log\tp{1/2}\ge\uprho_0$ which, by $\uprho_\star>\uprho_+$ , implies that $G(\uprho_\star)<0$. Conversely, we have $\uprho_-\tp{\tilde{\m{A}}} = \log\m{H}_-\tp{\tilde{\m{A}}}<\log\tp{\tilde{\m{A}}/2}\le\uprho_0$ which implies that $G(\uprho_-)>0$.
\end{proof}

See Figure~\ref{the left and right handed bore regions} for a plot of the regions~\eqref{the sets of different chirality}. Next, we study the zeros $(\uprho_\pm,0)$ of the map $\Phi$ from~\eqref{The shallow water ODEs v4}.

\begin{lemC}[Linear character of equilibria]\label{lem on linear character of equilibira}
Define the maps $\lambda_{\pm}:\R\to\C$ via
\begin{equation}
    \lambda_{\pm}\tp{x_1} = \tp{-G(x_1) \pm \sqrt{\tp{G(x_1)}^2 + 4F'(x_1)}}/2,
\end{equation}
where $F$ and $G$ are as in~\eqref{the hamiltonian and dissipation}. The following hold.
\begin{enumerate}
    \item We have that $\m{spec}\tp{D\Phi\tp{\uprho_{\pm},0}} = \tcb{\lambda_-\tp{\uprho_\pm},\lambda_+\tp{\uprho_\pm}}\subset\C$.
    \item If $G(\uprho_+)>0$, then $\tp{\uprho_+,0}$ is an asymptotically stable equilibrium for $\Phi$ forward in time (a sink). On the other hand, if $G(\uprho_+)<0$, then $\tp{\uprho_+,0}$ is an asymptotically stable equilibrium for $\Phi$ backward in time (a source).
    \item We have $\lambda_{\pm}(\uprho_-)\in\R$ and $\lambda_{-}\tp{\uprho_-}<0<\lambda_{+}\tp{\uprho_+}$. In particular, $\tp{\uprho_-,0}$ is a hyperbolic equilibrium for $\Phi$.
    \item The stable manifold at $\tp{\uprho_-,0}$ is tangent to $\tp{1,\lambda_-\tp{\uprho_-}}$, while the unstable manifold at $\tp{\uprho_-,0}$ is tangent to $\tp{1,\lambda_+\tp{\uprho_-}}$.
    \item If $G(\uprho_-)>0$ we have the bounds
    \begin{equation}\label{important_bound_1}
        \lambda_-\tp{\uprho_-}<-v'(\uprho_-)\text{ and }\lambda_+\tp{\uprho_-}<v'(\uprho_-),
    \end{equation}
    while if $G(\uprho_-)<0$, we have the bounds
    \begin{equation}\label{important_bound_2}
        \lambda_-(\uprho_-)>-v'(\uprho_-)\text{ and }\lambda_+\tp{\uprho_-}>v'(\uprho_-).
    \end{equation}
    
\end{enumerate}
\end{lemC}
\begin{proof}
    The first item is a standard calculation. The second and third items follow from $F'(\uprho_-)>0$ and $F'(\uprho_+)<0$ (see first item of Lemma~\ref{lem on preliminary facts}). The fourth item is now a standard calculation and application of the (un)stable manifold theorem (see, for instance, Section 2.7 in Perko~\cite{MR1801796}).
    
    For the fifth item consider first the equality
    \begin{equation}\label{start here}
        \tp{-G(\uprho_-) + 2\sqrt{F'(\uprho_-)}}^2  = \tp{G(\uprho_-)}^2 + 4F'(\uprho_-) - 4G(\uprho_-)\sqrt{F'(\uprho_-)}.
    \end{equation}
    If $G(\uprho_-)>0$, then the final term above is negative and so we deduce that
    \begin{equation}
        -G(\uprho_-) + 2\sqrt{F'(\uprho_-)}<\sqrt{\tp{G(\uprho_-)}^2 + 4F'(\uprho_-)}.
    \end{equation}
    Since $v'(\uprho_-) = \sqrt{F'(\uprho_-)}$ we deduce from the above that $-v'(\uprho_-)>\lambda_-\tp{\uprho_-}$.
    
    On the other hand, if we suppose that $G(\uprho_-)<0$, then the last term in~\eqref{start here} is positive and so we find 
    \begin{equation}
        -G(\uprho_-) + 2\sqrt{F'(\uprho_-)}>\sqrt{\tp{G(\uprho_-)}^2 + 4F'(\uprho_-)}
    \end{equation}
    which implies that $\lambda_-\tp{\uprho_-}>-v'(\uprho_-)$. Similar arguments will verify the right hand inequalities in~\eqref{important_bound_1} and~\eqref{important_bound_2}.
\end{proof}

Recall that the region $R$ is defined in the fourth item of Lemma~\ref{lem on preliminary facts} and the map $\Phi$ is defined in~\eqref{The shallow water ODEs v4}. The following result says, in particular, that if $G>0$ on the region $R$, then this region is invariant forward in time for the flow of $\Phi$ whereas if $G<0$ on $R$ then the region is backwards in time invariant. See Figure~\ref{comparison of the trapping regions} for a depiction of the flow of $\Phi$ over the region $R$ in either of these two cases.
\begin{propC}[Trapping region]\label{prop on trapping region}
    The following hold.
    \begin{enumerate}
        \item Assume $G(\uprho_\star)>0$. For all initial data $X_0\in R$ there exists a unique smooth $X:[0,\infty)\to\R^2$ solving the initial value problem $X(0) = X_0$ and $X'(t) = \Phi(X(t))$ for all $t\in[0,\infty)$. Additionally, $X(t)\in R$ for all $t\in[0,\infty)$ and, if $X(0)\neq\tp{\uprho_-,0}$, then $X(t)\to\tp{\uprho_+,0}$ as $t\to\infty$. Moreover, $X$ obeys the bounds
        \begin{equation}\label{positive bounds}
            \sqrt{G(\uprho_\star)}\tnorm{e_2\cdot X}_{L^2\tp{(0,\infty)}} + \tnorm{e_2\cdot X}_{L^\infty\tp{(0,\infty)}}\le 2\max_{[\uprho_-,\uprho_\star]}v.
        \end{equation}
        \item Assume $G(\uprho_\star)>0$. The global stable manifold at $\tp{\uprho_-,0}$ intersects $R$ precisely at $\tp{\uprho_-,0}$, and the global unstable manifold at $\tp{\uprho_-,0}$ intersected with $R$ consists of a heteroclinic orbit from $\tp{\uprho_-,0}$ to $\tp{\uprho_+,0}$.
        \item Assume $G(\uprho_-)<0$. For all terminal data $X_0\in R$ there exists a unique smooth $X:(-\infty,0]\to\R^2$ solving the terminal value problem $X(0) = X_0$ and $X'(t) = \Phi(X(t))$ for all $t\in(-\infty,0]$. Additionally, $X(t)\in R$ for all $t\in(-\infty,0]$ and, if $X(0)\neq\tp{\uprho_-,0}$, then $X(t)\to\tp{\uprho_{+},0}$ as $t\to-\infty$. Moreover, $X$ obeys the bounds
        \begin{equation}\label{negative bounds}
            \sqrt{-G(\uprho_-)}\tnorm{e_2\cdot X}_{L^2\tp{(-\infty,0)}} + \tnorm{e_2\cdot X}_{L^\infty\tp{(-\infty,0)}}\le 2\max_{[\uprho_-,\uprho_\star]}v.
        \end{equation}
        \item Assume $G(\uprho_-)<0$. The global unstable manifold at $\tp{\uprho_-,0}$ intersects $R$ precisely at $\tp{\uprho_-,0}$, and the global stable manifold at $\tp{\uprho_-,0}$ intersected with $R$ consists of a heteroclinic orbit from $\tp{\uprho_+,0}$ to $\tp{\uprho_-,0}$.
    \end{enumerate}
\end{propC}
\begin{proof}
    We will only prove the first two items, working forward in time.  The third and fourth items follow from similar arguments, working backward in time.  Assume, then, that $G(\uprho_\star)>0$ in what follows.

    We begin with some simple observations. Let $\tau_{\pm} = \tp{1,\lambda_{\pm}\tp{\uprho_-}}$ for $\lambda_{\pm}\tp{\uprho_-}$ as in Lemma~\ref{lem on linear character of equilibira}. The fourth item of the lemma tells us that $\tau_\pm$ are tangent to the unstable $(+)$ and stable $(-)$ manifolds at $\tp{\uprho_-,0}$ while the fifth item of the lemma tells us that the line generated by $\tau_-$ is contained in a closed cone that intersects $R$ precisely at $\tp{\uprho_-,0}$ and that $r\tau_+\in R^\circ$ for $0<r$ small. 
    
    The set $R$ is compact, so we may choose a time horizon $T_R>0$ such that for any data $X_0\in R$ the ODE $X'(t) = \Phi(X(t))$, $X(0) = X_0$ admits a unique solution on the interval $t\in[0,T_R]$.  Recall from Lemma~\ref{lem on preliminary facts} that the function $G$ is decreasing and that when $\m{g}>0$ it has a unique zero at $\uprho_0 > \uprho_\star$.  In light of this strict separation and the compactness of $R$, we may assume without loss of generality that $T_R>0$ is chosen so that 
    \begin{equation}\label{this is not a useful label}
            G(e_1 \cdot X(t)) \ge  G(\uprho_\star)/2 \text{ for all } t \in [0,T_R] \text{ and } X_0 \in R. 
    \end{equation}

    Fix arbitrary data $X_0 \in R$; we will show that its solution trajectory satisfies $X(t) \in R$ for $t \in [0,T_R]$ by way of an energy argument.  To this end, define the map $V:[\uprho_-,\infty)\to\R$ via~\eqref{the potential energy of the system} and note that $v$ and $V$ (where the former is defined in~\eqref{the maximal velocity function}) are related via $v(x_1) = \sqrt{-2 V(x_1)}$ for $x_1\in[\uprho_-,\uprho_\star]$. Moreover, $V\tp{\uprho_\star} = 0$ and $V(x_1)>0$ for all $x_1>\uprho_\star$.

    Now let $\rho:[0,T_R]\to\R$ be defined via $\rho = e_1\cdot X$. Then, by the definition of $\Phi$ (see~\eqref{The shallow water ODEs v4}), $\rho$ is a solution to the second order differential equation
    \begin{equation}
        \rho''(t) = F(\rho(t)) - G(\rho(t))\rho'(t)\text{ for }t\in[0,T_R].
    \end{equation}
    We multiply the above equation by $2\rho'(t)$ and use the definitions of $v$ and $V$ together with~\eqref{this is not a useful label} to deduce that
    \begin{equation}
        [\tp{\rho'}^2]'= -2[V(\rho)]' - 2G(\rho)\tp{\rho'}^2,
    \end{equation}
    and hence 
    \begin{equation}\label{energy estimates}
    \tp{\rho'(t)}^2 +  G\tp{\uprho_\star}\int_0^t\tp{\rho'(s)}^2\;\m{d}s\le -2V\tp{\rho(t)} + \tsb{\tp{\rho'(0)}^2 - v\tp{\rho(0)}^2} \text{ for } t\in[0,T_R].  
    \end{equation}
    As $X_0 = \tp{\rho(0),\rho'(0)}\in R$ the final term in square brackets above is nonpositive and hence we deduce that
    \begin{equation}\label{It's a trap}
        \tp{\rho'(t)}^2\le -2V(\rho(t))\text{ for all }t\in[0,T_R].
    \end{equation}
    Since $-V(x_1)<0$ for $x_1>\uprho_\star$ we deduce from this that $\rho(t)\le\uprho_\star$ for all $t\in[0,T_R]$. On the other hand, if there exists an $s\in[0,T_R]$ for which $\rho(s) = \uprho_-$ then $V(\uprho_-) = 0$ and~\eqref{It's a trap} implies that $\rho'(s) = 0$ as well and thus $X(s) = (\uprho_-,0)$ is the left equilibrium point, which is only possible if $X$ is constant. If there were $t\in(0,T_R]$ for which $\rho(t)<\uprho_-$, then the intermediate value theorem implies the existence of $s\in(0,t)$ for which $\rho(s) = \uprho_-$, but we have just argued that this is inconsistent with $\rho(t)<\uprho_-$. Upon combining these observations, we deduce that
    \begin{equation}\label{things are a bit different}
        \uprho_-\le\rho(t)\le\uprho_\star\text{ for all }t\in[0,T_R].
    \end{equation}
    With this in hand, we may return to~\eqref{It's a trap} and recall that $v(x)^2 = 2V(x)$ to conclude that
    \begin{equation}
        X(t)\in\tp{\rho(t),\rho'(t)}\in R\text{ for all }t\in[0,T_R].
    \end{equation}
    Since the time horizon $T_R>0$ is valid for all initial data in $R$, we may now iterate this argument to deduce that the solution exists and remains in $R$ for $[0,kT_R]$ for arbitrary $k\in\N$. The estimate~\eqref{positive bounds} is now immediate from~\eqref{energy estimates}, which continues to hold for all $t\in[0,\infty)$.
    
    Since $\rho$ is smooth, the bound~\eqref{positive bounds} implies that $\rho'(t) \to 0$ as $t \to \infty$.  Thus, if we write $\omega_+(X_{0})$ for the forward limit set, then 
    \begin{equation}
        \omega_+(X_{0}) \subseteq R \cap \{(x,y)\in\R^2 \;:\; y=0 \}.
    \end{equation}
    On the other hand, the form of the vector field $\Phi$ prevents the set on the right from containing cycles or limit orbits from $(\uprho_-,0)$ to $(\uprho_+,0)$, so the Poincar\'e-Bendixson theorem (see, for instance, Section 3.7 in Perko~\cite{MR1801796}) requires that either $\omega_+(X_{0}) = \{(\uprho_-,0)\}$ or $\omega_+(X_{0}) = \{(\uprho_+,0)\}$.  Since $R$ is compact, we deduce from this that $X(t) \to (\uprho_\pm,0)$ as $t \to \infty.$    If $X_0 = (\uprho_-,0)$, then obviously $\omega_+(X_0) = (\uprho_-,0)$.  Assume, then, that $X_0 \neq (\uprho_-,0)$ and  $X(t) \to (\uprho_-,0)$ as $t \to \infty$.  The stable manifold theorem shows that the trajectory must be tangent to $\tau_-$ at $(\uprho_-,0)$, but we saw above that $\tau_-$ is contained in a cone outside $R$, so this is impossible. Hence, if $X_0 \neq (\uprho_-,0)$ then $X(t) \to (\uprho_+,0)$ as $t \to \infty$.  This completes the proof of the first item.

    We now turn to the proof of the second item.  The convergence result from the first item shows that the stable manifold for $(\uprho_-,0)$ intersects $R$ precisely at $(\uprho_-,0)$.  On the other hand, the same convergence result and the initial observation about $\tau_+$ shows that the intersection of $R$ with the unstable manifold to $(\uprho_-,0)$ consists of a heteroclinic orbit from $(\uprho_-,0)$ to $(\uprho_+,0)$.
\end{proof}

We now say more about the heteroclinic orbits produced by the previous result.
\begin{thmC}[Distinguished heteroclinic orbits]\label{thm on distinguished heteroclinic orbits}
    The following hold. 
    \begin{enumerate}
        \item If $G(\uprho_\star)>0$, there exists a smooth $\rho:\R\to[\uprho_-,\uprho_\star]$ solving the ODE~\eqref{The shallow water ODEs v3} and satisfying $\lim_{t\to\pm\infty}\rho(t) = \uprho_{\pm}$. Moreover, there exist constants $\al,C_m\in\R^+$ for $m\in\N$ such that 
        \begin{equation}\label{who said your life's a bore}
            \sup_{t\le 0}e^{\al|t|}\tabs{\rho(t) - \uprho_-} + \sup_{t\ge 0 }e^{\al|t|}\tabs{\rho(t) - \uprho_+}\le C_0 \text{ and } \sup_{t\in\R}e^{\al|t|}\tabs{\rho^{(m)}(t)}\le C_m \text{ for }m \ge 1.
        \end{equation}
        \item If $G(\uprho_-)<0$, there exists a smooth $\rho:\R\to[\uprho_-,\uprho_\star]$ solving the ODE~\eqref{The shallow water ODEs v3} and satisfying $\lim_{t\to\pm\infty}\rho(t) = \uprho_{\mp}$. Moreover, there exist constants $\al,C_m\in\R^+$ for $m\in\N$ such that 
        \begin{equation}\label{who said your life's a bore variant}
            \sup_{t\le 0}e^{\al|t|}\tabs{\rho(t) - \uprho_+} + \sup_{t\ge 0 }e^{\al|t|}\tabs{\rho(t) - \uprho_-}\le C_0 \text{ and } \sup_{t\in\R}e^{\al|t|}\tabs{\rho^{(m)}(t)}\le C_m \text{ for } m \ge 1.
        \end{equation}
    \end{enumerate}
\end{thmC}
\begin{proof}
    We only prove the first item, as the second follows by mirrored arguments.  Suppose then that $G(\uprho_\star)>0$.

    We know from the second item of Proposition~\ref{prop on trapping region} that the global unstable manifold for the flow of $\Phi$ intersected with $R$ consists of a heteroclinic orbit from $(\uprho_-,0)$ to $(\uprho_+,0)$.  Pick a point $X_0 \in \m{int}R \setminus\tcb{(\uprho_+,0)}$ on this orbit. Then the proposition provides a unique global smooth solution $X : \R \to R$ to the initial value problem $X'(t) = \Phi(X(t))$ for $t\in\R$ with $X(0) = X_0$, and the solution satisfies $X(t) \to (\uprho_\pm,0)$ as $t \to \pm \infty$. Write $\Gamma = \Bar{X(\R)} = X(\R) \cup \{(\uprho_-,0), (\uprho_+,0)\}$, and define $\rho : \R \to [\uprho_-,\uprho_\star]$ via $\rho = e_1\cdot X$.
    
    Let $r>0$ be such that $R \subset B(0,r)$ and pick $\varphi \in C^\infty_c(\R^2)$ such that $\varphi =1$ in $B(0,2r)$.  We then define $\Psi \in C^\infty_c(\R^2;\R^2)$ via $\Psi = \varphi \Phi$, and let $\psi : \R \times \R^2 \to \R^2$ denote its flow map, i.e. $\partial_t \psi(t,x) = \Psi(\psi(t,x))$ and $\psi(0,x) = x$ for $t\in\R$ and $x\in\R^2$. Standard ODE theory guarantees that $\psi(t,\cdot) : \R^2 \to \R^2$ is a smooth diffeomorphism for $t \in \R$, and that $\psi(t, \psi(s,x)) = \psi(t+s,x)$ for $s,t \in \R$ and $x \in \R^2$.  Since $\Phi$ and $\Psi$ agree in $B(0,2r)$, we have that $X(t) = \psi(t,X_0)$.  Additionally, since $\Psi$ is Lipschitz, we can choose a constant $L >0$ such that 
\begin{equation}\label{he_said_my_lifes_a_bore_1}
    \abs{\psi(t,x) - \psi(t,y)} \le e^{\abs{t}L} \abs{x-y} \text{ for all } t \in \R \text{ and }x,y \in \R^2.
\end{equation}

Next, we define the smooth diffeomorphism $f=\psi(-1,\cdot) : \R^2 \to \R^2$.  Theorems 6.9 and 6.17 in Irwin~\cite{MR1867353} then provide $\chi>0$ and constants $A_0,A_1>0$ such that 
\begin{equation}\label{he_said_my_lifes_a_bore_2}
    \abs{f^n(x) - f^n(y)} \le A_0 e^{-A_1 n} \abs{x-y} \text{ for all } x,y \in \Gamma \cap \Bar{B((\uprho_-,0),\chi)} \text{ and } n \in \N,
\end{equation}
where $f^n$ denotes the $n-$fold iteration of $f$.  Pick $T >0$ large enough such that $X(t) \in \Bar{B((\uprho_-,0),\chi)}$ for $t \le -T$.  Now let $t \le -T$ and write $t =-T -n - \tau$ for $n \in \N$ and $\tau \in [0,1)$.   Then for $y = (\uprho_-,0)$ we can combine~\eqref{he_said_my_lifes_a_bore_1} and~\eqref{he_said_my_lifes_a_bore_2} to bound
\begin{multline}
    \abs{X(t) - y} = \abs{\psi(t,X_0) - \psi(t,y)} = \abs{\psi(-n,\psi(-T-\tau,X_0)) - \psi(-n,\psi(-T-\tau,y))}  \\
    = \abs{f^n(\psi(-T-\tau,X_0) - f^n(\psi(-T-\tau,y))} 
    \le A_0 e^{-A_1 n} \abs{\psi(-T-\tau,X_0) - \psi(-T-\tau,y)} \\
    \le A_0 e^{-A_1 n} e^{L(T+\tau)} \abs{X_0-y} = A_0  e^{A_1 t} e^{(L-A_1)(T+\tau)} \abs{X_0-y} \le A_0 e^{(L-A_1)(T+1)} e^{A_1 t}  \abs{X_0-y}.
\end{multline}
From this we deduce that there exists $M_0>0$ such that 
\begin{equation}\label{he_said_my_lifes_a_bore_3}
    \sup_{t \le 0} e^{-A_1 t} \abs{X(t)-(\uprho_-,0) } = M_0 < \infty.
\end{equation}

On the other hand, we know from Lemma~\ref{lem on linear character of equilibira} that the point $(\uprho_+,0)$ is an asymptotically stable equilibrium point for $\Phi$.  Consequently, there exists $M_1,A_2>0$ such that 
\begin{equation}\label{he_said_my_lifes_a_bore_4}
    \sup_{t \ge 0} e^{A_2 t} \abs{X(t)-(\uprho_+,0) } = M_1 < \infty.
\end{equation}
The left hand bound in~\eqref{who said your life's a bore} now follows from~\eqref{he_said_my_lifes_a_bore_3} and~\eqref{he_said_my_lifes_a_bore_4} upon setting $\alpha = \min\{A_1,A_2\} >0$.  

We now prove the right hand bound in~\eqref{who said your life's a bore}.  Since 
\begin{equation}
    X' = \Phi(X) = \Phi(X) - \Phi(\uprho_\pm,0)
\end{equation}
and $\Phi$ is Lipschitz in $\Bar{B(0,r)}$, the bounds~\eqref{he_said_my_lifes_a_bore_3} and~\eqref{he_said_my_lifes_a_bore_4} imply that $\sup_{t \in \R} e^{\alpha \abs{t}} \tabs{X'(t)} < \infty$. In turn, we may differentiate the ODE $X' = \Phi(X)$  $m$ times and apply the Fa\'a di Bruno formula to bound 
\begin{equation}
\tabs{X^{(m+1)}(t)} \lesssim_\Phi  \sum_{k=1}^m \sum_{n_1 + \cdots + n_k = m}  \prod_{j=1}^k \tabs{X^{(n_j)}(t)}.   
\end{equation}
This and a simple induction argument reveal that $\sup_{t \in \R} e^{\alpha \abs{t}} \tabs{X^{(m)}(t)} < \infty$ for all $1\le m\in\N$, which then yields the desired result.
\end{proof}

\begin{rmkC}[Nonuniqueness of heteroclinic orbit solutions]\label{rmk on translation degeneracy}
    The maps $\rho:\R\to[\uprho_-,\uprho_\star]$ produced by Theorem~\ref{thm on distinguished heteroclinic orbits} are not uniquely determined by the properties: 1) $\rho$ is a solution to the Li\'enard equation~\eqref{The shallow water ODEs v3} and 2) $\rho$ has the limits $\uprho_{\pm}$ at $\pm\infty$. There is, in fact, a translation invariance and we find that for all $t_0\in\R$ the function $\rho(\cdot - t_0)$ also satisfies these two properties.
\end{rmkC}

\subsection{Perturbation theory for heteroclinic orbits}\label{subsection on perturbation theory for heteroclinic orbits}

In contrast with the analysis of the specific ODEs of Section~\ref{subsection on distinguished shallow water bore solutions}, our aim here is to consider in abstract what happens when we have heteroclinic orbit solutions to a certain autonomous ODE and then make a small nonautonomous perturbation to the equations.  We will show that if the original heteroclinic connects a hyperbolic point to an asymptotically stable point, then under a general class of perturbations the nonautonomous equation admits a nearby global solution; the dependence of the family of solutions on the perturbation is also studied. Of course, the dynamical systems literature is ancient and vast and may contain a version of our main result here, Theorem~\ref{thm on perturbations of heteroclinic orbits}; as such, we make no claim that this material is entirely novel. We refer, for instance, to Barreira and Valls~\cite{MR2368551} for related results on nonautonomous perturbation theory near hyperbolic equilibria. However, we are unaware of a reference with the precise form of the result we need.

\begin{defnC}[Hypotheses for heteroclinic orbit perturbation theory]\label{defn of hypotheses for heteroclininc perturbation theory}
    We define the following for use in this subsection.
    \begin{enumerate}
        \item Let $\Lambda$ be a complete bounded metric space with metric $\m{dist}_{\Lambda}$; suppose that $\m{diam}\Lambda\le 2$.
        \item Let $\Psi:\Lambda\times\R\times\R^d\to\R^d$ be a continuous function admitting the following decomposition
            \begin{equation}\label{this is a sparse equation}
                \Psi(\lambda,t,x) = \Psi_0\tp{x} + \Psi_1\tp{\lambda,t,x}
            \end{equation}
        where $\Psi_0:\R^d\to\R^d$ is smooth and for each $\lambda\in\Lambda$ we have that $\Psi_1(\lambda,\cdot,\cdot):\R\times\R^d\to\R^d$ is smooth for all $k_1,k_2\in\N$ the partial derivative $D_{\R}^{k_1}D_{\R^d}^{k_2}\Psi_1:\Lambda\times\R\times\R^d\to\R^d$ is continuous. Here we are using $D_{\R}$ and $D_{\R^d}$ to indicate differentiation with respect to the $\R$ and $\R^d$ factors in the domain of $\Psi_1$.
        
        \item We suppose that $\Psi_0^{-1}\tcb{0} \supseteq \tcb{\m{x}_{\m{h}},\m{x}_{\m{a}}}$ with $\m{x}_{\m{h}}\neq\m{x}_{\m{a}}$ and denote
        \begin{equation}
            M_{\m{h}} = D_{\R^d}\Psi_0(\m{x}_{\m{h}})\text{ and }M_{\m{a}} = D_{\R^d}\Psi_0\tp{\m{x}_{\m{a}}}.
        \end{equation}
        We posit further
        \begin{enumerate}
            \item The matrix $M_{\m{a}}$ is asymptotically stable in the sense that its spectrum has only negative real part.
            \item The matrix $M_{\m{h}}$ is hyperbolic in the sense that there exists a direct sum decomposition of nontrivial $M_{\m{h}}$-invariant subspaces $\R^d = E^{\m{s}}\oplus E^{\m{u}}$ with $M_{\m{h}}E^{\m{s}}\subseteq E^{\m{s}}$, $M_{\m{h}}E^{\m{u}}\subseteq E^{\m{u}}$ such that $M_{\m{h}}\res E^{\m{s}}$ has only spectrum with negative real part and $M_{\m{h}}\res E^{\m{u}}$ as only spectrum of positive real part. Associated to this direct sum decomposition, we have the projection operators $\Pi_{\m{s}}:\R^d\to E^{\m{s}}$ and $\Pi_{\m{u}}:\R^d\to E^{\m{u}}$.
        \end{enumerate}
        These spectral hypotheses on the matrices $M_{\m{h}}$ and $M_{\m{a}}$ provide  constants $K,\al\in\R^+$ such that for all $x\in\R^d$ and all $t\in\R$ we have the following semigroup estimates
        \begin{equation}\label{_the semigroup bounds_}
            |e^{tM_{\m{a}}}x|\le Ke^{-\al t}|x|,\quad|e^{tM_{\m{h}}}\Pi_{\m{s}}x|\le Ke^{-\al t}|\Pi_{\m{s}}x|,
            \text{ and }
            |e^{tM_{\m{h}}}\Pi_{\m{u}}x|\le Ke^{\al t}|\Pi_{\m{u}}x|.
        \end{equation}

        \item We give quantitative measures of the functions $\Psi_0$ and $\Psi_1$ appearing in the decomposition of $\Psi$~\eqref{this is a sparse equation}. For fixed $R\in\R^+$ large enough so that $\m{x}_{\m{h}},\m{x}_{\m{a}}\in\Bar{B(0,R/2)}$, define, for $n\in\N$, $\lambda,\tilde{\lambda}\in\Lambda$ (with $\lambda\neq\tilde{\lambda}$), and $t\in\R$
        \begin{equation}\label{important_quantities_1}
            \Sampi_n = \sup_{x\in\Bar{B(0,2R)}}\sum_{k=0}^n|D_{\R^d}^k\Psi_0(x)|,\quad \kkappa_n\tp{\lambda,t} = \sup_{x\in\Bar{B(0,2R)}}\sum_{k_1+k_2\le n}|D_{\R}^{k_1}D_{\R^d}^{k_2}\Psi_1\tp{\lambda,t,x}|,
        \end{equation}
        and
        \begin{equation}\label{important_quantities_2}
            \qoppa_n(\lambda,\tilde{\lambda},t) =  \tp{\m{dist}_\Lambda\tp{\lambda,\tilde{\lambda}}}^{-1}\sup_{x\in\Bar{B(0,2R)}}\sum_{k_1+k_2\le n}\tabs{D_{\R}^{k_1}D_{\R^d}^{k_2}\Psi_1\tp{\lambda,t,x} - D_{\R}^{k_1}D_{\R^d}^{k_2}\Psi_1\tp{\tilde{\lambda},t,x}}.
        \end{equation}
        We also define the values
        \begin{equation}\label{important_quantities_3}
            \Bar{\kkappa}_n = \sup_{\lambda\in\Lambda}\tnorm{\kkappa_n\tp{\lambda,\cdot}}_{\tp{L^\infty\cap L^2}\tp{\R}},\quad\Bar{\qoppa}_n = \sup_{\lambda,\tilde{\lambda}\in\Lambda,\;\lambda\neq\tilde{\lambda}}\tnorm{\qoppa_n\tp{\lambda,\tilde{\lambda},\cdot}}_{\tp{L^\infty\cap L^2}\tp{\R}}
        \end{equation}
        and suppose that $\Bar{\kkappa}_n,\Bar{\qoppa}_n<\infty$ for each $n\in\N$. 
    \end{enumerate}
\end{defnC}

Our interest lies in the study of smooth solutions $X:I\to\R^d$ to the family of ODEs
\begin{equation}\label{The general peturbed ODE IVP}
    X'(t) = \Psi\tp{\lambda,t,X(t)},\quad X(T) = X_{\m{i}}
\end{equation}
with $T\in I$, $I\subseteq\R$ an interval, $X_{\m{i}}\in\R^d$, and $\Psi$ as in the second item of Definition~\ref{defn of hypotheses for heteroclininc perturbation theory}. Since $\Psi$ is smooth and hence locally Lipschitz, standard ODE theory implies that solutions to~\eqref{The general peturbed ODE IVP} are always unique and defined in a maximal open interval containing $T$, which we denote by  $\mathcal{J}_{\m{max}}^{\lambda,T}(X_{\m{i}}) \subseteq \R$ in order to indicate the dependence on the parameter $\lambda \in \Lambda$ and the data $X_{\m{i}} \in \R^d$.  This allows us to uniquely define the flow $\psi_\lambda$ associated with~\eqref{The general peturbed ODE IVP} as the map $\psi_\lambda\tp{\cdot,T,X_{\m{i}}}: \mathcal{J}_{\m{max}}^{\lambda,T}(X_{\m{i}}) \to \R^d $ via $\psi_\lambda\tp{t,T,X_{\m{i}}} = X(t)$. We will also use the notation
\begin{equation}\label{pin it down}
    \Psi_{\m{q}}(\lambda,t,x) = \Psi_0\tp{\m{x}_{\m{q}} + x} - M_{\m{q}}x + \Psi_1\tp{\lambda,t,\m{x}_{\m{q}} + x}
\end{equation}
for $\m{q}\in\tcb{\m{h},\m{a}}$, $\lambda\in\Lambda$, $t\in\R$, and $x\in\R^d$.

We are now ready to state our first result, which concerns the behavior of the hyperbolic equilibrium $\m{x}_{\m{h}}$. This is a generalization of the classical unstable manifold theorem for dynamical systems. A somewhat related general theory containing results on nonautonomous stable manifold theory can be found in the work of Barreira and Valls~\cite{MR2368551}.

\begin{propC}[Unstable manifolds with parameter and nonautonomy]\label{prop on unstable manifolds with parameter and nonautonomy}
    Assume Definition~\ref{defn of hypotheses for heteroclininc perturbation theory}. There exist $0<\del_{1},\theta,\sig\le1$, depending on $R$, $\Sampi_2$, $K$, and $\al$, such that if $\Bar{\kkappa}_1\le\del_1$ then for each $T\le 0$ there exists a Lipschitz continuous map
    \begin{equation}\label{v is for very good manifold}
        \pmb{V}_T:\Lambda\times \tp{E^{\m{u}}\cap\Bar{B(0,\sig\theta R)}}\to E^{\m{s}}\cap\Bar{B(0,\theta R)}
    \end{equation}
    and the following hold for $\lambda\in\Lambda$, $T\in(-\infty,0]$.
    \begin{enumerate}
        \item The set of terminal conditions with bounded and square summable backward trajectories for~\eqref{The general peturbed ODE IVP}, namely 
        \begin{multline}\label{the unstable set}
            \mathcal{I}\tp{\lambda,T} = \{X_{\m{i}}\in\Bar{B(\m{x}_{\m{h}},\theta R)}\;:\;|\Pi_{\m{u}}\tp{X_{\m{i}} - \m{x}_{\m{h}}}|\le\sig\theta R, \; (-\infty,T] \subseteq \mathcal{J}_{\m{max}}^{\lambda,T}(X_{\m{i}}), 
         \\   \text{ and } \tnorm{\psi_\lambda\tp{\cdot,T,X_{\m{i}}} - \m{x}_{\m{h}}}_{\tp{L^\infty\cap L^2}\tp{(-\infty,T)}}\le\theta R\},
        \end{multline}
        is nonempty. In fact, we have the equality
        \begin{equation}
            \mathcal{I}\tp{\lambda,T} = \tcb{\m{x}_{\m{h}} + x + \pmb{V}_T\tp{\lambda,x}\;:\;x\in E^{\m{u}}\cap\Bar{B(0,\sig\theta R)}}
        \end{equation}
        \item The induced solution map
        \begin{equation}\label{the _induced_ solution map}
            \Lambda\times\tp{E^{\m{u}}\cap\Bar{B(0,\sig\theta R)}}\ni\tp{\lambda,x}\mapsto\psi_\lambda\tp{\cdot,T,\m{x}_{\m{h}} + x + \pmb{V}_T\tp{\lambda,x}} - \m{x}_{\m{h}}\in\tp{C^1\cap W^{1,\infty}\cap H^1}\tp{(-\infty,T];\R^d}
        \end{equation}
        is Lipschitz continuous with constant depending only on $R$, $\Sampi_2$, $K$, $\al$, and $\Bar{\qoppa}_0$.
        \item There exists a constant $C\in\R^+$, depending only on $R$, $\Sampi_2$, $K$, and $\al$, such that we have the estimate
        \begin{equation}
            |\pmb{V}_T\tp{\lambda,x} - \pmb{V}_T\tp{\tilde{\lambda},x}|\le C\Bar{\qoppa}_0\m{dist}_{\Lambda}\tp{\lambda,\tilde{\lambda}}
        \end{equation}
        for all $\lambda,\tilde{\lambda}\in\Lambda$ and $x\in E^{\m{u}}\cap\Bar{B(0,\sig\theta R)}$.
    \end{enumerate}
\end{propC}
\begin{proof}
    We divide the proof into several steps.

    \emph{Step 1 -- Auxiliary linear problem:}      We first embark on the  study of the auxiliary linear inhomogeneous equation
    \begin{equation}\label{_the inhomogeneous linear equation_}
        X'(t) = M_{\m{h}}X(t) + f(t) \text{ for } t\in(-\infty,T],
    \end{equation}
    where $T\in(-\infty,0]$ and  $f:(-\infty,T]\to\R^d$ is a given bounded and continuous function, seeking bounded and continuously differentiable solutions $X:(-\infty,T]\to\R^d$. Note that standard linear ODE theory implies that any solution to~\eqref{_the inhomogeneous linear equation_} is necessarily global in time (but not necessarily globally bounded).  We claim the following equivalence: for such $f$ fixed, $X:(-\infty,T]\to\R^d$ is a bounded and continuously differentiable solution to~\eqref{_the inhomogeneous linear equation_} if and only if there exists $x_{\m{u}} \in E^{\m{u}}$ such that 
    \begin{equation}\label{_representation formula for bounded solutions_}
        X(t) = e^{\tp{t - T}M_{\m{h}}}x_{\m{u}} - \int_{t}^T e^{\tp{t - \tau}M_{\m{h}}}\Pi_{\m{u}}f(\tau)\;\m{d}\tau + \int_{-\infty}^{t}e^{\tp{t - \tau}M_{\m{h}}}\Pi_{\m{s}}f(\tau)\;\m{d}\tau \text{ for } t \le T.
    \end{equation}
    Note that this equivalence determines the terminal data $X(T)$ in terms of $x_{\m{u}}$ and $f$ via 
    \begin{equation}
        X(T) = \Pi_{\m{u}}X\tp{T} + \Pi_{\m{s}}X\tp{T} \text{ with }\Pi_{\m{u}}X\tp{T}=x_{\m{u}} \text{ and } \Pi_{\m{s}}X\tp{T}= \int_{-\infty}^{T}e^{\tp{T - \tau}M_{\m{h}}}\Pi_{\m{s}}f(\tau)\;\m{d}\tau.
    \end{equation}

    One direction of the claim is easy: due to~\eqref{_the semigroup bounds_} we readily check that the formula~\eqref{_representation formula for bounded solutions_} defines a bounded and continuously differentiable solution to~\eqref{_the inhomogeneous linear equation_}. It remains to establish the converse.  The Duhamel formula ensures  that for any choice of $r,t\le T$ and bounded solution $X$ to~\eqref{_the inhomogeneous linear equation_} we have
    \begin{equation}\label{_Duhamels principle_}
        X(t) = e^{\tp{t - r}M_{\m{h}}}X(r) + \int_{r}^{t}e^{\tp{t - \tau}M_{\m{h}}}f(\tau)\;\m{d}\tau.
    \end{equation}
    We apply  $\Pi_{\m{s}}$ to~\eqref{_Duhamels principle_} and send $r\to-\infty$;  since $X$ is bounded and we have the semigroup bounds~\eqref{_the semigroup bounds_},  the first term on the right hand side of~\eqref{_Duhamels principle_} vanishes, and we are left with the formula for the stable part of the solution:
    \begin{equation}\label{_the unstable part of X_}
        \Pi_{\m{s}}X(t) = \int_{-\infty}^t e^{\tp{t - \tau}M_{\m{h}}}\Pi_{\m{s}}f(\tau)\;\m{d}\tau.
    \end{equation}
    Similarly, we can take $r=T$ in~\eqref{_Duhamels principle_} and apply the operator $\Pi_{\m{u}}$ to deduce that the unstable part of $X$ is given by
    \begin{equation}\label{_the stable part of X_}
        \Pi_{\m{u}}X(t) = e^{(t-T)M_{\m{h}}}\Pi_{\m{u}}X(T) - \int_{t}^Te^{\tp{t - \tau}M_{\m{h}}}\Pi_{\m{u}}f(\tau)\;\m{d}\tau.
    \end{equation}
    We sum~\eqref{_the unstable part of X_} and~\eqref{_the stable part of X_} to deduce formula~\eqref{_representation formula for bounded solutions_}, which completes the proof of the converse and of the claimed equivalence.

   \emph{Step 2 -- Solution operator and fixed point formulation:}   By using~\eqref{_the semigroup bounds_} with the previous equivalence and formula~\eqref{_representation formula for bounded solutions_}, we deduce the following estimate on any bounded solution to~\eqref{_the inhomogeneous linear equation_}:
    \begin{equation}\label{_bounds on the linear inhomogeneous problem_}
        \sup_{t\le T}\tabs{X(t)}\le K\tabs{\Pi_{\m{u}}X(T)} + \tp{2K/\al}\sup_{t\le T}\tabs{f(t)}.
    \end{equation}
    If we suppose further that $f$ is square summable, then any bounded solution to~\eqref{_the inhomogeneous linear equation_} is also square summable and from~\eqref{_representation formula for bounded solutions_} and Young's inequality we may also deduce the estimate
    \begin{equation}\label{_more bounds on the linear inhomogeneous problem_}
        \bp{\int_{-\infty}^{T}|X(t)|^2\;\m{d}t}^{1/2}\le\f{K}{\tp{2\al}^{1/2}}\tabs{\Pi_{\m{u}}X(T)} + \f{2K}{\al}\bp{\int_{-\infty}^T|f(t)|^2\;\m{d}t}^{1/2}.
    \end{equation}
    The estimates~\eqref{_bounds on the linear inhomogeneous problem_} and~\eqref{_more bounds on the linear inhomogeneous problem_} allow us to define the bounded linear operator
    \begin{equation}
        \pmb{L}_T:E^{\m{u}}\times\tp{C^0\cap L^\infty\cap L^2}\tp{(-\infty,T];\R^d}\to\tp{C^0\cap L^\infty\cap L^2}\tp{(-\infty,T];\R^d} \text{ via } \pmb{L}_T\tp{x_{\m{u}},f} = X
    \end{equation}
    where $X$ is given by~\eqref{_representation formula for bounded solutions_}.

    Next, we consider the differential equation in the left of~\eqref{The general peturbed ODE IVP}. Suppose that we are given a bounded, square summable, and continuously differentiable function $X:(-\infty,T]\to\R^d$ such that $X + \m{x}_{\m{h}}$ is a solution to these equations for some value of $\lambda\in\Lambda$. Then we also have a bounded solution to the linear inhomogeneous equation~\eqref{_the inhomogeneous linear equation_} with 
    \begin{equation}\label{_f is for feedback_}
        f\tp{t} = \Psi_{\m{h}}\tp{\lambda,t,X(t)},\quad t\in(-\infty,T],
    \end{equation}
    where $\Psi_{\m{h}}$ is as in~\eqref{pin it down}. The finiteness of~\eqref{important_quantities_1} and~\eqref{important_quantities_3} ensure that the boundedness and square summability of $X$ implies the boundedness and square summability of $f$ in~\eqref{_f is for feedback_}. In fact, we calculate directly that
    \begin{equation}\label{u is for ukelele}
        \tnorm{f}_{L^\infty\cap L^2} \le \Sampi_2\tnorm{X}^2_{L^\infty\cap L^2} + \Bar{\kkappa}_0
    \end{equation}
    Hence the previous equivalence~\eqref{_representation formula for bounded solutions_} forces $X$ to obey the fixed point identity
    \begin{equation}\label{_the fixed point identity_}
        X = \pmb{L}_T\tp{\Pi_{\m{u}}X(T),\Psi_{\m{h}}\tp{\lambda,\cdot,X}}.
    \end{equation}

    \emph{Step 3 -- Fixed point argument:}  With the above considerations in mind, we will now set up a fixed point argument that produces solutions to~\eqref{_the fixed point identity_} (and hence, by adding $\m{x}_{\m{h}}$, solutions to~\eqref{The general peturbed ODE IVP}). Given $0<\theta,\sig\le 1$, define the operator
    \begin{equation}
        \pmb{N}_T:\Lambda\times \tp{\Bar{B(0,\sig\theta R)}\cap E^{\m{u}}}\times\tp{\Bar{B(0,\theta R)}\cap \tp{C^0\cap L^\infty\cap L^2}\tp{(-\infty,T];\R^d}}\to\tp{C^0\cap L^\infty\cap L^2}\tp{(-\infty,T];\R^d}
    \end{equation}
    via $\pmb{N}_T\tp{\lambda,x_{\m{u}},X} = \pmb{L}_T\tp{x_{\m{u}},\Psi_{\m{h}}\tp{\lambda,\cdot,X}}$.

    Due to the bounds on $\pmb{L}_T$ from~\eqref{_bounds on the linear inhomogeneous problem_} and~\eqref{_more bounds on the linear inhomogeneous problem_}, along with~\eqref{u is for ukelele}, we deduce that
    \begin{equation}
        \tnorm{\pmb{N}_T\tp{\lambda,x_{\m{u}},X}}_{L^\infty\cap L^2}\le K\tp{1 + \tp{2\al}^{-1/2}}\sig\theta R + \tp{2K/\al}\tp{\Sampi_2\tp{\theta R}^2 + \Bar{\kkappa}_0}
    \end{equation}
    and
    \begin{equation}
        \tnorm{\pmb{N}_T\tp{\lambda,x_{\m{u}},X} - \pmb{N}_T\tp{\lambda,x_{\m{u}},Y}}_{L^\infty\cap L^2}\le\tp{2K/\al}\tp{\Sampi_2\theta R + \Bar{\kkappa}_1}\tnorm{X - Y}_{L^\infty\cap L^2}.
    \end{equation}
    Therefore, if we take
    \begin{equation}\label{___values___}
        \theta=\min\tcb{1,\al\tp{6R\Sampi_2 K}^{-1}},\quad\sig=\min\tcb{1,\tp{3K\tp{1 + \tp{2\al}^{-1/2}}}^{-1}},\quad\Bar{\kkappa}_1\le\al\min\tcb{1,\theta R}\tp{6K}^{-1}
    \end{equation}
    then  $\pmb{N}_T\tp{\lambda,x_{\m{u}},\cdot}$ maps $\Bar{B(0,\theta R)}$ into itself with a Lipschitz constant at most $2/3$.   We also check that $\pmb{N}_T$ is Lipschitz in its first and second arguments as well, as it obeys 
    \begin{equation}\label{it turns out this estimate is somewhat important}
        \tnorm{\pmb{N}_T\tp{\lambda,x_{\m{u}},X} - \pmb{N}_T\tp{\tilde{\lambda},y_{\m{u}},X}}_{L^\infty\cap L^2}\le K\tp{1 + \tp{2\al}^{-1/2}}|x_{\m{u}} - y_{\m{u}}| + \tp{2K/\al}\Bar{\qoppa}_0\m{dist}_{\Lambda}\tp{\lambda,\tilde{\lambda}}
    \end{equation}
    for all $\lambda,\tilde{\lambda}\in\Lambda$, $x_{\m{u}},y_{\m{u}}\in E^{\m{u}}\cap\Bar{B(0,\sig\theta R)}$, and $X\in\Bar{B(0,\theta R)}\cap\tp{C^0\cap L^\infty\cap L^2}$. The Banach contraction mapping principle with parameter (see, for example, Theorem C.7 in Irwin~\cite{MR1867353}) thus applies, and we acquire a Lipschitz continuous map
    \begin{equation}\label{fixed point are more fun than you realize}
        \pmb{F}_T:\Lambda\times\tp{\Bar{B(0,\sig\theta R)}\cap E^{\m{u}}}\to\Bar{B(0,\theta R)}\cap\tp{C^0\cap L^\infty\cap L^2}\tp{(-\infty,T];\R^d}
    \end{equation}
    with the property that $X = \pmb{F}_T\tp{\lambda,x_{\m{u}}}$ is the unique (in the ball $\Bar{B(0,\theta R)}\subset\tp{C^0\cap L^\infty\cap L^2}\tp{(-\infty,T];\R^d}$) solution to $X = \pmb{N}_T\tp{\lambda,x_{\m{u}},X}$ and hence, after adding $\m{x}_{\m{h}}$, is the unique bounded solution to~\eqref{The general peturbed ODE IVP} with $\Pi_{\m{u}}\tp{X_{\m{i}} - \m{x}_{\m{h}}} = x_{\m{u}}$.

    \emph{Step 4 -- Conclusion:}     We may now see that the set~\eqref{the unstable set} is indeed nonempty, since for every choice of $\lambda\in\Lambda$ and $x_{\m{u}}\in E^{\m{u}}$ we have the inclusion $\m{x}_{\m{h}} + [\pmb{F}_T\tp{\lambda,x_{\m{u}}}]\tp{T}\in\mathcal{I}\tp{\lambda,T}$. On the other hand, for any $X_{\m{i}}\in\mathcal{I}\tp{\lambda,T}$ we have, by definition, an associated bounded and square summable function $X = \psi_\lambda(\cdot,T,X_{\m{i}}) - \m{x}_{\m{h}}$ such that $X+\m{x}_{\m{h}}$ solves~\eqref{The general peturbed ODE IVP}. By the previous calculations, the fixed point identity~\eqref{_the fixed point identity_} holds, so the uniqueness assertion of the contraction mapping principle implies that $X = \pmb{F}_T\tp{\lambda,\Pi_{\m{u}}X(T)}$.
    
    We have just argued that
    \begin{equation}
        \mathcal{I}\tp{\lambda,T} = \tcb{\m{x}_{\m{h}} + [\pmb{F}_T\tp{\lambda,x_{\m{u}}}](T)\;:\;x_{\m{u}}\in E^{\m{u}}\cap\Bar{B(0,\sig\theta R)}}.
    \end{equation}
    From~\eqref{_the fixed point identity_} we also see that $\Pi_{\m{u}}[\pmb{F}_T\tp{\lambda,x_{\m{u}}}](T) = x_{\m{u}}$. Hence we may define the map~\eqref{v is for very good manifold} via 
    \begin{equation}
        \pmb{V}_T(\lambda,x_{\m{u}}) = \Pi_{\m{s}}[\pmb{F}_T\tp{\lambda,x_{\m{u}}}](T)
    \end{equation}
    for $\lambda\in\Lambda$ and $x_{\m{u}}\in E^{\m{u}}\cap\Bar{B(0,\sig\theta R)}$, and conclude that the first item holds.   The third item is a simple calculation that follows from the Lipschitz estimate~\eqref{it turns out this estimate is somewhat important} and the fact that $\pmb{N}_T$ is $2/3$-contractive in its final argument. In fact, we may take $C = 6K/\al$.

    It remains to prove the second item. The map in~\eqref{the _induced_ solution map} is none other than $\pmb{F}_T$, and we therefore know it to be Lipschitz in the target space~\eqref{fixed point are more fun than you realize} of one fewer derivative. For any $\lambda$ and $x_{\m{u}}$ belonging to the domain of $\pmb{F}_T$ we have established that $\pmb{F}_T\tp{\lambda,x_{\m{u}}}$ solves~\eqref{_the fixed point identity_} and is therefore continuously differentiable. We then calculate that its derivative obeys the identity
    \begin{equation}
        \tsb{\pmb{F}_T\tp{\lambda,x_{\m{u}}}}' = M_{\m{h}}\pmb{F}_T\tp{\lambda,x_{\m{u}}} + \Psi_{\m{h}}\tp{\lambda,\cdot,\pmb{F}_{\m{T}}\tp{\lambda,x_{\m{u}}}},
    \end{equation}
    so it is then a simple matter to use the above to promote the Lipschitz continuity of the map~\eqref{fixed point are more fun than you realize} to the claimed~\eqref{the _induced_ solution map}.
\end{proof}

Our second result studies behavior near the asymptotically stable equilibrium $\m{x}_{\m{a}}$ and gives sufficient conditions for the existence of a family of forward in time attracting trajectories for~\eqref{The general peturbed ODE IVP}. This is a generalization of the classical result on the stability of a sink equilibrium (see, for instance, Theorem 1 in Section 2.9 of Perko~\cite{MR1801796}).

\begin{propC}[Generalized attractor stability]\label{prop on attractor stability}
    Assume Definition~\ref{defn of hypotheses for heteroclininc perturbation theory}. Let $T\in(-\infty,0]$ and suppose that $\mathsf{X}:[T,\infty)\to\R^d$ is a given continuously differentiable solution to the equation $\mathsf{X}'(t) = \Psi_0\tp{\mathsf{X}\tp{t}}$ for $t\ge T$, satisfying $\lim_{t\to\infty}\mathsf{X}(t) = \m{x}_{\m{a}}$, $\mathsf{X}(T) = \mathsf{X}_{\m{i}}\in\Bar{B(0,R)}$, and $\tnorm{\mathsf{X} - \m{x}_{\m{a}}}_{L^\infty\cap L^2}\le R/2$.

    There exist $0<\del_2,\theta,\sig\le1$, depending only on $R$, $\Sampi_2$, $K$, $\al$, $T$, $\Psi_0$, and $\mathsf{X}$, such that if $\Bar{\kkappa}_1\le\del_2$ then there is a Lipschitz continuous map
    \begin{equation}\label{the rather attractive solution operator}
        \pmb{U}_T:\Lambda\times\Bar{B\tp{0,\sig\theta R}} \to \tp{C^1\cap W^{1,\infty}\cap H^1}\tp{[T,\infty);\R^d}
    \end{equation}
    such that the following hold.
    \begin{enumerate}
        \item For all $\tp{\lambda,x}$ in the domain of~\eqref{the rather attractive solution operator} we have that $X = \mathsf{X} + \pmb{U}_T\tp{\lambda,x}$ is a solution to the ODE~\eqref{The general peturbed ODE IVP} in the time interval $[T,\infty)$ with $X_{\m{i}} = \mathsf{X}_{\m{i}} + x$.
        \item For all $(\lambda,x)$ in the domain of~\eqref{the rather attractive solution operator} we have $\pmb{U}_T\tp{\lambda,x}\in\Bar{B(0,\theta R)}\subset\tp{C^0\cap L^\infty\cap L^2}\tp{[T,\infty);\R^d}$.
        \item The Lipschitz constant of $\pmb{U}_T$ depends only on $R$, $\Sampi_2$, $K$, $\al$, $\Bar{\qoppa}_0$, $T$, $\Psi_0$, and $\mathsf{X}$.
    \end{enumerate}
\end{propC}
\begin{proof}
    We begin by noting that any bounded and continuously differentiable solution $X:[T,\infty)\to\R^d$ to equation~\eqref{The general peturbed ODE IVP} satisfies 
    \begin{equation}\label{occasionally useful}
        X(t) = e^{\tp{t - T}M_{\m{a}}}X_{\m{i}} + \int_{T}^te^{\tp{t -\tau}M_{\m{a}}}\tp{\Psi\tp{\lambda,\tau,X(\tau)} - M_{\m{a}}X(\tau)}\;\m{d}\tau \text{ for } t \ge T.
    \end{equation}
    A similar identity is satisfied by $\mathsf{X}$ if we replace $\Psi$ above with $\Psi_0$ and $X_{\m{i}}$ by $\mathsf{X}_{\m{i}}$. By subtracting the latter from the former and writing $X = \mathsf{X} + Y$ and $X_{\m{i}} = \mathsf{X}_{\m{i}} + Y_{\m{i}}$, we deduce that $Y:[T,\infty)\to\R^d$ obeys the equation
    \begin{equation}\label{Y fixed point version 1}
        Y(t) = e^{\tp{t - T}M_{\m{a}}}Y_{\m{i}} + \pmb{P}(Y)(t) + \pmb{R}_\lambda\tp{Y}\tp{t} \text{ for } t\ge T,
    \end{equation}
    where the linear map $\pmb{P}$ is defined by
    \begin{equation}\label{the compact perturbation}
        \pmb{P}(Y)(t) = \int_{T}^te^{\tp{t - \tau}M_{\m{a}}}\tp{D_{\R^d}\Psi_0\tp{\mathsf{X}\tp{\tau}} - M_{\m{a}}}Y\tp{\tau}\;\m{d}\tau
    \end{equation}
    and the nonlinear map $\pmb{R}_\lambda$ is given by
    \begin{equation}
        \pmb{R}_\lambda\tp{Y}(t) = \int_T^te^{\tp{t - \tau}M_{\m{a}}}\tsb{\Psi_0\tp{\mathsf{X} + Y} - \Psi_0\tp{\mathsf{X}} - D_{\R^d}\Psi_0\tp{\mathsf{X}}Y + \Psi_1\tp{\lambda,\cdot,\mathsf{X} + Y}}\tp{\tau}\;\m{d}\tau. 
    \end{equation}

    We now aim to establish that $\pmb{P}$ is a bounded and compact operator on the space $\tp{C^0\cap L^\infty\cap L^2}\tp{[T,\infty);\R^d}$.  In light of~\eqref{_the semigroup bounds_}, it's clear that $\pmb{P}$ is actually a bounded linear map with values in the space $\tp{C^1\cap W^{1,\infty}\cap H^1}\tp{[T,\infty);\R^d}$, with operator norm depending only on on $K$, $\al$, and $\Sampi_2$.
    The theorems of Rellich–Kondrachov and Arzel\`a-Ascoli establish the compactness of the embedding 
        \begin{equation}
        \tp{C^1\cap W^{1,\infty}\cap H^1}\tp{[T,\tilde{T}];\R^d}\emb\tp{C^0\cap L^{\infty}\cap L^2}\tp{[T,\tilde{T}];\R^d}
    \end{equation}
    for every $\tilde{T}\ge T$.   The compactness of $\pmb{P}$ now follows from this by way of a standard argument that stitches together this compactness on bounded intervals with the uniform decay of the tails granted by the fact that $\lim_{t\to\infty}\tp{D_{\R^d}\Psi_0\tp{\mathsf{X}\tp{t}} - M_{\m{a}}} = 0$.

    With the compactness of $\pmb{P}$ in hand, we now know that the operator $I - \pmb{P} \in\mathcal{L}\tp{\tp{C^0\cap L^\infty\cap L^2}\tp{[T,\infty);\R^d}}$ is Fredholm with index zero. If $Z\in\m{ker}\tp{I - \pmb{P}}$ then a simple application of the Gr\"onwall inequality implies that $Z = 0$. Therefore, the operator $I - \pmb{P}$ is actually invertible and so we may apply $\tp{I - \pmb{P}}^{-1}$ to identity~\eqref{Y fixed point version 1} to see that the perturbation $Y$ obeys the fixed point identity
    \begin{equation}\label{Y fixed point version 2}
        Y = \pmb{Q}Y_{\m{i}} + \tp{I - \pmb{P}}^{-1}\circ\pmb{R}_{\lambda}\tp{Y}.
    \end{equation}
    where $\pmb{Q}Y_{\m{i}}$ denotes $\tp{I - \pmb{P}}^{-1}$ applied to the function $[T,\infty)\ni t\mapsto e^{\tp{t - T}M_{\m{a}}}Y_{\m{i}}\in\R^d$.

    Next, we set up a contraction mapping argument that produces solutions to~\eqref{Y fixed point version 2}. Given $0<\theta,\sig<1$ we define the map
    \begin{equation}
        \pmb{N}_T:\Lambda\times\Bar{B(0,\sig\theta R)}\times\tp{\Bar{B(0,\theta R)}\cap\tp{C^0\cap L^\infty\cap L^2}\tp{[T,\infty);\R^d}} \to \tp{C^0\cap L^\infty\cap L^2}\tp{[T,\infty);\R^d}
    \end{equation}
    via $\pmb{N}_T\tp{\lambda,Y_{\m{i}},Y} = \pmb{Q}Y_{\m{i}} + \tp{I - \pmb{P}}^{-1}\circ\pmb{R}_\lambda\tp{Y}$. We directly estimate that
    \begin{equation}
        \tnorm{\pmb{N}_T\tp{\lambda,Y_{\m{i}},Y}}_{L^\infty\cap L^2}\le\tnorm{\pmb{Q}}\sig\theta R + \tnorm{\tp{I - \pmb{P}}^{-1}}\tp{K/\al}\tp{\Sampi_2\tp{\theta R}^2 + \Bar{\kkappa}_0}.
    \end{equation}
    Therefore, if we take
    \begin{equation}
        \theta = \min\tcb{1,\al\tp{3R\Sampi_2K\tnorm{\tp{I - \pmb{P}}^{-1}}}^{-1}},\quad\sig = \min\tcb{1,\tp{3\tnorm{\pmb{Q}}}^{-1}},\quad\Bar{\kkappa}_1\le\al\min\tcb{1,\theta R}\tp{3K\tnorm{(I - \pmb{P})^{-1}}}^{-1},
    \end{equation}
    then  $\pmb{N}_T\tp{\lambda,Y_{\m{i}},\cdot}$ maps $\Bar{B(0,\theta R)}$ into itself and
    \begin{equation}
        \tnorm{\pmb{N}_T\tp{\lambda,Y_{\m{i}},Y} - \pmb{N}_T\tp{\lambda,Y_{\m{i}},Z}}_{L^\infty\cap L^2}\le\tnorm{\tp{I - \pmb{P}}^{-1}}\tp{K/\al}\tp{\theta R + \Bar{\kkappa}_1}\tnorm{Y - Z}_{L^\infty\cap L^2}\le\tp{2/3}\tnorm{Y - Z}_{L^\infty\cap L^2}
    \end{equation}
    The latter estimate shows that $\pmb{N}_T$ is a contraction in its final argument, but we can also check that $\pmb{N}_T$ is Lipschitz in its first and second arguments as well: 
    \begin{equation}
    \tnorm{\pmb{N}_T\tp{\lambda,Y_{\m{i}},Y} - \pmb{N}_{T}\tp{\tilde{\lambda},Z_{\m{i}},Y}}_{L^\infty\cap L^2}\le\tnorm{Q}\tabs{Y_{\m{i}} - Z_{\m{i}}} + \tp{K/\al}\Bar{\qoppa}_0\m{dist}_{\Lambda}\tp{\lambda,\tilde{\lambda}}.
    \end{equation}
    The Banach contraction mapping principle with parameter, see e.g. Theorem C.7 in Irwin~\cite{MR1867353}, then provides a Lipschitz continuous map
    \begin{equation}\label{wow another contraction}
        \pmb{U}_T:\Lambda\times\Bar{B(0,\sig\theta R)}\to\tp{C^0\cap L^\infty\cap L^2}\tp{[T,\infty);\R^d}
    \end{equation}
    with the property that $Y = \pmb{U}_T\tp{\lambda,Y_{\m{i}}}$ is the unique (in the ball $\Bar{B(0,\theta R)}$) solution to $Y = \pmb{N}_T\tp{\lambda,Y_{\m{i}},Y}$. We can undo the sequence of reductions made at the beginning of the proof to deduce that the function $X = \mathsf{X} + Y$ solves~\eqref{The general peturbed ODE IVP} with $X_{\m{i}} = \mathsf{X}_{\m{i}} + Y_{\m{i}}$.

    We have now established the first and second items, and it remains only to show that $\pmb{U}_T$ is a Lipschitz map into the target space given in~\eqref{the rather attractive solution operator}. This is achieved via a similar strategy to that used at the end of the proof of Proposition~\ref{prop on unstable manifolds with parameter and nonautonomy}.  Since $\mathsf{X} + \pmb{U}_T\tp{\lambda,x}$ is a solution to~\eqref{occasionally useful} and $\mathsf{X}$ is continuously differentiable, we deduce that $\pmb{U}_T\tp{\lambda,x}$ is continuously differentiable as well. Hence, its derivative obeys the identity
    \begin{equation}
        [\pmb{U}_T\tp{\lambda,x}]'\tp{t} =\Psi_0\tp{\tp{\mathsf{X} + \pmb{U}_T\tp{\lambda,x}}\tp{t}} - \Psi_0\tp{\mathsf{X}\tp{t}} +  \Psi_1\tp{\lambda,t,\tp{\mathsf{X} + \pmb{U}_T\tp{\lambda,x}}(t)}
    \end{equation}
    for $t\ge T$, and it is then a simple matter to use the above to promote the Lipschitz continuity of the map~\eqref{wow another contraction} to the claimed~\eqref{the rather attractive solution operator}.
\end{proof}

Our third result combines the analysis of Propositions~\ref{prop on unstable manifolds with parameter and nonautonomy} and~\ref{prop on attractor stability} to yield a certain type of stability for heteroclinic orbits.

\begin{thmC}[Perturbations of heteroclinic orbits]\label{thm on perturbations of heteroclinic orbits}
    Assume Definition~\ref{defn of hypotheses for heteroclininc perturbation theory}. Suppose $\lambda_0\in\Lambda$ is such that $\Psi_1\tp{\lambda_0,t,x}=0$ for all $t\in\R$ and $x\in\R^d$.  Further suppose that $\mathsf{X}:\R\to\R^d$ is a given continuously differentiable solution to the equation $\mathsf{X}'\tp{t} = \Psi_0\tp{\mathsf{X}\tp{t}}$ for $t\in\R$, satisfying $\lim_{t\to\infty}\mathsf{X}\tp{t} = \m{x}_{\m{a}}$, $\lim_{t\to-\infty}\mathsf{X}\tp{t} = \m{x}_{\m{h}}$, and $\max\tcb{\tnorm{\mathsf{X} - \m{x}_{\m{h}}}_{\tp{L^\infty\cap L^2}\tp{(-\infty,0)}},\tnorm{\mathsf{X} - \m{x}_{\m{a}}}_{\tp{L^\infty\cap L^2}\tp{(0,\infty)}}}\le R/2$. 
    
    There exists $0<\updelta\le1$, depending only on $R$, $\Sampi_2$, $K$, $\al$, $\Psi_0$, and $\mathsf{X}$,  with the property that if $\max\tcb{\Bar{\kkappa}_1,\Bar{\qoppa}_0}\le\updelta$ then there exists a Lipschitz continuous map
    \begin{equation}\label{B is for bore map}
        \pmb{B}:\Lambda\to\tp{C^1\cap W^{1,\infty}\cap H^1}\tp{\R;\R^d}
    \end{equation}
    possessing the following properties.
    \begin{enumerate}
        \item We have $\pmb{B}\tp{\lambda_0} = 0$, and for all $\lambda\in\Lambda$ the function $X = \mathsf{X} + \pmb{B}\tp{\lambda}$ is a solution to the ODE  $X'(t) = \Psi\tp{\lambda,t,X(t)}$ for $t \in \R$.
        \item The Lipschitz constant of $\pmb{B}$ as in~\eqref{B is for bore map} depends only on $R$, $\Sampi_2$, $K$, $\al$, $\Psi_0$, and $\mathsf{X}$.
    \end{enumerate}
\end{thmC}
\begin{proof}
    Let $\updelta = \max\tcb{\Bar{\kkappa}_1,\Bar{\qoppa}_0}$. The hypotheses on $\mathsf{X}$ ensure that for each $0<\ep\le R$, there exists a $T_\ep\in(-\infty,0]$ with the property that
    \begin{equation}\label{is this using the axiom of choice}
        \sup_{t\le T_\ep}\tabs{\mathsf{X}(t) -  \m{x}_{\m{h}}} + \bp{\int_{-\infty}^{T_\ep}\tabs{\mathsf{X}\tp{\tau} - \m{x}_{\m{h}}}^2\;\m{d}\tau}^{1/2}\le\ep.
    \end{equation}
    Assume that  $0<\updelta\le\del_1$, where $\del_1$ is the parameter from Proposition~\ref{prop on unstable manifolds with parameter and nonautonomy}.  The hypotheses of the proposition are thus satisfied, so its third item provides  $0<\ep_1\le R$ and, for every $T\le0$, a Lipschitz continuous map
    \begin{equation}\label{hike the appalachian trail}
        \pmb{A}_T:\Lambda\times\tp{\Bar{B(0,2\ep_1)}\cap E^{\m{u}}}\to\tp{C^1\cap W^{1,\infty}\cap H^1}\tp{(-\infty,T];\R^d}
    \end{equation}
    such that for all $\lambda$ and $x_{\m{u}}$ belonging to the domain of~\eqref{hike the appalachian trail} the function $X = \m{x}_{\m{h}} + \pmb{A}_T\tp{\lambda,x_{\m{u}}}$ is a solution to~\eqref{The general peturbed ODE IVP} on the interval $(-\infty,T]$ with $X(T)\in\mathcal{I}\tp{\lambda,T}$ the unique element with $\Pi_{\m{u}}\tp{X(T) - \m{x}_{\m{h}}} = x_{\m{u}}$.

    The definition~\eqref{is this using the axiom of choice} of $T_{\ep_1}$ ensures that $\mathsf{X}\tp{T_{\ep_1}}\in\mathcal{I}\tp{\lambda,T}$. Let us define $\Bar{x}_{\m{u}} = \Pi_{\m{u}}\tp{\mathsf{X}\tp{T_{\ep_1}} - \m{x}_{\m{h}}}$. Since $\Psi_1\tp{\lambda_0,\cdot,\cdot}=0$ we have that on $(-\infty,T_{\ep_1}]$
    \begin{equation}\label{uniqueness is kinda cool I guess}
        \m{x}_{\m{h}} + \pmb{A}_{T_{\ep_1}}\tp{\lambda_0,\Bar{x}_{\m{u}}} = \psi_{\lambda_0}\tp{\cdot,T_{\ep_1},\mathsf{X}\tp{T_{\ep_1}}} = \mathsf{X}.
    \end{equation}
    We are now in a position to define the map
    \begin{equation}\label{left half definition}
        \pmb{B}^-:\Lambda\to\tp{C^1\cap W^{1,\infty}\cap H^1}\tp{(-\infty,T_{\ep_1}];\R^d} \text{ via } 
        \pmb{B}^{-}\tp{\lambda} = \m{x}_{\m{h}} - \mathsf{X} + \pmb{A}_{T_{\ep_1}}\tp{\lambda,\Bar{x}_{\m{u}}}.
    \end{equation}
    Since $\mathsf{X} - \m{x}_{\m{h}}\in\tp{C^1\cap W^{1,\infty}\cap H^1}\tp{(-\infty,T_{\ep_1}];\R^d}$, $\pmb{B}^-$ is well-defined.  Moreover, this map inherits the Lipschitz continuity of $\pmb{A}_{T_{\ep_1}}$, and the calculation~\eqref{uniqueness is kinda cool I guess} shows that $\pmb{B}^-\tp{\lambda_0} = 0$. Finally, since  $X^- = \mathsf{X} + \pmb{B}^-\tp{\lambda} = \m{x}_{\m{h}} + \pmb{A}_{T_{\ep_1}}\tp{\lambda,\Bar{x}_{\m{u}}}$, we have that   $(X^-)'(t) = \Psi\tp{\lambda,t,X^-(t)}$ for $t \le T_{\ep_1}$ for every $\lambda\in\Lambda$.

    The goal now is to use Proposition~\ref{prop on attractor stability} to extend the definition of $\pmb{B}^-$ to all of $\R$.  To achieve this we assume in addition that $\updelta\le\del_2$, the parameter from the proposition.  This then  grants us $0<\ep_2\le R$ and a Lipschitz continuous map
    \begin{equation}\label{Tchaikovsky is hard to spell}
        \pmb{U}_{T_{\ep_1}}:\Lambda\times\Bar{B(0,\ep_2)}\to\tp{C^1\cap W^{1,\infty}\cap H^1}\tp{[T_{\ep_1},\infty);\R^d}
    \end{equation}
    such the for all $\lambda$, $x$ in the domain of~\eqref{Tchaikovsky is hard to spell} we have that $X = \mathsf{X} + \pmb{U}_{T_{\ep_1}}\tp{\lambda,x}$ is a solution to~\eqref{The general peturbed ODE IVP} on $[T_{\ep_1},\infty)$ with $X_{\m{i}} = \mathsf{X}\tp{T_{\ep_1}} + x$.

    Further decreasing $\updelta$ if necessary, the third item of Proposition~\ref{prop on unstable manifolds with parameter and nonautonomy} combined with the facts that $\m{diam}\Lambda\le 2$ and $\pmb{B}^-\tp{\lambda_0} = 0$ allows us to arrange for $\tsb{\pmb{B}^-\tp{\lambda}}(T_{\ep_1})\in\Bar{B(0,\ep_2)}$ for all $\lambda\in\Lambda$. In turn, we may define the map
    \begin{equation}
        \pmb{B}^+:\Lambda\to\tp{C^1\cap W^{1,\infty}\cap H^1}\tp{[T_{\ep_1},\infty);\R^d} \text{ via }
        \pmb{B}^+\tp{\lambda} = \pmb{U}_{T_{\ep_1}}\tp{\lambda,\tsb{\pmb{B}^-\tp{\lambda}}\tp{T_{\ep_1}}}.
    \end{equation}
    $\pmb{B}^+$ is a composition of Lipschitz maps and hence Lipschitz as well. For all $\lambda\in\Lambda$, the map $X = \mathsf{X} + \pmb{B}^+\tp{\lambda}$ solves~\eqref{The general peturbed ODE IVP} on $[T_{\ep_1},\infty)$ with $X_{\m{i}} = \mathsf{X}\tp{T_{\ep_{1}}} + \pmb{B}^-\tp{\lambda}\tp{T_{\ep_1}}$. By uniqueness and the identity $\pmb{B}^-\tp{\lambda_0} = 0$, we find that $X = \mathsf{X}$     when $\lambda = \lambda_0$, which in turn implies that that $\pmb{B}^+\tp{\lambda_0} = 0$.

    To complete the proof and construct the map of~\eqref{B is for bore map}, we concatenate $\pmb{B}^-$ and $\pmb{B}^+$; indeed, for $\lambda\in\Lambda$ we define 
    \begin{equation}\label{Patch Adams}
        \tsb{\pmb{B}\tp{\lambda}}\tp{t} = \begin{cases}
            \tsb{\pmb{B}^-\tp{\lambda}}\tp{t}&\text{if }t\le T_{\ep_1},\\
            \tsb{\pmb{B}^+\tp{\lambda}}\tp{t}&\text{if }t>T_{\ep_1}.
        \end{cases}
    \end{equation}
    It is then a simple matter to check that
    \begin{equation}\label{it checks the left and right hand limits on its paper}
        \lim_{t\to T_{\ep_1},\;t<T_{\ep_1}}\tsb{\pmb{B}\tp{\lambda}}\tp{t} =  \m{x}_{\m{h}} - \mathsf{X}\tp{T_{\ep_1}} + \pmb{A}_{T_{\ep_1}}\tp{\lambda,\Bar{x}_u} = \lim_{t\to T_{\ep_1},\;t>T_{\ep_1}}\tsb{\pmb{B}\tp{\lambda}}\tp{t},
    \end{equation}
    and hence $\pmb{B}(\lambda)$ as defined in~\eqref{Patch Adams} is continuous on $\R$.  By design, we also know that $X=\mathsf{X} + \pmb{B}\tp{\lambda}$ solves $X'(t) = \Psi\tp{\lambda,t,X(t)}$ for $t$ in each of the intervals $(-\infty,T_{\ep_1})$ and $\tp{T_{\ep_1},\infty}$. Therefore, \eqref{it checks the left and right hand limits on its paper} implies that $\pmb{B}\tp{\lambda}$ is continuously differentiable in all of $\R$.  Finally, the Lipschitz continuity of $\pmb{B}$ as a map between the spaces of ~\eqref{B is for bore map} is inherited from the Lipschitz continuity of $\pmb{B}^-$ and $\pmb{B}^+$.
\end{proof}

Our final result regarding perturbations of heteroclinic orbits studies higher regularity of the solutions to the perturbed problem.

\begin{coroC}[Improved regularity]\label{coro on improved regularity}
    For each $\N\ni n\ge 2$ the map $\pmb{B}$ of~\eqref{B is for bore map} produced by Theorem~\ref{thm on perturbations of heteroclinic orbits} is actually a Lipschitz continuous between the spaces
    \begin{equation}\label{more derivatives on the bore map makes me more money}
        \pmb{B}:\Lambda\to\tp{C^n\cap W^{n,\infty}\cap H^n}\tp{\R;\R^d}.
    \end{equation}
    Moreover, the Lipschitz constant of~\eqref{more derivatives on the bore map makes me more money} depends only on $R$, $\Sampi_{n+1}$, $\Bar{\kkappa}_{n}$, $\Bar{\qoppa}_{n-1}$, $K$, $\al$, $\Psi_0$, and $\mathsf{X}$.
\end{coroC}
\begin{proof}
    For $\lambda\in\Lambda$ we calculate 
    \begin{equation}
        \tsb{\pmb{B}\tp{\lambda}}'\tp{t} = \Psi_0\tp{\tp{\mathsf{X} + \pmb{B}\tp{\lambda}}\tp{t}} - \Psi_0\tp{\mathsf{X}\tp{t}} + \Psi_1\tp{\lambda,t,\tp{\mathsf{X}\tp{t} + \pmb{B}\tp{\lambda}}\tp{t}} \text{ for } t \in \R.
    \end{equation}
    This identity then allows for a standard inductive regularity promotion argument, which establishes the Lipschitz continuity of~\eqref{more derivatives on the bore map makes me more money}.  We omit further details for the sake of brevity.
\end{proof}


\section{PDEs in thin domains}\label{section on pdes in thin domains}

In Section~\ref{section on analysis and perturbations of shallow water heteroclinic orbits} we have shown the existence of heteroclinic orbit solutions to the one-dimensional shallow water equations~\eqref{The shallow water ODEs}.  We expect the shallow water solutions to describe our desired Navier-Stokes solutions to leading order. The task for this section is for us to develop tools that can be used for the quantitative connection between PDE theory in two-dimensional thin domains and the one-dimensional shallow water ODE theory. These will then be essential throughout Sections~\ref{section on the shallow water and residual equations} and~\ref{section on analysis of bore waves} when we set up and execute a fixed point argument.

\subsection{Thin domain toolbox}\label{subsection on thin domain toolbox}

The purpose of this subsection is to pay special attention to the constants arising in certain Sobolev-theoretic inequalities and constructions in $\ep$-thin domains for $0<\ep\ll1$.

We embark with the study of extension operators and right inverses to the divergence.
\begin{lemC}[Right inverses to the trace and the divergence]\label{lem on anti-trace and anti-divergence}
   Given $\ep,\del\in(0,1)$ the following hold for a constant $C\in\R^+$ depending only on $\del$.
    \begin{enumerate}
        \item There exists a bounded linear operator $\mathcal{E}:H^{1/2}\tp{\Sigma_0} \to H^1\tp{\Omega_\ep}$ such that for all $\phi\in H^{1/2}\tp{\Sigma_0}$ we have 
        \begin{equation}\label{the extension operator properties}
            \m{Tr}_{\Sigma_0}\mathcal{E}\phi = \phi,\quad\sup_{0\le y\le\ep}\tabs{\mathscr{F}\tsb{\mathcal{E}\phi}(\cdot,y)}\le\tabs{\mathscr{F}[\phi]},
            \text{ and }
            \tnorm{\mathcal{E}\phi}_{H^1\tp{\Omega_\ep}}\le C\tnorm{\phi}_{H^{1/2}\tp{\Sigma_0}}.
        \end{equation}
        \item Suppose that $h\in W^{1,\infty}\tp{\R}$ satisfies $\tnorm{h}_{W^{1,\infty}}\le\del^{-1}$ and $\del\le h\le\del^{-1}$. There exists a bounded linear operator $\Uppi_h:L^2\tp{\Omega_\ep}\to H^1\tp{\Omega_\ep;\R^2}$ such that for all $f\in L^2\tp{\Omega_\ep}$ 
        \begin{equation}\label{the anti-divergence properties}
            \grad^{\mathcal{A}_h}\cdot\Uppi_hf = f,\quad\m{Tr}_{\Sigma_0}\Uppi_h f = 0,
            \text{ and }
            \tnorm{\Uppi_h f}_{H^1\tp{\Omega_\ep}}\le C\tnorm{f}_{L^2\tp{\Omega_\ep}},
        \end{equation}
        where $\mathcal{A}_h$ is the geometry matrix as defined in~\eqref{geometry matrices}.
    \end{enumerate}
\end{lemC}
\begin{proof}
    The extension operator $\mathcal{E}$ of the first item can be taken to be the standard Poisson extension: in terms of the horizontal Fourier transform, we define
    \begin{equation}
        \mathscr{F}[\mathcal{E}\phi]\tp{\xi,y} = e^{-y|\xi|}\mathscr{F}[\phi]\tp{\xi} 
        \text{ for }
        \tp{\xi,y}\in\R\times\tp{0,\ep} \text{ and } \phi\in H^{1/2}\tp{\Sigma_0}.
    \end{equation}
    Then~\eqref{the extension operator properties} follow from Plancherel's theorem and elementary computations.

    For the second item we first note that thanks to Proposition 2.1 in Leoni and Tice~\cite{MR4630597} (see also Lemma 2.2 in Stevenson and Tice~\cite{MR4787851}) there exists a bounded linear right inverse to the divergence operator $\tilde{\Uppi}:L^2\tp{\Omega_{1/\del}}\to H^1\tp{\Omega_{1/\del};\R^2}$ satisfying
    \begin{equation}
        \grad\cdot\tilde{\Uppi} f = f
        \text{ and }
        \m{Tr}_{\Sigma_0}\tilde{\Uppi}f = 0 \text{ for all }
        f\in L^2\tp{\Omega_{1/\del}}
    \end{equation}
    and whose operator norm depends only on $\del$. On the other hand, for the $h$-flattening map
    \begin{equation}\label{this is the flattening map}
        F_h^\ep:\begin{cases}
            L^2\tp{\Omega_{\ep h}}\to L^2\tp{\Omega_\ep},\\
            H^1\tp{\Omega_{\ep h}}\to H^1\tp{\Omega_\ep},
        \end{cases}
        \text{ defined by } F^\ep_h f =  f\circ\mathfrak{F}_h,
    \end{equation}
    where $\mathfrak{F}_h:\Omega_\ep\to\Omega_{\ep h}$ is as in Section~\ref{subsection on flattening, etc}, it is straightforward to check that $F^\ep_h$ is invertible and satisfies
    \begin{equation}\label{estimates on these flattening and unflattening maps}
        \sum_{n=0}^1\tp{\tnorm{F^\ep_h}_{\mathcal{L}\tp{H^n\tp{\Omega_{\ep h}};H^n\tp{\Omega_\ep}}} + \tnorm{\tp{F^\ep_h}^{-1}}_{\mathcal{L}\tp{H^n\tp{\Omega_{\ep }};H^n\tp{\Omega_{\ep h}}}} }\le C
    \end{equation}
    for a constant $C$ depending only on $\del$.  We now define $\Uppi_h:L^2\tp{\Omega_\ep}\to H^1\tp{\Omega_\ep;\R^2}$ via
    \begin{equation}
        \Uppi_h f = F^\ep_h\circ \rho_{\ep h}\circ\tilde{\Uppi} \circ \iota_{\ep h}\circ \tp{F^\ep_h}^{-1}f, 
    \end{equation}
    where $\iota_{\ep h}:L^2\tp{\Omega_{\ep h}}\to L^2\tp{\Omega_{\ep/\del}}$ is the (nonexpansive) trivial extension operator and $\rho_{\ep h}: H^{1}\tp{\Omega_{\ep/\del}}\to H^1\tp{\Omega_{\ep h}}$ is the (nonexpansive) restriction operator. Our construction ensures that the operator norm of $\Uppi_h$ depends only on $\del$.  In turn, elementary computations yield the first two identities of~\eqref{the anti-divergence properties}. 
\end{proof}

We next study traces. This necessitates the introduction of boundary norms adapted to thin domains.

\begin{lemC}[Adapted boundary norms]\label{lem on adapted boundary norms}
    There exists $C\in\tp{1,\infty}$ such that for all $\ep\in\tp{0,1}$ and all $\phi\in H^{1/2}\tp{\R}$ the norms
    \begin{equation}\label{the small and large trace norms}
        \tnorm{\phi}^{\tp{\ep,-}}_{H^{1/2}\tp{\R}}= \inf_{\phi_0 + \phi_1 = \phi}\tp{\tnorm{\phi_0}_{H^{1/2}\tp{\R}} + \ep^{1/2}\tnorm{\phi_1}_{H^1\tp{\R}}} 
        \text{ and }
        \tnorm{\phi}^{\tp{\ep,+}}_{H^{1/2}\tp{\R}} = \tnorm{\phi}_{H^{1/2}\tp{\R}} + \ep^{-1/2}\tnorm{\phi}_{L^2\tp{\R}}
    \end{equation}
    satisfy the following.
    \begin{enumerate}
        \item We have
        \begin{equation}
            C^{-1}\ep^{1/2}\tnorm{\phi}_{H^{1/2}\tp{\R}}\le\tnorm{\phi}^{\tp{\ep,-}}_{H^{1/2}\tp{\R}}\le\tnorm{\phi}_{H^{1/2}\tp{\R}}\le\tnorm{\phi}^{\tp{\ep,+}}_{H^{1/2}\tp{\R}}\le\ep^{-1/2}C\tnorm{\phi}_{H^{1/2}\tp{\R}}.
        \end{equation}
        \item We have
        \begin{equation}\label{small trace norm equivalence}
            C^{-1}\tnorm{\phi}^{\tp{\ep,-}}_{H^{1/2}\tp{\R}}\le\bp{\ep\int_{(-1/\ep,1/\ep)}\tbr{\xi}^2\tabs{\mathscr{F}[\phi]\tp{\xi}}^2\;\m{d}\xi + \int_{\R\setminus\tp{-1/\ep,1/\ep}}\tbr{\xi}\tabs{\mathscr{F}[\phi]\tp{\xi}}^2\;\m{d}\xi}^{1/2}\le C\tnorm{\phi}^{\tp{\ep,-}}_{H^{1/2}\tp{\R}}.
        \end{equation}
        \item We have
        \begin{equation}\label{large trace norm equivalence}
            C^{-1}\tnorm{\phi}^{\tp{\ep,+}}_{H^{1/2}\tp{\R}}\le\bp{\ep^{-1}\int_{\tp{-1/\ep,1/\ep}}\tabs{\mathscr{F}[\phi]\tp{\xi}}^2\;\m{d}\xi + \int_{\R\setminus\tp{-1/\ep,1/\ep}}\tbr{\xi}\tabs{\mathscr{F}[\phi]\tp{\xi}}^2\;\m{d}\xi}^{1/2}\le C\tnorm{\phi}^{\tp{\ep,+}}_{H^{1/2}\tp{\R}}.
        \end{equation}
    \end{enumerate}
\end{lemC}
\begin{proof}
    The second and third items readily imply the first, and the third item is an elementary consequence of the Plancherel theorem.  To prove the second item we split a general $\phi\in H^{1/2}\tp{\R}$ according to $\phi = \phi_0 + \phi_1$ with $\phi_0 = \mathds{1}_{(-1/\ep,1/\ep)}(D)\phi$ and $\phi_1 = \mathds{1}_{\R\setminus(-1/\ep,1/\ep)}(D)$ in order to see that the left hand inequality in~\eqref{small trace norm equivalence} holds. To prove the right hand inequality we first note  that
    \begin{multline}\label{Tariffs}
        \ep^{1/2}\tnorm{\mathds{1}_{\tp{-1/\ep,1/\ep}}\tp{D}\psi}_{H^1\tp{\R}}\le C\tnorm{\mathds{1}_{\tp{-1/\ep,1/\ep}}\tp{D}\psi}_{H^{1/2}\tp{\R}},\\\tnorm{\mathds{1}_{\R\setminus\tp{-1/\ep,1/\ep}}\tp{D}\psi}_{H^{1/2}}\le C\ep^{1/2}\tnorm{\mathds{1}_{\R\setminus\tp{-1/\ep,1/\ep}}\tp{D}\psi}_{H^1\tp{\R}}
    \end{multline}
    and then apply~\eqref{Tariffs} to the pieces of a general decomposition of $\phi = \phi_0 + \phi_1$.
\end{proof}

The boundary norms of Lemma~\ref{lem on adapted boundary norms} play an important role our next result. We refer to the work of Fern\'andez Bonder, Mart\'inez, and Rossi~\cite{MR2037752} for more general results on thin domain trace inequalities.

\begin{lemC}[Traces]\label{lem on traces in thin domains}
    There exists $C\in\R^+$ such that the following  hold for all $\ep\in(0,1)$, $f\in H^1\tp{\Omega_\ep}$, and $\phi\in H^{1/2}\tp{\R}$.
    \begin{enumerate}
        \item We have the bound $\tnorm{\m{Tr}_{\Sigma_\ep}f - \m{Tr}_{\Sigma_0}f}_{H^{1/2}\tp{\R}}^{\tp{\ep,+}}\le C\tnorm{f}_{H^{1}\tp{\Omega_\ep}}$.

        \item We have the bound $\tnorm{\m{Tr}_{\Sigma_\ep}f}^{\tp{\ep,-}}_{H^{1/2}\tp{\R}} + \tnorm{\m{Tr}_{\Sigma_0}f}^{\tp{\ep,-}}_{H^{1/2}\tp{\R}}\le C\tnorm{f}_{H^{1}\tp{\Omega_\ep}}$.
        
        \item For $\Sigma\in\tcb{\Sigma_0,\Sigma_\ep}$ we have the bound
        \begin{equation}
            \babs{\int_{\Sigma}\phi f} + \babs{\int_{\Sigma}\tbr{D}^{1/2}\phi\cdot\tbr{D}^{-1/2}\pd_1 f}\le C\tnorm{\phi}^{\tp{\ep,+}}_{H^{1/2}\tp{\R}}\tnorm{f}_{H^{1}\tp{\Omega_\ep}}.
        \end{equation}
        \item If we assume additionally that $\m{Tr}_{\Sigma_0}f = 0$, then
        \begin{equation}
            \babs{\int_{\Sigma_\ep}\phi f} + \babs{\int_{\Sigma_\ep}\tbr{D}^{1/2}\phi\cdot\tbr{D}^{-1/2}\pd_1f}\le C\tnorm {\phi}^{\tp{\ep,-}}_{H^{1/2}\tp{\R}}\tnorm{f}_{H^1\tp{\Omega_\ep}}.
        \end{equation}
    \end{enumerate}
\end{lemC}
\begin{proof}
    For the first item it suffices to prove
    \begin{equation}\label{difference of traces estimates}
        \ep^{-1/2}\tnorm{\m{Tr}_{\Sigma_\ep}f - \m{Tr}_{\Sigma_0}f}_{L^2\tp{\R}} + \tnorm{\m{Tr}_{\Sigma_\ep}f - \m{Tr}_{\Sigma_0}f}_{\dot{H}^{1/2}\tp{\R}}\le C\tnorm{\grad f}_{L^2\tp{\Omega_\ep}}.
    \end{equation}
    To prove this we will reduce to the case for the domain $\Omega_1$ via a scaling argument. The validity of~\eqref{difference of traces estimates} for the domain $\Omega_1$ is a consequence of standard trace theory. Then for  $\ep\in\tp{0,1}$ and $f\in H^1\tp{\Omega_\ep}$  the function defined by $f_\ep\tp{x,y} = f(\ep x,\ep y)$ belongs to $H^1\tp{\Omega_1}$, so we may invoke the special case of~\eqref{difference of traces estimates} in $\Omega_1$.  We get the claimed estimate in $\Omega_\ep$ by then noting that
    \begin{equation}
        \tnorm{\grad f_\ep}_{L^2\tp{\Omega_1}} = \tnorm{\grad f}_{L^2\tp{\Omega_\ep}},\;\tnorm{g_\ep}_{\dot{H}^{1/2}\tp{\R}} = \tnorm{g}_{\dot{H}^{1/2}\tp{\R}},
        \text{ and }
        \tnorm{g_\ep}_{L^2\tp{\R}} = \ep^{-1/2}\tnorm{g}_{L^2\tp{\R}},
    \end{equation}
    where $g_\ep = \tp{\m{Tr}_{\Sigma_1} - \m{Tr}_{\Sigma_0}}f_\ep$ and $g = \tp{\m{Tr}_{\Sigma_1} - \m{Tr}_{\Sigma_0}}f$.

    For the second item we only prove the trace estimate onto $\Sigma_0$, as the case $\Sigma_\ep$ is similar. The restrictions to the strip $\Omega_\ep$ of functions in $C^\infty_c\tp{\R^2}$ is a dense subspace of $H^1\tp{\Omega_\ep}$, so we only need to establish the bound for $f\in H^1\tp{\Omega_\ep}$ belonging to this subspace.  For any $y\in(0,\ep)$ and $\xi\in\R$ we may then apply the fundamental theorem of calculus to the difference $\tabs{\mathscr{F}[f]\tp{\xi,y}}^2 - \tabs{\mathscr{F}\tp{\xi,0}}^2$ and then use Cauchy-Schwarz to bound
    \begin{equation}\label{there looks like wiggle room, but there is not}
        \tbr{\xi}\tabs{\mathscr{F}[f]\tp{\xi,0}}^2\le\tbr{\xi}\tabs{\mathscr{F}[f]\tp{\xi,y}}^2 + 2\bp{\int_0^\ep\tbr{\xi}^2\tabs{\mathscr{F}[f]\tp{\xi,z}}^2\;\m{d}z}^{1/2}\bp{\int_0^\ep\tabs{\pd_2\mathscr{F}[f]\tp{\xi,z}}^2\;\m{d}z}^{1/2}.
    \end{equation}
    Next we employ the second item of Lemma~\ref{lem on adapted boundary norms}, which permits us to estimate separately the high and low mode norms.  For the low mode piece we multiply~\eqref{there looks like wiggle room, but there is not} by $\tbr{\xi}$, integrate over $\tp{\xi,y}\in\tp{-1/\ep,1/\ep}\times\tp{0,\ep}$, apply the inequalities of Young and Cauchy-Schwarz, use that $\tbr{\xi}\le C\ep^{-1}$, and invoke Plancherel's theorem to arrive at the bound
    \begin{equation}\label{ineq_a}
        \ep^{1/2}\tnorm{\mathds{1}_{\tp{-1/\ep,1/\ep}}\tp{D}\m{Tr}_{\Sigma_0}f}_{H^1\tp{\R}}\le\tnorm{f}_{H^1\tp{\Omega_\ep}}.
    \end{equation}
    On the other hand, to get the high mode estimate, we  integrate~\eqref{there looks like wiggle room, but there is not} over $\tp{\xi,y}\in\tp{\R\setminus\tp{-1/\ep,1/\ep}}\times\tp{0,\ep}$ and use $\tbr{\xi}^{-1}\lesssim\ep$ along with the aforementioned tricks to obtain the bound
    \begin{equation}\label{ineq_b}
        \tnorm{\mathds{1}_{\R\setminus\tp{-1/\ep,1/\ep}}\tp{D}\m{Tr}_{\Sigma_0}f}_{H^{1/2}\tp{\R}}\le\tnorm{f}_{H^1\tp{\Omega_\ep}}.
    \end{equation}
    Summing~\eqref{ineq_a} and~\eqref{ineq_b} gives the claimed estimate of the second item.

    We now turn our attention to the third item.  We employ the equivalence second item of Lemma~\ref{lem on adapted boundary norms} together with Cauchy-Schwarz and  Plancherel to see that     
    \begin{multline}
        \babs{\int_{\Sigma}\phi f} + \babs{\int_{\Sigma}\tbr{D}^{1/2}\phi\cdot\tbr{D}^{-1/2}\pd_1 f} 
        \lesssim
        \int_{(-1/\ep,1/\ep)}\tsb{\ep^{-1/2}\tabs{\mathscr{F}\tsb{\phi}\tp{\xi}}}\cdot\tsb{\ep^{1/2}\tbr{\xi}\tabs{\mathscr{F}\tsb{\m{Tr}f}\tp{\m{\xi}}}}\;\m{d}\xi \\
        + \int_{\R\setminus\tp{-1/\ep,1/\ep}}\tsb{\tbr{\xi}^{1/2}\tabs{\mathscr{F}\tsb{\phi}\tp{\xi}}}\cdot\tsb{\tbr{\xi}^{1/2}\tabs{\mathscr{F}\tsb{\m{Tr}f}\tp{\m{\xi}}}}\;\m{d}\xi \lesssim\tnorm{\phi}_{H^{1/2}\tp{\R}}^{\tp{\ep,+}}\tnorm{\m{Tr}f}_{H^{1/2}\tp{\R}}^{\tp{\ep,-}}.
    \end{multline}
    The third item now follows by combining this with the trace estimate from the second item.      The proof of the fourth item is nearly identical, except that now since  $\m{Tr}_{\Sigma_0}f = 0$ we can use the first item to bound $\tnorm{\m{Tr}_{\Sigma_\ep}f}_{H^{1/2}\tp{\R}}^{\tp{\ep,+}}$,  which permits $\phi$ to be estimated in the weaker norm $\tnorm{\phi}^{\tp{\ep,-}}_{H^{1/2}\tp{\R}}$.
\end{proof}

We continue by considering Poincar\'e and Sobolev type inequalities. Similar results in three spatial dimensions have been obtained in Section 2 of Temam and Ziane~\cite{MR1401403}. We also  refer to the work of Maru\v si\'c~\cite{MR1930860}.

\begin{lemC}[Poincar\'e and Sobolev]\label{lem on poincare and sobolev in thin domains}
    For each $q\in[2,\infty]$ there exists $C_q\in\R^+$ such that  the following hold for all $\ep\in\tp{0,1}$.
    \begin{enumerate}
        \item If $q < \infty$ and  $f\in H^1\tp{\Omega_\ep}$, then
        \begin{equation}\label{deviation estimates}
            \tnorm{f - \m{Tr}_{\Sigma_\ep}f}_{L^q\tp{\Omega_\ep}} + \tnorm{f - \m{Tr}_{\Sigma_0}f}_{L^q\tp{\Omega_\ep}} + \bnorm{f - \f{1}{\ep}\int_0^\ep f\tp{\cdot,y}\;\m{d}y}_{L^q\tp{\Omega_\ep}}\le\ep^{2/q} C_q\tnorm{\grad f}_{L^2\tp{\Omega_\ep}},
        \end{equation}
        where for $\chi\in\tcb{0,\ep}$ the symbol $f - \m{Tr}_{\Sigma_\chi}f$ denotes the function $\Omega_\ep\ni\tp{x,y}\mapsto f(x,y) - \tp{\m{Tr}_{\Sigma_\chi}f}\tp{x}\in\R$.
        \item If $q < \infty$ and  $f\in H^1\tp{\Omega_\ep}$, then 
        \begin{equation}\label{Lq embedding}
            \ep^{1/2 - 1/q}\tnorm{f}_{L^q\tp{\Omega_\ep}}\le C_q\tnorm{f}_{H^1\tp{\Omega_\ep}}.
        \end{equation}
        \item For all $f\in H^2\tp{\Omega_\ep}$ we have the estimate
        \begin{equation}\label{Linfty embedding}
            \ep^{1/2}\tnorm{f}_{L^\infty\tp{\Omega_\ep}}\le C_\infty \tnorm{f}_{H^2\tp{\Omega_\ep}}.
        \end{equation}
    \end{enumerate}
\end{lemC}
\begin{proof}
    To prove the first item, we begin by reducing to the case $\ep = 1$. Assume that~\eqref{deviation estimates} holds for $\Omega_1$ and let $\ep\in(0,1)$ and $f\in H^1\tp{\Omega_\ep}$. Then the function $f_\ep(x,y) = f(\ep x,\ep y)$ belongs to $H^1\tp{\Omega_1}$ and so we may invoke the special case of~\eqref{deviation estimates} in $\Omega_1$. We get the claimed estimate in $\Omega_\ep$ by noting that
    \begin{equation}
        \tnorm{\grad f_\ep}_{L^2\tp{\Omega_1}} = \tnorm{\grad f}_{L^2\tp{\Omega_\ep}} 
        \text{ and }
        \tnorm{f_\ep - Af_\ep}_{L^q\tp{\Omega_1}} = \ep^{-2/q}\tnorm{f - Af}_{L^q\tp{\Omega_\ep}}
    \end{equation}
    where $Af$ and $Af_\ep$ refer to one of the trace or averaging operations in~\eqref{deviation estimates}.  This completes the proof of the reduction, so we now consider the case $\ep =1$.
     For any $\ell\in\Z$ we let $Q_\ell = (\ell - 1, \ell + 1)\times(0,1)$. Standard Sobolev space arguments yield the bound
    \begin{equation}\label{the case in a unit cube}
        \tnorm{f - \m{Tr}_{\Sigma_1}f}_{L^q\tp{Q_\ell}} + \tnorm{f - \m{Tr}_{\Sigma_0}f}_{L^q\tp{Q_\ell}} + \bnorm{f - \int_0^1f(\cdot,y)\;\m{d}y}_{L^q\tp{Q_\ell}}\le C_q\tnorm{\grad f}_{L^2\tp{Q_\ell}}
    \end{equation}
    for a constant $C_q\in\R^+$ depending only $q$. Raising expression~\eqref{the case in a unit cube} to the $q^{\m{th}}$ power, summing over $\ell\in\Z$, taking the $q^{\m{th}}$-root, and using $q/2\ge 1$ to invoke Minkowski's inequality then proves~\eqref{deviation estimates} in the $\Omega_1$ case. This completes the proof of the first item.

    We turn our attention to the second item. Given $f\in H^1\tp{\Omega_\ep}$ we decompose $f = f^0 + f^1$, where $f^0 = f - \mathcal{E}\m{Tr}_{\Sigma_0}f$ and $f^1 = \mathcal{E}\m{Tr}_{\Sigma_0}f$, with $\mathcal{E}$  the extension operator from Lemma~\ref{lem on anti-trace and anti-divergence}.  Then, thanks to~\eqref{deviation estimates}, the fact that $\m{Tr}_{\Sigma_0}f^0 = 0$, and~\eqref{the extension operator properties}, we initially have
    \begin{equation}\label{__a__}
        \tnorm{f}_{L^q}\le C_q\ep^{2/q}\tp{\tnorm{f}_{H^1\tp{\Omega_\ep}} + \tnorm{\m{Tr}_{\Sigma_0}f}_{H^{1/2}}} + \tnorm{f_1}_{L^q}.
    \end{equation}
    On the other hand, the Hausdorff-Young inequality along with~\eqref{the extension operator properties} gives
    \begin{equation}\label{__b__}
        \tnorm{f_1}_{L^q}\le\bp{\int_0^\ep\bp{\int_{\R}\tabs{\mathscr{F}[f_1]\tp{\xi,y}}^{q/(q-1)}\;\m{d}\xi}^{q - 1}\;\m{d}y}^{1/q}\le\ep^{1/q}\tnorm{\mathscr{F}\tsb{\m{Tr}_{\Sigma_0}f}}_{L^{q/(q-1)}}.
    \end{equation}
    H\"older's inequality with the product of $\tbr{\cdot}^{-1/2}$ and $\tbr{\cdot}^{1/2}\m{Tr}_{\Sigma_0}f$ allows us to estimate further
    \begin{equation}\label{__c__}
        \tnorm{\mathscr{F}\tsb{\m{Tr}_{\Sigma_0}f}}_{L^{q/(q-1)}}\le C_q\tnorm{\m{Tr}_{\Sigma_0}f}_{H^{1/2}}.
    \end{equation}
    Synthesizing~\eqref{__a__}, \eqref{__b__}, and~\eqref{__c__} with the second item of Lemma~\ref{lem on traces in thin domains} and the first item of Lemma~\ref{lem on adapted boundary norms} gives~\eqref{Lq embedding}.

    We now prove the third item. The strategy is similar to the proof of the second item of Lemma~\ref{lem on traces in thin domains}. Let $f\in H^2\tp{\Omega_\ep}$ be the restriction to $\Omega_\ep$ of a function in $C^\infty_c\tp{\R^2}$. By arguing as in~\eqref{there looks like wiggle room, but there is not}, we find that for any $y,z\in(0,\ep)$ and $\xi\in\R$ it holds
    \begin{equation}
        \tbr{\xi}^2\tabs{\mathscr{F}[f]\tp{\xi,y}}^2\le\tbr{\xi}^2\tabs{\mathscr{F}[f]\tp{\xi,z}}^2 + 2\bp{\int_0^\ep\tbr{\xi}^2\tabs{\mathscr{F}[f]\tp{\xi,s}}^2\;\m{d}s}^{1/2}\bp{\int_0^\ep\tbr{\xi}^2\tabs{\pd_2\mathscr{F}[f]\tp{\xi,s}}^2\;\m{d}s}^{1/2}.
    \end{equation}
    Integrating over $\tp{\xi,z}\in\Omega_\ep$, taking the square root, and invoking the Cauchy-Schwarz and Young's inequalities, we then find that
    \begin{equation}
        \ep^{1/2}\tnorm{\m{Tr}_{\Sigma_y}f}_{H^1\tp{\R}}\le  C\tnorm{f}_{H^2\tp{\Omega_\ep}}.
    \end{equation}
    Inequality~\eqref{Linfty embedding} now follows from $H^1\tp{\R}\emb L^\infty\tp{\R}$ and taking the supremum over $y\in(0,\ep)$.
\end{proof}

We now consider the Korn inequality in thin domains. We have the following preliminary result.
\begin{lemC}[Korn, I]\label{lem on korn in a star-shaped domain}
    Let $\ep,\del\in(0,1)$ and $h\in W^{1,\infty}\tp{\R}$ satisfy $\tnorm{h}_{W^{1,\infty}}\le\del^{-1}$ and $\del\le h\le\del^{-1}$. There exists a constant $C\in\R^+$, depending only on $\del$, such that for all $x_0\in\R$ and $\varphi\in H^1\tp{\Omega_{\ep h}\cap Q^\ep\tp{x_0};\R^2}$ we have the estimate
    \begin{equation}\label{lem on korn in a star-shaped domain est}
        \bp{\int_{\Omega_{\ep h}\cap Q^\ep\tp{x_0}}\babs{\grad\varphi - \f{1}{|\Omega_{\ep h}\cap Q^\ep\tp{x_0}|}\int_{\Omega_{\ep h}\cap Q^\ep\tp{x_0}}\f12\mathbb{A}\psi}^2}^{1/2}\le C\tnorm{\mathbb{D}\varphi}_{L^2\tp{\Omega_{\ep h}\cap Q^\ep\tp{x_0}}}
    \end{equation}
    where $Q^\ep\tp{x_0} = (x_0 - \ep\del^2/4,x_0 + \ep\del^2/4)\times\tp{0,\ep/\del}$, $\mathbb{A}\varphi = \grad\varphi - \grad\varphi^{\m{t}}$, and $|\Omega^\ep\cap Q^\ep\tp{x_0}|$ is the planar Lebesgue measure of the set $\Omega_{\ep h}\cap Q^\ep\tp{x_0}$
\end{lemC}
\begin{proof}
    We aim to directly invoke Corollary 3.20 from Lewicka~\cite{MR4592573}.  To do so we only need to check that the domain $\Omega_{\ep h}\cap Q^\ep\tp{x_0}$ is star-shaped with respect to an interior ball and the inverse ratio of this interior ball's diameter to the diameter of $\Omega_{\ep h}\cap Q^\ep\tp{x_0}$ is bounded above by a function of $\del$ alone.

    To prove that $\Omega_{\ep h}\cap Q^\ep\tp{x_0}$ is star-shaped with respect to a ball, we consider the open sets indexed by $z\in(x_0 - \ep\del^2/4, x_0 + \ep\del^2/4)$ and defined via
    \begin{equation}\label{convexobro}
        C^\ep\tp{z} = \tcb{(x,y)\in Q^\ep(x_0)\;:\;y/\ep < h(z) - \del^{-1}|x - z|}.
    \end{equation}
    The set $C^\ep\tp{z}$ is convex, as it is the intersection of a cone with a rectangle. The bound $\tnorm{h}_{W^{1,\infty}}\le\del^{-1}$ then implies that $C^\ep(z)\subseteq \Omega_{\ep h}\cap Q^\ep(x_0)$; indeed, if $(x,y)\in C^\ep(z)$, then $|h(x) - h(z)|\le\del^{-1}|x - z|$ and so the defining inequality in~\eqref{convexobro} forces $0<y/\ep<h(x)$ and hence $(x,y)\in\Omega_{\ep h}$.  In fact, since $(x,y)\in C^\ep(x)$ for all pairs $(x,y)\in\Omega_{\ep h}\cap Q^\ep\tp{x_0}$, we have the identity 
    \begin{equation}\label{yep that's a union}
        \bigcup_{|z - x_0|<\ep\del^2/4}C^\ep\tp{z} = \Omega_{\ep h}\cap Q^\ep\tp{x_0}.
    \end{equation}
    Thus, if we could find an open ball 
    \begin{equation}\label{the condition on the ball B}
          \es \neq B\subseteq\bigcap_{|z - x_0|<\ep\del^2/4} C^\ep(z)
    \end{equation}    
    then the convexity of each $C^\ep\tp{z}$ and~\eqref{yep that's a union} would imply that $\Omega_{\ep h}\cap Q^\ep\tp{x_0}$ is star-shaped with respect to $B$.  We will prove this.

     We claim that the choice $B = B((x_0,\ep\del/4),\ep\del^2/8)$ satisfies~\eqref{the condition on the ball B}.  We have $B\subset(x_0 - \ep\del^2/4,x_0 + \ep\del^2/4)\times(0,\ep\del/2)$, so it suffices to prove that the set on the right is a subset of $C^\ep(z)$. Supposing that $|x - x_0|<\ep\del^2/4$ and $0<y<\ep\del/2$, we use $h(z)\ge\del$ to check that
    \begin{equation}
        h(z) - \del^{-1}|x - z|\ge\del - \ep\del/2\ge\del/2>y/\ep,
    \end{equation}
    and hence $(x,y)\in C^\ep\tp{z}$.  This completes the proof that $\Omega_{\ep h}\cap Q^\ep\tp{x_0}$ is star-shaped with respect to $B$. We can thus conclude that~\eqref{lem on korn in a star-shaped domain est} holds with the constant $C$ depending on $\m{diam}\tp{\Omega_{\ep h}\cap Q^\ep\tp{x_0}}/\m{diam}B\le 2\ep\del^{-1}/\tp{\ep\del^2/8} = 16\del^{-3}$.
\end{proof}

Next, we present a version of the  Korn inequality, closely related to and inspired by the theory developed by Lewicka and M\"uller~\cite{MR2795715}. The key differences between what they prove and what we do are: 1) we are concerned with the unbounded thin domains $\Omega_{\ep h}$; and 2) we are permitting the $H^1\tp{\Omega_{\ep h};\R^2}$ norm to not only be controlled by the symmetric gradient, but also a small trace norm.  We remark that a version of a thin domain Korn inequality in three spatial dimensions for a flat slab appears in Section 3.1.2 of Bresch and Noble~\cite{MR2975338}.

\begin{propC}[Korn, II]\label{prop on Korn in thin domains}
    Let $\ep,\del\in\tp{0,1}$ and $h\in W^{1,\infty}\tp{\R}$ satisfy $\tnorm{h}_{W^{1,\infty}}\le \del^{-1}$ and $\del\le h\le \del^{-1}$. There exists a constant $C\in\R^+$, depending only on $\del$, such that for all $\varphi\in H^1\tp{\Omega_\ep;\R^2}$ satisfying $\m{Tr}_{\Sigma_0}\varphi_2 =0$ we have the estimate
    \begin{equation}\label{the thin domain Korn inequality}
        \tnorm{\varphi}_{H^1\tp{\Omega_\ep}}\le C\tp{\ep^{1/2}\tnorm{\varphi_1}_{L^2\tp{\Sigma_0}} + \tnorm{\mathbb{D}^{\mathcal{A}_h}\varphi}_{L^2\tp{\Omega_\ep}}}.
    \end{equation}
\end{propC}
\begin{proof}
    We divide the proof into steps.

    \emph{Step 1 -- Reductions:}  To begin, we make a pair of reductions. First, we claim that for $\del$ fixed it suffices to prove~\eqref{the thin domain Korn inequality} is valid for all $0 < \ep \le \ep_0$ for some $\ep_0 \in (0,1)$ depending on $\delta$.  Indeed, once this is established, the remaining range $\ep \in (\ep_0,1)$ may be recovered via an elementary scaling argument.

    Second, we claim that it suffices to prove the existence of a constant $C$, depending only on $\del$, such that for all $\psi\in H^1\tp{\Omega_{\ep h};\R^2}$ satisfying $\m{Tr}_{\Sigma_0}\psi_2 = 0$ we have the estimate
    \begin{equation}\label{the reduced thin Korn inequality}
        \tnorm{\grad\psi}_{L^2\tp{\Omega_{\ep h}}}\le C\tp{\tnorm{\psi}_{L^2\tp{\Omega_{\ep h}}} + \tnorm{\mathbb{D}\psi}_{L^2\tp{\Omega_{\ep h}}}}.
    \end{equation}
    To prove the claim, let us suppose that~\eqref{the reduced thin Korn inequality} holds and let $\varphi\in H^1\tp{\Omega_\ep;\R^2}$ satisfy $\varphi_2 = 0$ on $\Sigma_0$. We use the flattening map $F_h^\ep$ constructed in~\eqref{this is the flattening map} from the proof of Lemma~\ref{lem on anti-trace and anti-divergence} and let $\psi = \tp{F_h^\ep}^{-1}\varphi\in H^1\tp{\Omega_{\ep h};\R^2}$, which satisfies $\m{Tr}_{\Sigma_0}\psi = 0$ along with the estimates
    \begin{equation}\label{some important flattening bounds}
        \tnorm{\varphi}_{H^1\tp{\Omega_\ep}}\le C\tnorm{\psi}_{H^1\tp{\Omega_{\ep h}}},\;\tnorm{\psi}_{L^2\tp{\Omega_{\ep h}}}\le C\tnorm{\varphi}_{L^2\tp{\Omega_\ep}},
        \text{ and }
        \tnorm{\mathbb{D}\psi}_{L^2\tp{\Omega_{\ep h}}}\le C\tnorm{\mathbb{D}^{\mathcal{A}_h}\varphi}_{L^2\tp{\Omega_\ep}}
    \end{equation}
    for $C$ depending only on $\del$. We can then invoke the assumed reduced estimate~\eqref{the reduced thin Korn inequality} for $\psi$ to find that
    \begin{equation}\label{another important intermediate estimate}
        \tnorm{\varphi}_{H^1\tp{\Omega_\ep}}\lesssim\tnorm{\psi}_{H^1\tp{\Omega_{\ep h}}}\lesssim\tnorm{\psi}_{L^2\tp{\Omega_{\ep h}}} + \tnorm{\mathbb{D}\psi}_{L^2\tp{\Omega_{\ep h}}}\lesssim\tnorm{\varphi}_{L^2\tp{\Omega_\ep}} + \tnorm{\mathbb{D}^{\mathcal{A}_h}\varphi}_{L^2\tp{\Omega_\ep}},
    \end{equation}
    again for implicit constants depending only on $\del$. To move from~\eqref{another important intermediate estimate} to~\eqref{the thin domain Korn inequality} we take $\ep\in(0,1)$ sufficiently small and employ the estimates from the first item of Lemma~\ref{lem on poincare and sobolev in thin domains} to bound
    \begin{equation}
        \tnorm{\varphi}_{L^2\tp{\Omega_\ep}}\le\ep^{1/2}\tnorm{\varphi_1}_{L^2\tp{\Sigma_0}} + \tnorm{\varphi_1 - \m{Tr}_{\Sigma_0}\varphi_1}_{L^2\tp{\Omega_\ep}} + \tnorm{\varphi_2}_{L^2\tp{\Omega_\ep}}\le\ep^{1/2}\tnorm{\varphi_1}_{L^2\tp{\Sigma_0}} + \ep\tnorm{\grad\varphi}_{L^2\tp{\Omega_\ep}}
    \end{equation}
    and then adsorb.  This proves the second claim.

    \emph{Step 2 -- The gradient replacement:}     We have now reduced to proving the estimate~\eqref{the reduced thin Korn inequality} for sufficiently small $\ep$.  The next step is to argue as in the proof of Theorem 5.1 in~\cite{MR2795715}:  we claim that for each $\psi\in H^1\tp{\Omega_{\ep h};\R^2}$ there exists $R\psi\in \tp{W^{1,\infty}\cap H^1}\tp{\Sigma_0;\R^{2\times 2}_{\m{asym}}}$ for which
    \begin{equation}\label{the estimates for the smooth curl dude}
        \tnorm{\grad\psi - R\psi}_{L^2\tp{\Omega_{\ep h}}}\le C\tnorm{\mathbb{D}\psi}_{L^2\tp{\Omega_{\ep h}}}
        \text{ and }
        \tnorm{\tp{R\psi}'}_{L^2\tp{\Sigma_0}}\le C\ep^{-3/2}\tnorm{\mathbb{D}\psi}_{L^2\tp{\Omega_{\ep h}}},
    \end{equation}
    with the constant $C$ depending only on $\del$.  
    
    In broad strokes, the map $R\psi$ will be taken to be a mollified version of $\f12\mathbb{A}\psi$ (where this operator is introduced in Lemma~\ref{lem on korn in a star-shaped domain}). More precisely, let us fix $\vartheta\in C^\infty_c\tp{\R}$ with $\m{supp}\vartheta \subset (-\del^2/4,\del^2/4)$, $\vartheta\ge0$, and $\int_{\R}\vartheta = 1$. Then define $\Theta:\R\times\Omega_{\ep h}\to\R$ via $\Theta(x,z) = \vartheta\tp{(z_1 - x)/\ep}/(\ep^2h(x))$ and $R\psi:\Sigma_0\to\R^{2\times 2}_{\m{asym}}$ via
    \begin{equation}\label{smoothed out curl dude}
        R\psi(x) = \f12\int_{\Omega_{\ep h}}\Theta\tp{x,z}\mathbb{A}\psi\tp{z}\;\m{d}z.
    \end{equation}
    The kernels $\Theta$ satisfy $\Theta(x,z)=0$ when $|x - z_1|\ge\ep\del^2/4$ and satisfy the additional conditions 
    \begin{equation}\label{The only estimates on the kernel youll need}
        \int_{\Omega_{\ep h}}\Theta(x,z)\;\m{d}z = 1,
        \quad
        \int_{\Omega_{\ep h}}D_1\Theta\tp{x,z}\;\m{d}z = 0,
        \quad 
        \tabs{\Theta(x,z)}\le C\ep^{-2},
        \text{ and }
        \tabs{D_1\Theta(x,z)}\le C\ep^{-3}
    \end{equation}
    with $C$ depending only on $\del$. Now, given any $x_0\in\R$, we employ the notation Lemma~\ref{lem on korn in a star-shaped domain} to write
    \begin{equation}
        A_{x_0}\psi = \f{1}{|\Omega_{\ep h}\cap Q^\ep\tp{x_0}|}\int_{\Omega_{\ep h}\cap Q^\ep\tp{x_0}}\f12\mathbb{A}\psi.
    \end{equation}
    By using~\eqref{The only estimates on the kernel youll need}, Cauchy-Schwarz, and Lemma~\ref{lem on korn in a star-shaped domain}, we then estimate the difference
    \begin{equation}\label{initial estimate for R}
        |R\psi\tp{x_0} - A_{x_0}\psi|\le\int_{\Omega_{\ep h}}\Theta(x_0,z)\babs{\f12\mathbb{A}\psi\tp{z} - A_{x_0}\psi}\;\m{d}z\le C\ep^{-1}\tnorm{\mathbb{D}\psi}_{L^2\tp{\Omega_{\ep h}\cap Q^{\ep}\tp{x_0}}}.
    \end{equation}
    In a similar fashion, we also have
    \begin{equation}\label{initial estimate for derivative of R}
        |\tp{R\psi}'\tp{x_0}|\le\int_{\Omega_{\ep h}}|D_1\Theta(x_0,z)|\babs{\f12\mathbb{A}\psi(z) - A_{x_0}\psi}\;\m{d}z\le C\ep^{-2}\tnorm{\mathbb{D}\psi}_{L^2\tp{\Omega_{\ep h}\cap Q^\ep\tp{x_0}}}.
    \end{equation}
    From~\eqref{initial estimate for derivative of R} we may also derive the estimate
    \begin{equation}\label{more advanced estimate on the derivative of R}
        \sup_{|x-x_0|<\ep\del^2/4}|\tp{R\psi}'\tp{x}|\le C\ep^{-2}\tnorm{\mathbb{D}\psi}_{L^2\tp{\Omega_{\ep h}\cap Q^{2\ep}\tp{x_0}}},
    \end{equation}
    which is a local Lipschitz estimate for $R\psi$ near $x_0$; hence, in particular we have that
    \begin{equation}\label{pull a supremum out of your hat}
        \sup_{|x - x_0|<\ep\del^2/4}\tabs{R\psi(x) - R\psi\tp{x_0}}\le C\ep^{-1}\tnorm{\mathbb{D}\psi}_{L^2\tp{\Omega_{\ep h}\cap Q^{2\ep}\tp{x_0}}}.
    \end{equation}
    We now let $\tp{x,y}\in\Omega_{\ep h}\cap Q^\ep\tp{x_0}$ and use~\eqref{initial estimate for R} and~\eqref{pull a supremum out of your hat} to write
    \begin{multline}
        \tabs{\grad\psi\tp{x,y} - R\psi\tp{x}}\le\tabs{\grad\psi(x,y) - A_{x_0}\psi} + \tabs{A_{x_0}\psi - R\psi\tp{x_0}} + \tabs{R\psi\tp{x_0} - R\psi(x)}\\
        \le C\tabs{\grad\psi\tp{x,y} - A_{x_0}\psi} + C\ep^{-1}\tnorm{\mathbb{D}\psi}_{L^2\tp{\Omega_{\ep h}\cap Q^{2\ep}\tp{x_0}}}.
    \end{multline}
    Now we take the norm in $L^2\tp{\Omega_{\ep h}\cap Q^\ep\tp{x_0}}$ of the above expression and once again invoke  Lemma~\ref{lem on korn in a star-shaped domain} to obtain the estimate
    \begin{equation}\label{almost half way there}
        \tnorm{\grad\psi - R\psi}_{L^2\tp{\Omega_{\ep h}\cap Q^\ep(x_0)}}\le C\tnorm{\mathbb{D}\psi}_{L^2\tp{\Omega_{\ep h}\cap Q^{2\ep}\tp{x_0}}}.
    \end{equation}
    By squaring~\eqref{almost half way there} and  summing over $x_0\in \ep\del^2\Z/4$ we obtain the left estimate of~\eqref{the estimates for the smooth curl dude}. The right hand estimate in~\eqref{the estimates for the smooth curl dude} can be obtained from~\eqref{more advanced estimate on the derivative of R} by noting that the latter implies
    \begin{equation}
        \tnorm{\tp{R\psi}'}_{L^2\tp{x_0 - \ep\del^2/4,x_0 + \ep\del^2/4}}\le C\ep^{-3/2}\tnorm{\mathbb{D}\psi}_{L^2\tp{\Omega_{\ep h}\cap Q^{2\ep}\tp{x_0}}};
    \end{equation}
    therefore, upon squaring and summing over the appropriate lattice again, we obtain the desired bound.  The proof of the claim is now complete.

    \emph{Step 3 -- Integration by parts tricks:}     At this point we have successfully established~\eqref{the estimates for the smooth curl dude}. This reduces proving~\eqref{the reduced thin Korn inequality} to obtaining further estimates on the map $R\psi$.  Assume now that $\psi\in H^1\tp{\Omega_{\ep h};\R^2}$ satisfies $\m{Tr}_{\Sigma_0}\psi_2 = 0$.  As $R\psi\in \R^{2\times 2}_{\m{asym}}$, we know that $R\psi_{11} = R\psi_{22} = 0$ and $R\psi_{12} = - R\psi_{21}$, so we only need to obtain estimates on $R\psi_{12}$.
    
    By using the expression~\eqref{smoothed out curl dude} along with the identity $\mathbb{A}\psi_{12} = \mathbb{D}\psi_{12} - 2\pd_1\psi_2$ we split  $R\psi_{12}  = R^1\psi - 2R^2\psi$ for 
    \begin{equation}
        R^1\psi(x) = \int_{\Omega_{\ep h}}\Theta\tp{x,z}\mathbb{D}\psi\tp{z}_{12}\;\m{d}z
        \text{ and }
        R^2\psi(x) = \int_{\Omega_{\ep h}}\Theta(x,z)\pd_1\psi_2\tp{z}\;\m{d}z.
    \end{equation}
    Thanks to~\eqref{The only estimates on the kernel youll need} and Cauchy-Schwarz, we have the pointwise estimate
    \begin{equation}
        |R^1\psi(x)|\le C\ep^{-1}\bp{\int_{\Omega_{\ep h}\cap Q^\ep\tp{x}}\tabs{\mathbb{D}\psi\tp{z}}^2\;\m{d}z}^{1/2},
    \end{equation}
    which, by Tonelli's theorem, yields the bound   
    \begin{equation}\label{the estimate on R1}
        \tnorm{R^1\psi}_{L^2\tp{\R}}\le C\ep^{-1/2}\tnorm{\mathbb{D}\psi}_{L^2\tp{\Omega_{\ep h}}}.
    \end{equation}

    Next, we turn to the $R^2\psi$ term.  In fact, the elementary inner-product inequality
    \begin{equation}\label{the hilbert space identity}
        \tnorm{R\psi_{12}}_{L^2\tp{\R}}^2\le \tnorm{R^1\psi}^2_{L^2\tp{\R}} + 4\tabs{\tbr{R\psi_{12},R^2\psi}_{L^2\tp{\R}}}
    \end{equation}
    allows us to focus only on the inner-product between $R\psi_{12}$ and $R^2\psi$.  The definition of $\Theta$ allows us to compute 
    \begin{equation}
      D_{2,1}\Theta\tp{x,z}  = - D_1\Theta\tp{x,z} - \tp{h'(x)/h(x)}\Theta(x,z),
    \end{equation}
    where $D_{2,1}$ denotes the derivative with respect to $z_1$.   From this identity, Fubini's theorem, and two integrations by parts, we then learn that 
    \begin{multline}\label{integration by parts in wild}
        \tbr{R\psi_{12},R^2\psi}_{L^2\tp{\R}} = \int_{\R}\int_{\Omega_{\ep h}}\tp{R\psi_{12}}'(x) \Theta(x,z)\psi_2\tp{z}\;\m{d}z\;\m{d}x + \int_{\R}\f{h'(x)}{h(x)}\int_{\Omega_{\ep h}}R\psi_{12}(x) \Theta\tp{x,z}\psi_2\tp{z}\;\m{d}z\;\m{d}x\\
          -\ep\int_{\R}\int_{\R}h'(w)R\psi_{12}\tp{x}\Theta(x,w,\ep h(w))\psi_2\tp{w,\ep h(w)}\;\m{d}w\;\m{d}x = \bf{I}_1 + \bf{I}_2 + \bf{I}_3.
    \end{multline}
    We will next estimate $\bf{I}_i$ for $i\in\tcb{1,2,3}$ individually.

    For $\bf{I}_1$ we use repeated applications of Cauchy-Schwarz, the estimates in~\eqref{The only estimates on the kernel youll need}, the $\tp{R\psi}'$ estimate of~\eqref{the estimates for the smooth curl dude}, and Fubini's theorem:
    \begin{equation}\label{I11}
        \bf{I}_1\le\f{C}{\ep}\int_{\R}\tabs{\tp{R\psi_{12}}'(x)}\bp{\int_{\Omega_{\ep h}\cap Q^\ep\tp{x}}\tabs{\psi_2\tp{z}}^2\;\m{d}z}^{1/2}\;\m{d}x\le C\ep^{-2}\tnorm{\mathbb{D}\psi}_{L^2\tp{\Omega_{\ep h}}}\tnorm{\psi_2}_{L^2\tp{\Omega_{\ep h}}}.
    \end{equation}
    To estimate the second term on the right here, we employ the flattening map from~\eqref{this is the flattening map} to change into $\Omega_\ep$; in light of the estimates~\eqref{estimates on these flattening and unflattening maps}, the vanishing trace $\m{Tr}_{\Sigma_0}F^\ep_h\psi_2 = 0$, and the first item of Lemma~\ref{lem on poincare and sobolev in thin domains}, we may bound
    \begin{equation}\label{I12}
        \tnorm{\psi_2}_{L^2\tp{\Omega_{\ep h}}}\lesssim \tnorm{F^\ep_h\psi_2}_{L^2\tp{\Omega_\ep}}\lesssim\ep\tnorm{F^\ep_h\psi}_{H^1\tp{\Omega_\ep}}\lesssim\ep\tnorm{\psi}_{H^1\tp{\Omega_{\ep h}}}.
    \end{equation}
    Combining~\eqref{I11} and~\eqref{I12} then grants the bound
    \begin{equation}\label{estimate on I1}
        \bf{I}_1\le C\ep^{-1}\tnorm{\mathbb{D}\psi}_{L^2\tp{\Omega_{\ep h}}}\tnorm{\psi}_{H^1\tp{\Omega_{\ep h}}}.
    \end{equation}
    
    For $\bf{I}_2$ we use similar strategy in order to get
    \begin{equation}
        \bf{I}_2\le \f{C}{\ep}\int_{\R}\tabs{R\psi_{12}\tp{x}}\bp{\int_{\Omega_{\ep h}\cap Q^\ep\tp{x}}\tabs{\psi_2\tp{z}}^2\;\m{d}z}^{1/2}\;\m{d}x\le C\ep^{-1/2}\tnorm{R\psi_{12}}_{L^2\tp{\R}}\tnorm{\psi_2}_{L^2\tp{\Omega_{\ep h}}}.
    \end{equation}
    Applying~\eqref{I12} again then yields
    \begin{equation}\label{estimate on I2}
        \bf{I}_2\le C\ep^{1/2}\tnorm{R\psi_{12}}_{L^2\tp{\R}}\tnorm{\psi}_{H^1\tp{\Omega_{\ep h}}}.
    \end{equation}
    
    For $\bf{I}_3$, we use~\eqref{The only estimates on the kernel youll need} and Fubini's theorem followed by the Cauchy-Schwarz and Young's inequalities to bound
    \begin{equation}
        \bf{I}_3\le C\int_{\R}\bp{\f{1}{\ep}\int_{w-\ep\del^2/4}^{w+\ep\del^2/4}\tabs{R\psi_{12}\tp{x}}\;\m{d}x}\tabs{\psi_2\tp{w,\ep h(w)}}\;\m{d}w\le C\tnorm{R\psi_{12}}_{L^2\tp{\R}}\tnorm{\m{Tr}_{\Sigma_\ep}F^\ep_h\psi_2}_{L^2\tp{\R}}.
    \end{equation}
    Again appealing to ~\eqref{estimates on these flattening and unflattening maps} and using the identity $\m{Tr}_{\Sigma_0}F^\ep_h\psi_2=0$ with  the first item of Lemma~\ref{lem on traces in thin domains}, we further estimate
    \begin{equation}\label{estimate on I3}
        \bf{I}_3\le C\ep^{1/2}\tnorm{R\psi_{12}}_{L^2\tp{\R}}\tnorm{\psi}_{H^1\tp{\Omega_{\ep h}}}.
    \end{equation}

    Synthesizing~\eqref{estimate on I1}, \eqref{estimate on I2}, and~\eqref{estimate on I3} with~\eqref{the estimate on R1}, \eqref{the hilbert space identity}, and~\eqref{integration by parts in wild} shows that
    \begin{equation}
        \tnorm{R\psi_{12}}^2_{L^2\tp{\R}}\lesssim \ep^{-1}\tnorm{\mathbb{D}\psi}^2_{L^2\tp{\Omega_{\ep h}}} + \ep^{-1}\tnorm{\mathbb{D}\psi}_{L^2\tp{\Omega_{\ep h}}}\tnorm{\psi}_{H^1\tp{\Omega_{\ep h}}} + \ep^{1/2}\tnorm{R\psi_{12}}_{L^2\tp{\R}}\tnorm{\psi}_{H^1\tp{\Omega_{\ep h}}}.
    \end{equation}
    Absorbing the $R\psi_{12}$ term on the left and using the trivial bound $\tabs{\mathbb{D}\psi} \le 2 \tabs{\nabla \psi}$, we then deduce from this that
    \begin{equation}\label{final estimate on Rpsi}
        \tnorm{R\psi_{12}}_{L^2\tp{\R}} \lesssim\ep^{-1/2}\tnorm{\mathbb{D}\psi}_{L^2\tp{\Omega_{\ep h}}}^{1/2}\tnorm{\psi}_{H^1\tp{\Omega_{\ep h}}}^{1/2} + \ep^{1/2}\tnorm{\psi}_{H^1\tp{\Omega_{\ep h}}}.
    \end{equation}

   \emph{Step 4 -- Conclusion:}      Finally, we are ready to conclude. For any $\psi\in H^1\tp{\Omega_{\ep h};\R^2}$ that satisfies $\m{Tr}_{\Sigma_0}\psi_2 = 0$ we use~\eqref{the estimates for the smooth curl dude} and~\eqref{final estimate on Rpsi} to estimate
\begin{multline}
        \tnorm{\grad\psi}_{L^2\tp{\Omega_{\ep h}}}\le \tnorm{\grad\psi - R\psi}_{L^2\tp{\Omega_{\ep h}}} + \tnorm{R\psi}_{L^2\tp{\Omega_{\ep h}}}
        \lesssim 
        \tnorm{\mathbb{D}\psi}_{L^2\tp{\Omega_{\ep h}}}
        + \ep^{1/2} \tnorm{R\psi_{12}}_{L^2\tp{\R}}
        \lesssim \tnorm{\mathbb{D}\psi}_{L^2\tp{\Omega_{\ep h}}}^{1/2}\tnorm{\psi}_{H^1\tp{\Omega_{\ep h}}}^{1/2} \\ + \ep\tnorm{\psi}_{H^1\tp{\Omega_{\ep h}}} 
        \lesssim 
        \tnorm{\mathbb{D}\psi}_{L^2\tp{\Omega_{\ep h}}}^{1/2}(\tnorm{\psi}_{L^2\tp{\Omega_{\ep h}}} + \tnorm{\nabla \psi}_{L^2\tp{\Omega_{\ep h}}} )^{1/2} + \ep (\tnorm{\psi}_{L^2\tp{\Omega_{\ep h}}} + \tnorm{\nabla \psi}_{L^2\tp{\Omega_{\ep h}}} )
        .
    \end{multline}
    By taking $\ep>0$ sufficiently small (depending on $\del$) and employing Cauchy's inequality,  we may adsorb the $\nabla \psi$ terms onto the left and arrive at the desired bound~\eqref{the reduced thin Korn inequality}.
    \end{proof}

\subsection{The Stokes equations with stress boundary conditions}\label{subsection on the stokes equations with stress boundary conditions}

We now turn our attention to a system of linear Stokes equations in an $\ep$-thin strip domain with the Navier slip boundary condition on the bottom and stress boundary conditions on the top. The system for the unknown velocity $u:\Omega_\ep\to\R^2$ and pressure $p:\Omega_\ep\to\R$ is given by
\begin{equation}\label{linear stokes in a thin domain}
    \begin{cases}
        \grad^{\mathcal{A}_h}p - \upmu\grad^{\mathcal{A}_h}\cdot\mathbb{D}^{\mathcal{A}_h}u = \varphi^1,\quad \grad^{\mathcal{A}_h}\cdot u = \varphi^2&\text{in }\Omega_\ep,\\
        -\tp{p - \upmu\mathbb{D}^{\mathcal{A}_h}u}\mathcal{N}_{\ep h} = \psi^1&\text{on }\Sigma_\ep,\\
        u_2=0,\quad\upmu\pd_2^{\mathcal{A}_h}u_1 - \ep\m{a}u_1 = \psi^2&\text{on }\Sigma_0.
    \end{cases}
\end{equation}
The given data are the functions $\varphi^1:\Omega_\ep\to\R^2$, $\varphi^2:\Omega_\ep\to\R$, $\psi^1:\Sigma_\ep\to\R^2$, and $\psi^2:\Sigma_0\to\R$. The equations~\eqref{linear stokes in a thin domain} come with a functional parameter $h:\R\to\R^+$ which determines the matrices $\mathcal{A}_h$ as in~\eqref{geometry matrices} and the vectors $\mathcal{N}_{\ep h}$, defined after~\eqref{parent free boundary navier-stokes system}. We think of the parameters $\upmu,\m{a}>0$ as fixed and track the dependence of the solution operator on the function $h$, with estimates uniform for $0<\ep\ll1$.

We begin with a series of definitions of function spaces and norms that we will employ in our analysis of~\eqref{linear stokes in a thin domain}. First, we discuss the role of the functional parameter $h$ appearing in system~\eqref{linear stokes in a thin domain}; its container function space is
\begin{equation}\label{the Z space}
    \bf{Z} = \tcb{h\in W^{2,3}_{\loc}\tp{\R}\;:\;h = h^0 + h^1,\;h^0\in W^{2,3}\tp{\R},\;h^1\in W^{2,\infty}\tp{\R}}.
\end{equation}
The norm on this space, which makes it Banach, is given by
\begin{equation}\label{the Z space, norm}
    \tnorm{h}_{\bf{Z}} = \inf_{h = h^0 + h^1}\tp{\tnorm{h^0}_{W^{2,3}} + \tnorm{h^1}_{W^{2,\infty}}}.
\end{equation}
Given $\del\in(0,1)$, we define the following closed, convex subsets of $\bf{Z}$:
\begin{equation}\label{in this set we have Korn}
    Z_\del = \tcb{h\in\bf{Z}\;:\;\tnorm{h}_{\bf{Z}}\le\del^{-1},\;\del\le h\le\del^{-1}}.
\end{equation}
We will assume that $h\in Z_\del$ and allow estimates for solutions to~\eqref{linear stokes in a thin domain} to depend on $\del$.  Two remarks about our choices here are in order. First, for every $h\in\bf{Z}$ with $\inf h>0$ there exists $\del>0$ such that $h\in Z_\del$, so assuming $h \in Z_\del$ in not a particularly strenuous constraint. Second, our specific choice of $\bf{Z}$ is primarily for convenience and simplicity in our analysis of~\eqref{linear stokes in a thin domain} for its intended application, and in principal many other choices for the $h$-parameter function space could work.

Next, we define the spaces for the velocity and pressure. For $n\in\tcb{0,1}$ the closed subspace of vector fields that are tangent to $\Sigma_0$ is denoted by
\begin{equation}\label{tangent to the bottom vector fields}
    H^{n+1}_{\m{tan}}\tp{\Omega_\ep;\R^2} = \tcb{u\in H^{n+1}\tp{\Omega_\ep;\R^2}\;:\;\m{Tr}_{\Sigma_0}u_2 = 0}.
\end{equation}
Using~\eqref{tangent to the bottom vector fields}, we define the domain space 
\begin{equation}\label{the original norms on X}
    \bf{X}^n_\ep = H^{n+1}_{\m{tan}}\tp{\Omega_\ep;\R^2}\times H^n\tp{\Omega_\ep}
     \text{ with norm }
    \tnorm{u,p}_{\bf{X}^n_\ep} = \tnorm{u}_{H^{n+1}} + \tnorm{p}_{H^n}.
\end{equation}
We  also need to equip $\bf{X}^1_\ep$ with a family of larger norms:
\begin{equation}\label{stronger, adapted norms}
    \tnorm{u,p}_{\bf{X}^1_\ep}^{(+)} = \tnorm{u,p}_{\bf{X}^1_\ep} + \tnorm{\pd_2 u_1}_{H^{1/2}\tp{\Sigma_0}}^{\tp{\ep,+}} + \tnorm{\pd_2 u_1}_{H^{1/2}\tp{\Sigma_\ep}}^{\tp{\ep,+}}.
\end{equation}
While each of the norms~\eqref{stronger, adapted norms} is equivalent to~\eqref{the original norms on X}, we see that from Lemma~\ref{lem on adapted boundary norms} that the constants witnessing this equivalence are degenerating as $\ep>0$ tends to zero.

The weak codomain space is
\begin{equation}
    \bf{Y}^0_\ep = \tp{H^1_{\m{tan}}\tp{\Omega_\ep;\R^2}}^\ast\times L^2\tp{\Omega_\ep}
    \text{ with norm }
    \tnorm{\varphi^1,\varphi^2}_{\bf{Y}^0_\ep} = \tnorm{\varphi^1}_{\tp{H^1_{\m{tan}}}^\ast} + \tnorm{\varphi^2}_{L^2}.
\end{equation}
The strong codomain space is
\begin{equation}\label{the definition of the strong codomain space}
    \bf{Y}^1_\ep = L^2\tp{\Omega_\ep;\R^2}\times H^1\tp{\Omega_\ep}\times H^{1/2}\tp{\Sigma_\ep;\R^2}\times H^{1/2}\tp{\Sigma_0}
\end{equation}
with norm  given by
\begin{equation}\label{norm on Y1}
    \tnorm{\varphi^1,\varphi^2,\psi^1,\psi^2}_{\bf{Y}^1_\ep} = \tnorm{\varphi^1}_{L^2\tp{\Omega_\ep}} + \tnorm{\varphi^2}_{H^1\tp{\Omega_\ep}} + \tnorm{\psi^1_1}_{H^{1/2}\tp{\Sigma_\ep}}^{(\ep,+)} + \tnorm{\psi^1_2}_{H^{1/2}\tp{\Sigma_\ep}}^{\tp{\ep,-}} + \tnorm{\psi^2}_{H^{1/2}\tp{\Sigma_0}}^{\tp{\ep,+}},
\end{equation}
where we recall that the special $H^{1/2}\tp{\R}$ norms used here are defined in Lemma~\ref{lem on adapted boundary norms}.

Our first result studies the weak formulation of system~\eqref{linear stokes in a thin domain}.

\begin{propC}[Theory of weak solutions]\label{prop on theory of weak solutions}
    Let $\ep,\del\in(0,1)$ and $h\in Z_\del$. There exists a constant $C\in\R^+$, depending only on $\del$, $\upmu$, and $\m{a}$ such the following hold.
    \begin{enumerate}
        \item The bilinear form $B^{\ep}_h:H^1_{\m{tan}}\tp{\Omega_\ep;\R^2}\times H^1_{\m{tan}}\tp{\Omega_\ep;\R^2}\to\R$ given by
        \begin{equation}\label{bilinear form}
            B^{\ep}_h(u,v) = \f{\upmu}{2}\int_{\Omega_\ep}h\mathbb{D}^{\mathcal{A}_h}u:\mathbb{D}^{\mathcal{A}_h}v + \ep\m{a}\int_{\Sigma_0}u_1v_1
        \end{equation}
        is well-defined and continuous. Moreover for all $u,v\in H^1_{\m{tan}}\tp{\Omega_\ep;\R^2}$ we have the estimates
        \begin{equation}\label{estimates on the bilinear form}
            \tabs{B^\ep_h\tp{u,v}}\le C\tnorm{u}_{H^1}\tnorm{v}_{H^1}
            \text{ and }
            \tnorm{u}_{H^1}^2\le C B^\ep_h(u,u).
        \end{equation}
        \item The linear map $I^\ep_h:\bf{X}^0_\ep\to\tp{H^1_{\m{tan}}\tp{\Omega_\ep;\R^2}}^\ast$ given by
        \begin{equation}\label{the weak map I}
            \tbr{I^\ep_h\tp{u,p},v}_{H^1_{\m{tan}},\tp{H^1_{\m{tan}}}^\ast} = B^\ep_h\tp{u,v} - \int_{\Omega_\ep} h p \grad^{\mathcal{A}_h}\cdot v
        \end{equation}
        is well-defined and continuous. Moreover $\tnorm{I^\ep_h\tp{u,p}}_{\tp{H^1_{\m{tan}}}^\ast}\le C\tnorm{u,p}_{\bf{X}^0_\ep}$ for all $(u,p)\in\bf{X}^0_\ep$.
        \item The linear map $J^\ep_h:\bf{X}^0_\ep\to\bf{Y}^0_\ep$ given by $J^\ep_h(u,p) = \tp{I^\ep_h\tp{u,p},\grad^{\mathcal{A}_h}\cdot u}$ is well-defined, continuous, and invertible. Moreover, for all $\tp{u,p}\in\bf{X}^0_\ep$ we have the estimates
        \begin{equation}\label{estimates for weak solutions}
            \tnorm{u,p}_{\bf{X}^0_\ep}\le C\tnorm{J^\ep_h\tp{u,p}}_{\bf{Y}^0_\ep}
            \text{ and }
            \tnorm{J^\ep_h\tp{u,p}}_{\bf{Y}^0_\ep}\le C\tnorm{u,p}_{\bf{X}^0_\ep}.
        \end{equation}
    \end{enumerate}
\end{propC}
\begin{proof}
    Throughout the proof we permit implicit constants to depend only on $\del$, $\upmu$, and $\m{a}$.
    
    For the first item, the boundedness estimate for $B^\ep_h$ in~\eqref{estimates on the bilinear form} follows from the fact that $\tnorm{h,\mathcal{A}_h}_{L^\infty}\lesssim1$, Cauchy-Schwarz, and the trace estimate of the second item of Lemma~\ref{lem on traces in thin domains}. On the other hand, the coercive estimate in~\eqref{estimates on the bilinear form} is a consequence of the thin domain Korn inequality, Proposition~\ref{prop on Korn in thin domains}.

    The second item follows directly from the first item and Cauchy-Schwarz.

    We now turn our attention to the third item. The map $J^\ep_h$ is well-defined, continuous, and satisfies the right hand bound in~\eqref{estimates for weak solutions} due to simple computations and the first two items. We now prove the left hand estimate in~\eqref{estimates for weak solutions} and the invertibility of this operator.
    
    Let $\tp{\varphi^1,\varphi^2}\in\bf{Y}^0_\ep$. For $\tp{u,p}\in\bf{X}^0_\ep$ we use the right inverse to the divergence $\Uppi_h:L^2\tp{\Omega_\ep}\to H^1_{\m{tan}}\tp{\Omega_{\ep};\R^2}$ from the second item of Lemma~\ref{lem on anti-trace and anti-divergence} to derive the following equivalence:
    \begin{equation}
        J^\ep_h\tp{u,p} = \tp{\varphi^1,\varphi^2}\lra \begin{cases}
            w = u - \Uppi_h\varphi^2\in H^1_{\m{tan}}\tp{\Omega_\ep;\R^2},\\
            \tilde{\varphi}^1 = \varphi^1 - I^\ep_h\tp{\Uppi_h\varphi^2,0}\in\tp{H^1_{\m{tan}}\tp{\Omega_\ep;\R^2}}^\ast,\\
            J^\ep_h\tp{w,p} = \tp{\tilde{\varphi}^1,0}.
        \end{cases}
    \end{equation}
    As $\tnorm{u}_{H^1}\lesssim\tnorm{w}_{H^1} + \tnorm{\varphi^2}_{L^2}$ and $\tnorm{\tilde{\varphi}^1}_{\tp{H^1_{\m{tan}}}^\ast}\lesssim\tnorm{\varphi^1,\varphi^2}_{\bf{Y}^0_\ep}$, we have reduced to proving that there exists a unique $\tp{w,p}\in\bf{X}^0_\ep$ satisfying $J^\ep_h\tp{w,p} = \tp{\tilde{\varphi}^1,0}$ and  $\tnorm{w,p}_{\bf{X}^0_\ep}\lesssim\tnorm{\tilde{\varphi}^1}_{\tp{H^1_{\m{tan}}}^\ast}$.

    We now prove the latter estimate, which will also grant us uniqueness. We test the equation $I^\ep_h\tp{w,p} = \tilde{\varphi}^1$  with $w$ and note that $\grad^{\mathcal{A}_h}\cdot w = 0$ implies that the $p$ contributions vanish; this yields
    \begin{equation}
        B^\ep_h\tp{w,w} = \tbr{\tilde{\varphi}^1,w}_{\tp{H^1_{\m{tan}}}^\ast,H^1_{\m{tan}}}.
    \end{equation}
    By the coercive estimate in~\eqref{estimates on the bilinear form} it follows that $\tnorm{w}_{H^1}\lesssim\tnorm{\tilde{\varphi}^1}_{\tp{H^1_{\m{tan}}}^\ast}$. Now, to obtain a bound on $p$, we test the equation $I^\ep_h\tp{w,p} = \tilde{\varphi}^1$ with $\Uppi_h p\in H^1_{\m{tan}}\tp{\Omega_\ep;\R^2}$ and use the identity $\grad^{\mathcal{A}_h}\cdot\Uppi_h p=p$ to equate
    \begin{equation}
        \int_{\Omega_\ep}h\tabs{p}^2 = B^\ep_h\tp{w,\Uppi_h p} - \tbr{\tilde{\varphi}^1,\Uppi_h p}_{\tp{H^1_{\m{tan}}}^\ast,H^1_{\m{tan}}}.
    \end{equation}
    Since $\tnorm{\Uppi_hp}_{H^1}\lesssim\tnorm{p}_{L^2}$ and $\del\le h$, we deduce that $\tnorm{p}_{L^2}\lesssim\tnorm{\tilde{\varphi}^1}_{\tp{H^1_{\m{tan}}}^\ast}$.

    It remains to establish the existence of solutions to the reduced equation. Consider the following Hilbert space that is a closed subspace of $H^1_{\m{tan}}\tp{\Omega_\ep;\R^2}$:
    \begin{equation}\label{the solinoidal, tangential subspace}
        {_{\m{sol}_h}}H^1_{\m{tan}}\tp{\Omega_\ep;\R^2} = \tcb{v\in H^1_{\m{tan}}\tp{\Omega_\ep;\R^2}\;:\;\grad^{\mathcal{A}_h}\cdot v = 0}.
    \end{equation}
    The estimates~\eqref{estimates on the bilinear form} from the first item show that the restriction of the bilinear form $B^\ep_h$ to the product of the space~\eqref{the solinoidal, tangential subspace} with itself satisfies the hypotheses of the Lax-Milgram theorem (see, for instance, Theorem 6 in Chapter 6 of Lax~\cite{MR1892228}). The restriction of $\tilde{\varphi}^1$ to ${_{\m{sol}_h}}H^1_{\m{tan}}\tp{\Omega_\ep;\R^2}$ defines an element in the dual space to~\eqref{the solinoidal, tangential subspace} and so, by Lax-Milgram, there exists $w\in{_{\m{sol}_h}}H^1_{\m{tan}}\tp{\Omega_\ep;\R^2}$ with the property that
    \begin{equation}\label{sol identity}
        B^\ep_h\tp{w,v} = \tbr{\tilde{\varphi}^1,v}_{\tp{H^1_{\m{tan}}}^\ast,H^1_{\m{tan}}}
        \text{ for all }
        v\in{_{\m{sol}_h}}H^1_{\m{tan}}\tp{\Omega_\ep;\R^2}.
    \end{equation}
    Next, we define the bounded linear functional $Q\in\tp{L^2\tp{\Omega_\ep}}^\ast$ via
    \begin{equation}
        \tbr{Q,q}_{\tp{L^2}^\ast,L^2} = B^\ep_h\tp{w,\Uppi_hq} - \tbr{\tilde{\varphi}^1,\Uppi_hq}_{\tp{H^1_{\m{tan}}}^\ast,H^1_{\m{tan}}}
    \end{equation}
    and employ the Riesz representation theorem  (see, for instance, Theorem II.4 in Reed and Simon~\cite{MR493419}) to produce a $p\in L^2\tp{\Omega_\ep}$ such that 
    \begin{equation}\label{riesz identity}
        \tbr{Q,q}_{\tp{L^2}^\ast,L^2} = \int_{\Omega_\ep} h pq \text{ for }q\in L^2\tp{\Omega_\ep}.
    \end{equation}
    To conclude, we use~\eqref{sol identity} and~\eqref{riesz identity} to prove that for any $v\in H^1_{\m{tan}}\tp{\Omega_\ep;\R^2}$ it holds that
    \begin{multline}
        \tbr{I^\ep_h\tp{w,p},v}_{\tp{H^1_{\m{tan}}}^\ast,H^1_{\m{tan}}} = \tbr{I^\ep_h\tp{w,p},v - \Uppi_h\grad^{\mathcal{A}_h}\cdot v}_{\tp{H^1_{\m{tan}}}^\ast,H^1_{\m{tan}}} + \tbr{I^\ep_h\tp{w,p},\Uppi_h\grad^{\mathcal{A}_h}\cdot v}_{_{\tp{H^1_{\m{tan}}}^\ast,H^1_{\m{tan}}}}\\
        = \tbr{\tilde{\varphi}^1, v - \Uppi_h\grad^{\mathcal{A}_h}\cdot v}_{_{\tp{H^1_{\m{tan}}}^\ast,H^1_{\m{tan}}}} + B^\ep_h\tp{w,\Uppi_h\grad^{\mathcal{A}_h}\cdot v}_{\tp{H^1_{\m{tan}}}^\ast,H^1_{\m{tan}}} - \tbr{Q,\grad^{\mathcal{A}_h}\cdot v}_{\tp{L^2}^\ast,L^2} = \tbr{\tilde{\varphi}^1,v}_{\tp{H^1_{\m{tan}}}^\ast,H^1_{\m{tan}}},
    \end{multline}
    and hence $J^\ep_h\tp{w,p} = \tp{\tilde{\varphi}^1,0}$.
\end{proof}

Our next result studies the strong formulation of~\eqref{linear stokes in a thin domain}. We associate to the system a family of partial differential operators depending on $h\in Z_\del$ and establish a priori estimates and Lipschitz continuity. The strategy for regularity in the following mirrors that which is typical for the Stokes equations, e.g. see Sections 6 and 7 of Chapter IV in Boyer and Fabrie~\cite{MR2986590} and references therein. However, we must enumerate the details precisely as to keep track of the thin domain parameter $\ep$ and the adapted norms.

\begin{propC}[Theory of strong solutions]\label{prop on theory of strong solutions}
    Let $\ep,\del\in\tp{0,1}$, $h\in Z_\del$. The following hold for a constant $C\in\R^+$ depending only $\del$, $\upmu$, and $\m{a}$.
    \begin{enumerate}
        \item For $h\in Z_\del$, the linear map $L^\ep_h:\bf{X}^1_\ep\to\bf{Y}^1_\ep$ given via
        \begin{equation}\label{definition of the linear differential operator}
            L^\ep_h\tp{u,p} = \tp{\grad^{\mathcal{A}_h}p - \upmu\grad^{\mathcal{A}_h}\cdot\mathbb{D}^{\mathcal{A}_h}u,\grad^{\mathcal{A}_h}\cdot u,-\tp{p - \upmu\mathbb{D}^{\mathcal{A}_h}u}\mathcal{N}_{\ep h},\upmu\pd_2^{\mathcal{A}_h}u_1 - \ep\m{a}u_1}
        \end{equation}
        is well-defined and continuous. Moreover, for all $(u,p)\in\bf{X}^1_\ep$ we have the estimates
        \begin{equation}\label{a priori estimates and bounds for strong solutions}
            \tnorm{u,p}_{\bf{X}^1_\ep}^{\tp{+}}\le C\tnorm{L^\ep_h\tp{u,p}}_{\bf{Y}^1_\ep}
            \text{ and }
            \tnorm{L^\ep_h\tp{u,p}}_{\bf{Y}^1_\ep}\le C\tnorm{u,p}_{\bf{X}^1_\ep}^{\tp{+}}.
        \end{equation}
        \item For $h,\tilde{h}\in Z_\del$ and $(u,p)\in\bf{X}^1_\ep$, we have the estimate
        \begin{equation}\label{forward Lipschitz for strong solutions}
            \tnorm{\tp{L^\ep_h - L^\ep_{\tilde{h}}}\tp{u,p}}_{\bf{Y}^1_\ep}\le C\tnorm{h - \tilde{h}}_{\bf{Z}}\tnorm{u,p}_{\bf{X}^1_\ep}^{\tp{+}}.
        \end{equation}
    \end{enumerate}
\end{propC}
\begin{proof}
    Throughout the proof we permit implicit constants or the symbol $C$ to depend only on $\del$, $\upmu$, and $\m{a}$. We divide the proof into steps.
    
    \emph{Step 1 -- The $\mathcal{A}$-gradient:} We begin with the following auxiliary claim: for $n\in\tcb{0,1}$ and $h\in Z_\del$ the map
    \begin{equation}\label{geometry gradient}
        G^\ep_h:H^{n+1}\tp{\Omega_{\ep}} \to H^n\tp{\Omega_\ep;\R^2} 
        \text{ defined by }
        G^\ep_h\varphi = \grad^{\mathcal{A}_h}\varphi
    \end{equation}
    is well-defined, continuous, and obeys the bounds
    \begin{equation}\label{geometry gradient estimates}
        \tnorm{G^\ep_h\varphi}_{H^n\tp{\Omega_\ep}}\le C\tnorm{\varphi}_{H^{n+1}\tp{\Omega_\ep}}
        \text{ and }
        \tnorm{\tp{G^\ep_h - G^\ep_{\tilde{h}}}\varphi}_{H^n\tp{\Omega_\ep}}\le C\tnorm{h - \tilde{h}}_{\bf{Z}}\tnorm{\varphi}_{H^{n+1}\tp{\Omega_\ep}}
    \end{equation}
    for all $\varphi\in H^{n+1}\tp{\Omega_\ep}$ and $h,\tilde{h}\in Z_\del$. The case $n = 0$ is straightforward: the embedding $\bf{Z}\emb W^{1,\infty}\tp{\R}$ leads to the bounds
    \begin{equation}\label{bounds on the A matrix}
        \tnorm{\mathcal{A}_h}_{L^\infty}\le C 
        \text{ and }
        \tnorm{\mathcal{A}_h - \mathcal{A}_{\tilde{h}}}_{L^\infty}\le C\tnorm{h - \tilde{h}}_{\bf{Z}}
    \end{equation}
    so the case $n=0$ is handled with a direct estimate.

    For the case $n=1$, it suffices to study $\pd_j G^\ep_h\varphi$ mapping into $L^2\tp{\Omega_\ep;\R^{2}}$ for $j\in\tcb{1,2}$. We use the product rule to expand
    \begin{equation}
        \pd_j G^\ep_h\varphi = G^\ep_h\pd_j\varphi + \tp{\pd_j\mathcal{A}_h}\grad\varphi.
    \end{equation}
The already established $n=0$ mapping properties in~\eqref{geometry gradient estimates} permit us to ignore the first term on the right hand side above and focus on the second one.  

When $j=2$ we calculate that 
\begin{equation}
    \pd_2\mathcal{A}_h = -\bpm0&h'/h\\0&0\epm,
    \text{ and hence }
    \tnorm{\pd_2\mathcal{A}_h}_{L^\infty}\le C
    \text{ and }
    \tnorm{\pd_2\mathcal{A}_h - \pd_2\mathcal{A}_{\tilde{h}}}_{L^\infty}\le C\tnorm{h - \tilde{h}}_{\bf{Z}}.
\end{equation}
The desired bounds on $\tp{\pd_2\mathcal{A}_h}\grad\varphi$ now follow.

On the other hand, when $j =  1$, we need to be more careful. We now calculate that
\begin{equation}\label{the decomposition of partial 1 A}
    \pd_1\mathcal{A}_h = A^1_h + A^2_h,
    \text{ for }
     A_h^1 = \bpm0&y(h')^2/h^2\\0&-h'/h\epm
     \text{ and }
     A^2_h = \bpm0&-yh''/h\\0&0\epm.
\end{equation}
By the same strategy as before, we see that the $A^1_h$ contribution satisfies
\begin{equation}\label{the estimate on the good guy in the decomposition}
    \tnorm{A^1_h}_{L^\infty}\le C
    \text{ and }
    \tnorm{A^1_h  - A^1_{\tilde{h}}}_{L^\infty}\le C\tnorm{h - \tilde{h}}_{\bf{Z}},
\end{equation}
which means we only need  to consider the remaining $A^2_h$ contribution. According to the definition of the space $\bf{Z}$ given in~\eqref{the Z space} and~\eqref{the Z space, norm}, we may decompose $h = h^0 + h^1$ with $\tnorm{h^0}_{W^{2,3}}\le 2\del^{-1}$ and  $\tnorm{h^1}_{W^{2,\infty}}\le 2\del^{-1}$. In turn, after using the bounds $y<\ep$ and $\del\le h$ we have that
\begin{equation}\label{this is the one which requires more attention}
    \tnorm{A^2_h\grad\varphi}_{L^2\tp{\Omega_\ep}}\lesssim\ep\tnorm{\tp{h^0}''\pd_2\varphi}_{L^2\tp{\Omega_\ep}} + \ep\tnorm{\tp{h^1}''\pd_2\varphi}_{L^2\tp{\Omega_\ep}}\lesssim\ep\tnorm{\varphi}_{H^1} + \ep\tnorm{h^0}_{W^{2,3}\tp{\R}}\tnorm{\pd_2\varphi}_{L^6\tp{\Omega_\ep}}.
\end{equation}
The second item of Lemma~\ref{lem on poincare and sobolev in thin domains} provides the estimate $\ep^{1/3}\tnorm{\pd_2\varphi}_{L^6\tp{\Omega_\ep}}\lesssim\tnorm{\varphi}_{H^2\tp{\Omega_\ep}}$. Combining these observations with~\eqref{this is the one which requires more attention} yields the estimate $\tnorm{A^2_h\grad\varphi}_{L^2\tp{\Omega_\ep}}\lesssim \ep^{2/3}\tnorm{\varphi}_{H^2\tp{\Omega_\ep}}$, but then similar arguments show that
\begin{equation}
    \tnorm{\tp{A^2_h - A^2_{\tilde{h}}}\grad\varphi}_{L^2\tp{\Omega_\ep}}\lesssim\ep^{2/3}\tnorm{h - \tilde{h}}_{\bf{Z}}\tnorm{\varphi}_{H^2\tp{\Omega_\ep}}.
\end{equation}

We combine the above bounds on $A^1_h$, $A^2_h$ to acquire the desired estimates on $\tp{\pd_1\mathcal{A}_h}\grad\varphi$, which completes the proof of the claims~\eqref{geometry gradient} and~\eqref{geometry gradient estimates} in the case $n=1$.

\emph{Step 2 -- Lipschitz bounds on the operator norm:} We are now in a position to prove the second estimate of~\eqref{a priori estimates and bounds for strong solutions} and the Lipschitz estimate of~\eqref{forward Lipschitz for strong solutions}. This will be achieved by splitting the operator $L^\ep_h$ into various pieces.

Consider first the stress tensor maps ${_1}S^\ep_h,{_2}S^\ep_h:\bf{X}^1_\ep\to H^1\tp{\Omega_\ep;\R^{2\times 2}}$ defined via
\begin{equation}\label{the stress tensor decomposition}
    {_1}S^\ep_h\tp{u,p} = pI - 2\upmu\bpm\pd_1^{\mathcal{A}_h}u_1&0\\0&\pd_2^{\mathcal{A}_h}u_2\epm
    \text{ and }
    {_2}S^\ep_h\tp{u,p} = -\upmu\bpm0&\pd^{\mathcal{A}_h}_2u_1 + \pd^{\mathcal{A}_h}_1u_2\\\pd^{\mathcal{A}_h}_2u_1 + \pd^{\mathcal{A}_h}_1u_2&0\epm.
\end{equation}
The estimates we have established  for the  maps $G^\ep_h$ from~\eqref{geometry gradient} allow us to read off the following bounds for $j\in\tcb{1,2}$:
\begin{equation}\label{the stress tensor bounds}
        \tnorm{{_j}S^\ep_h\tp{u,p}}_{H^1\tp{\Omega_\ep}}\le C\tnorm{u,p}_{\bf{X}^1_\ep}
        \text{ and }
        \tnorm{\tp{{_j}S^\ep_h - {_j}S^\ep_{\tilde{h}}}\tp{u,p}}_{H^1\tp{\Omega_\ep}}\le C\tnorm{h - \tilde{h}}_{\bf{Z}}\tnorm{u,p}_{\bf{X}^1_\ep}.
\end{equation}
Similarly, we have the following estimates for the $\grad^{\mathcal{A}_h}$-divergence of ${_j}S^\ep_h$:
\begin{equation}\label{the div stress tensor bounds}
    \tnorm{\grad^{\mathcal{A}_h}\cdot {_j}S^\ep_h\tp{u,p}}_{L^2\tp{\Omega_\ep}}\le C\tnorm{u,p}_{\bf{X}^1_\ep}
    \text{ and }
     \tnorm{\tp{\grad^{\mathcal{A}_h}\cdot {_j}S^\ep_h - \grad^{\mathcal{A}_{\tilde{h}}}\cdot {_j}S^\ep_{\tilde{h}}}\tp{u,p}}_{L^2\tp{\Omega_\ep}}\le C\tnorm{h - \tilde{h}}_{\bf{Z}}\tnorm{u,p}_{\bf{X}^1_\ep}.
\end{equation}

Summing ${_1}S^\ep_h + {_2}S^\ep_h$ and employing ~\eqref{the div stress tensor bounds} then yields bounds for the  momentum equation components of $L^\ep_h$:
\begin{equation}\label{momentum equation bounds, I}
    \tnorm{\grad^{\mathcal{A}_h}p - \upmu\grad^{\mathcal{A}_h}\cdot\mathbb{D}^{\mathcal{A}_h}u}_{L^2\tp{\Omega_\ep}}\le C\tnorm{u,p}_{\bf{X}^1_\ep}, \text{ and }
\end{equation}
\begin{equation}\label{momentum equation bounds, II}
    \tnorm{\tp{\grad^{\mathcal{A}_h} - \grad^{\mathcal{A}_{\tilde{h}}}}p - \upmu\tp{\grad^{\mathcal{A}_h}\cdot\mathbb{D}^{\mathcal{A}_h} - \grad^{\mathcal{A}_{\tilde{h}}}\cdot\mathbb{D}^{\mathcal{A}_{\tilde{h}}}}u}_{L^2\tp{\Omega_\ep}}\le C\tnorm{h - \tilde{h}}_{\bf{Z}}\tnorm{u,p}_{\bf{X}^1_\ep}.
\end{equation}
Using~\eqref{the stress tensor bounds} with $p=0$ and taking the trace of ${_1}S^\ep_h$, we also may read off the $L^\ep_h$ divergence equation bounds:
\begin{equation}\label{divergence equation bounds}
    \tnorm{\grad^{\mathcal{A}_h}\cdot u}_{H^1\tp{\Omega_\ep}}\le C\tnorm{u,p}_{\bf{X}^1_\ep}
    \text{ and }
    \tnorm{\tp{\grad^{\mathcal{A}_h} - \grad^{\mathcal{A}_{\tilde{h}}}}\cdot u}_{H^1\tp{\Omega_\ep}}\le C\tnorm{h - \tilde{h}}_{\bf{Z}}\tnorm{u,p}_{\bf{X}^1_\ep}.
\end{equation}
Note that~\eqref{momentum equation bounds, I}, \eqref{momentum equation bounds, II}, and~\eqref{divergence equation bounds} give stronger Lipschitz estimates than what is asserted in~\eqref{forward Lipschitz for strong solutions}, as we have not yet encountered the adapted norms~\eqref{stronger, adapted norms}.

The bulk components of $L^\ep_h$ are now handled, so we next turn our attention to the boundary components, for which case the adapted norms~\eqref{stronger, adapted norms} will now play a role.  We first consider the Navier slip component of $L^\ep_h$. Due to the embedding $\bf{Z}\emb W^{1,\infty}\tp{\R}$, the boundedness of the product map $H^{1/2}\tp{\R}\times W^{1,\infty}\tp{\R}\to H^{1/2}\tp{\R}$, and the estimates from the second item of Lemma~\ref{lem on traces in thin domains} and the first item of Lemma~\ref{lem on adapted boundary norms}, we find the estimates
\begin{equation}\label{Navier slip I}
    \tnorm{\pd_2^{\mathcal{A}_h}u_1 - \ep\m{a}u_1}_{H^{1/2}\tp{\Sigma_0}}^{\tp{\ep,+}}\lesssim\tnorm{\pd_2 u_1}_{H^{1/2}\tp{\Sigma_0}}^{\tp{\ep,+}} + \tnorm{u_1}_{H^1\tp{\Omega_\ep}}\lesssim\tnorm{u,p}^{\tp{+}}_{\bf{X}^1_\ep},
\end{equation}
and similarly
\begin{equation}\label{Navier slip II}
    \tnorm{\tp{\pd_2^{\mathcal{A}_h} - \pd_2^{\mathcal{A}_{\tilde{h}}}}u_1}^{\tp{\ep,+}}_{H^{1/2}\tp{\Sigma_0}}\lesssim \tnorm{h - \tilde{h}}_{\bf{Z}}\tnorm{u,p}^{\tp{+}}_{\bf{X}^1_\ep}.
\end{equation}

Now we consider the stress boundary condition components of $L^\ep_h$. According to~\eqref{the stress tensor decomposition},  on $\Sigma_\ep$ we have the equality
\begin{equation}\label{starting to look at the boundaries}
    \tp{p - \upmu\mathbb{D}^{\mathcal{A}_h}u}\mathcal{N}_{\ep h} = {_1}S^\ep_h\tp{u,p}\mathcal{N}_{\ep h} + {_2}S^\ep_h\tp{u,p}\mathcal{N}_{\ep h}.
\end{equation}
We will resolve the terms on the right into components, recalling that  $\mathcal{N}_{\ep h} = \tp{-\ep h',1}$.  By using the boundedness of the product map $\tp{W^{1,3} + W^{1,\infty}}\tp{\R}\times H^{1/2}\tp{\R}\to H^{1/2}\tp{\R}$, the second item of Lemma~\ref{lem on traces in thin domains}, the first item of Lemma~\ref{lem on adapted boundary norms}, and the stress tensor bounds~\eqref{the stress tensor bounds} we get
\begin{equation}\label{DBC I}
    \tnorm{{_1}S^\ep_h\tp{u,p}\mathcal{N}_{\ep h}\cdot e_1}^{\tp{\ep,+}}_{H^{1/2}\tp{\Sigma_\ep}} + \tnorm{{_2}S^\ep_h\tp{u,p}\mathcal{N}_{\ep h}\cdot e_2}_{H^{1/2}\tp{\Sigma_\ep}}^{\tp{\ep,-}}\lesssim\tnorm{{_1}S^\ep_h\tp{u,p}}_{H^1\tp{\Omega_\ep}} + \tnorm{{_2}S^\ep_h\tp{u,p}}_{H^1\tp{\Omega_\ep}}\lesssim\tnorm{u,p}_{\bf{X}^1_\ep}
\end{equation}
and, by similar arguments,
\begin{equation}\label{DBC II}
    \tnorm{\tp{{_1}S^\ep_h\tp{u,p}\mathcal{N}_{\ep h} - {_1}S^\ep_{\tilde{h}}\tp{u,p}\mathcal{N}_{\ep\tilde{h}}}\cdot e_1}^{(\ep,+)}_{H^{1/2}\tp{\Sigma_\ep}} + \tnorm{\tp{{_2}S^\ep_h\tp{u,p}\mathcal{N}_{\ep h} - {_2}S^\ep_{\tilde{h}}\tp{u,p}\mathcal{N}_{\ep\tilde{h}}}\cdot e_2}^{\tp{\ep,-}}_{H^{1/2}\tp{\Sigma_\ep}}
    \lesssim\tnorm{h - \tilde{h}}_{\bf{Z}}\tnorm{u,p}_{\bf{X}^1_\ep}.
\end{equation}

For the term ${_1}S^\ep_h\tp{u,p}\mathcal{N}_{\ep h}\cdot e_2 = p - 2\upmu\pd_2^{\mathcal{A}_h}u_2$ we again use the second item of Lemma~\ref{lem on traces in thin domains} to get
\begin{equation}\label{DBC VI}
    \tnorm{{_1}S^\ep_h\tp{u,p}\mathcal{N}_{\ep h}\cdot e_2}_{H^{1/2}\tp{\Sigma_\ep}}^{\tp{\ep,-}}\lesssim\tnorm{p}^{\tp{\ep,-}}_{H^{1/2}\tp{\Sigma_\ep}} + \tnorm{ 2\upmu\pd_2^{\mathcal{A}_h}u_2}^{\tp{\ep,-}}_{H^{1/2}\tp{\Sigma_\ep}}\lesssim\tnorm{u,p}_{\bf{X}^1_\ep},
\end{equation}
but estimating the Lipschitz norm is no different:
\begin{equation}\label{DBC III}
    \tnorm{\tp{{_1}S^\ep_h\tp{u,p}\mathcal{N}_{\ep h} - {_1}S^\ep_{\tilde{h}}\tp{u,p}\mathcal{N}_{\ep\tilde{h}}}\cdot e_2}^{\tp{\ep,-}}_{H^{1/2}\tp{\Sigma_\ep}} \lesssim \tnorm{\tp{1/h - 1/\tilde{h}}\pd_2u_2}^{\tp{\ep,-}}_{H^{1/2}\tp{\Sigma_\ep}}\lesssim\tnorm{h - \tilde{h}}_{W^{1,\infty}}\tnorm{u,p}_{\bf{X}^1_\ep}.
\end{equation}

The story is similar for ${_2}S^\ep_h\tp{u,p}\mathcal{N}_{\ep h}\cdot e_1 = \pd_2^{\mathcal{A}_h}u_1 + \pd_1^{\mathcal{A}_h}u_2 = (\pd_2 u_1 + h\pd_1 u_2 - \ep h'\pd_2 u_2)/h$. Via the aforementioned product estimates, the definition of the adapted norms~\eqref{stronger, adapted norms}, the first and second items of Lemma~\ref{lem on traces in thin domains}, the first item of Lemma~\ref{lem on adapted boundary norms}, and the fact that $\m{Tr}_{\Sigma_0}\pd_1u_2 = 0$, we find that
\begin{equation}\label{DBC IV}
    \tnorm{{_2}S^\ep_h\tp{u,p}\mathcal{N}_{\ep h}\cdot e_1}^{\tp{\ep,+}}_{H^{1/2}\tp{\Sigma_\ep}}\lesssim\tnorm{\pd_2 u_1}_{H^{1/2}\tp{\Sigma_\ep}}^{\tp{\ep,+}} + \tnorm{\pd_1u_2}^{\tp{\ep,+}}_{H^{1/2}\tp{\Sigma_\ep}} + \ep\tnorm{\pd_2u_2}^{\tp{\ep,+}}_{H^{1/2}\tp{\Sigma_\ep}}\lesssim\tnorm{u,p}^{\tp{+}}_{\bf{X}^1_\ep}.
\end{equation}
For the difference, the same tricks give the bound
\begin{equation}\label{DBC V}
    \tnorm{\tp{{_2}S^\ep_h\tp{u,p}\mathcal{N}_{\ep h} - {_2}S^\ep_{\tilde{h}}\tp{u,p}\mathcal{N}_{\ep\tilde{h}}}\cdot e_1}^{\tp{\ep,+}}_{H^{1/2}\tp{\Sigma_\ep}}\lesssim\tnorm{h - \tilde{h}}_{\bf{Z}}\tnorm{u,p}^{\tp{+}}_{\bf{X}^1_\ep}.
\end{equation}

Combining~\eqref{momentum equation bounds, I}, \eqref{divergence equation bounds}, \eqref{Navier slip I}, \eqref{starting to look at the boundaries}, \eqref{DBC I}, \eqref{DBC VI}, and~\eqref{DBC IV} grants us the operator bound $\tnorm{L^\ep_h\tp{u,p}}_{\bf{Y}^1_\ep}\lesssim\tnorm{u,p}_{\bf{X}^1_\ep}^{\tp{+}}$ of the first item. On the other hand, combining~\eqref{momentum equation bounds, II}, \eqref{divergence equation bounds}, \eqref{Navier slip II}, \eqref{starting to look at the boundaries}, \eqref{DBC II}, \eqref{DBC III}, and~\eqref{DBC V} grants us the claimed Lipschitz bound~\eqref{forward Lipschitz for strong solutions} of the second item.

\emph{Step 3 -- The a priori estimate:} It remains only to prove the a priori estimate $\tnorm{u,p}^{\tp{+}}_{\bf{X}^1_\ep}\lesssim\tnorm{L^\ep_h\tp{u,p}}_{\bf{Y}^1_\ep}$ from the first item.  Let $(u,p)\in\bf{X}^1_\ep$ and define $\upsilon = (\varphi^1,\varphi^2,\psi^1,\psi^2)\in\bf{Y}^1_\ep$ via $L^\ep_h\tp{u,p} = \upsilon$ so that system~\eqref{linear stokes in a thin domain} is satisfied.  

For any $w\in H^1_{\m{tan}}\tp{\Omega_\ep;\R^2}$ we test the fist equation in system~\eqref{linear stokes in a thin domain} with $hw$ and integrate by parts to derive the identity
\begin{equation}\label{the weak formulation identity}
    \tbr{I^\ep_h\tp{u,p},w}_{\tp{H^1_{\m{tan}}}^\ast, H^1_{\m{tan}}} = \tbr{T_{\upsilon},w}_{\tp{H^1_{\m{tan}}}^\ast, H^1_{\m{tan}}},
\end{equation}
where $I^\ep_h$ is the linear map considered in the second item of Proposition~\ref{prop on theory of weak solutions} and $T_{\upsilon}\in\tp{H^1_{\m{tan}}\tp{\Omega_\ep;\R^2}}^\ast$ is defined by 
\begin{equation}\label{the T_upsilon map}
    \tbr{T_{\upsilon},w}_{\tp{H^1_{\m{tan}}}^\ast, H^1_{\m{tan}}} = \int_{\Omega_\ep}h\varphi^1\cdot w + \int_{\Sigma_\ep}\psi^1\cdot w - \int_{\Sigma_0}\psi^2 w_1.
\end{equation}
In the language of the third item of Proposition~\ref{prop on theory of weak solutions}, we have that $J^\ep_h\tp{u,p} = \tp{T_{\upsilon},\varphi^2}$ and so, from the same result, we have the (low regularity) estimate
\begin{equation}
    \tnorm{u,p}_{\bf{X}^0_\ep}\lesssim\tnorm{T_{\upsilon},\varphi^2}_{\bf{Y}^0_\ep} = \tnorm{T_\upsilon}_{(H^1_{\m{tan}})^\ast} + \tnorm{\varphi^2}_{L^2\tp{\Omega_\ep}}.
\end{equation}
The definition of the norm on $\bf{Y}^1_\ep$~\eqref{norm on Y1} shows that $\tnorm{\varphi^2}_{L^2}\le\tnorm{\upsilon}_{\bf{Y}^1_\ep}$. On the other hand, Cauchy-Schwarz and the third and fourth items of Lemma~\ref{lem on traces in thin domains} show that
\begin{equation}
    \tabs{\tbr{T_\upsilon,w}_{\tp{H^1_{\m{tan}}}^\ast,H^1_{\m{tan}}}}\lesssim\tp{\tnorm{\varphi^1}_{L^2\tp{\Omega_\ep}} + \tnorm{\psi^1_1}^{\tp{\ep,+}}_{H^{1/2}\tp{\Sigma_\ep}} + \tnorm{\psi^1_2}^{\tp{\ep,-}}_{H^{1/2}\tp{\Sigma_\ep}} + \tnorm{\psi^2}^{\tp{\ep,+}}_{H^{1/2}\tp{\Sigma_0}} }\tnorm{w}_{H^1\tp{\Omega_\ep}},
\end{equation}
which implies that $\tnorm{T_{\upsilon}}_{\tp{H^1_{\m{tan}}}^\ast}\lesssim\tnorm{\upsilon}_{\bf{Y}^1_\ep}$. Hence, $\tnorm{u,p}_{\bf{X}^0_\ep}\lesssim\tnorm{\upsilon}_{\bf{Y}^1_\ep}$.

We need to estimate one more derivative of $u$ and $p$ in terms of the data $\upsilon$; in doing so we will break to cases, first considering the tangential ($e_1$ direction) derivatives and then the normal ($e_2$ direction) derivatives.  For the tangential derivative  we again employ the weak formulation, but this time the one satisfied by $\pd_1 u$ and $\pd_1 p$.  To compute this we let $w\in H^2_{\m{tan}}\tp{\Omega_\ep;\R^2}$, insert $-\pd_1 w\in H^1_{\m{tan}}\tp{\Omega_\ep;\R^2}$ into the weak formulation~\eqref{the weak formulation identity}, and integrate by parts to  arrive at the identity
\begin{equation}\label{der_weak_1}
    \tbr{I^\ep_h\tp{\pd_1u,\pd_1 p},w}_{\tp{H^1_{\m{tan}}}^\ast,H^1_{\m{tan}}} + \tbr{C^\ep_h\tp{u,p},w}_{\tp{H^1_{\m{tan}}}^\ast,H^1_{\m{tan}}} = \tbr{T^1_{\upsilon},w}_{\tp{H^1_{\m{tan}}},^\ast,H^1_{\m{tan}}}
\end{equation}
where
\begin{multline}\label{der_weak_2}
    \tbr{C^\ep_h\tp{u,p},w}_{\tp{H^1_{\m{tan}}}^\ast,H^1_{\m{tan}}} = \int_{\Omega_\ep}\pd_1h\bp{\f{\upmu}{2}\mathbb{D}^{\mathcal{A}_h}u:\mathbb{D}^{\mathcal{A}_h}w - p\grad^{\mathcal{A}_h}\cdot w}\\
    + \int_{\Omega_\ep}\f{\upmu h}{2}\tp{\mathbb{D}^{\pd_1\mathcal{A}_h}u:\mathbb{D}^{\mathcal{A}_h}w + \mathbb{D}^{\mathcal{A}_h}u:\mathbb{D}^{\pd_1\mathcal{A}_{h}}w} - hp\grad^{\pd_1\mathcal{A}_h}\cdot w,
\end{multline}
and
\begin{equation}\label{der_weak_3}
    \tbr{T^1_{\upsilon},w}_{\tp{H^1_{\m{tan}}}^\ast,H^1_{\m{tan}}} = -\int_{\Omega_\ep}h\varphi^1\cdot\pd_1 w - \int_{\Sigma_\ep}\tbr{D}^{1/2}\psi^1\cdot \tbr{D}^{-1/2}\pd_1 w + \int_{\Sigma_0}\tbr{D}^{1/2}\psi^2\cdot\tbr{D}^{-1/2}\pd_1w_1.
\end{equation}
By a simple density argument, we find that identities~\eqref{der_weak_1}, \eqref{der_weak_2}, and~\eqref{der_weak_3} continue to hold for all $w\in H^1_{\m{tan}}\tp{\Omega_\ep;\R^2}$. We also find that 
\begin{equation}
    \grad^{\mathcal{A}_h}\cdot\pd_1 u + \grad^{\pd_j\mathcal{A}_h}\cdot u = \pd_1\varphi^2,
\end{equation}
and so the third item of Proposition~\ref{prop on theory of weak solutions} provides the estimate
\begin{equation}\label{set up for the tangential regularity estimate}
    \tnorm{\pd_1u,\pd_1 p}_{\bf{X}^0_\ep}\lesssim \tnorm{T^1_\upsilon}_{\tp{H^1_{\m{tan}}}^\ast} + \tnorm{\pd_1\varphi^2}_{L^2\tp{\Omega_\ep}} + \tnorm{C^\ep_h\tp{u,p}}_{\tp{H^1_{\m{tan}}}^\ast} + \tnorm{\grad^{\pd_1\mathcal{A}_h}\cdot u}_{L^2\tp{\Omega_\ep}}.
\end{equation}

We now estimate each of the four contributions to the right hand side of~\eqref{set up for the tangential regularity estimate}. The first of these is handled via the second and third items of Lemma~\ref{lem on traces in thin domains} and the definition of the norm~\eqref{norm on Y1}: 
\begin{equation}
    \tabs{\tbr{T^1_\upsilon,w}_{\tp{H^1_{\m{tan}}}^\ast,H^1_{\m{tan}}}}\lesssim\tp{\tnorm{\varphi^1}_{L^2\tp{\Omega_\ep}} + \tnorm{\psi^1_1}^{\tp{\ep,+}}_{H^{1/2}\tp{\Sigma_\ep}} + \tnorm{\psi^1_2}^{\tp{\ep,-}}_{H^{1/2}\tp{\Sigma_\ep}} + \tnorm{\psi^2}^{\tp{\ep,+}}_{H^{1/2}\tp{\Sigma_0}}}\tnorm{w}_{H^1\tp{\Omega_\ep}},
\end{equation}
which implies that $\tnorm{T^1_\upsilon}_{\tp{H^1_{\m{tan}}}^\ast}\lesssim\tnorm{\upsilon}_{\bf{Y}^1_\ep}$.
The second term in~\eqref{set up for the tangential regularity estimate} is directly estimated with $\tnorm{\pd_1\varphi^2}_{L^2\tp{\Omega_\ep}}\le\tnorm{\upsilon}_{\bf{Y}^1_\ep}$.

The third term in~\eqref{set up for the tangential regularity estimate} is more subtle. We initially use the embedding $\bf{Z}\emb W^{1,\infty}\tp{\R}$, the left hand estimate of~\eqref{bounds on the A matrix}, and Cauchy-Schwarz to bound
\begin{equation}
    \tabs{\tbr{C^\ep_h\tp{u,p},w}_{\tp{H^1_{\m{tan}}}^\ast,H^1_{\m{tan}}}}\lesssim\bp{\int_{\Omega_\ep}\tp{1 + \tabs{\pd_1\mathcal{A}_h}}^2\tp{\tabs{p}^2 + \tabs{\grad u}^2}}^{1/2}\tnorm{w}_{H^1\tp{\Omega_\ep}}.
\end{equation}
We combine this with the already established bound $\tnorm{u,p}_{\bf{X}^0_\ep}\lesssim\tnorm{\upsilon}_{\bf{Y}^1_\ep}$, the decomposition $\pd_1\mathcal{A}_h = A_h^1 + A_h^2$ from~\eqref{the decomposition of partial 1 A}, and the left hand estimate of~\eqref{the estimate on the good guy in the decomposition}, to the see that
\begin{equation}\label{another intermediate estimate}
    \tnorm{C^\ep_h\tp{u,p}}_{\tp{H^1_{\m{tan}}}^\ast}\lesssim\tnorm{\upsilon}_{\bf{Y}^1_\ep} + \ep\bp{\int_{\Omega_\ep}\tabs{h''}^2\tp{\tabs{p}^2 + \tabs{\grad u}^2}}^{1/2}.
\end{equation}
Now, by the definition of the norm~\eqref{the Z space, norm} on the space $\bf{Z}$ we have $h = h^0 + h^1$ with $\tnorm{h^0}_{W^{2,3}}\le 2\del^{-1}$ and $\tnorm{h^1}_{W^{2,\infty}}\le 2\del^{-1}$. We split the $h$ appearing in~\eqref{another intermediate estimate} according to this decomposition. The $h^0$ contribution is then bounded in $L^3\tp{\R}$, whereas the $h^1$ contribution can be estimated in $L^\infty\tp{\R}$. For the $h^0$ contribution we then use H\"older and estimate the resulting $L^6\tp{\Omega_\ep}$-norm on $\grad u$ and $p$ by interpolating between the $L^2\tp{\Omega_\ep}$ and $L^{12}\tp{\Omega_\ep}$ norms. Executing this gives us the bound
\begin{equation}\label{a nice C estimate}
    \tnorm{C^\ep_h\tp{u,p}}_{\tp{H^1_{\m{tan}}}^\ast}\lesssim\tnorm{\upsilon}_{\bf{Y}^1_\ep} + \ep\tnorm{\upsilon}_{\bf{Y}^1_\ep}^{1/5}\tnorm{p,\grad u}^{4/5}_{L^{12}\tp{\Omega_\ep}}.
\end{equation}
Finally, we use the second item of Lemma~\ref{lem on poincare and sobolev in thin domains} to bound $\tnorm{p,\grad u}_{L^{12}\tp{\Omega_\ep}}\lesssim \ep^{-5/12}\tnorm{u,p}_{\bf{X}^1_\ep}$.  The same strategy also allows us to estimate the final contribution in~\eqref{set up for the tangential regularity estimate}. Synthesizing these bounds, we derive the tangential regularity estimate
\begin{equation}\label{tangential regularity estimate}
    \tnorm{\pd_1 u,\pd_1 p}_{\bf{X}^0_\ep}\lesssim\tnorm{\upsilon}_{\bf{Y}^1_\ep} + \ep^{2/3}\tnorm{\upsilon}_{\bf{Y}^1_\ep}^{1/5}\tnorm{u,p}^{4/5}_{\bf{X}^1_\ep}.
\end{equation}

With the tangential estimates~\eqref{tangential regularity estimate} in hand, we now aim to promote to normal derivative estimates.  In doing so, our first task is to solve for $\pd_2 p$ in terms of quantities that can be estimated by $\tnorm{\upsilon}_{\bf{Y}^1_\ep}$ or $\tnorm{\pd_1 u,\pd_1 p}_{\bf{X}^0_\ep}$, and we accomplish this by passing through an intermediate quantity derived from the divergence equation in~\eqref{linear stokes in a thin domain}.  Indeed, we apply $\pd_2$ to this equation and solve for $\pd_2^2 u\cdot\tilde{\mathcal{N}}_h$, where $\tilde{\mathcal{N}}_h = \tp{-yh',1} = \mathcal{M}_h^{\m{t}}e_2$:
\begin{equation}\label{solving for some weirdo}
    \pd_2^2 u\cdot\tilde{\mathcal{N}}_h = h\pd_2\varphi^2 + h'\pd_2 u_1 - h\pd_1\pd_2 u_1.
\end{equation}
With~\eqref{solving for some weirdo} in mind, we  rewrite the momentum equation in~\eqref{linear stokes in a thin domain} as
\begin{equation}\label{equivalent form of the momentum equation}
    \grad^{\mathcal{A}_h}p - \upmu\Delta^{\mathcal{A}_h}u = \varphi^1 +\upmu\grad^{\mathcal{A}_h}\varphi^2,
\end{equation}
take the inner product of~\eqref{equivalent form of the momentum equation} with the vector field $\tilde{\mathcal{N}}_h$, and then substitute~\eqref{solving for some weirdo} where appropriate to derive the identity
\begin{multline}\label{solving for some other weirdo}
    \tp{1 + (yh')^2}\pd_2 p/h = yh'\pd_1 p + \tp{\varphi^1 + \upmu\grad^{\mathcal{A}_h}\varphi^2 + \pd_1^2 u - 2yh'\pd_1\pd_2 u + y\tp{2(h'/h)^2 - h''/h}\pd_2u}\cdot\tilde{\mathcal{N}}_h\\
     + \tp{1 + \tp{yh'}^2}\tp{h\pd_2\varphi^2 + h'\pd_2u_1 - h\pd_1\pd_2 u_1}/h^2.
\end{multline}
We next divide both sides of~\eqref{solving for some other weirdo} by $\tp{1 + \tp{yh'}^2}/h$ and take the norm in $L^2\tp{\Omega_\ep}$. The bounds $\tnorm{h,1/h}_{W^{1,\infty}}\lesssim1$, along with~\eqref{tangential regularity estimate} can be applied to each term with a single exception, the product $h''\pd_2u$. This term instead requires the exact same strategy we have seen in~\eqref{a nice C estimate}. We therefore arrive at the estimate
\begin{equation}\label{yet another intermediate estiamte, wow this proof is getting way longer than I expected}
    \tnorm{\pd_2 p}_{L^2\tp{\Omega_\ep}}\lesssim\tnorm{\upsilon}_{\bf{Y}^1_\ep} + \ep^{2/3}\tnorm{\upsilon}_{\bf{Y}^1_\ep}^{1/5}\tnorm{u,p}^{4/5}_{\bf{X}^1_\ep} + \ep\tnorm{h''\pd_2 u}_{L^2\tp{\Omega_\ep}}\lesssim\tnorm{\upsilon}_{\bf{Y}^1_\ep} + \ep^{2/3}\tnorm{\upsilon}_{\bf{Y}^1_\ep}^{1/5}\tnorm{u,p}^{4/5}_{\bf{X}^1_\ep}.
\end{equation}

To get the normal regularity of the velocity, we solve for $\pd_2^2 u$ in identity~\eqref{equivalent form of the momentum equation}:
\begin{equation}\label{believe it or not, this equation does it referenced}
    \upmu\tp{1 + \tp{yh'}^2}\pd_2^2 u/h^2 = \grad^{\mathcal{A}_h}p - \varphi^1 - \upmu\grad^{\mathcal{A}_h}\varphi^2 - \upmu\tp{\pd_1^2 u - 2yh'\pd_1\pd_2u/h + y\tp{2\tp{h'/h}^2 - h''/h}\pd_2 u}.
\end{equation}
Taking the $L^2\tp{\Omega_\ep;\R^2}$-norm, employing~\eqref{tangential regularity estimate} and~\eqref{yet another intermediate estiamte, wow this proof is getting way longer than I expected}, and using the same care previously mentioned with the product $h''\pd_2 u$, we find
\begin{equation}\label{normal regularity for u}
    \tnorm{\pd_2^2 u}_{L^2\tp{\Omega_\ep}}\lesssim\tnorm{\upsilon}_{\bf{Y}^1_\ep} + \ep^{2/3}\tnorm{\upsilon}_{\bf{Y}^1_\ep}^{1/5}\tnorm{u,p}^{4/5}_{\bf{X}^1_\ep}.
\end{equation}

We now combine~\eqref{tangential regularity estimate}, \eqref{yet another intermediate estiamte, wow this proof is getting way longer than I expected}, and~\eqref{normal regularity for u} with the estimates on weak solutions and a simple adsorption argument to deduce that
\begin{equation}\label{ugh we are almost done with this really long proof}
    \tnorm{u,p}_{\bf{X}^1_\ep}\lesssim\tnorm{u,p}_{\bf{X}^0_\ep} + \tnorm{\pd_1 u,\pd_1 p}_{\bf{X}^0_\ep} + \tnorm{\pd_2^2u,\pd_2 p}_{L^2\tp{\Omega_\ep}}\lesssim \tnorm{\upsilon}_{\bf{Y}^1_\ep}.
\end{equation}
The estimate~\eqref{ugh we are almost done with this really long proof} is nearly the left bound in~\eqref{a priori estimates and bounds for strong solutions}, but we still need to promote to the adapted norms~\eqref{stronger, adapted norms}.  Fortunately, this is now simple and we need only solve for the special boundary quantities identified by the adapted norms in the equations~\eqref{linear stokes in a thin domain} and then measure their size.   If we look to the Navier slip boundary condition, we find that
\begin{equation}\label{Navier slip bonus estimate}
    \tnorm{\pd_2 u_1}_{H^{1/2}\tp{\Sigma_0}}^{\tp{\ep,+}}\lesssim\ep\tnorm{u_1}^{\tp{\ep,+}}_{H^{1/2}\tp{\Sigma_0}} + \tnorm{\psi^2}_{H^{1/2}\tp{\Sigma_0}}^{\tp{\ep,+}}\lesssim\tnorm{\upsilon}_{\bf{Y}^1_\ep}.
\end{equation}
On the other hand, in the $e_1$-stress boundary condition, we have
\begin{equation}\label{e1 DBC bonus estimate}
    \tnorm{\pd_2u_1}_{H^{1/2}\tp{\Sigma_\ep}}^{\tp{\ep,+}}\lesssim \ep\tnorm{p - 2\upmu\pd_1^{\mathcal{A}_h}u_1}^{\tp{\ep,+}}_{H^{1/2}\tp{\Sigma_\ep}} + \tnorm{\pd_1^{\mathcal{A}_h}u_2}^{\tp{\ep,+}}_{H^{1/2}\tp{\Sigma_\ep}} + \tnorm{\psi^1_1}_{H^{1/2}\tp{\Sigma_\ep}}^{\tp{\ep,+}}\lesssim\tnorm{\upsilon}_{\bf{Y}^1_\ep}.
\end{equation}
The left estimate of~\eqref{a priori estimates and bounds for strong solutions} then follows by  combining~\eqref{ugh we are almost done with this really long proof} with~\eqref{Navier slip bonus estimate} and~\eqref{e1 DBC bonus estimate}.
\end{proof}

Our next result synthesizes a method of continuity argument with the previous result.

\begin{thmC}[Synthesis of linear analysis]\label{thm on synthesis of linear analysis}
    Let $\ep,\del\in(0,1)$. For all $h\in Z_\del$, the linear map $L^\ep_h:\bf{X}^1_\ep\to\bf{Y}^1_\ep$ defined in~\eqref{definition of the linear differential operator} is invertible. Moreover, for a constant $C\in\R^+$ depending only on $\del$, $\upmu$, and $\m{a}$, we have the following estimates for all $\upsilon = \tp{\varphi^1,\varphi^2,\psi^1,\psi^2}\in\bf{Y}^1_\ep$ and $h,\tilde{h}\in Z_\del$:
    \begin{equation}\label{estimates on the inverse stokes operator}
        \tnorm{\tp{L^\ep_h}^{-1}\upsilon}_{\bf{X}^1_\ep}^{\tp{+}}\le C\tnorm{\upsilon}_{\bf{Y}^1_\ep}
         \text{ and }
        \tnorm{\tsb{\tp{L^\ep_h}^{-1} - \tp{L^\ep_{\tilde{h}}}^{-1}}\upsilon}_{\bf{X}^1_\ep}^{\tp{+}}\le C\tnorm{h - \tilde{h}}_{\bf{Z}}\tnorm{\upsilon}_{\bf{Y}^1_\ep}.
    \end{equation}
\end{thmC}
\begin{proof}
    We begin by proving that the map~\eqref{definition of the linear differential operator} is invertible. Once this is known, the left hand operator norm estimate in~\eqref{estimates on the inverse stokes operator} follows immediately from estimate~\eqref{a priori estimates and bounds for strong solutions} from Proposition~\ref{prop on theory of strong solutions}.      By the method of continuity (see, for instance, Theorem 5.2 in Gilbarg and Trudinger~\cite{MR1814364} or, for a more general version, Theorem 4.51 in Abramovich and Aliprantis~\cite{MR1921782}), the Lipschitz estimate from the second item of Proposition~\ref{prop on theory of strong solutions}, and the a priori estimate from the first item of the aforementioned result, we see that it suffices to prove that when $h = 1\in Z_\del$ the operator $L^\ep_1:\bf{X}^1_\ep\to\bf{Y}^1_\ep$ is an isomorphism.        
    
    Proposition~\ref{prop on theory of strong solutions} shows that $L^\ep_1:\bf{X}^1_\ep\to\bf{Y}^1_\ep$ is injective, so we need only establish surjectivity here.  Let $\upsilon = \tp{\varphi^1,\varphi^2,\psi^1,\psi^2}\in\bf{Y}^1_\ep$ and define $T_{\upsilon}\in\tp{H^1_{\m{tan}}\tp{\Omega;\R^2}}^\ast$ as in~\eqref{the T_upsilon map}.  By the theory of weak solutions (specifically, the third item of Proposition~\ref{prop on theory of weak solutions}) we are guaranteed the existence of $(u,p)\in\bf{X}^0_\ep$ such that $J^\ep_1\tp{u,p} = \tp{T_\upsilon,\varphi^2}$.  It suffices to check that $(u,p)$ belongs to the higher regularity space $\bf{X}^1_\ep$, as this permits an integration by parts to learn that $L^\ep_1\tp{u,p} = \upsilon$. There are many ways to achieve this regularity promotion for $(u,p)$.  One possibility is to argue by standard elliptic regularity theory for systems  (see, for instance, Agmon, Douglis, and Nirenberg~\cite{adn2}).  An alternate approach is to observe that when $h=1$ we have that  $\mathcal{A}_h = I$, and so the system~\eqref{linear stokes in a thin domain} commutes with horizontal translations.  This then allows for a tangential to normal regularity promotion argument of the same form used in the proof of Theorem 2.5 in Stevenson and Tice~\cite{MR4787851}.  Taking one of these two routes, we obtain the regularity promotion and thereby  complete the proof that $L^\ep_1$ is an isomorphism.

    It remains to establish the Lipschitz estimate of~\eqref{estimates on the inverse stokes operator}.  Let $\upsilon\in\bf{Y}^1_\ep$ and $h,\tilde{h}\in Z_\del$. Set $\chi = \tp{L^\ep_h}^{-1}\upsilon - \tp{L^\ep_{\tilde{h}}}^{-1}\upsilon\in\bf{X}^1_\ep$. Starting with the identity $L^\ep_h\chi = \tp{L^\ep_{\tilde{h}} - L^\ep_{h}}\tp{L^\ep_{\tilde{h}}}^{-1}\upsilon$ we use the estimates~\eqref{a priori estimates and bounds for strong solutions} and~\eqref{forward Lipschitz for strong solutions} to get
    \begin{equation}
        \tnorm{\chi}_{\bf{X}^1_\ep}^{\tp{+}}\lesssim\tnorm{h - \tilde{h}}_{\bf{Z}}\tnorm{\tp{L^\ep_{\tilde{h}}}^{-1}\upsilon}_{\bf{X}^1_\ep}^{\tp{+}}\lesssim\tnorm{h - \tilde{h}}_{\bf{Z}}\tnorm{\upsilon}_{\bf{Y}^1_\ep},
    \end{equation}
    where the implicit constants depend only on $\del$, $\upmu$, and $\m{a}$.     
    The right hand estimate~\eqref{estimates on the inverse stokes operator} follows.
\end{proof}


\section{The shallow water and residual Stokes equations}\label{section on the shallow water and residual equations}

In comparison to what we have explored so far, this section is heavily computational in nature. We aim to quantify a relationship between the free boundary Navier-Stokes system~\eqref{FBINSE, param. tuned and funny flux}, the shallow water equations~\eqref{The shallow water ODEs}, and the linear Stokes equations~\eqref{linear stokes in a thin domain}, specifically when we are working around the distinguished shallow water heteroclinic orbit solutions produced by Theorem~\ref{thm on distinguished heteroclinic orbits}. This is a lengthy, but essential task.

\subsection{Ansatz and derivation}\label{subsection on ansatz and derivation}

Recall that in Sections~\ref{subsection on statement of the main theorems and discussion} and~\ref{subsection on distinguished shallow water bore solutions} we made the claim that the shallow water ODEs~\eqref{The shallow water ODEs} are, in a sense, the limit of the free boundary Navier-Stokes system~\eqref{FBINSE, param. tuned and funny flux} as the parameter $\ep>0$ tends to zero. The goal of this subsection is to begin making this claim more precise by carrying out a plethora of important calculations. We produce an ansatz for solutions to the Navier-Stokes system~\eqref{FBINSE, param. tuned and funny flux} which is the sum of a main part that is intimately related to the shallow water ODEs~\eqref{The shallow water ODEs} and a small residual. We then calculate the Stokes-type equations satisfied by the residual, given that the main part actually solves a perturbation of the shallow water equations.  

\begin{thmC}[The shallow water and residual Stokes equations]\label{thm on the shallow water and residual equations}

     Let $\ep,\m{A}\in\tp{0,1}$, $\upmu,\m{a}>0$, $\m{g},\upsigma\ge0$, and let $\Bar{\upgamma},\Bar{\m{A}}\in\R$ be defined via~\eqref{the parameter tuning identities 440 hz}. Consider the following six assertions.
     \begin{enumerate}
         \item The functions
\begin{equation}\label{ansatz_0}
    H,U,U_1,U_2,P,P_1,P_2\in W^{\infty,\infty}\tp{\R},\quad\tp{\eta,u,p}\in H^{5/2}\tp{\Sigma_\ep}\times H^2\tp{\Omega_\ep;\R^2}\times H^1\tp{\Omega_\ep},
    \text{ and }
     \mathfrak{r}_1,\mathfrak{r}_3,\mathfrak{r}_4\in H^\infty\tp{\R},
\end{equation}
are given and satisfy the bound $\inf(H + \ep^2 \eta) >0$.

    \item  The triple $(\pmb{\zeta},\pmb{u},\pmb{p})\in H^{5/2}_\loc\tp{\R}\times H^2_\loc\tp{\Omega_\ep;\R^2}\times H^1_\loc\tp{\Omega_\ep}$ is defined via
\begin{equation}\label{ansatz_1}
    \pmb{\zeta} = H + \ep^2\eta,\quad \pmb{u} = X + \ep^2 u, \text{ and } \pmb{p} = P + \ep^2 Q + \ep^2 p,
\end{equation}
where $X,V,W\in H^{2}_\loc\tp{\Omega_\ep;\R^2}$ and $Q\in H^1_\loc\tp{\Omega_\ep}$ are determined via (writing $y$ for the vertical variable in $\Omega_\ep$)
\begin{equation}\label{ansatz_2}
    X = V + \ep^2 W,\quad V = Ue_1 - y\pmb{\zeta} U'e_2,\quad W = \bp{\f{y\pmb{\zeta}}{\ep}U_1 + \f12\bp{\f{y\pmb{\zeta}}{\ep}}^2U_2}e_1 - \ep\bp{\f{1}{2}\bp{\f{y\pmb{\zeta}}{\ep}}^2U_1' + \f16\bp{\f{y\pmb{\zeta}}{\ep}}^3U_2'}e_2,
\end{equation}
and
\begin{equation}\label{ansatz_3}
    Q = \f{y\pmb{\zeta}}{\ep}P_1 + \f12\bp{\f{y\pmb{\zeta}}{\ep}}^2P_2.
\end{equation}

\item   The adversarial contribution $\mathfrak{L}\eta\in H^{3/2}\tp{\Sigma_\ep}$ is defined via
        \begin{equation}\label{adversarial linear operator}
            \mathfrak{L}\eta = \tp{P - \m{g}H - 2\upmu U'}\eta' + \upmu\eta\tp{U_2 - U''} - \m{g}\eta H'.
        \end{equation}

\item The residual forcing terms $\mathfrak{f}_1,\mathfrak{f}_2\in L^2_\loc\tp{\Omega_\ep}$ and $\mathfrak{f}_0,\mathfrak{f}_3,\mathfrak{f}_4\in H^{1/2}_\loc\tp{\Sigma_\ep}$ are defined via
        \begin{equation}\label{residual forcing no 0}
            \mathfrak{f}_0 = \tp{1 - \pd_1^2}\tp{\Bar{\upgamma}H - \pmb{\zeta}^2U_1/2 - \pmb{\zeta}^3U_2/6 - \Bar{\m{A}}},
        \end{equation}
        \begin{equation}\label{residual forcing no 1}
            \mathfrak{f}_1 = -\Bar{\upgamma}U' + \tp{V - \tp{4 + \ep^2\Bar{\upgamma}}e_1}\cdot\grad^{\mathcal{A}_{\pmb{\zeta}}}W_1 + W\cdot\grad^{\mathcal{A}_{\pmb{\zeta}}}V_1 + \ep^2 W\cdot\grad^{\mathcal{A}_{\pmb{\zeta}}}W_1 + \ep^2u\cdot\grad^{\mathcal{A}_{\pmb{\zeta}}}u_1 + \pd_1^{\mathcal{A}_{\pmb{\zeta}}}Q - \upmu\tp{\pd_1^{\mathcal{A}_{\pmb{\zeta}}}}^2W_1,
        \end{equation}
        \begin{equation}\label{residual forcing no 2}
            \mathfrak{f}_2 = \Bar{\upgamma}y\pmb{\zeta}U'' + \tp{V - \tp{4 + \ep^2\Bar{\upgamma}}e_1}\cdot\grad^{\mathcal{A}_{\pmb{\zeta}}}W_2 + W\cdot\grad^{\mathcal{A}_{\pmb{\zeta}}}V_2 + \ep^2 W\cdot\grad^{\mathcal{A}_{\pmb{\zeta}}}W_2 +  \ep^2u\cdot\grad^{\mathcal{A}_{\pmb{\zeta}}}u_2 - \upmu\tp{\pd_1^{\mathcal{A}_{\pmb{\zeta}}}}^2W_2,
        \end{equation}
        \begin{equation}\label{residual forcing no 3}
            \mathfrak{f}_3 = \ep Q\pmb{\zeta}' - 2\ep\upmu\pd_1^{\mathcal{A}_{\pmb{\zeta}}}W_1\pmb{\zeta}' + \upmu\pd_1^{\mathcal{A}_{\pmb{\zeta}}}W_2 + \upsigma\mathcal{H}\tp{\ep\pmb{\zeta}}\pmb{\zeta}' - \ep^3\m{g}\eta\eta',
        \end{equation}
        and
        \begin{equation}\label{residual forcing no 4}
            \mathfrak{f}_4 = -Q + \upmu\pmb{\zeta}\pmb{\zeta}'U'' + \upmu\mathbb{D}^{\mathcal{A}_{\pmb{\zeta}}}W\mathcal{N}_{\ep\pmb{\zeta}}\cdot e_2 - \ep^{-1}\upsigma\mathcal{H}\tp{\ep\pmb{\zeta}}.
        \end{equation}

    \item The following system of ODEs is satisfied in $\R$:
        \begin{equation}\label{_SWE_0}
            \tp{4 - U}H - \m{A} = 0,\quad\tp{U - 4}U' + P' - \upmu U'' - \upmu U_2 - 4\m{a} + \ep^2\mathfrak{r}_1 = 0,
        \end{equation}
        \begin{equation}\label{_SWE_1}
            P_1 + \upmu U_1' = 0,\quad P_2 + \tp{4 - U}U'' + \tp{U'}^2 + \upmu U''' + \upmu U_2' = 0,\quad \upmu U_1 - \m{a}U = 0,
        \end{equation}
        \begin{equation}\label{_SWE_2}
            \tp{P - \m{g}H - 2\upmu U'}H' + \upmu U_1 + \upmu H\tp{U_2 - U''} + \ep\mathfrak{r}_3 = 0,\quad -P - 2\upmu U' + \m{g}H + \ep^2\mathfrak{r}_4 = 0.
        \end{equation}
    
    \item The following system of PDEs is satisfied:
        \begin{equation}\label{first form of the residual PDEs}
            \begin{cases}
                  \tp{X - \tp{4 + \ep^2\Bar{\upgamma}}e_1}\cdot\grad^{\mathcal{A}_{\pmb{\zeta}}}u + u\cdot\grad^{\mathcal{A}_{\pmb{\zeta}}}X + \grad^{\mathcal{A}_{\pmb{\zeta}}}p - \upmu\Delta^{\mathcal{A}_{\pmb{\zeta}}}u + \tp{\mathfrak{f}_1 - \mathfrak{r}_1,\mathfrak{f}_2} = 0,\quad \grad^{\mathcal{A}_{\pmb{\zeta}}}\cdot u = 0&\text{in }\Omega_\ep,\\
                 -\tp{p - \upmu\mathbb{D}^{\mathcal{A}_{\pmb{\zeta}}}u}\mathcal{N}_{\ep\pmb{\zeta}} + \tp{\ep\mathfrak{L}\eta,\m{g}\eta} + \tp{\mathfrak{f}_3 - \mathfrak{r}_3,\mathfrak{f}_4 - \mathfrak{r}_4} = 0&\text{on }\Sigma_\ep,\\
                 u_2 = 0,\quad \upmu\pd_2^{\mathcal{A}_{\pmb{\zeta}}}u_1 - \ep\m{a}u_1 = 0&\text{on }\Sigma_0,\\
                \tp{1 - \pd_1^2}\tp{\tp{\tp{4 + \ep^2\Bar{\upgamma}} - U}\eta} - \f{\pmb{\zeta}}{\ep}\int_0^\ep u_1\tp{\cdot,y}\;\m{d}y - \f{1}{\ep}\pd_1\tp{u\cdot\mathcal{N}_{\ep\pmb{\zeta}}} + \mathfrak{f}_0 = 0&\text{on }\Sigma_\ep.
            \end{cases}
        \end{equation}
     \end{enumerate}
If  all six assertions are satisfied, then the triple $\tp{\pmb{\zeta},\pmb{u},\pmb{p}}$ of~\eqref{ansatz_1} is a solution to the free boundary Navier-Stokes system~\eqref{FBINSE, param. tuned and funny flux}.  Conversely, if the first through fifth assertions are satisfied and the triple $\tp{\pmb{\zeta},\pmb{u},\pmb{p}}$ of~\eqref{ansatz_1} is a solution to~\eqref{FBINSE, param. tuned and funny flux}, then the  sixth item is satisfied by the residuals $\tp{\eta,u,p}$.
\end{thmC}

The proof of this theorem is given at the end of this subsection.  In essence, the proof consists of a large number of computations, which we have recorded below in more self-contained pieces.  We emphasize that the role of the functional parameters $\mathfrak{r}_1$, $\mathfrak{r}_3$, and $\mathfrak{r}_4$, which do not explicitly appear in~\eqref{ansatz_1}, \eqref{ansatz_2}, or~\eqref{ansatz_3}, is to specify a coupling between the shallow water ODE part and the Stokes PDE residual part.  Moreover, we do not need a functional parameter in the second momentum equation, which explains the lack of $\mathfrak{r}_2$ above.  The precise choice of these functions is not made until Section~\ref{subsection on the shallow water bore map}.

We now embark on the computation of the residuals, handling each component of the system of equations~\eqref{FBINSE, param. tuned and funny flux} separately. We begin with the conservation of relative velocity flux equation.

\begin{propC}[Conservation of relative velocity flux's residuals]\label{prop on conservation of flux residuals}
    Assuming the first and second hypotheses of Theorem~\ref{thm on the shallow water and residual equations}, the following are equivalent.
    \begin{enumerate}
        \item We have the following equality on $\Sigma_\ep$:
        \begin{equation}\label{conservation of flux equation}
            \bp{4+ \ep^2\Bar{\upgamma} - \f{1}{\ep}\int_0^\ep\pmb{u}_1\tp{\cdot,y}\;\m{d}y}\pmb{\zeta} - \tp{4 + \ep^2\Bar{\upgamma}}\pd_1^2\pmb{\zeta} - \f{1}{\ep}\pd_1\tp{\pmb{u}\cdot\mathcal{N}_{\ep\pmb{\zeta}}} = \m{A} + \ep^2\Bar{\m{A}}.
        \end{equation}
        \item We have the following equality on $\Sigma_\ep$:
        \begin{equation}\label{residual conservation of flux}
            \tp{1 - \pd_1^2}\tp{\tp{\tp{4 + \ep^2\Bar{\upgamma}} - U}\eta} - \f{\pmb{\zeta}}{\ep}\int_0^\ep u_1\tp{\cdot,y}\;\m{d}y - \f{1}{\ep}\pd_1\tp{u\cdot\mathcal{N}_{\ep\pmb{\zeta}}} + \mathfrak{f}_0 + \ep^{-2}\mathfrak{S}_0 = 0,
        \end{equation}
        where
        \begin{equation}\label{shallow water no 0}
            \mathfrak{S}_0 = \tp{1 - \pd_1^2}\tp{\tp{4 - U}H - \m{A}}
        \end{equation}
        and $\mathfrak{f}_0$ is given by~\eqref{residual forcing no 0}.
    \end{enumerate}
\end{propC}
\begin{proof}
    The strategy is simply to plug the ansatz~\eqref{ansatz_1} into the final equation of~\eqref{FBINSE, param. tuned and funny flux}, divide by $\ep^2$, and sort the resulting terms.   We first compute the residual vertical average of $\pmb{u}_1$ as
    \begin{equation}\label{__0}
        \ep^{-3}\int_0^\ep\pmb{u}_1\tp{\cdot,y}\;\m{d}y = \ep^{-2}U + \f12\pmb{\zeta}U_1 + \f16\pmb{\zeta}^2U_2 + \ep^{-1}\int_0^\ep u_1\tp{\cdot,y}\;\m{d}y.
    \end{equation}
    In turn, we find that
    \begin{multline}\label{__1}
        \ep^{-2}\bp{\bp{4+ \ep^2\Bar{\upgamma} - \f{1}{\ep}\int_0^\ep\pmb{u}_1\tp{\cdot,y}\;\m{d}y}\pmb{\zeta} - \tp{\m{A} + \ep^2\Bar{\m{A}}}} = \ep^{-2}\tp{\tp{4 - U}H - \m{A}} \\+ \bp{\Bar{\upgamma}H - \f12\pmb{\zeta}^2U_1 - \f{1}{6}\pmb{\zeta}^3U_2 - \Bar{\m{A}}} + \tp{4 + \ep^2\Bar{\upgamma} - U}\eta - \f{\pmb{\zeta}}{\ep}\int_0^\ep u_1\tp{\cdot,y}\;\m{d}y.
    \end{multline}
    On the other hand, we compute
    \begin{equation}\label{__2}
        -\ep^{-3}\pd_1\tp{\pmb{u}\cdot\mathcal{N}_{\ep\pmb{\zeta}}} =\ep^{-2}\tp{UH}'' + \bp{\f12\pmb{\zeta}^2U_1 + \f16\pmb{\zeta}^3U_2}'' + \tp{U\eta}'' - \ep^{-1}\pd_1\tp{u\cdot\mathcal{N}_{\ep\pmb{\zeta}}}
    \end{equation}
    and
    \begin{equation}\label{__3}
        -\ep^{-2}\tp{4 + \ep^2\Bar{\upgamma}}\pd_1^2\pmb{\zeta} = -\ep^{-2}4H'' - \Bar{\upgamma}H'' - \tp{4 + \ep^2\Bar{\upgamma}}\eta''.
    \end{equation}

    The shallow water contribution~\eqref{shallow water no 0} is the sum of the first terms of the right hand sides of~\eqref{__1}, \eqref{__2}, and~\eqref{__3}. The residual contribution~\eqref{residual forcing no 0} is the sum of the second grouping of terms. The remaining terms then belong to the residual conservation of relative momentum flux~\eqref{residual conservation of flux}.
\end{proof}

Our next two results handle the momentum equation's components.

\begin{propC}[First momentum equation's residuals]\label{prop on first momentum residuals}
    Assuming the first and second hypotheses of Theorem~\ref{thm on the shallow water and residual equations}, the following are equivalent.
    \begin{enumerate}
        \item We have the following equality in $\Omega_\ep$:
        \begin{equation}\label{first momentum equation}
            \tp{\pmb{u} - \tp{4 + \ep^2\Bar{\upgamma}}e_1}\cdot\grad^{\mathcal{A}_{\pmb{\zeta}}}\pmb{u}_1 + \pd_1^{\mathcal{A}_{\pmb{\zeta}}}\pmb{p} - \upmu\Delta^{\mathcal{A}_{\pmb{\zeta}}}\pmb{u}_1 = 4\m{a}.
        \end{equation}
        \item We have the following equality in $\Omega_\ep$:
        \begin{equation}\label{first residual momentum equation}
            \tp{X - \tp{4 + \ep^2\Bar{\upgamma}}e_1}\cdot\grad^{\mathcal{A}_{\pmb{\zeta}}}u_1 + u\cdot\grad^{\mathcal{A}_{\pmb{\zeta}}}X_1 + \pd_1^{\mathcal{A}_{\pmb{\zeta}}}p - \upmu\Delta^{\mathcal{A}_{\pmb{\zeta}}}u_1 + \mathfrak{f}_1 - \mathfrak{r}_1 + \ep^{-2}\mathfrak{S}_1 = 0
        \end{equation}
        where
        \begin{equation}\label{shallow water no 1}
            \mathfrak{S}_1 =\tp{U - 4}U' + P' -\upmu U'' - \upmu U_2 - 4\m{a} + \ep^2\mathfrak{r}_1
        \end{equation}
        and $\mathfrak{f}_1$ is given by~\eqref{residual forcing no 1}.
    \end{enumerate}
\end{propC}
\begin{proof}
    Again, the strategy is to plug the ansatz~\eqref{ansatz_1} into the first equation of system~\eqref{FBINSE, param. tuned and funny flux}, divide by $\ep^2$, and sort the resulting terms.  We begin with the advective derivative, calculating that
    \begin{multline}\label{advective derivative 1}
        \ep^{-2}\tp{\pmb{u} - \tp{4 + \ep^2\Bar{\upgamma}}e_1}\cdot\grad^{\mathcal{A}_{\pmb{\zeta}}}\pmb{u}_1 = \ep^{-2}\tp{X - \tp{4 + \ep^2\Bar{\upgamma}}e_1}\cdot\grad^{\mathcal{A}_{\pmb{\zeta}}}X_1\\
        + \tp{X - \tp{4 + \ep^2\Bar{\upgamma}}e_1}\cdot\grad^{\mathcal{A}_{\pmb{\zeta}}}u_1 + u\cdot\grad^{\mathcal{A}_{\pmb{\zeta}}}X_1 + \ep^2 u\cdot\grad^{\mathcal{A}_{\pmb{\zeta}}}u_1.
    \end{multline}
    The second and third terms are the residual advective derivative in~\eqref{first residual momentum equation} while the final term is sorted to~\eqref{residual forcing no 1}. The first term requires further dissection: 
    \begin{multline}
        \ep^{-2}\tp{X - \tp{4 + \ep^2\Bar{\upgamma}}e_1}\cdot\grad^{\mathcal{A}_{\pmb{\zeta}}}X_1 = \ep^{-2}\tp{U - 4}U' -\Bar{\upgamma}U'\\+ \tp{V - \tp{4 + \ep^2\Bar{\upgamma}}e_1}\cdot\grad^{\mathcal{A}_{\pmb{\zeta}}}W_1 + W\cdot\grad^{\mathcal{A}_{\pmb{\zeta}}}V_1 + \ep^2 W\cdot\grad^{\mathcal{A}_{\pmb{\zeta}}}W_1.
    \end{multline}
    The first term is then a shallow water~\eqref{shallow water no 1} contribution, while what remains are residual forcing~\eqref{residual forcing no 1}.

    For the pressure's contribution we calculate
    \begin{equation}\label{pressure contribution 1}
        \ep^{-2}\pd_1^{\mathcal{A}_{\pmb{\zeta}}}\pmb{p} = \ep^{-2}P' + \pd_1^{\mathcal{A}_{\pmb{\zeta}}}p + \pd_1^{\mathcal{A}_{\pmb{\zeta}}}Q.
    \end{equation}
    The first grouping of terms belongs to the category~\eqref{shallow water no 1}. The second term is the residual pressure in~\eqref{first residual momentum equation}. The final term is sorted to~\eqref{residual forcing no 1}.

    For the dissipation's contribution, we compute
    \begin{equation}\label{dissipation contribution 1}
        \ep^{-2}\Delta^{\mathcal{A}_{\pmb{\zeta}}}\pmb{u}_1 = \ep^{-2}\tp{U'' + U_2} + \Delta^{\mathcal{A}_{\pmb{\zeta}}}u_1 + \tp{\pd_1^{\mathcal{A}_{\pmb{\zeta}}}}^2W_1.
    \end{equation}
    The first grouping of terms belongs to the category~\eqref{shallow water no 1}. The second term is the residual dissipation in~\eqref{first residual momentum equation}. The final term is sorted to~\eqref{residual forcing no 1}.
\end{proof}

Next, we handle the second component.

\begin{propC}[Second momentum equation's residuals]\label{prop on second momentum residuals}
    Assuming the first and second hypotheses of Theorem~\ref{thm on the shallow water and residual equations}, the following are equivalent.
    \begin{enumerate}
        \item We have the following  equality in $\Omega_\ep$:
        \begin{equation}\label{second momentum equation}
            \tp{\pmb{u} - \tp{4 + \ep^2\Bar{\upgamma}}e_1}\cdot\grad^{\mathcal{A}_{\pmb{\zeta}}}\pmb{u}_2 + \pd_2^{\mathcal{A}_{\pmb{\zeta}}}\pmb{p} - \upmu\Delta^{\mathcal{A}_{\pmb{\zeta}}}\pmb{u}_2 = 0.
        \end{equation}
        \item We have the following equality in $\Omega_\ep$:
        \begin{equation}\label{second residual momentum equation}
            \tp{X - \tp{4 + \ep^2\Bar{\upgamma}}e_1}\cdot\grad^{\mathcal{A}_{\pmb{\zeta}}}u_2 + u\cdot\grad^{\mathcal{A}_{\pmb{\zeta}}}X_2 + \pd_2^{\mathcal{A}_{\pmb{\zeta}}}p - \upmu\Delta^{\mathcal{A}_{\pmb{\zeta}}}u_2 + \mathfrak{f}_2 + \ep^{-2}\mathfrak{S}_2 = 0,
        \end{equation}
        where
        \begin{equation}\label{shallow water no 2}
            \mathfrak{S}_2 = \ep\tp{P_1 + \upmu U_1'} + y\pmb{\zeta}\tp{P_2 + \tp{4 - U}U'' + \tp{U'}^2 + \upmu U''' + \upmu U_2'}
        \end{equation}
        and $\mathfrak{f}_2$ is given by~\eqref{residual forcing no 2}.
    \end{enumerate}
\end{propC}
\begin{proof}
    The strategy is the same as that of Proposition~\ref{prop on first momentum residuals}. We begin by considering the advective derivative.
    \begin{multline}
        \ep^{-2}\tp{\pmb{u} - \tp{4 + \ep^2\Bar{\upgamma}}e_1}\cdot\grad^{\mathcal{A}_{\pmb{\zeta}}}\pmb{u}_2 = \ep^{-2}\tp{X - (4 + \ep^2\Bar{\upgamma})e_1}\cdot\grad^{\mathcal{A}_{\pmb{\zeta}}}X_2\\
        + \tp{X - \tp{4 + \ep^2\Bar{\upgamma}}e_1}\cdot\grad^{\mathcal{A}_{\pmb{\zeta}}}u_2 + u\cdot\grad^{\mathcal{A}_{\pmb{\zeta}}}X_2 + \ep^2 u\cdot\grad^{\mathcal{A}_{\pmb{\zeta}}}u_2.
    \end{multline}
    The right hand side's second and third terms are sorted to~\eqref{second residual momentum equation} and the final term belongs to~\eqref{residual forcing no 2}. The first term we now expand further
    \begin{multline}
        \ep^{-2}\tp{X - (4 + \ep^2\Bar{\upgamma})e_1}\cdot\grad^{\mathcal{A}_{\pmb{\zeta}}}X_2 = \ep^{-2}y\pmb{\zeta}\tp{\tp{4 + \ep^2\Bar{\upgamma} - U}U'' + \tp{U'}^2} \\+ \tp{V - \tp{4 + \ep^2\Bar{\upgamma}}e_1}\cdot\grad^{\mathcal{A}_{\pmb{\zeta}}}W_2
        + W\cdot\grad^{\mathcal{A}_{\pmb{\zeta}}}V_2 + \ep^2 W\cdot\grad^{\mathcal{A}_{\pmb{\zeta}}}W_2.
    \end{multline}
    The term $\Bar{\upgamma}y\pmb{\zeta}U''$ along with the final three terms are residual forcing~\eqref{residual forcing no 2} while the rest of the first term belongs to~\eqref{shallow water no 2}.

    For the pressure's contribution, we calculate
    \begin{equation}
        \ep^{-2}\pd_2^{\mathcal{A}_{\pmb{\zeta}}}\pmb{p} =\ep^{-2}\tp{\ep P_1 + y\pmb{\zeta}P_2} + \pd_2^{\mathcal{A}_{\pmb{\zeta}}}p
    \end{equation}
    and sort the first term to~\eqref{shallow water no 2} and the second term to~\eqref{second residual momentum equation}.

    Finally, we consider the dissipation's contribution by computing that
    \begin{equation}
        \ep^{-2}\Delta^{\mathcal{A}_{\pmb{\zeta}}}\pmb{u}_2 = -\ep^{-2}\tp{\ep U_1' + y\pmb{\zeta}\tp{U''' + U_2'}} + \tp{\pd_1^{\mathcal{A}_{\pmb{\zeta}}}}^2W_2 + \Delta^{\mathcal{A}_{\pmb{\zeta}}} u_2.
    \end{equation}
    We sort the first term to~\eqref{shallow water no 2}, the second to~\eqref{residual forcing no 2}, and the third to~\eqref{second residual momentum equation}.
\end{proof}

The structure of the ansatz leads to a pleasantly simple divergence equation on remainders.

\begin{propC}[Incompressibility constraint's residuals]\label{prop on incompressibility residuals}
    Assuming the first and second hypotheses of Theorem~\ref{thm on the shallow water and residual equations}, the following are equivalent
    \begin{enumerate}
        \item We have the following equality in $\Omega_\ep$:
        \begin{equation}\label{divergence equation}
            \grad^{\mathcal{A}_{\pmb{\zeta}}}\cdot\pmb{u} = 0.
        \end{equation}
        \item We have the following equality in $\Omega_\ep$:
        \begin{equation}\label{residual divergence equation}
            \grad^{\mathcal{A}_{\pmb{\zeta}}}\cdot u = 0.
        \end{equation}
    \end{enumerate}
\end{propC}
\begin{proof}
    By construction $\grad^{\mathcal{A}_{\pmb{\zeta}}}\cdot X = 0$ and $\pmb{u} = X + \ep^2 u$.
\end{proof}

Next, we take a look at the $e_1$-dynamic boundary condition in~\eqref{FBINSE, param. tuned and funny flux}. In contrast with the previous computations, we now encounter a new object category for the residual equation. This is the adversarial linear operator which appears in equation~\eqref{adversarial linear operator}.

\begin{propC}[First dynamic boundary condition's residuals]\label{prop on first DBC residual}
    Assuming the first and second hypotheses of Theorem~\ref{thm on the shallow water and residual equations}, the following are equivalent.
    \begin{enumerate}
        \item We have the following equality on $\Sigma_\ep$:
        \begin{equation}\label{first DBC}
            -\tp{\pmb{p} - \upmu\mathbb{D}^{\mathcal{A}_{\pmb{\zeta}}}\pmb{u}}\mathcal{N}_{\ep\pmb{\zeta}}\cdot e_1 + \tp{\m{g}\pmb{\zeta} - \ep\upsigma\mathcal{H}\tp{\ep\pmb{\zeta}}}\mathcal{N}_{\ep\pmb{\zeta}}\cdot e_1 = 0.
        \end{equation}
        \item We have the following equality on $\Sigma_\ep$:
        \begin{equation}\label{first residual DBC}
            -\tp{p - \upmu\mathbb{D}^{\mathcal{A}_{\pmb{\zeta}}}u}\mathcal{N}_{\ep\pmb{\zeta}}\cdot e_1 + \ep\mathfrak{L}\eta + \mathfrak{f}_3 - \mathfrak{r}_3 + \ep^{-1}\mathfrak{S}_3 = 0,
        \end{equation}
        where
        \begin{equation}\label{shallow water no 3}
            \mathfrak{S}_3 =\tp{P - \m{g}H - 2\upmu U'}H' + \upmu U_1 + \upmu H\tp{U_2 - U''} + \ep\mathfrak{r}_3,
        \end{equation}
        $\mathfrak{L}\eta$ is given by~\eqref{adversarial linear operator}, and $\mathfrak{f}_3$ is given by~\eqref{residual forcing no 3}.
    \end{enumerate}
\end{propC}
\begin{proof}
Again the strategy is similar to that of Proposition~\ref{prop on first momentum residuals}. First we compute that the pressure satisfies
    \begin{equation}\label{initial pressure}
        -\ep^{-2}\pmb{p}\mathcal{N}_{\ep\pmb{\zeta}}\cdot e_1 = \ep^{-1}P\pmb{\zeta}' - \tp{Q + p}\mathcal{N}_{\ep\pmb{\zeta}}\cdot e_1.
    \end{equation}
The $Q$ term above belongs to~\eqref{residual forcing no 3}, while the $p$ term is seen in~\eqref{first residual DBC}.

We also compute that the viscous stress satisfies
\begin{multline}\label{initial viscous stress}
    \ep^{-2}\mathbb{D}^{\mathcal{A}_{\pmb{\zeta}}}\pmb{u}\mathcal{N}_{\ep\pmb{\zeta}}\cdot e_1 = -\ep^{-1}2\pd_1^{\mathcal{A}_{\pmb{\zeta}}}X_1\pmb{\zeta}' + \ep^{-2}\tp{\pd_2^{\mathcal{A}_{\pmb{\zeta}}}X_1 + \pd_1^{\mathcal{A}_{\pmb{\zeta}}}X_2} + \mathbb{D}^{\mathcal{A}_{\pmb{\zeta}}}u\mathcal{N}_{\ep\pmb{\zeta}}\cdot e_1\\
     =\ep^{-1}\tp{-2U'\pmb{\zeta}' + U_1 + \pmb{\zeta}U_2 - \pmb{\zeta}U''} - 2\ep\pd_1^{\mathcal{A}_{\pmb{\zeta}}}W_1\pmb{\zeta}' + \pd_1^{\mathcal{A}_{\pmb{\zeta}}}W_2 + \mathbb{D}^{\mathcal{A}_{\pmb{\zeta}}}u\mathcal{N}_{\ep\pmb{\zeta}}\cdot e_1.
\end{multline}
The right hand side's second and third terms are sorted to~\eqref{residual forcing no 3}, and the final term is the residual viscous stress in~\eqref{first residual DBC}.

Now we consider the gravity-capillary operator:
\begin{equation}\label{initial gravity capillary}
    \ep^{-2}\tp{\m{g}\pmb{\zeta} - \ep\upsigma\mathcal{H}\tp{\ep\pmb{\zeta}}}\mathcal{N}_{\ep\pmb{\zeta}}\cdot e_1 = -\ep^{-1}\m{g}\pmb{\zeta}\pmb{\zeta}' + \upsigma\mathcal{H}\tp{\ep\pmb{\zeta}}\pmb{\zeta}'.
\end{equation}
The final term above is also a residual forcing, and so is given to~\eqref{residual forcing no 3}.

It remains to examine the right hand side contributions in~\eqref{initial pressure}, \eqref{initial viscous stress}, and~\eqref{initial gravity capillary} which have a leading coefficient of $\ep^{-1}$. We sum these and expand $\pmb{\zeta} = H + \ep^2\eta$ to get
\begin{multline}
    \ep^{-1}\tp{\tp{P - \m{g}\pmb{\zeta} - 2\upmu U'}\pmb{\zeta}' + \upmu U_1 + \upmu\pmb{\zeta}\tp{U_2 - U''}} = \ep^{-1}\tp{\tp{P - \m{g}H - 2\upmu U'}H' + \upmu U_1 + \upmu H\tp{U_2 - U''}}\\
    + \ep\tp{\tp{P - 2\upmu U'}\eta' + \upmu\eta\tp{U_2 - U''} - \m{g}H\eta' - \m{g}\eta H' - \ep^2\m{g}\eta\eta'}.
\end{multline}
The first grouping of terms is the shallow water contribution in~\eqref{shallow water no 3}. The next group is the sum of the adversarial linear operator~\eqref{adversarial linear operator} with a forcing term belonging to~\eqref{residual forcing no 3}.
\end{proof}

Now we consider the $e_2$-dynamic boundary condition.

\begin{propC}[Second dynamic boundary condition's residuals]\label{prop on second dynamic boundary condition residual}
    Assuming the first and second hypotheses of Theorem~\ref{thm on the shallow water and residual equations}, the following are equivalent.
    \begin{enumerate}
        \item We have the following equality on $\Sigma_\ep$:
        \begin{equation}\label{second DBC}
            -\tp{\pmb{p} - \upmu\mathbb{D}^{\mathcal{A}_{\pmb{\zeta}}}\pmb{u}}\mathcal{N}_{\ep\pmb{\zeta}}\cdot e_2 + \tp{\m{g}\pmb{\zeta} - \ep\upsigma\mathcal{H}\tp{\ep\pmb{\zeta}}}\mathcal{N}_{\ep\pmb{\zeta}}\cdot e_2 = 0.
        \end{equation}
        \item We have the following equality on $\Sigma_\ep$:
        \begin{equation}\label{second residual DBC}
            -\tp{p - \upmu\mathbb{D}^{\mathcal{A}_{\pmb{\zeta}}}u}\mathcal{N}_{\ep\pmb{\zeta}}\cdot e_2 + \m{g}\eta + \mathfrak{f}_4 - \mathfrak{r}_4 + \ep^{-2}\mathfrak{S}_4 = 0,
        \end{equation}
        where
        \begin{equation}\label{shallow water no 4}
            \mathfrak{S}_4 = -P - 2\upmu U' + \m{g}H + \ep^2\mathfrak{r}_4
        \end{equation}
        and $\mathfrak{f}_4$ is given by~\eqref{residual forcing no 4}.
    \end{enumerate}
\end{propC}
\begin{proof}
    For the pressure contribution, we compute that
    \begin{equation}
        -\ep^{-2}\pmb{p}\mathcal{N}_{\ep\pmb{\zeta}}\cdot e_2 = -\ep^{-2}P - Q - p
    \end{equation}
    and sort the first, second, and third terms into~\eqref{shallow water no 4}, \eqref{residual forcing no 4}, and~\eqref{second residual DBC}, respectively.

    For the viscous stress contribution, we compute that
    \begin{equation}
        \ep^{-2}\mathbb{D}^{\mathcal{A}_{\pmb{\zeta}}}\pmb{u}\mathcal{N}_{\ep\pmb{\zeta}}\cdot e_2 = -\ep^{-2}2U' + \pmb{\zeta}\pmb{\zeta}'U'' + \mathbb{D}^{\mathcal{A}_{\pmb{\zeta}}}W\mathcal{N}_{\ep\pmb{\zeta}}\cdot e_2 + \mathbb{D}^{\mathcal{A}_{\pmb{\zeta}}}u\mathcal{N}_{\ep\pmb{\zeta}}\cdot e_2.
    \end{equation}
    The first term is of the shallow water category~\eqref{shallow water no 4}. The second and third terms are residual forcing~\eqref{residual forcing no 4}. The final term is the residual viscous stress appearing in~\eqref{second residual DBC}.

    Finally, for the gravity-capillary operator we find that
    \begin{equation}
        \ep^{-2}\tp{\m{g}\pmb{\zeta} - \ep\upsigma\mathcal{H}\tp{\ep\pmb{\zeta}}}\mathcal{N}_{\ep\pmb{\zeta}}\cdot e_2 = \ep^{-2}\m{g}H + \m{g}\eta - \ep^{-1}\upsigma\mathcal{H}\tp{\ep\pmb{\zeta}}
    \end{equation}
    and sort the right hand side's first term to~\eqref{shallow water no 4}, second term to~\eqref{second residual DBC}, and final term to~\eqref{residual forcing no 4}.
\end{proof}

Finally we handle the residuals for the Navier slip boundary condition.
\begin{propC}[Navier slip residuals]\label{prop on navier slip residuals}
    Assuming the first and second hypotheses of Theorem~\ref{thm on the shallow water and residual equations}, the following are equivalent.
    \begin{enumerate}
        \item We have the following equalities on $\Sigma_0$:
        \begin{equation}\label{navier slip}
            \pmb{u}_2 = 0
            \text{ and }
            \upmu\pd_2^{\mathcal{A}_{\pmb{\zeta}}}\pmb{u}_1 - \ep\m{a}\pmb{u}_1 = 0.
        \end{equation}
        \item We have the following equalities on $\Sigma_0$:
        \begin{equation}\label{residual navier slip}
            u_2 = 0 \text{ and } \upmu\pd_2^{\mathcal{A}_{\pmb{\zeta}}}u_1 - \ep\m{a}u_1 + \ep^{-1}\mathfrak{S}_5 = 0
        \end{equation}
        where
        \begin{equation}\label{shallow water no 5}
            \mathfrak{S}_5 = \upmu U_1 - \m{a}U.
        \end{equation}
    \end{enumerate}
\end{propC}
\begin{proof}
    Since $\pmb{u} = X + \ep^2 u$ and by construction $X\cdot e_2 = 0$ on $\Sigma_0$, we find that $\pmb{u}_2=0$ if and only if $u_2=0$ on $\Sigma_0$. Next, we simply compute that
    \begin{equation}
        \ep^{-2}\tp{\upmu\pd_2^{\mathcal{A}_{\pmb{\zeta}}}\pmb{u}_1 - \ep\m{a}\pmb{u}_1} = \ep^{-1}\tp{\upmu U_1 - \m{a}U} + \upmu\pd_2^{\mathcal{A}_{\pmb{\zeta}}}u_1 - \ep\m{a}u_1
    \end{equation}
    and sort the first grouping of terms to~\eqref{shallow water no 5} and the remaining terms to~\eqref{residual navier slip}.
\end{proof}

We now synthesize the calculations of Propositions~\ref{prop on conservation of flux residuals}--\ref{prop on navier slip residuals} in order to prove Theorem~\ref{thm on the shallow water and residual equations}.

\begin{proof}[Proof of Theorem~\ref{thm on the shallow water and residual equations}]
    If the ODEs identities~\eqref{_SWE_0}, \eqref{_SWE_1}, and~\eqref{_SWE_2} hold, then we have $\mathfrak{S}_0 = \mathfrak{S}_1 = \dots = \mathfrak{S}_5 = 0$, where these are the shallow water contributions defined in~\eqref{shallow water no 0}, \eqref{shallow water no 1}, \eqref{shallow water no 2}, \eqref{shallow water no 3}, \eqref{shallow water no 4}, and~\eqref{shallow water no 5}. Thus, the satisfaction of $\pmb{\zeta}>0$ and~\eqref{first form of the residual PDEs} along with the sufficient directions of the previous family of propositions yield the first result, which establishes that $\tp{\pmb{\zeta},\pmb{u},\pmb{\eta}}$ solve~\eqref{FBINSE, param. tuned and funny flux}.  On the other hand, the necessary directions of these proposition yields the converse statement.
\end{proof}

\subsection{The shallow water bore map}\label{subsection on the shallow water bore map}

The free functional parameters $\mathfrak{r}_1$, $\mathfrak{r}_3$, and $\mathfrak{r}_4$ of Theorem~\ref{thm on the shallow water and residual equations} determine how the ODEs~\eqref{_SWE_0}, \eqref{_SWE_1}, and~\eqref{_SWE_2} are coupled to the PDEs~\eqref{first form of the residual PDEs}. The simple choice $\mathfrak{r}_1 = \mathfrak{r}_3 = \mathfrak{r}_4 = 0$, which corresponds to a complete decoupling, would allow us to first solve a closed system for $H$, $U$, $U_1$, $U_2$, $P$, $P_1$, and $P_2$ and then use these as coefficients and data terms in the residual PDEs for $\eta$, $u$, and $p$.  Unfortunately, with the complete decoupling strategy the solvability of the resulting PDE for the residuals is far from clear. The issue is the presence of adversarial linear terms in~\eqref{first form of the residual PDEs}, namely the terms $\tp{\ep\mathfrak{L}\eta,\m{g}\eta}$ appearing in the residual dynamic boundary condition and $u\cdot\grad^{\mathcal{A}_{\pmb{\zeta}}}X_1$ appearing in the $e_1$-residual momentum equation, which perturb us away from the Stokes system studied in Section~\ref{subsection on the stokes equations with stress boundary conditions}. Our strategy, therefore, is to judiciously select a mechanism of coupling $\mathfrak{r}_1$, $\mathfrak{r}_3$, and $\mathfrak{r}_4$ which pacifies these linear enemies.

The choice of coupling made here adds small perturbations (having coefficient $\ep$) to the ODEs~\eqref{_SWE_0}, \eqref{_SWE_1}, and~\eqref{_SWE_2} that are functions of the residual variables $\eta$, $u$, and $p$ as well as the shallow water variables $H$, $U$, $U_1$, $U_2$, $P$, $P_1$, and $P_2$. When $\ep = 0$ things are again decoupled and the ODEs reduce to the one-dimensional viscous shallow water equations with laminar drag, for which bore wave solutions are found in Section~\ref{subsection on distinguished shallow water bore solutions}. Then, when $\ep>0$, we can solve the perturbed shallow water ODEs with the theory of Section~\ref{subsection on perturbation theory for heteroclinic orbits} and produce a \emph{bore map}, which solves the family of perturbed equations, producing bore profiles as a function of the residual variables.

Our first result specifies the form of the perturbations and derives an equivalent form of the ODEs.   In what follows we encounter, for $\Xi\in\tp{1,\infty}$, the Fourier multiplication operator $\uppi_\Xi = \upchi\tp{D/\Xi}$ with $\upchi\in C^\infty_c\tp{B(0,2)}$ a fixed, even cut-off satisfying $0\le\upchi\le1$ and $\upchi = 1$ on $B(0,1)$. We also encounter a new variable $w\in H^1\tp{\R}$  that is meant to be a proxy for the vertical average of the residual $u_1$.

\begin{propC}[Coupling specification]\label{prop on the specification of the coupling}
    Let $\upmu,\m{a}>0$, $\m{g}\ge0$, and $\m{A}\in\tp{0,1}$, and define the set $D = \tcb{z\in\R^6\;:\; z_1 + e^{z_5} > 0}$.  Then there exists a smooth function
    \begin{equation}\label{dont write out the formula for this function}
        \aleph: D \to \R^2 \text{ for which }\aleph(0,0,0,0,\cdot,\cdot) = 0
    \end{equation}
    such that the following holds.     If $\ep,\del\in(0,1)$, $\Xi\in(1,\infty)$, $H, U, U_1, U_2, P, P_1, P_2\in W^{\infty,\infty}\tp{\R}$, and $\eta,w\in H^1\tp{\R}$ satisfy $\del\le H$, $|\ep^2\uppi_\Xi\eta|\le\del/2$, and $U\le 4 - \del$, and the functions $\mathfrak{r}_1,\mathfrak{r}_3,\mathfrak{r}_4\in H^\infty\tp{\R}$ are defined via
    \begin{equation}\label{the specified coupling equation}
        \mathfrak{r}_1 = \tp{\tp{U - 4}\uppi_{\Xi}w}',\quad \mathfrak{r}_3 = \ep\mathfrak{L}\uppi_\Xi\eta,
        \text{ and }
        \mathfrak{r}_4 = \m{g}\uppi_\Xi\eta,
    \end{equation}
    where $\mathfrak{L}$ is the linear map defined in~\eqref{adversarial linear operator}, then the following are equivalent.
    \begin{enumerate}
        \item The system of ODEs~\eqref{_SWE_0}, \eqref{_SWE_1}, and~\eqref{_SWE_2} from Theorem~\ref{thm on the shallow water and residual equations} is satisfied.
        
        \item There exists $\rho\in W^{\infty,\infty}\tp{\R}$ such that $H = e^\rho$, the smooth vector field $Y = \tp{\rho,\rho'}:\R\to\R^2$ is a solution to the differential equation
        \begin{equation}\label{peturbed shallow water equation}
            Y'(t) = \Phi(Y(t)) + \Phi_1(t,Y(t))
        \end{equation}
        where $\Phi:\R^2\to\R^2$ is the smooth vector field defined in~\eqref{The shallow water ODEs v4} and $\Phi_1:\R\times[\log\del,\infty)\times\R\to\R^2$ is determined by~\eqref{dont write out the formula for this function}  by
        \begin{equation}\label{peturbed shallow water residual}
            \Phi_1\tp{t,x,y} = \aleph\tp{\ep^2\uppi_\Xi\eta(t),\ep^2\uppi_\Xi\eta'(t),\ep^2\uppi_\Xi w(t), \ep^2\uppi_\Xi w'(t),x,y},
        \end{equation}
        and the functions $U$,  $P$, $U_1$,  $U_2$, $P_1$, and $P_2$  are determined by $H$ and $\eta$ via 
        \begin{equation}\label{SWE___0}
            U = 4 - \m{A}/H,\quad P = \m{g}H - 2\upmu U' + \ep^2\m{g}\uppi_{\Xi}\eta, 
        \end{equation}
        \begin{equation}\label{SWE___1}
            U_1 = \f{\m{a}}{\upmu}U,\quad U_2 = U'' - \f{\upmu U_1 + \tp{P - \m{g}H - 2\upmu U'}\tp{H' + \ep^2\uppi_\Xi\eta'} - \m{g}H'\ep^2\uppi_\Xi\eta}{\upmu\tp{H + \ep^2\uppi_\Xi\eta}},
        \end{equation}
        \begin{equation}\label{SWE___2}
            P_1 = -\upmu U_1',\quad P_2 = \tp{U - 4}U'' - (U')^2 - \upmu\tp{U''' + U_2'}.
        \end{equation}
    \end{enumerate}
\end{propC}
\begin{proof}
    Throughout the proof we will use the shorthand $\lambda_1 = \ep^2\uppi_\Xi\eta$ and $\lambda_2 = \ep^2\uppi_\Xi w$.   Since $\del\le H$ and $|\lambda_1|\le\del/2$, the ODEs~\eqref{_SWE_0} are equivalent to the first identity in~\eqref{SWE___0} along with the equation
    \begin{equation}\label{SWE___3}
        \tp{H + \lambda_1}\tp{\tp{U - 4}U' + P' - \upmu U'' - \upmu U_2 - 4\m{a} + \tp{\tp{U - 4}\lambda_2}'} = 0.
    \end{equation}
    Furthermore~\eqref{_SWE_1} is clearly equivalent to $U_1 = \m{a}U/\upmu$ and identities~\eqref{SWE___2}. The second identity in~\eqref{_SWE_2} is equivalent to the second identity in~\eqref{SWE___0}. Finally, after some algebraic manipulation, we find that the first identity of~\eqref{_SWE_2} is equivalent to the second identity in~\eqref{SWE___1}.

    Now we can simplify the second identity in~\eqref{SWE___1} further to get that
    \begin{equation}\label{SWE___4}
        -\upmu\tp{H + \lambda_1}\tp{U'' + U_2} = \m{a}U -2\upmu\tp{H + \lambda_1}U'' - 4\upmu\tp{H' + \lambda_1'}U' + \m{g}\lambda_1\lambda_1'.
    \end{equation}
    Then we substitute~\eqref{SWE___4} along with the expression for $P$ from~\eqref{SWE___0} into~\eqref{SWE___3} to arrive at the identity
    \begin{equation}
        \tp{H + \lambda_1}\tp{U - 4}U' + \m{a}U - 4\upmu\tp{\tp{H + \lambda_1}U'}' + \tp{H + \lambda_1}\tp{\m{g}\tp{H' + \lambda_1'} - 4\m{a}}
        + \tp{H + \lambda_1}\tp{\tp{U - 4}\lambda_2}' + \m{g}\lambda_1\lambda_1' = 0.
    \end{equation}
    After some computation, the above can be arranged to have the form
    \begin{equation}\label{another intermediate shallow water equation}
        4\upmu\tp{HU'}' = H\tp{U - 4}U' + \m{a}U + \m{g}HH' - 4\m{a}H + \widehat{\aleph}\tp{\lambda_1,\lambda_1',\lambda_2,\lambda_2',H,H',U,U'}
    \end{equation}
    where $\widehat{D} = \tcb{z\in\R^8\;:\;z_1 + z_5>0}$ and $\widehat{\aleph}:\widehat{D} \to\R$ is a smooth function, depending only on the fixed values of $\upmu$, $\m{a}$, $\m{g}$, and $\m{A}$, that satisfies $\widehat{\aleph}\tp{0,0,0,0,\cdot,\cdot,\cdot,\cdot} = 0$.

    Notice that~\eqref{another intermediate shallow water equation} is a small, nonautonomous perturbation of the shallow water equations~\eqref{The shallow water ODEs}.  We now proceed as in Section~\ref{subsection on distinguished shallow water bore solutions} by writing $U = 4 - \m{A}/H$ and $H = e^\rho$ to transform~\eqref{another intermediate shallow water equation} into a single second order ODE for $\rho$:
    \begin{equation}\label{yet another intermediate shallow water equation}
        \rho'' = F\tp{\rho} - G\tp{\rho}\rho' + \tp{4\upmu\m{A}}^{-1}\widehat{\aleph}\tp{\lambda_1,\lambda_1',\lambda_2,\lambda_2',e^\rho,e^\rho\rho',4 - \m{A}e^{-\rho},\m{A}e^{-\rho}\rho'},
    \end{equation}
    with $F$ and $G$ defined in~\eqref{the hamiltonian and dissipation}.   To complete the proof  we set $Y = \tp{\rho,\rho'}:\R\to\R^2$ and define the $\aleph$ of~\eqref{dont write out the formula for this function} via
    \begin{equation}
        \aleph(z) = \tp{0,\tp{4\upmu\m{A}}^{-1}\widehat{\aleph}\tp{z_1,z_2,z_3,z_4,e^{z_5},e^{z_5}z_6,4 - \m{A}e^{-z_5},\m{A}e^{-z_5}z_6}},
    \end{equation}
    in order deduce that~\eqref{peturbed shallow water equation} and~\eqref{peturbed shallow water residual} hold.
\end{proof}

We now come to the main result of this subsection, which produces a solution operator to the system of ODEs~\eqref{_SWE_0}, \eqref{_SWE_1}, and~\eqref{_SWE_2} when the coupling $\mathfrak{r}_1$, $\mathfrak{r}_3$, and $\mathfrak{r}_4$ is determined by~\eqref{the specified coupling equation} of Proposition~\ref{prop on the specification of the coupling}. Recall that the regions $\mathfrak{C}_{\iota}$ are defined in Lemma~\ref{lem on sufficient conditions for a sign} and $\m{H}_\pm = \m{H}_\pm\tp{\m{A}}$ are defined in~\eqref{sweq}

\begin{propC}[The bore map, I]\label{prop on the bore map}
    Let $\upmu,\m{a}>0$, $\iota\in\tcb{-1,1}$, and $\tp{\m{g},\m{A}}\in\mathfrak{C}_\iota$. There exist
    \begin{itemize}
        \item numbers $\al,\be\in\tp{0,1}$, $M\in\tp{3/\m{H}_-\tp{\m{A}},\infty}$, and a closed ball $\Bar{B(0,\be)}\subset\tsb{\tp{C^{10}\cap W^{10,\infty}\cap H^{10}}\tp{\R}}^2$,
        \item smooth functions $\mathsf{H}, \mathsf{U}, \mathsf{U}_1, \mathsf{U}_2, \mathsf{P}\in W^{\infty,\infty}\tp{\R}$ and $\mathsf{P}_1,\mathsf{P}_2\in H^\infty\tp{\R}$,
        \item Lipschitz continuous maps
        \begin{equation}\label{Siviglia_Dos}
        \begin{array}{ll}
            \mathfrak{h},\mathfrak{u},\mathfrak{u}_1:\Bar{B(0,\be)} \to \tp{C^{10} \cap W^{10,\infty} \cap H^{10}} \tp{\R}, & 
            \mathfrak{u}_2:\Bar{B(0,\be)} \to \tp{C^{8} \cap W^{8,\infty} \cap H^{8}} \tp{\R}, \\
            \mathfrak{p},\mathfrak{p}_1:\Bar{B(0,\be)} \to \tp{C^{9} \cap W^{9,\infty} \cap H^{9}} \tp{\R}, &
            \mathfrak{p}_2:\Bar{B(0,\be)} \to \tp{C^{7} \cap W^{7,\infty} \cap H^{7}} \tp{\R},
        \end{array}
    \end{equation}
    \end{itemize}
    such that the following hold for all $\tp{\lambda_1,\lambda_2},\tp{\tilde{\lambda}_1,\tilde{\lambda}_2}\in\Bar{B(0,\be)}$.
    \begin{enumerate}
        \item The pair $\tp{\mathsf{H},\mathsf{U}}$ is a solution to the shallow water ODEs~\eqref{The shallow water ODEs}; moreover we have the limits
        \begin{multline}\label{useful_limit_2_Dos}
            \lim_{t\to\pm\infty}\mathsf{H}\tp{\iota t} = \m{H}_\pm,\quad\lim_{t\to\pm\infty}\mathsf{U}\tp{\iota t} = 4 - \f{\m{A}}{\m{H}_\pm},\quad\lim_{t\to\pm\infty}\mathsf{U}_1\tp{\iota t} = \f{\m{a}}{\upmu}\bp{4 - \f{\m{A}}{\m{H}_\pm}},\\\lim_{t\to\pm\infty}\mathsf{U}_2\tp{\iota t} = -\f{\m{a}}{\upmu}\bp{\f{4}{\m{H}_{\pm}} - \f{\m{A}}{\m{H}^2_\pm}},\quad\lim_{t\to\pm\infty}\mathsf{P}\tp{\iota t} = \m{g}\m{H}_{\pm},\quad\lim_{t\to\pm\infty}\mathsf{P}_1\tp{\iota t} = 0,
            \text{ and }
            \lim_{t\to\pm\infty}\mathsf{P}_2\tp{\iota t} = 0.
        \end{multline}
        \item For all $\mathsf{Z}\in\tcb{\mathsf{H},\mathsf{U},\mathsf{U}_1,\mathsf{U}_2,\mathsf{P},\mathsf{P}_1,\mathsf{P}_2}$ we have the bounds
        \begin{multline}\label{useful_limit_3_Dos}
            \tnorm{\mathsf{Z}}_{W^{10,\infty}}\le M, 
             \quad 
            \sup_{t\le 0}e^{\al|t|}\tabs{\mathsf{Z}\tp{t} - \lim_{\tau\to-\infty}\mathsf{Z}\tp{\tau}} + \sup_{t\ge 0}e^{\al|t|}\tabs{\mathsf{Z}\tp{t} - \lim_{\tau\to\infty}\mathsf{Z}\tp{\tau}}\le M, \\
             \text{ and }
           \sup_{1 \le m \le 10} \sup_{t\in\R}e^{\al|t|}\tabs{\mathsf{Z}^{\tp{m}}\tp{t}}\le M.
        \end{multline}
        \item For all $\mathfrak{Z}\in\tcb{\mathfrak{h},\mathfrak{u},\mathfrak{u}_1,\mathfrak{u}_2,\mathfrak{p},\mathfrak{p}_1,\mathfrak{p}_2}$ we have that $\mathfrak{Z}\tp{0,0} = 0$, $\tnorm{\mathfrak{Z}\tp{\lambda_1,\lambda_2}}_{W^{7,\infty}\cap H^7}\le M$,
        \begin{equation}\label{rocky mountain spotted fever_Dos}
            \tnorm{\mathfrak{Z}\tp{\lambda_1,\lambda_2} - \mathfrak{Z}\tp{\tilde{\lambda}_1,\tilde{\lambda}_2}}_{W^{7,\infty}\cap H^7}\le M\tnorm{\tp{\lambda_1,\lambda_2} - \tp{\tilde{\lambda}_1,\tilde{\lambda}_2}}_{W^{10,\infty}\cap H^{10}},
        \end{equation}
        and if $\tp{\lambda_1,\lambda_2}\in\tsb{H^\infty\tp{\R}}^2$ then $\mathfrak{Z}\tp{\lambda_1,\lambda_2}\in H^\infty\tp{\R}$.

        \item We have the bounds 
       \begin{equation}
        \m{H}_-\le\inf\mathsf{H}, 
        \; 
        \tnorm{\mathfrak{h}\tp{\lambda_1,\lambda_2}}_{L^\infty\tp{\R}}\le\m{H}_-/3, 
        \;
        \tp{2/3}\m{H}_-\le\inf\tp{\mathsf{H} + \mathfrak{h}\tp{\lambda_1,\lambda_2}}, 
        \text{ and }
        \mathsf{U} + \mathfrak{u}\tp{\lambda_1,\lambda_2}\le 4-1/M.   
       \end{equation} 
        
        \item If for some $\ep\in\tp{0,1}$, $\Xi\in\tp{1,\infty}$, and $\tp{\eta,w}\in\tsb{H^1\tp{\R}}^2$ we have $\lambda_1 = \ep^2\uppi_\Xi\eta$ and $\lambda_2 = \ep^2\uppi_\Xi w$ (where $\uppi_\Xi$ is the mollification operator defined prior to Proposition~\ref{prop on the specification of the coupling}), then the functions
                \begin{multline}\label{you wouldnt really want to type this mess more than once_Dos}
            H = \mathsf{H} + \mathfrak{h}\tp{\lambda_1,\lambda_2},\quad U = \mathsf{U} + \mathfrak{u}\tp{\lambda_1,\lambda_2},\quad U_1 = \mathsf{U}_1 + \mathfrak{u}_1\tp{\lambda_1,\lambda_2},\quad U_2 = \mathsf{U}_2 + \mathfrak{u}_2\tp{\lambda_1,\lambda_2},\\
            P = \mathsf{P} + \mathfrak{p}\tp{\lambda_1,\lambda_2},\quad P_1 = \mathsf{P}_1 + \mathfrak{p}_1\tp{\lambda_1,\lambda_2},
            \text{ and }
            P_2 = \mathsf{P}_2 + \mathfrak{p}_2\tp{\lambda_1,\lambda_2}
        \end{multline}
        satisfy the system of ODEs~\eqref{_SWE_0}, \eqref{_SWE_1}, and~\eqref{_SWE_2} from Theorem~\ref{thm on the shallow water and residual equations}, under the coupling specification~\eqref{the specified coupling equation}.
    \end{enumerate}
\end{propC}
\begin{proof}
    
    To begin we define the function $\mathsf{H} \in W^{\infty,\infty}\tp{\R}$ via $\mathsf{H} = e^{\varrho}$, where  $\varrho:\R\to[\uprho_-,\uprho_\star]$ is a smooth map provided by  Theorem~\ref{thm on distinguished heteroclinic orbits}, using the first item when $\iota =1$ and $G(\uprho_\star) >0$, and the second item when $\iota=-1$ and $G(\uprho_-) < 0$.   We emphasize that the theorem does not provide a unique $\varrho$ (see Remark~\ref{rmk on translation degeneracy}), so we choose arbitrarily here.   Regardless of the choice, when $\iota =1$ the solution is a heteroclinic orbit from the hyperbolic equilibrium $\uprho_-$ to the sink $\uprho_+$ that obeys the estimates~\eqref{who said your life's a bore}, and when $\iota =-1$ the solution is a heteroclinic orbit from the source $\uprho_+$  to the hyperbolic equilibrium $\uprho_-$ that obeys the estimates~\eqref{who said your life's a bore variant}.  These estimates translate to those stated in the second item for $\mathsf{Z} = \mathsf{H}$, provided that $M$ is large enough.  

    For the sake of brevity, in the remainder of the proof we will only consider the case $\iota =1$, as the perturbation theory from Section~\ref{subsection on perturbation theory for heteroclinic orbits} is adapted to the case in which the background solution connects the unstable manifold on the left to the sink on the right.  The case $\iota=-1$ may be handled by first reversing the background solution $\mathsf{H}$, arguing as in the case $\iota=1$, and then reversing the perturbed solutions again.
    
    We now endeavor to construct the $\mathfrak{h}$-map in~\eqref{Siviglia_Dos} when $\iota =1$.  Proposition~\ref{prop on the specification of the coupling} suggests we should employ the theory of Section~\ref{subsection on perturbation theory for heteroclinic orbits} to attack the equation~\eqref{peturbed shallow water equation}.     
    To this end, define the distinguished heteroclinic orbit $\mathsf{Y} = \tp{\varrho,\varrho'}$, which is a solution to $\mathsf{Y}' = \Phi\tp{\mathsf{Y}}$ satisfying $\lim_{t\to\pm\infty}\mathsf{Y}\tp{t} = \tp{\uprho_\pm,0}$, where $\Phi:\R^2\to\R^2$ is the smooth vector field defined in~\eqref{The shallow water ODEs v4}.
    
    Given any $0<\be<\m{H}_-/6$, which we will shrink throughout the proof as needed, we define the complete metric space 
    \begin{equation}\label{the complete metric space}
        \Lambda = \Bar{B(0,\be)}\subset\tsb{\tp{C^{10}\cap W^{10,\infty}\cap H^{10}}\tp{\R}}^2,\quad \tnorm{\cdot}_{\Lambda} = \tnorm{\cdot}_{W^{10,\infty}\cap H^{10}}.
    \end{equation}
    We then take the functions $\Psi_0:\R^2\to\R^2$ and $\Psi_1:\Lambda\times\R\times\R^2\to\R^2$ of the decomposition $\Psi = \Psi_0 + \Psi_1$ from~\eqref{this is a sparse equation} to be
    \begin{equation}\label{why is this proof in particular so difficult to write?}  
        \Psi_0(x) = \Phi(x)  
        \text{ and }
        \Psi_1(\lambda,t,x) = \aleph\tp{\lambda_1(t),\lambda_1'\tp{t},\lambda_2(t),\lambda_2'\tp{t},x}\theta(e^{x_1}/\m{H}_-)
    \end{equation}
    where  $\aleph$ is the map granted by Proposition~\ref{prop on the specification of the coupling}, and $\theta:\R\to\R$ is a smooth function that satisfies $0\le\theta\le 1$ along with $\theta(z) = 0$ for $z\le 1/2$ and $\theta(z)=1$ for $z\ge 2/3$.  The purpose of the cut-off $\theta$ is to allow $\Psi_1$ to be globally defined, since $\aleph$ is only defined on $D$.   Note that $\Psi_0^{-1}\tcb{0}=\tcb{\tp{\uprho_-,0},\tp{\uprho_+,0}}$ and that $\tp{\uprho_-,0}$ and $\tp{\uprho_+,0}$ are hyperbolic and asymptotically stable equilibrium points for $\Psi_0$, respectively.   
    
    Next, we  introduce the parameter $1\le R\in\R^+$ to determine the quantities~\eqref{important_quantities_1}, \eqref{important_quantities_2}, and~\eqref{important_quantities_3}, and assume it is large enough  that
    \begin{equation}
        |\uprho_\star| + |\uprho_-| + \tnorm{\mathsf{Y} - \tp{\uprho_+,0}}_{L^\infty\cap L^2\tp{(0,\infty)}} + \tnorm{\mathsf{Y} - \tp{\uprho_-,0}}_{L^\infty\cap L^2\tp{(-\infty,0)}}\le R/10.
    \end{equation}
    For $n\in\tcb{0,1,\dots,9}$ the quantities $\Bar{\kkappa}_n$ and $\Bar{\qoppa}_n$ are finite. By using~\eqref{dont write out the formula for this function}, it is not hard to see that
    \begin{equation}
        \Bar{\kkappa}_n\lesssim\be \text{ and }    \Bar{\qoppa}_n\lesssim\be
    \end{equation}
    for implicit constants depending only on $R$, $\upmu$, $\m{a}$, $\m{g}$, and $\m{A}$.  Taking $\be$ to be sufficiently small, depending only on $\mathsf{Y}$ and the aforementioned parameters, we may guarantee that the hypotheses of Theorem~\ref{thm on perturbations of heteroclinic orbits} and Corollary~\ref{coro on improved regularity} are satisfied.  These then grant us a Lipschitz map
    \begin{equation}\label{_IM_RUNNING_OUT_OF_LABELS_}
        \mathfrak{Y}:\Bar{B(0,\beta)}\to\tp{C^9\cap W^{9,\infty}\cap H^9}\tp{\R;\R^2}
    \end{equation}
    such that $\mathfrak{Y}(0) = 0$ and for all $\lambda\in\Bar{B(0,\be)}$ the function $Y = \mathsf{Y} + \mathfrak{Y}\tp{\lambda}$ is a solution to the ODE $Y' = \Psi_0\tp{Y} + \Psi_1\tp{\lambda(\cdot),\cdot,Y}$. Further shrinking $\be>0$  if necessary, we may impose the extra condition that the cut-off function $\theta$ in~\eqref{why is this proof in particular so difficult to write?} only returns unity when evaluated at the $Y$  produced via the map~\eqref{_IM_RUNNING_OUT_OF_LABELS_}.

    We are now ready to build the $\mathfrak{h}$-map in~\eqref{Siviglia_Dos} by setting
    \begin{equation}
        \mathfrak{h}(\lambda) = \mathsf{H} \tp{e^{\mathfrak{Y}_1\tp{\lambda}} - 1}.
    \end{equation}
    At first glance, \eqref{_IM_RUNNING_OUT_OF_LABELS_} seems to suggest that $\mathfrak{h}$ maps into the wrong codomain, with $9$ derivatives instead of the stated $10$, but the promotion to $10$ indeed holds thanks to the identity $\mathfrak{Y}_1' = \mathfrak{Y}_2$.
     
     The properties $\mathfrak{h}\tp{0,0} = 0$ and $\tnorm{\mathfrak{h}\tp{\lambda_1,\lambda_2}}_{L^\infty\tp{\R}}\le\m{H}_-/3$ follow as long as we take $\be$ to be small enough, depending on $\upmu$, $\m{a}$, $\m{g}$, $\m{A}$, and $\mathsf{H}$.  That $\mathfrak{h}\tp{\lambda}\in H^\infty\tp{\R}$ when  $\lambda\in\Bar{B(0,\be)}\cap\tsb{H^\infty\tp{\R}}^2$, follows from an argument like that used in Corollary~\ref{coro on improved regularity}. The Lipschitz constant of $\mathfrak{h}$ may be taken as a function of only $\upmu$, $\m{a}$, $\m{g}$, $\m{A}$, and $\mathsf{H}$.  
     
     Thanks to the properties of the map~\eqref{_IM_RUNNING_OUT_OF_LABELS_} and the identity $\log\tp{\mathsf{H} + \mathfrak{h}\tp{\lambda}} = \mathsf{Y}_1 + \mathfrak{Y}\tp{\lambda}_1$ we find that if for some $\ep\in(0,1)$, $\Xi\in\tp{1,\infty}$, and $\tp{\eta,w}\in\tsb{H^1\tp{\R}}^2$ we have $\lambda_1=\ep^2\uppi_\Xi\eta$ and $\lambda_2 = \ep^2\uppi_\Xi w$, the function $\rho = \log\tp{\mathsf{H} + \mathfrak{h}\tp{\lambda_1,\lambda_2}}\in W^{\infty,\infty}\tp{\R}$ is a solution to the ODEs~\eqref{peturbed shallow water equation} and~\eqref{peturbed shallow water residual}.

    With $\mathsf{H}$ and $\mathfrak{h}$ in hand, we now turn to the construction of the functions $\mathsf{U}$, $\mathsf{U}_1$, $\mathsf{U}_2$, $\mathsf{P}$, $\mathsf{P}_1$, and $\mathsf{P}_2$ along with the remaining maps in~\eqref{Siviglia_Dos}. Once this is done, the claimed properties from the first, second, third, and fourth items follow from straightforward computations that we omit for brevity.      We define
    \begin{multline}\label{hello there distinguished gentlemen}
        \mathsf{U} = 4 - \f{\m{A}}{\mathsf{H}},\quad \mathsf{P} = \m{g}\mathsf{H} - 2\upmu\mathsf{U}',\quad \mathsf{U}_1 = \f{\m{a}}{\upmu}\mathsf{U},\quad\mathsf{U}_2 =\mathsf{U}'' -  \f{\upmu\mathsf{U}_1 + \tp{\mathsf{P} - \m{g}\mathsf{H} - 2\upmu\mathsf{U}'}\mathsf{H}'}{\upmu\mathsf{H}},\\
        \mathsf{P}_1 = - \upmu\mathsf{U}_1', 
        \text{ and }
        \mathsf{P}_2 = \tp{\mathsf{U} - 4}\mathsf{U}'' - \tp{\mathsf{U}'}^2 - \upmu\tp{\mathsf{U}''' + \mathsf{U}_2'}.
    \end{multline}
    Then, for $\tp{\lambda_1,\lambda_2}\in\Bar{B(0,\be)}$ we define
    \begin{multline}
        \mathfrak{u}(\lambda_1,\lambda_2) = \f{\m{A}}{\mathsf{H}} - \f{\m{A}}{\mathsf{H} + \mathfrak{h}\tp{\lambda_1,\lambda_2}},\quad\mathfrak{p}\tp{\lambda_1,\lambda_2} = \m{g}\mathfrak{h}\tp{\lambda_1,\lambda_2} - 2\upmu\tsb{\mathfrak{u}\tp{\lambda_1,\lambda_2}}' + \m{g}\lambda_1,\\\mathfrak{p}_1\tp{\lambda_1,\lambda_2} = -\upmu\tsb{\mathfrak{u}_1\tp{\lambda_1,\lambda_2}}',\quad
        \mathfrak{u}_1\tp{\lambda_1,\lambda_2} = \f{\m{a}}{\upmu}\mathfrak{u}\tp{\lambda_1,\lambda_2},\quad \mathfrak{u}_2\tp{\lambda_1,\lambda_2} = - \mathsf{U}_2 + \tp{\mathsf{U} + \mathfrak{u}\tp{\lambda_1,\lambda_2}}'' \\- \f{\upmu\tp{\mathsf{U}_1 + \mathfrak{u}_1\tp{\lambda_1,\lambda_2}} + \tp{-4\upmu\tp{\mathsf{U} + \mathfrak{u}\tp{\lambda_1,\lambda_2}}' + \m{g}\lambda_1}\tp{\mathsf{H} + \mathfrak{h}\tp{\lambda_1,\lambda_2} + \lambda_1}' - \m{g}\tp{\mathsf{H} + \mathfrak{h}\tp{\lambda_1,\lambda_2}}'\lambda_1}{\upmu\tp{\mathsf{H} + \mathfrak{h}\tp{\lambda_1,\lambda_2} + \lambda_1}}, \\
        \text{ and }
        \mathfrak{p}_2\tp{\lambda_1,\lambda_2} = -\mathsf{P}_2 + \tp{\mathsf{U} + \mathfrak{u}\tp{\lambda_1,\lambda_2} - 4}\tp{\mathsf{U} + \mathfrak{u}\tp{\lambda_1,\lambda_2}}'' - \tp{\tsb{\mathsf{U} + \mathfrak{u}\tp{\lambda_1,\lambda_2}}'}^2 \\- \upmu\tp{\tp{\mathsf{U} + \mathfrak{u}\tp{\lambda_1,\lambda_2}}''' - \tp{\mathsf{U}_2 + \mathfrak{u}_2\tp{\lambda_1,\lambda_2}}'}.
    \end{multline}

     To complete the proof, we now justify the fifth item by invoking the equivalence of Proposition~\ref{prop on the specification of the coupling}: our construction exactly ensures that $Y = \tp{\rho,\rho'}$ solves~\eqref{peturbed shallow water equation} and~\eqref{peturbed shallow water residual}, while $U$, $U_1$, $U_2$, $P$, $P_1$, and $P_2$ are exactly determined by $H = e^\rho$ and $\eta$ as in~\eqref{SWE___0}, \eqref{SWE___1}, and~\eqref{SWE___2}. 
\end{proof}

\begin{rmkC}[On the domain space for the bore map]\label{rmk on the domain space for the bore map}
    The choice of the number of derivatives to keep track of in the domains of the maps~\eqref{Siviglia_Dos} being equal to ten is somewhat arbitrary. The only mandate is that we feed the bore maps something with enough regularity so that the outputs have enough derivatives to comfortably handle the computations of the next section.
\end{rmkC}

Proposition~\ref{prop on the bore map} provides precise information about the various pieces that contribute to the bore map.  For our purposes later it will be convenient to trade some of this precision for uniformity in terms of spaces and derivative counts.  The following result records this consolidated information about the bore map.

\begin{coroC}[The bore map, II]\label{coro on bore map III}
Assume the hypotheses and notation of Proposition~\ref{prop on the bore map}.  For all $Z\in\tcb{H,U,U_1,U_2,P,P_1,P_2}$ the proposition allows us to view the sums
    \begin{equation}\label{generic bound}
        H = \mathsf{H} + \mathfrak{h},\; U = \mathsf{U} + \mathfrak{u},U_1 = \mathsf{U}_1 + \mathfrak{u}_1,\; U_2 = \mathsf{U}_2 + \mathfrak{u}_2,\; P = \mathsf{P} + \mathfrak{p},\; P_1 =  \mathsf{P}_1 + \mathfrak{p}_1,\; P_2 = \mathsf{P}_2 + \mathfrak{p}_2
    \end{equation}
    as maps $Z:\Bar{B(0,\be)}\to\tp{C^7\cap W^{7,\infty}}\tp{\R}$ for $\Bar{B(0,\be)}\subset\tsb{\tp{C^{10}\cap W^{10,\infty}\cap H^{10}}\tp{\R}}^2$.  For all $\tp{\lambda_1,\lambda_2},\tp{\tilde{\lambda}_1,\tilde{\lambda}_2}\in\Bar{B(0,\be)}$ we have the estimates
    \begin{equation}\label{yall got any more of them derivatives}
        \tnorm{Z(\lambda_1,\lambda_2)}_{W^{7,\infty}} + \tnorm{\tsb{Z\tp{\lambda_1,\lambda_2}}'}_{H^6}\le2M
    \end{equation}
    and
    \begin{equation}\label{generic lipschitz bound}
        \tnorm{Z(\lambda_1,\lambda_2) - Z(\tilde{\lambda}_1,\tilde{\lambda}_2)}_{W^{7,\infty}\cap H^7}\le M\tnorm{\tp{\lambda_1,\lambda_2} - \tp{\tilde{\lambda}_1,\tilde{\lambda}_2}}_{W^{10,\infty}\cap H^{10}}.
    \end{equation}
\end{coroC}
\begin{proof}
    These are trivial consequences of Proposition~\ref{prop on the bore map}.
\end{proof}


\subsection{Properties of residual forcing and driving terms}\label{subsection on properties of residual forcing and nonlinearities}

With the specification of $\mathfrak{r}_1$, $\mathfrak{r}_3$, and $\mathfrak{r}_4$ and the bore map theory of Section~\ref{subsection on the shallow water bore map} in hand, we are now in a position to study the residual forcing terms $\mathfrak{f}_0$, $\mathfrak{f}_1$, $\mathfrak{f}_2$, $\mathfrak{f}_3$, and $\mathfrak{f}_4$ appearing in system~\eqref{first form of the residual PDEs}.  In particular, we are interested in viewing these terms as maps built from the various subcomponents we constructed in the previous subsection and understanding their resulting mapping properties.  To accomplish this we first need to introduce some notation.  

To measure the free surface residual $\eta$, for $\ep\in\tp{0,1}$ we employ the following $\ep$-dependent norm on $H^{5/2}\tp{\Sigma_\ep}$:
\begin{equation}\label{epsilon dependent norms for the free surface residual}
    \tnorm{\eta}_{H^{5/2}\tp{\Sigma_\ep}}^{\tp{\ep}} = \ep^{1/2}\tnorm{\eta}_{H^2\tp{\Sigma_\ep}} + \ep\tnorm{\eta}_{H^{5/2}\tp{\Sigma_\ep}}.
\end{equation}
Next, we introduce the function $\ep_0 :(1,\infty) \to (0,1)$ defined by $\ep_0\tp{R} = \sqrt{\m{H}_-/\tp{ 6C  R}}$ where $1\le C<\infty$ is a constant from the one-dimensional Morrey inequality: $\norm{g}_{L^\infty(\R)} \le C \norm{g}_{H^1\tp{\R}}$.  The main use of this function is in the following observation:  if $f,g : \R \to \R$ satisfy $f \ge 2\m{H}_-/3$ and $\norm{g}_{H^1\tp{\R}} \le R$, then  for $0 < \ep < \ep_0(R)$ we have that $f + \ep^2 g \ge 2\m{H}_-/3 - \ep^2 \norm{g}_{L^\infty\tp{\R}} \ge 2\m{H}_-/3 - \ep^2 C R  \ge \m{H}_-/2$.  

The purpose of the following definition is to fix our choice of parameters and specify the functional dependence of the terms appearing in Theorem~\ref{thm on the shallow water and residual equations} on the bore map theory of Section~\ref{subsection on the shallow water bore map}.

\begin{defnC}[Ansatz mapping terms and residual forcing terms]\label{defn big steppah}
Let $\upmu,\m{a}>0$, $\iota\in\tcb{-1,1}$, $\tp{\m{g},\m{A}}\in\mathfrak{C}_\iota$, and $\upsigma\ge0$. Thanks to Proposition~\ref{prop on the bore map} and Corollary~\ref{coro on bore map III} there exists $\be\in\tp{0,1}$ and bore map constituents that are Lipschitz continuous mappings
\begin{equation}\label{ansatz, but now precise}
    H,U,U_1,U_2,P,P_1,P_2:\Bar{B(0,\be)}\subset\tsb{\tp{C^{10}\cap W^{10,\infty}\cap H^{10}}\tp{\R}}^2\to\tp{C^7\cap W^{7,\infty}}\tp{\R}
\end{equation}
defining a solution operator to the system of ODEs~\eqref{_SWE_0}, \eqref{_SWE_1}, and~\eqref{_SWE_2} in the sense made precise by the fifth item of Proposition~\ref{prop on the bore map}. Note that $H$  satisfies  $H\ge\tp{2/3}\m{H}_-$.   Given $\ep\in\tp{0,1}$, $\lambda\in\Bar{B(0,\be)}$, and $\eta\in H^{5/2}\tp{\Sigma_\ep}$, we define the following functions
\begin{equation}\label{any suggestions for how to make this stuff nicer}
    \begin{split}
        \pmb{\zeta}(\ep,\lambda,\eta) & = H(\lambda) + \ep^2\eta : \R \to \R, \\
        W_1(\ep,\lambda,\eta) &= \f{y\pmb{\zeta}\tp{\ep,\lambda,\eta}}{\ep}U_1\tp{\lambda} + \f12\bp{\f{y\pmb{\zeta}\tp{\ep,\lambda,\eta}}{\ep}}^2U_2\tp{\lambda} : \R^2 \to \R, \\
        W_2\tp{\ep,\lambda,\eta} &=  - \ep\bp{\f12\bp{\f{y\pmb{\zeta}\tp{\ep,\lambda,\eta}}{\ep}}^2U_1\tp{\lambda}' + \f16\bp{\f{y\pmb{\zeta}\tp{\ep,\lambda,\eta}}{\ep}}^3U_2\tp{\lambda}'} : \R^2 \to \R, \\
        W(\ep,\lambda,\eta) &= W_1\tp{\ep,\lambda,\eta}e_1 + W_2\tp{\ep,\lambda,\eta}e_2 : \R^2 \to \R^2, \\
        V\tp{\ep,\lambda,\eta} &= U\tp{\lambda}e_1 - y\pmb{\zeta}\tp{\ep,\lambda,\eta}U(\lambda)'e_2 : \R^2 \to \R^2, \\
        X(\ep,\lambda,\eta) &= V(\ep,\lambda,\eta) + \ep^2 W\tp{\ep,\lambda,\eta} : \R^2 \to \R^2, \\
         Q(\ep,\lambda,\eta) &= \f{y\pmb{\zeta}\tp{\ep,\lambda,\eta}}{\ep}P_1\tp{\lambda} + \f12\bp{\f{y\pmb{\zeta}\tp{\ep,\lambda,\eta}}{\ep}}^2P_2\tp{\lambda} : \R^2 \to \R.
    \end{split} 
\end{equation}

Next, let $R \in (1,\infty)$ and $\ep_0:\tp{1,\infty}\to\tp{0,1}$ be as above and suppose that $0<\ep<\ep_0\tp{R}$.  We then define 
the sets
\begin{multline}\label{theses sets prevent illegal curry}
    E^R_\ep=\tcb{\tp{\lambda,\eta}\in\Bar{B(0,\be)}\times H^{5/2}\tp{\Sigma_\ep}\;:\;\tnorm{\eta}^{\tp{\ep}}_{H^{5/2}\tp{\Sigma_\ep}}\le R}\text{ and }\\
    F^R_\ep=\tcb{\tp{\lambda,\eta,u}\in\Bar{B(0,\be)}\times H^{5/2}\tp{\Sigma_\ep}\times H^2_{\m{tan}}\tp{\Omega_\ep;\R^2}:\;\max\tcb{\tnorm{\eta}^{\tp{\ep}}_{H^{5/2}\tp{\Sigma_\ep}},\tnorm{u}_{H^2\tp{\Omega_\ep}}}\le R}.
\end{multline}
Note that if $(\lambda,\eta) \in E^R_\ep$ then the choice of $\ep_0$ implies that $\pmb{\zeta}(\ep,\lambda,\eta) \ge \m{H}_- /2$, which in turn means that the geometry matrix $\mathcal{A}_{\pmb{\zeta}\tp{\ep,\lambda,\eta}}$ from~\eqref{geometry matrices} is well-defined. In turn, we define the following maps, noting that each is well-defined on the given domain but that the appropriateness of the codomain will be verified in the subsequent result:
\begin{equation}\label{no more illegal curry}
    \mathfrak{f}_0\tp{R,\ep;\cdot},\;\mathfrak{f}_3\tp{R,\ep;\cdot},\;\mathfrak{f}_4\tp{R,\ep;\cdot}:E^R_\ep\to H^{1/2}\tp{\Sigma_\ep}\text{ and }\mathfrak{f}_1\tp{R,\ep;\cdot},\;\mathfrak{f}_2\tp{R,\ep;\cdot}:F^R_\ep\to L^2\tp{\Omega_\ep}
\end{equation}
are defined via
    \begin{equation}\label{really bro needs a label no cap}
        \mathfrak{f}_0\tp{R,\ep;\lambda,\eta} = \tp{1 - \pd_1^2}\tp{\Bar{\upgamma}H(\lambda) - \tp{H(\lambda) + \ep^2\eta}^2U_1(\lambda)/2 - \tp{H(\lambda) + \ep^2\eta}^3U_2\tp{\lambda}/6 - \Bar{\m{A}}},
    \end{equation}
    \begin{multline}\label{I am sorry for how ugly this is}
        \mathfrak{f}_1\tp{R,\ep;\lambda,\eta,u} = -\Bar{\upgamma}U(\lambda)' + \tp{V(\ep,\lambda,\eta) - \tp{4 + \ep^2\Bar{\upgamma}}e_1}\cdot\grad^{\mathcal{A}_{\pmb{\zeta}\tp{\ep,\lambda,\eta}}}W_1\tp{\ep,\lambda,\eta} \\
        + W\tp{\ep,\lambda,\eta}\cdot\grad^{\mathcal{A}_{\pmb{\zeta}\tp{\ep,\lambda,\eta}}}X_1\tp{\ep,\lambda,\eta} 
        +\pd_1^{\mathcal{A}_{\pmb{\zeta}\tp{\ep,\lambda,\eta}}}Q\tp{\ep,\lambda,\eta} - \upmu\tp{\pd_1^{\mathcal{A}_{\pmb{\zeta}\tp{\ep,\lambda,\eta}}}}^2W_1\tp{\ep,\lambda,\eta} + \ep^2u\cdot\grad^{\mathcal{A}_{\pmb{\zeta}\tp{\ep,\lambda,\eta}}}u_1,
    \end{multline}
    \begin{multline}\label{ugly expression of f2}
        \mathfrak{f}_2\tp{R,\ep;\lambda,\eta,u} = \Bar{\upgamma}y\pmb{\zeta}\tp{\ep,\lambda,\eta}U\tp{\lambda}'' + \tp{V\tp{\lambda} - \tp{4 + \ep^2\Bar{\upgamma}}e_1}\cdot\grad^{\mathcal{A}_{\pmb{\zeta}\tp{\ep,\lambda,\eta}}}W_2\tp{\ep,\lambda,\eta} \\+ W\tp{\ep,\lambda,\eta}\cdot\grad^{\mathcal{A}_{\pmb{\zeta}\tp{\ep,\lambda,\eta}}}X_2\tp{\ep,\lambda,\eta}-\upmu\tp{\pd_1^{\mathcal{A}_{\pmb{\zeta}\tp{\ep,\lambda,\eta}}}}^2W_2\tp{\ep,\lambda,\eta} + \ep^2 u\cdot\grad^{\mathcal{A}_{\pmb{\zeta}\tp{\ep,\lambda,\eta}}}u_2,
    \end{multline}
     \begin{multline}\label{country roads take me home}
        \mathfrak{f}_3\tp{R,\ep;\lambda,\eta} = \ep Q(\ep,\lambda,\eta)\pmb{\zeta}\tp{\ep,\lambda,\eta}' - 2\ep\upmu\pd_1^{\mathcal{A}_{\pmb{\zeta}\tp{\ep,\lambda,\eta}}}W_1\tp{\ep,\lambda,\eta}\pmb{\zeta}\tp{\ep,\lambda,\eta}' + \upmu\pd_1^{\mathcal{A}_{\pmb{\zeta}\tp{\ep,\lambda,\eta}}}W_2\tp{\ep,\lambda,\eta} \\+ \upsigma\mathcal{H}\tp{\ep\pmb{\zeta}\tp{\ep,\lambda,\eta}}\pmb{\zeta}\tp{\ep,\lambda,\eta}' - \ep^3\m{g}\eta\eta',
    \end{multline}
    and
    \begin{equation}\label{please sir may I have a label}
        \mathfrak{f}_4\tp{R,\ep;\lambda,\eta} = -Q\tp{\ep,\lambda,\eta} + \upmu\pmb{\zeta}\tp{\ep,\lambda,\eta}\pmb{\zeta}\tp{\ep,\lambda,\eta}'U\tp{\lambda}'' + \upmu\mathbb{D}^{\mathcal{A}_{\pmb{\zeta}\tp{\ep,\lambda,\eta}}}W\tp{\ep,\lambda,\eta}\mathcal{N}_{\ep\pmb{\zeta}\tp{\ep,\lambda,\eta}}\cdot e_2 - \ep^{-1}\upsigma\mathcal{H}\tp{\ep\pmb{\zeta}\tp{\ep,\lambda,\eta}}.
    \end{equation}
\end{defnC}

The codomains of the various maps~\eqref{no more illegal curry} in the previous definition are determined by the data space (\eqref{the definition of the strong codomain space} and~\eqref{norm on Y1}) associated with the well-posedness theory for the thin domain Stokes equations with stress boundary conditions, which we develop in Section~\ref{subsection on the stokes equations with stress boundary conditions}. Our next result studies the mapping properties of the residual forcing terms introduced in the previous definition in these spaces. We recall the adapted boundary norms introduced in Lemma~\ref{lem on adapted boundary norms}. We also emphasize that the $\mathfrak{f}_0$, given by~\eqref{really bro needs a label no cap}, estimate in the following is the key place where the parameter tuning~\eqref{the parameter tuning identities 440 hz} is required; it is exploited in order to force $\mathfrak{f}_0$ to vanish at infinity and be square summable. 

\begin{propC}[Mapping properties of the residual forcing terms]\label{prop on the mapping propertieso of the residual forcing}
Assume the hypotheses of Definition~\ref{defn big steppah}. There exists functions $\ep_1:\tp{1,\infty}\to\tp{0,1}$ and $C_1:\tp{1,\infty}\to\tp{1,\infty}$ with $\ep_1\le\ep_0$ such that the following hold. First, we have the lower bound
\begin{equation}\label{this bound is here for today then its gone}
    \inf_{ \substack{ 1<R<\infty \\ 0<\ep<\ep_0\tp{R}} }\inf_{\tp{\lambda,\eta}\in E^R_\ep}\inf_{\R}\tp{H(\lambda) + \ep^2\eta}\ge\m{H}_-/2.
\end{equation}
Second, we have the uniform range estimates:
\begin{multline}\label{residual forcing bounds}
    \sup_{\substack{1<R<\infty \\ 0<\ep<\ep_1\tp{R}} } \sup_{\tp{\lambda,\eta}\in E^R_\ep}\ssb{\tnorm{\mathfrak{f}_0\tp{R,\ep;\lambda,\eta}}_{H^{1/2}\tp{\Sigma_\ep}} + \tnorm{\mathfrak{f}_3\tp{R,\ep;\lambda,\eta}}^{(\ep,+)}_{H^{1/2}\tp{\Sigma_\ep}} + \tnorm{\mathfrak{f}_4\tp{R,\ep;\lambda,\eta}}^{(\ep,-)}_{H^{1/2}\tp{\Sigma_\ep}}}<\infty\text{ and }\\
    \sup_{\substack{1<R<\infty\\0<\ep<\ep_1\tp{R}}}\sup_{\tp{\lambda,\eta,u}\in F^R_\ep}\ssb{\tnorm{\mathfrak{f}_1\tp{R,\ep;\lambda,\eta,u}}_{L^2\tp{\Omega_\ep}} + \tnorm{\mathfrak{f}_2\tp{R,\ep;\lambda,\eta,u}}_{L^2\tp{\Omega_\ep}}}<\infty.
\end{multline}
Third, we have the uniform Lipschitz estimates:
\begin{equation}\label{lipschitz 0__}
    \sup_{\substack{1<R<\infty\\0<\ep<\ep_1\tp{R}}}\sup_{\tp{\lambda,\eta}\neq\tp{\tilde{\lambda},\tilde{\eta}}}\f{\tnorm{\mathfrak{f}_0\tp{R,\ep;\lambda,\eta} - \mathfrak{f}_0\tp{R,\ep;\tilde{\lambda},\tilde{\eta}}}_{H^{1/2}\tp{\Sigma_\ep}}}{\tnorm{\lambda - \tilde{\lambda}}_{\Lambda} + C_1\tp{R}\ep\tnorm{\eta - \tilde{\eta}}^{\tp{\ep}}_{H^{5/2}\tp{\Sigma_\ep}}}<\infty;
\end{equation}
writing $s_3 = +$ and $s_4 = -$,
\begin{equation}\label{lipschitz 034__}
    \sup_{\substack{1<R<\infty\\0<\ep<\ep_1\tp{R}}}\sup_{\tp{\lambda,\eta}\neq\tp{\tilde{\lambda},\tilde{\eta}}}  
    \f{\tnorm{\mathfrak{f}_j\tp{R,\ep;\lambda,\eta} - \mathfrak{f}_j\tp{R,\ep;\tilde{\lambda},\tilde{\eta}}}^{\tp{\ep,s_j}}_{H^{1/2}\tp{\Sigma_\ep}} }{\tnorm{\lambda - \tilde{\lambda}}_{\Lambda} + C_1\tp{R}\ep\tnorm{\eta - \tilde{\eta}}^{\tp{\ep}}_{H^{5/2}\tp{\Sigma_\ep}}}<\infty \text{ for } j \in \{3,4\};
\end{equation}
and for $j \in \{1,2\}$,   
\begin{equation}\label{lipscthiz 12__}
    \sup_{\substack{1<R<\infty\\0<\ep<\ep_1\tp{R}}}\sup_{\tp{\lambda,\eta,u}\neq\tp{\tilde{\lambda},\tilde{\eta},\tilde{u}}}\f{\tnorm{\mathfrak{f}_j\tp{R,\ep;\lambda,\eta,u} - \mathfrak{f}_j\tp{R,\ep;\tilde{\lambda},\tilde{\eta},\tilde{u}}}_{L^2\tp{\Omega_\ep}}} 
    {\tnorm{\lambda - \tilde{\lambda}}_{\Lambda} + C_1\tp{R}\ep^2\tnorm{\eta - \tilde{\eta}}_{W^{1,\infty}\tp{\Sigma_\ep}} + C_1\tp{R}\ep^{3/2}\tnorm{u - \tilde{\eta}}_{H^{2}\tp{\Omega_\ep}}}<\infty,
\end{equation}
where $\tnorm{\cdot}_{\Lambda}$ is as in~\eqref{the complete metric space}.
\end{propC}
\begin{proof}
    For the most part the stated estimates follow from straightforward computations and estimates, so for the sake of brevity we will only highlight the pieces that require some delicate care and omit details for the rest. We begin by noting that estimate~\eqref{this bound is here for today then its gone} is a direct consequence of the discussion preceding Definition~\ref{defn big steppah}.
    
    We now handle each term in order.
    
    \emph{Step 0 -- The term $\mathfrak{f}_0$:} Estimate of~\eqref{lipschitz 0__} is a straightforward consequence of estimate~\eqref{generic lipschitz bound} from Corollary~\eqref{coro on bore map III} along with the fact that $H^{5/2}\tp{\R}$ is an algebra.  In view of this Lipschitz estimate and the assumed bounds on $\lambda$ and $\eta$, it then suffices to establish the $\mathfrak{f}_0$-bound of~\eqref{residual forcing bounds} when $\lambda = 0$ and $\eta = 0$, which is equivalent to establishing the inclusion
    \begin{equation}\label{sure would like to finish this paper any day now}
        \tp{1 - \pd_1^2}^{-1}\mathfrak{f}_0\tp{\ep,0,0} = \Bar{\upgamma}\mathsf{H} - \mathsf{H}^2\mathsf{U}_1/2 - \mathsf{H}^3\mathsf{U}_2/6 - \m{A}\in H^{5/2}\tp{\Sigma_\ep},
    \end{equation}
    where $\mathsf{H} = H(0)$ and $\mathsf{U}_1$, and $\mathsf{U}_2$ are defined in~\eqref{hello there distinguished gentlemen}. We invoke the identities
    \begin{equation}
        \mathsf{U} = 4 - \m{A}/\mathsf{H},\quad\mathsf{U}_1 = \tp{\m{a}/\upmu}\mathsf{U}, 
        \text{ and }
        \mathsf{U}_2 =-\m{a}\mathsf{U}/\upmu\mathsf{H} + \mathsf{U}'' - 4\mathsf{H}'\mathsf{U}'/\mathsf{H}
    \end{equation}
    to rewrite the function in~\eqref{sure would like to finish this paper any day now} as
    \begin{equation}
        \tp{1 - \pd_1^2}^{-1}\mathfrak{f}_0(\ep,0,0) =\Bar{\upgamma}\mathsf{H} - \Bar{\m{A}} - \tp{4\m{a}/3\upmu}\mathsf{H}^3 - \tp{\m{a}/3\upmu}\mathsf{H}\tp{4\mathsf{H} - 4\mathsf{H}^2 - \m{A}} - \mathsf{H}^3\tp{\mathsf{U}'' - 4\mathsf{H}'\mathsf{U}'/\mathsf{H}}/6.
    \end{equation}
    Due to the exponential decay of derivatives provided by the second item of Proposition~\ref{prop on the bore map},  the third term  together with the derivatives of the first two terms  belong to $H^{\infty}\tp{\R}$ and, in particular
    \begin{equation}\label{cat food tastes good}
        \tnorm{\tp{\Bar{\upgamma}\mathsf{H} - \Bar{\m{A}} - \tp{4\m{a}/3\upmu}\mathsf{H}^3}'  }_{H^{3/2}} + \tnorm{(4\mathsf{H} - 4\mathsf{H}^2 - \m{A} )'}_{H^{3/2}} +         \tnorm{ \mathsf{H}^3\tp{\mathsf{U}'' - 4\mathsf{H}'\mathsf{U}'/\mathsf{H}}/6}_{H^{5/2}}  \lesssim 1.
    \end{equation}
    Thus, 
    \begin{equation}\label{almost at the end of this proof yay}
        \tnorm{\tp{1 - \pd_1^2}^{-1}\mathfrak{f}_0\tp{\ep,0,0}}_{H^{5/2}}\lesssim \tnorm{\Bar{\upgamma}\mathsf{H} - \Bar{\m{A}} - \tp{4\m{a}/3\upmu}\mathsf{H}^3}_{L^2} + \tnorm{4\mathsf{H} - 4\mathsf{H}^2 - \m{A}}_{L^2} + 1.
    \end{equation}
    In~\eqref{cat food tastes good} and~\eqref{almost at the end of this proof yay} the implicit constants depend only on the physical parameters initialized by Definition~\ref{defn big steppah} and the bore map~\eqref{ansatz, but now precise}. So it only remains to verify that each of the above polynomials-in-$\mathsf{H}$ belongs to $L^{2}\tp{\Sigma_\ep}$. While $\mathsf{H}$ itself does not belong to $L^2\tp{\Sigma_\ep}$, we do know that its limits at $\pm \infty$, namely $\m{H}_{\pm}$, satisfy the relations~\eqref{sweq} and~\eqref{the parameter tuning identities 440 hz}.  Hence the polynomials-in-$\mathsf{H}$ from~\eqref{almost at the end of this proof yay} vanish at $\pm\infty$, and we may invoke the fundamental theorem of calculus and the aforementioned exponential decay of derivatives to deduce that these functions vanish exponentially quickly at $\pm\infty$, and thus belong to $L^2\tp{\Sigma_\ep}$.

    \emph{Step 1 -- The term $\mathfrak{f}_1$:} The $\mathfrak{f}_1$-parts of the bounds~\eqref{residual forcing bounds} and~\eqref{lipscthiz 12__} are mostly straightforward, so we only point out the terms requiring some subtlety.
    
    Firstly, we observe that $V$, $W$, and $Q$ of~\eqref{any suggestions for how to make this stuff nicer} can be thought of as variable coefficient polynomials in $y/\ep$. Since $y\in\tp{0,\ep}$, this ratio $y/\ep$ is bounded independently of $\ep$. However, this structure shows that the application of $\pd_2^{\mathcal{A}_{\pmb{\zeta}}}$ to any one of these introduces a factor of $1/\ep$. This is not a problem because each of these $e_2$-type derivatives is paired with either $V_2$ or $W_2$, which have an additional factor of either $y$ or $\ep$ to compensate. On the other hand, the derivatives $\pd_1^{\mathcal{A}_{\pmb{\zeta}}}$ are essentially harmless, as the following computation elucidates:
    \begin{equation}
        V\tp{\ep,\lambda,\eta}\cdot\grad^{\mathcal{A}_{\pmb{\zeta}\tp{\ep,\lambda,\eta}}}W_1\tp{\ep,\lambda,\eta} = U\tp{\lambda}\pd_1^{\mathcal{A}_{\pmb{\zeta}\tp{\ep,\lambda,\eta}}} W_1\tp{\ep,\lambda,\eta} -\tp{y/\ep}\pmb{\zeta}\tp{\ep,\lambda,\eta}U\tp{\lambda}'\tp{U_1\tp{\lambda} + \tp{y/\ep}\pmb{\zeta}\tp{\ep,\lambda,\eta}U_2\tp{\lambda}}.
    \end{equation}
    
    Secondly, we note that $U$, $V_1$, $W_1$, and $X_1$ do not belong to square summable Sobolev spaces, but after application of $\pd_1^{\mathcal{A}_{\pmb{\zeta}}}$ they do become square summable, as a consequence of estimate~\eqref{yall got any more of them derivatives} from Corollary~\ref{coro on bore map III}.  On the other hand the second components $V_2$, $W_2$, $X_2$ do belong to square summable spaces.  As a consequence of these two facts, we deduce that each of the advective style derivatives in~\eqref{I am sorry for how ugly this is}, i.e. the first, second, and third grouping of terms, must map into square summable spaces. 
    
    Next, except possibly for the penultimate term in~\eqref{I am sorry for how ugly this is}, no term involves more than one derivative of $\pmb{\zeta}\tp{\ep,\lambda,\eta}$, which means that control on $\eta$ in $W^{1,\infty}\tp{\Sigma_\ep}$ is sufficient for these.  In fact, this is also true for the penultimate term, as can be see by writing it out:
    \begin{equation}
    \tp{\pd_1^{\mathcal{A}_{\pmb{\zeta}\tp{\ep,\lambda,\eta}}}}^2W_1\tp{\ep,\lambda,\eta} = \tp{y/\ep}\pmb{\zeta}\tp{\ep,\lambda,\eta}U_1\tp{\lambda}'' + \tp{\tp{y/\ep}\pmb{\zeta}\tp{\ep,\lambda,\eta}}^2U_2\tp{\lambda}''/2.
    \end{equation}

    Finally, the coefficient $\ep^{3/2}$ present in the velocity contributions on the right hand side of~\eqref{lipscthiz 12__} can be explained through the thin domain Sobolev embedding of the third item of Lemma~\ref{lem on poincare and sobolev in thin domains}. For example, we have the estimate
    \begin{equation}
        \tnorm{\ep^2u\cdot\grad^{\mathcal{A}_{{\pmb{\zeta}\tp{\ep,\lambda,\eta}}}}u_1}_{L^2\tp{\Omega_\ep}}\lesssim \ep^2\tnorm{u}_{L^\infty\tp{\Omega_\ep}}\tnorm{u}_{H^{1}\tp{\Omega_\ep}}\lesssim \ep^{3/2}\tnorm{u}_{H^1\tp{\Omega_\ep}},
    \end{equation}
    where the implicit constants depend only the physical parameters initialized by Definition~\ref{defn big steppah}, the bore map~\eqref{ansatz, but now precise}, and $R$.

    \emph{Step 2 -- The term $\mathfrak{f}_2$:} The same arguments used to handle $\mathfrak{f}_1$ in the previous step work here as well, so we omit further detail.

    \emph{Step 3 -- The term $\mathfrak{f}_3$:} Again we only highlight the key observations.  The norm $\tnorm{\cdot}^{\tp{\ep,+}}_{H^{1/2}\tp{\Sigma_\ep}}$ is the sum of the $H^{1/2}\tp{\Sigma_\ep}$ norm with $\ep^{-1/2}$ times the $L^2\tp{\Sigma_\ep}$ norm,     so we may consider each of these separately.  The first three terms in~\eqref{country roads take me home} are no trouble to bound thanks to Corollary~\ref{coro on bore map III}, each term having an additional factor of $\ep$ to offset the $\ep^{-1/2}$ built into the norm (for see $W_2$ in~\eqref{any suggestions for how to make this stuff nicer}), and the fact that $Q$, $\pd_1^{\mathcal{A}_{\pmb{\zeta}}}W_1$, and $W_2$ are all square summable.

    To handle the capillary contribution in~\eqref{country roads take me home} we expand
    \begin{equation}\label{curvature expansion}
        \ep^{-1}\mathcal{H}\tp{\ep\pmb{\zeta}\tp{\ep,\lambda,\eta}} = \pmb{\zeta}\tp{\ep,\lambda,\eta}''\cdot\tp{1 + \ep^2\tabs{\pmb{\zeta}\tp{\ep,\lambda,\eta}'}^2}^{-3/2}.
    \end{equation}
    The bound of the norm in $L^2\tp{\Sigma_\ep}$ is then straightforward, whereas for the $H^{1/2}\tp{\Sigma_\ep}$-norm we use the continuity of the product map $H^{1/2}\tp{\R\Sigma_\ep}\times H^1\tp{\Sigma_\ep}\to H^{1/2}\tp{\Sigma_\ep}$ and put the second factor in~\eqref{curvature expansion} in $H^1\tp{\Sigma_\ep}$.   For the final term in~\eqref{country roads take me home} we write $2\eta\eta'=\tp{\eta^2}'$ and use the fact that $H^1\tp{\Sigma_\ep}$ and $H^{3/2}\tp{\Sigma_\ep}$ are algebras.

    \emph{Step 4 -- The term $\mathfrak{f}_4$:} The norm employed for $\mathfrak{f}_4$ satisfies $\tnorm{\cdot}^{\tp{\ep,-}}_{H^{1/2}\tp{\Sigma_\ep}}\le\tnorm{\cdot}_{H^{1/2}\tp{\Sigma_\ep}}$, so it suffices to prove the $\mathfrak{f}_4$-part of the claimed estimates~\eqref{residual forcing bounds} and~\eqref{lipschitz 034__} for the stronger $H^{1/2}\tp{\Sigma_\ep}$ norm, and it turns out there is no need to take advantage of any gains given by the smaller norm.

    As in the previous step, handling the first and second terms of~\eqref{please sir may I have a label} is straightforward and the final capillary term may also be felled with inspection of~\eqref{curvature expansion}.   To see that the third term also does not require any new ideas, we can write it out fully as
    \begin{equation}\label{jajambo}
        \mathbb{D}^{\mathcal{A}_{\pmb{\zeta}\tp{\ep,\lambda,\eta}}}W\tp{\ep,\lambda,\eta}\mathcal{N}_{\ep\pmb{\zeta}\tp{\ep,\lambda,\eta}}\cdot e_2 = -\ep\pmb{\zeta}\tp{\ep,\lambda,\eta}'\tp{\pd_1^{\mathcal{A}_{\pmb{\zeta}\tp{\ep,\lambda,\eta}}}W_2\tp{\ep,\lambda,\eta} + \pd_2^{\mathcal{A}_{\pmb{\zeta}\tp{\ep,\lambda,\eta}}}W_1\tp{\ep,\lambda,\eta}} + 2\pd_2^{\mathcal{A}_{\pmb{\zeta}\tp{\ep,\lambda,\eta}}}W_2\tp{\ep,\lambda,\eta}.
    \end{equation}
    Firstly, since $W_2$ and $\pmb{\zeta}'$ are square summable the entire expression inherits this property.  As we discussed in Step 1 above, the components of $\pd_2^{\mathcal{A}_{\pmb{\zeta}}}W$ come with a factor of $\ep^{-1}$ (from the differentiation) but the first grouping of terms in~\eqref{jajambo} has an explicit extra $\ep$ and the final term with $W_2$ has an implicit one. On balance~\eqref{jajambo} comes with a net coefficient of $\ep^0$ or smaller and so it can be harmlessly bounded.
\end{proof}

At this point we have analyzed the forcing type residuals $\mathfrak{f}_0$, $\mathfrak{f}_1$, $\mathfrak{f}_2$, $\mathfrak{f}_3$, and $\mathfrak{f}_4$ appearing in system~\eqref{first form of the residual PDEs}. However, there are a few more discrepancies between this system and the linear Stokes system~\eqref{linear stokes in a thin domain} that have yet to be analyzed.  These are terms of linear driving type, which we define now.

\begin{defnC}[Residual driving terms]\label{defn underground methods}
    Assume the hypotheses of Definition~\ref{defn big steppah}. We further build upon the maps~\eqref{ansatz, but now precise} and~\eqref{any suggestions for how to make this stuff nicer} and the sets~\eqref{theses sets prevent illegal curry} by defining the following functions for $\Xi,R\in\tp{1,\infty}$ and $0<\ep<\ep_0\tp{R}$:
    \begin{equation}
        \mathfrak{d}^\Xi_1\tp{R,\ep;\cdot},\;\mathfrak{d}^\Xi_2\tp{R,\ep;\cdot}:F^R_\ep\to L^2\tp{\Omega_\ep}\text{ and }\mathfrak{d}_3^\Xi\tp{R,\ep;\cdot},\;\mathfrak{d}_4^\Xi\tp{R,\ep;\cdot}:E^R_\ep\to H^{1/2}\tp{\Sigma_\ep}
    \end{equation}
    given by
        \begin{multline}\label{first driver full expression}
        \mathfrak{d}^\Xi_1\tp{R,\ep;\lambda,\eta,u} = \tp{X\tp{\ep,\lambda,\eta} - \tp{4 + \ep^2\Bar{\upgamma}}e_1}\cdot\grad^{\mathcal{A}_{\pmb{\zeta}\tp{\ep,\lambda,\eta}}}u_1 \\+ u\cdot\grad^{\mathcal{A}_{\pmb{\zeta}\tp{\ep,\lambda,\eta}}}X_1\tp{\ep,\lambda,\eta} - \bp{\tp{U(\lambda) - 4}\uppi_\Xi\bp{\f{1}{\ep}\int_0^\ep u_1\tp{\cdot,y}\;\m{d}y}}',
    \end{multline}
        \begin{equation}\label{d2 expression}
        \mathfrak{d}^\Xi_2\tp{R,\ep;\lambda,\eta,u} = \tp{X\tp{\ep,\lambda,\eta} - \tp{4 + \ep^2\Bar{\upgamma}}}\cdot\grad^{\mathcal{A}_{\pmb{\zeta}\tp{\ep,\lambda,\eta}}}u_2 + u\cdot\grad^{\mathcal{A}_{\pmb{\zeta}\tp{\ep,\lambda,\eta}}}X_2\tp{\ep,\lambda,\eta},
    \end{equation}
    \begin{equation}\label{d3 expression}
        \mathfrak{d}^\Xi_3\tp{R,\ep;\lambda,\eta} = \ep\tp{P(\lambda) - \m{g}H\tp{\lambda} - 2\upmu U'\tp{\lambda}}\tp{\eta - \uppi_\Xi\eta}' + \ep\upmu\tp{U_2\tp{\lambda} - U\tp{\lambda}''}\tp{\eta - \uppi_\Xi\eta} - \ep\m{g}H\tp{\lambda}'\tp{\eta - \uppi_\Xi\eta},
    \end{equation}
    and
    \begin{equation}\label{d4 expression}
        \mathfrak{d}_4^\Xi\tp{R,\ep;\lambda,\eta} = \m{g}\tp{\eta - \uppi_\Xi\eta}.
    \end{equation}
\end{defnC}

Our next result studies the mapping properties of the residual driving terms from the previous definition.  Here the key point is that the $\Xi$ parameter allows us to make certain constants appearing in the estimates small, even when the parameter $R$ is large.  This will play an essential role later in a fixed point argument.

\begin{propC}[Mapping properties of the residual driving terms]\label{prop on mapping properties of the residual driving terms}
    Assume the hypotheses of Definition~\ref{defn big steppah}. There exists functions $\ep_2:\tp{1,\infty}\to\tp{0,1}$ and $C_2:\tp{1,\infty}\to\tp{1,\infty}$, with $\ep_2\le\ep_0$, such that the following hold. First, we again have the positivity~\eqref{this bound is here for today then its gone}. Second, we have the uniform range estimates:
    \begin{multline}\label{driver bounds}
        \sup_{\substack{1<R<\infty\\0<\ep<\ep_2\tp{R}}}\sup_{1<\Xi<\infty}\sup_{\tp{\lambda,\eta,u}\in F^R_\ep}\tp{1+R/\Xi}^{-1}\ssb{\tnorm{\mathfrak{d}_1^\Xi\tp{R,\ep;\lambda,\eta,u}}_{L^2\tp{\Omega_\ep}} + \tnorm{\mathfrak{d}^\Xi_2\tp{R,\ep;\lambda,\eta,u}}_{L^2\tp{\Omega_\ep}}}<\infty,\\
        \sup_{\substack{1<R<\infty\\0<\ep<\ep_2\tp{R}}}\sup_{1<\Xi<\infty}\sup_{\tp{\lambda,\eta}\in E^R_\ep}\tp{1 + R/\Xi}^{-1}\ssb{\tnorm{\mathfrak{d}_3^\Xi\tp{R,\ep;\lambda,\eta}}^{\tp{\ep,+}}_{H^{1/2}\tp{\Sigma_\ep}} + \tnorm{\mathfrak{d}_4^\Xi\tp{R,\ep;\lambda,\eta}}^{\tp{\ep,-}}_{H^{1/2}\tp{\Sigma_\ep}}}<\infty.
    \end{multline}
    Third, we have the uniform Lipschitz estimates for $j\in\tcb{1,2}$
    \begin{equation}\label{driver lippy 12}
        \sup_{\substack{1<R<\infty\\0<\ep<\ep_2\tp{R}}}\sup_{1<\Xi<\infty}\sup_{\tp{\lambda,\eta,u}\neq\tp{\tilde{\lambda},\tilde{\eta},\tilde{u}}}\f{\tnorm{\mathfrak{d}_j^\Xi\tp{R,\ep;\lambda,\eta,u} - \mathfrak{d}_j^\Xi\tp{R,\ep;\tilde{\lambda},\tilde{\eta},\tilde{u}}}_{L^2\tp{\Omega_\ep}}}{\tnorm{\lambda - \tilde{\lambda}}_{\Lambda} + C_2\tp{R}\ep^2\tnorm{\eta - \tilde{\eta}}_{W^{1,\infty}\tp{\Sigma_\ep}} + \tp{C_2\tp{R}\ep + 1/\Xi}\tnorm{u - \tilde{u}}_{H^2\tp{\Omega_\ep}}}<\infty
    \end{equation}
    and for $j\in\tcb{3,4}$, with the understanding that $s_3 = +$ and $s_4 = -$,
    \begin{equation}\label{driver lippy 34}
        \sup_{\substack{1<R<\infty\\0<\ep<\ep_2\tp{R}}}\sup_{1<\Xi<\infty}\sup_{\tp{\lambda,\eta}\neq\tp{\tilde{\lambda},\tilde{\eta}}}\f{\tnorm{\mathfrak{d}^\Xi_j\tp{R,\ep;\lambda,\eta} - \mathfrak{d}^\Xi_j\tp{R,\ep;\tilde{\lambda},\tilde{\eta}}}^{\tp{\ep,s_j}}_{H^{1/2}\tp{\Sigma_\ep}}}{\tnorm{\lambda - \tilde{\lambda}}_{\Lambda} + \tp{1/\Xi}\tnorm{\eta - \tilde{\eta}}^{\tp{\ep}}_{H^{5/2}\tp{\Sigma_\ep}}}<\infty,
    \end{equation}
    where $\tnorm{\cdot}_{\Lambda}$ is as in~\eqref{the complete metric space}.
\end{propC}
\begin{proof}
We follow the same strategy used in the proof of Proposition~\ref{prop on the mapping propertieso of the residual forcing} and elect to only highlight the parts requiring some more delicate care and suppress the remaining details. Again, we handle the terms in order.

\emph{Step 1 -- The term $\mathfrak{d}_1$:} We expand~\eqref{first driver full expression} into the sum of two expressions, one that is made small by taking $\ep$ sufficiently small and the other that is made small by taking $\Xi$ sufficiently large; more precisely, we split
    \begin{equation}\label{this is a decomposition into an epsilon piece and a Xi piece}
        \mathfrak{d}^\Xi_1\tp{R,\ep;\lambda,\eta,u} = {_1}\mathfrak{d}_1\tp{R,\ep;\lambda,\eta,u} + {_2}\mathfrak{d}^\Xi_1\tp{R,\ep;\lambda,u},
    \end{equation}
    where
    \begin{multline}\label{d11}
        {_1}\mathfrak{d}_1\tp{R,\ep;\lambda,\eta,u} = \bp{\tp{U(\lambda) - 4}\bp{u_1 - \f{1}{\ep}\int_0^\ep u_1\tp{\cdot,y}\;\m{d}y}}' \\ - y\tp{\tp{U\tp{\lambda} - 4}\pmb{\zeta}\tp{\ep,\lambda,\eta}}'\pd_2^{\mathcal{A}_{\pmb{\zeta}\tp{\ep,\lambda,\eta}}}u_1 + u_1\bp{\ep y\pmb{\zeta}\tp{\ep,\lambda,\eta}U_1\tp{\lambda}' + \f12\tp{y\pmb{\zeta}\tp{\ep,\lambda,\eta}}^2U_2\tp{\lambda}'}\\
         + u_2\tp{\ep U_1\tp{\lambda} + y\pmb{\zeta}\tp{\ep,\lambda,\eta}U_2\tp{\lambda}} + \ep^2\tp{W\tp{\ep,\lambda,\eta} - \Bar{\upgamma}e_1}\cdot\grad^{\mathcal{A}_{\pmb{\zeta}\tp{\ep,\lambda,\eta}}}u_1
    \end{multline}
    and
    \begin{equation}\label{d12}
        {_2}\mathfrak{d}^\Xi_1\tp{R,\ep;\lambda,u} = \bp{\tp{U(\lambda) - 4}\bp{\f{1}{\ep}\int_0^\ep u_1\tp{\cdot,y}\;\m{d}y - \uppi_\Xi\bp{\f{1}{\ep}\int_0^\ep u_1\tp{\cdot,y}\;\m{d}y}}}'.
    \end{equation}

    To establish the  $\mathfrak{d}_1$-parts of the claimed estimates~\eqref{driver bounds} and~\eqref{driver lippy 12}, it is sufficient to prove that the maps in the decomposition~\eqref{this is a decomposition into an epsilon piece and a Xi piece} satisfy the following. There exists $N\in\tp{1,\infty}$ with the property that for $R\in\tp{1,\infty}$ there exists $0<\ep_R\le\ep_1(R)$ and there exists $C_R\ge C_1\tp{R}$ such that for all $0<\ep<\ep_R$, for all $\Xi\in\tp{1,\infty}$, and for all $\tp{\lambda,u,\eta},\tp{\tilde{\lambda},\tilde{u},\tilde{\eta}}\in F^R_\ep$ we have the bounds
    \begin{equation}\label{decomp 1 2 bounds}
        \tnorm{{_1}\mathfrak{d}_1\tp{R,\ep;\lambda,\eta,u}}_{L^2\tp{\Omega_\ep}}\le N,\quad\tnorm{{_2}\mathfrak{d}^\Xi_1\tp{\ep,R;\lambda,u}}_{L^2\tp{\Omega_\ep}}\le N\tp{R/\Xi}, 
    \end{equation}
    \begin{equation}\label{decomp 1 Lip}
        \tnorm{{_1}\mathfrak{d}_1\tp{R,\ep;\lambda,\eta,u} - {_1}\mathfrak{d}_1\tp{R,\ep;\tilde{\lambda},\tilde{\eta},\tilde{u}}}_{L^2\tp{\Omega_\ep}}\le N\tp{\tnorm{\lambda - \tilde{\lambda}}_{\Lambda} + C_R\ep^2\tnorm{\eta - \tilde{\eta}}_{W^{1,\infty}} + C_R\ep\tnorm{u - \tilde{u}}_{H^2\tp{\Omega_\ep}}},
    \end{equation}
    and
    \begin{equation}\label{decomp 2 Lip}
        \tnorm{{_2}\mathfrak{d}^\Xi_1\tp{R,\ep;\lambda,u} - {_2}\mathfrak{d}^\Xi_1\tp{R,\ep;\tilde{\lambda},\tilde{u}}}_{L^2\tp{\Omega_\ep}}\le N\tp{\tnorm{\lambda - \tilde{\lambda}}_{\Lambda} + \tp{1/\Xi}\tnorm{u - \tilde{u}}_{H^2\tp{\Omega_\ep}}}.
    \end{equation}

    Now, the estimates~\eqref{decomp 1 2 bounds}, \eqref{decomp 1 Lip}, and~\eqref{decomp 2 Lip} are mostly direct and only require minor commentary.  For the first term in~\eqref{d11} the key is to expand with the product rule and apply the average-deviation estimate from the first item of Lemma~\ref{lem on poincare and sobolev in thin domains} (with $q=2$)  to both $u_1$ and $\pd_1u_1$.  The remaining terms in ${_1}\mathfrak{d}_1$ are small due to the presence of either an $\ep$ or a $y$.  On the other hand, the key for ${_2}\mathfrak{d}_1^\Xi$ is to take advantage of the low-frequency vanishing through the bounds
    \begin{equation}
        \tnorm{f - \uppi_\Xi f}_{L^2\tp{\Omega_\ep}}\le \tp{10/\Xi}\tnorm{f}_{H^1\tp{\Omega_\ep}} \text{ and }
        \tnorm{\uppi_\Xi f}_{L^2\tp{\Omega_\ep}}\le\tnorm{f}_{L^2\tp{\Omega_\ep}}
        \text{ for }
        f\in H^1\tp{\Omega_\ep},
    \end{equation}
    which we apply to $f \in \tcb{ \f1\ep\int_0^\ep u_1\tp{\cdot,y}\;\m{d}y , \f1\ep\int_0^\ep\pd_1u_1\tp{\cdot,y}\;\m{d}y}$.

\emph{Step 2 -- The term $\mathfrak{d}_2$:} As in the proof of the previous step, we begin by expanding the expression~\eqref{d2 expression} more carefully:
    \begin{multline}\label{expansion of second driving term}
        \mathfrak{d}_2\tp{\ep,\lambda,\eta,u} =-U'\tp{\lambda}u_2 + \tp{U\tp{\lambda} - 4}\pd_1u_2 - y\tp{\tp{U\tp{\lambda - 4}\pmb{\zeta}\tp{\ep,\lambda,\eta}}'\pd^{\mathcal{A}_{\pmb{\zeta}\tp{\ep,\lambda,\eta}}}_2u_2 + \pmb{\zeta}\tp{\ep,\lambda,\eta}U(\lambda)''u_1} \\ + \ep^2\tp{\tp{W\tp{\ep,\lambda,\eta} - \Bar{\upgamma}e_1}\cdot\grad^{\mathcal{A}_{\pmb{\zeta}\tp{\ep,\lambda,\eta}}}u_2 + u\cdot\grad^{\mathcal{A}_{\pmb{\zeta}\tp{\ep,\lambda,\eta}}}W_2\tp{\ep,\lambda,\eta}}.
    \end{multline}
    Due to the inclusion $u\in H^2_{\m{tan}}\tp{\Omega_\ep;\R^2}$, we have that $\m{Tr}_{\Sigma_0}u_2 = \m{Tr}_{\Sigma_0}\pd_1u_2 = 0$; consequently, the first two terms in~\eqref{expansion of second driving term} are $\ep$-small according to the first item of Lemma~\ref{lem on poincare and sobolev in thin domains}.  The third grouping of terms is similarly small due to the $y$-coefficient. Finally, the last grouping of terms is small due to the coefficient $\ep^2$ and the observation that $\pd_2^{\mathcal{A}_{\pmb{\zeta}}}W_2$ is an order one contribution (see~\eqref{any suggestions for how to make this stuff nicer}). These comments and  straightforward computations then yield the desired bounds.

\emph{Step 3 -- The term $\mathfrak{d}_3$:} Recall that from Lemma~\ref{lem on adapted boundary norms} we have $\tnorm{\cdot}^{\tp{\ep,+}}_{H^{1/2}\tp{\Sigma_\ep}} = \tnorm{\cdot}_{H^{1/2}\tp{\Sigma_\ep}} + \ep^{-1/2}\tnorm{\cdot}_{L^2\tp{\Sigma_\ep}}$. The strategy here is the same as the one used on the ${_2}\mathfrak{d}_1^\Xi$-term in the proof of the first step:  low-frequency vanishing provides the bounds
   \begin{equation}
        \ep^{-1/2}\tnorm{\ep\tp{\eta - \uppi_\Xi\eta}}_{H^1\tp{\Sigma_\ep}}\le\tp{10/\Xi}\tp{\ep^{1/2}\tnorm{\eta}_{H^2\tp{\Sigma_\ep}}} 
        \text{ and }
        \tnorm{\ep\tp{\eta - \uppi_\Xi\eta}}_{H^{3/2}\tp{\Sigma_\ep}}\le\tp{10/\Xi}\tp{\ep\norm{\eta}_{H^{5/2}\tp{\Sigma_\ep}}}
    \end{equation}
    which leads to
    \begin{equation}
        \tnorm{\ep\tp{\eta - \uppi_\Xi\eta}}_{H^{1/2}\tp{\Sigma_\ep}}^{\tp{\ep,+}} + \tnorm{\ep\tp{\eta - \uppi_\Xi\eta}'}_{H^{1/2}\tp{\Sigma_\ep}}^{\tp{\ep,+}}\le\tp{20/\Xi}\tnorm{\eta}_{H^{5/2}\tp{\Sigma_\ep}}^{\tp{\ep}}.
    \end{equation}
    The $\mathfrak{d}_3$-parts of the claimed estimates~\eqref{driver bounds} and~\eqref{driver lippy 34} now follow by routine arguments.

\emph{Step 4 -- The term $\mathfrak{d}_4$:} By the definition of the boundary norm $\tnorm{\cdot}^{\tp{\ep,-}}_{H^{1/2}\tp{\Sigma_\ep}}$ from Lemma~\ref{lem on adapted boundary norms} we read off the bound $\tnorm{\cdot}^{\tp{\ep,-}}_{H^{1/2}\tp{\Sigma_\ep}}\le\ep^{1/2}\tnorm{\cdot}_{H^1\tp{\Sigma_\ep}}$. In turn, it suffices to prove the $\mathfrak{d}_4$ part of estimates~\eqref{driver bounds} and~\eqref{driver lippy 34} in this stronger norm. The result now follows by vanishing  low mode estimate
    \begin{equation}
        \ep^{1/2}\tnorm{\eta - \uppi_\Xi \eta}_{H^1\tp{\Sigma_\ep}}\le\tp{10/\Xi}\tnorm{\eta}^{\tp{\ep}}_{H^{5/2}\tp{\Sigma_\ep}}.
    \end{equation}
\end{proof}

\section{Analysis of bore waves}\label{section on analysis of bore waves}

\subsection{Existence}\label{subsection on existence}

In this subsection we prove the existence of bore wave solutions to the free boundary Navier-Stokes system~\eqref{FBINSE, param. tuned and funny flux}. The general strategy is to go though the solution ansatz introduced in Theorem~\ref{thm on the shallow water and residual equations}, which reformulates the problem in terms of solving a system of ODEs, whose solutions describe the leading order behavior, coupled to a system of PDEs for the residuals.

The careful estimates of Sections \ref{subsection on the stokes equations with stress boundary conditions}, \ref{subsection on the shallow water bore map}, and~\ref{subsection on properties of residual forcing and nonlinearities} get us well along the way to finding the desired solutions via a contraction mapping argument. This is the ultimate goal of this subsection, but we first require a preliminary study of the interaction between the linear part of the residual PDEs~\eqref{first form of the residual PDEs} and the forcing and driving terms already studied in Section~\ref{subsection on properties of residual forcing and nonlinearities}.

As is done in Section~\ref{subsection on properties of residual forcing and nonlinearities}, we will take the physical parameters and functionalization of the ansatz constituents, residual forcing terms, and residual driving terms to be instantiated by Definitions~\ref{defn big steppah} and~\ref{defn underground methods}. In particular, $\be\in\tp{0,1}$ is the size of the domain ball in~\eqref{ansatz, but now precise}, and $\ep_0:\tp{1,\infty}\to\tp{0,1}$ is the initial smallness threshold defined prior to Definition~\ref{defn big steppah}.

Our first result considers all but the final equation in~\eqref{first form of the residual PDEs}. We use the thin domain Stokes solution operator constructed in Section~\ref{subsection on the stokes equations with stress boundary conditions} to reformulate these equations as a fixed point identity for the velocity.  We then record estimates associated with this new formulation. This is one of the key places where the extra half of a derivative gain coming from the reformulation of the conservation of relative velocity flux equation is used (compare the final equations in systems~\eqref{flattened free boundary Navier-Stokes system} and~\eqref{FBINSE, param. tuned and funny flux}). Recall the definition of the norms~\eqref{epsilon dependent norms for the free surface residual} and the space~\eqref{tangent to the bottom vector fields}.

\begin{propC}[Velocity-pressure fixed point reformulation maps]\label{prop on velocity - pressure fixed point reformulation}
    Assume the hypotheses of Definition~\ref{defn big steppah}. For $R,\Xi\in\tp{1,\infty}$ and $0<\ep<\ep_0\tp{R}$ there exist functions $\mathfrak{R}^\Xi_1\tp{R,\ep;\cdot}:F^R_\ep\to H^2_{\m{tan}}\tp{\Omega_\ep;\R^2}$, $\mathfrak{R}^\Xi_2\tp{R,\ep;\cdot}:F^R_\ep\to H^1\tp{\Omega_\ep}$, $\ep_3:\tp{1,\infty}\to\tp{0,1}$, and $C_3:\tp{1,\infty}\to\tp{1,\infty}$, with $\ep_3\le\min\tcb{\ep_1,\ep_2}$, where $F^R_\ep$ is defined in~\eqref{theses sets prevent illegal curry} and the pair of functions $\ep_1$ and $\ep_2$ are from Propositions~\ref{prop on the mapping propertieso of the residual forcing} and~\ref{prop on mapping properties of the residual driving terms} (respectively), such that the following are satisfied.
    \begin{enumerate}
        \item For any $\tp{\lambda,\eta,u}\in F^R_\ep$ upon setting
        \begin{multline}\label{this is the definition of the R XI operators}
            \widehat{u} = \mathfrak{R}^{\Xi}_1\tp{R,\ep;\lambda,\eta,u},\quad \widehat{p} = \mathfrak{R}_2^\Xi\tp{R,\ep;\lambda,\eta,u},\quad h = \pmb{\zeta}\tp{\ep,\lambda,\eta},\\
            \varphi^1 = -\tp{\mathfrak{f}_1\tp{R,\ep;\lambda,\eta,u} + \mathfrak{d}^\Xi_1\tp{R,\ep;\lambda,\eta,u}, \mathfrak{f}_2\tp{R,\ep;\lambda,\eta,u} + \mathfrak{d}_2^\Xi\tp{R,\ep;\lambda,\eta,u}}, \text{ and }
            \\\psi^1 = -\tp{\mathfrak{f}_3\tp{R,\ep;\lambda,\eta} + \mathfrak{d}_3^\Xi\tp{R,\ep;\lambda,\eta},\mathfrak{f}_4\tp{R,\ep;\lambda,\eta} + \mathfrak{d}_4^\Xi\tp{R,\ep;\lambda,\eta}}
        \end{multline}
        we have that $L^\ep_h\tp{\widehat{u},\widehat{p}} = \tp{\varphi^1,0,\psi^1,0}$, where $L^\ep_h$ is the operator from Proposition~\ref{prop on theory of strong solutions} and $\pmb{\zeta}$, $\mathfrak{f}_1$, $\mathfrak{f}_2$, $\mathfrak{f}_3$, $\mathfrak{f}_4$, $\mathfrak{d}_1^\Xi$, $\mathfrak{d}_2^\Xi$, $\mathfrak{d}_3^\Xi$, and $\mathfrak{d}_4^\Xi$ are from Definitions~\ref{defn big steppah} and~\ref{defn underground methods}.
        \item We have the uniform range estimates:
        \begin{equation}\label{R1 R2 bound}
            \sup_{\substack{1<R<\infty\\0<\ep<\ep_3\tp{R}}}\sup_{1<\Xi<\infty}\sup_{\tp{\lambda,\eta,u}\in F^R_\ep}\tp{1 + R/\Xi}^{-1}\tnorm{\tp{\mathfrak{R}^\Xi_1,\mathfrak{R}^\Xi_2}\tp{R,\ep;\lambda,\eta,u}}_{\bf{X}^1_\ep}<\infty,
        \end{equation}
        and the uniform Lipschitz estimates:
        \begin{equation}\label{R1 R2 Lip}
            \sup_{\substack{1<R<\infty\\0<\ep<\ep_3\tp{R}}}\sup_{1<\Xi<\infty}\sup_{\tp{\lambda,\eta,u}\neq\tp{\tilde{\lambda},\tilde{\eta},\tilde{u}}}\f{\tnorm{\tp{\mathfrak{R}^\Xi_1,\mathfrak{R}^\Xi_2}\tp{R,\ep;\lambda,\eta,u} - \tp{\mathfrak{R}^\Xi_1,\mathfrak{R}^\Xi_2}\tp{R,\ep;\tilde{\lambda},\tilde{\eta},\tilde{u}}}_{\bf{X}^1_\ep}}{\tnorm{\lambda - \tilde{\lambda}}_{\Lambda} + \tp{C_3\tp{R}\ep + 1/\Xi}\tp{\tnorm{\eta - \tilde{\eta}}^{\tp{\ep}}_{H^{5/2}\tp{\Sigma_\ep}} + \tnorm{u - \tilde{u}}_{H^2\tp{\Omega_\ep}}}}<\infty,
        \end{equation}
        where $\tnorm{\cdot}_{\Lambda}$ is as in~\eqref{the complete metric space}.
    \end{enumerate}
\end{propC}
\begin{proof}
    
    Equation~\eqref{this is the definition of the R XI operators} tells us how to define $\mathfrak{R}^\Xi_1$ and $\mathfrak{R}^\Xi_2$, provided we can invert $L^\ep_h$, namely $\tp{\widehat{u},\widehat{p}} = \tp{L^\ep_h}^{-1}\tp{\varphi^1,0,\psi^1,0}$.  To check invertibility we note that Corollary~\ref{coro on bore map III}, paired with the definition of $\ep_0:\tp{1,\infty}\to\tp{0,1}$, provide a constant $\del_0>0$, depending only on $\m{H}_-$ and the number $M$ from the aforementioned corollary, such that for all $0<\ep<\ep_0\tp{R}$, we have $h = \pmb{\zeta}\tp{\ep,\lambda,\eta}\in Z_{\del_0}$, where the latter set is defined in~\eqref{in this set we have Korn}; Theorem~\ref{thm on synthesis of linear analysis} thus applies, allowing us to  work in conjunction with Definitions~\ref{defn big steppah} and~\ref{defn underground methods} to define the maps in this way.

    The operator bound~\eqref{R1 R2 bound} is then a consequence of the left hand estimate in~\eqref{estimates on the inverse stokes operator} coupled to the bounds~\eqref{residual forcing bounds} and~\eqref{driver bounds} of Propositions~\ref{prop on the mapping propertieso of the residual forcing} and~\ref{prop on mapping properties of the residual driving terms}.

    The Lipschitz continuity bound~\eqref{R1 R2 Lip} is a result of the right hand estimate in~\eqref{estimates on the inverse stokes operator} paired with the difference estimates~\eqref{lipschitz 034__}, \eqref{lipscthiz 12__}, \eqref{driver lippy 12}, and~\eqref{driver lippy 34} coming from the aforementioned propositions.
\end{proof}

Our second result now studies the final equation in system~\eqref{first form of the residual PDEs}. By using the first term, we may solve for the residual free surface in terms of the velocity up to some remainder.  In what follows we implement this idea and prove some important estimates.

\begin{propC}[Free surface fixed point reformulation map]\label{prop on free surface fixed point reformulation}
    Assume the hypotheses of Definition~\ref{defn big steppah}. For $R,\Xi\in\tp{1,\infty}$ and $0<\ep<\ep_0\tp{R}$ there exists functions $\mathfrak{R}_0\tp{R,\ep;\cdot}:F^R_\ep\to H^{5/2}\tp{\Sigma_\ep}$, $\ep_4:\tp{1,\infty}\to\tp{0,1}$, and $C_4:\tp{1,\infty}\to\tp{1,\infty}$, with $\ep_4\le\ep_1$, where $F^R_\ep$ is defined in~\eqref{theses sets prevent illegal curry} and $\ep_1$ is from Proposition~\ref{prop on the mapping propertieso of the residual forcing}, such that the following are satisfied.
    \begin{enumerate}
        \item For any $\tp{\lambda,\eta,u}\in F^R_\ep$, upon setting $\widehat{\eta} = \mathfrak{R}_0\tp{R,\ep;\lambda,\eta,u}$ we have the identity 
        \begin{equation}\label{the expression for the free surface fixed point reformulation}
            \tp{1 - \pd_1^2}\tp{\tp{4 + \ep^2\Bar{\upgamma} - U\tp{\lambda}}\widehat{\eta}} = \f{\pmb{\zeta}\tp{\ep,\lambda,\eta}}{\ep}\int_0^\ep u_1\tp{\cdot,y}\;\m{d}y + \f{1}{\ep}\pd_1\tp{u\cdot\mathcal{N}_{\ep\pmb{\zeta}\tp{\ep,\lambda,\eta}}} - \mathfrak{f}_0\tp{R,\ep;\lambda,\eta}
        \end{equation}
        on $\Sigma_\ep$, where $\mathfrak{f}_0$ is from~\eqref{really bro needs a label no cap}.
        \item We have the uniform range estimates:
        \begin{equation}\label{the r0 bound}
            \sup_{\substack{1<R<\infty\\0<\ep<\ep_4\tp{R}}}\sup_{\tp{\lambda,\eta,u}\in F^R_\ep}\tp{1 + \tnorm{u}_{H^2\tp{\Omega_\ep}}}^{-1}\tnorm{\mathfrak{R}_0\tp{R,\ep;\lambda,\eta,u}}_{H^{5/2}\tp{\Sigma_\ep}}^{\tp{\ep}}<\infty
        \end{equation}
        and the uniform Lipschitz estimates:
        \begin{equation}\label{the r0 Lip}
            \sup_{\substack{1<R<\infty\\0<\ep<\ep_4\tp{R}}}\sup_{\tp{\lambda,\eta,u}\neq\tp{\tilde{\lambda},\tilde{\eta},\tilde{u}}}\f{\tnorm{\mathfrak{R}_0\tp{R,\ep;\lambda,\eta,u} - \mathfrak{R}_0\tp{R,\ep;\tilde{\lambda},\tilde{\eta},\tilde{u}}}_{H^{5/2}\tp{\Sigma_\ep}}^{\tp{\ep}}}{\tnorm{\lambda - \tilde{\lambda}}_{\Lambda} + C_4\tp{R}\ep\tnorm{\eta - \tilde{\eta}}^{\tp{\ep}}_{H^{5/2}\tp{\Sigma_\ep}} + \tnorm{u - \tilde{u}}_{H^2\tp{\Omega_\ep}}}<\infty,
            \end{equation}
        where $\tnorm{\cdot}_{\Lambda}$ is as in~\eqref{the complete metric space}.
    \end{enumerate}
\end{propC}
\begin{proof}
    We begin with the observation that the number $\Bar{\upgamma}$ defined in~\eqref{the parameter tuning identities 440 hz} is strictly positive, so the fourth item of Proposition~\ref{prop on the bore map} provides a constant $1<M<\infty$ for which
    \begin{equation}\label{is this too many inf}
        \inf_{0<\ep<1}\inf_{\lambda\in\Bar{B(0,\be)}}\inf_{\R}\tp{4 + \ep^2\Bar{\upgamma} - U(\lambda)}\ge 1/M.
    \end{equation}
    With this bound in mind, the expression~\eqref{the expression for the free surface fixed point reformulation} suggests how we must define the map $\mathfrak{R}_0$; namely, for $1<R<\infty$, $0<\ep<\ep_0\tp{R}$, and $\tp{\lambda,\eta,u}\in F^R_\ep$ we set
    \begin{equation}\label{this is how we plan on defining this operator, just so you know}
        \mathfrak{R}_0\tp{R,\ep;\lambda,\eta,u} = \f{1}{4 + \ep^2\Bar{\upgamma} - U(\lambda)}\tp{1 - \pd_1^2}^{-1}\bp{\f{\pmb{\zeta}\tp{\ep,\lambda,\eta}}{\ep}\int_0^\ep u_1\tp{\cdot,y}\;\m{d}y + \f{1}{\ep}\pd_1\tp{u\cdot\mathcal{N}_{\ep\pmb{\zeta}\tp{\ep,\lambda,\eta}}} - \mathfrak{f}_0\tp{\ep,\lambda,\eta}}.
    \end{equation}

    Now the proof of the estimates~\eqref{the r0 bound} and~\eqref{the r0 Lip}  follows from routine computations, as soon as we point out the key observations.  Firstly, we recall that the norm in question, $\tnorm{\cdot}^{\tp{\ep}}_{H^{5/2}\tp{\Sigma_\ep}}$, is defined in~\eqref{epsilon dependent norms for the free surface residual}. It asks that we take the sum of $\ep^{1/2}$ multiplying the $H^2\tp{\Sigma_\ep}$-norm and $\ep$ multiplying the $H^{5/2}\tp{\Sigma_\ep}$ norm. In either case (thanks in part to~\eqref{is this too many inf}), we can handle the $\tp{4 + \ep^2\Bar{\upgamma} - U(\lambda)}^{-1}$ as a multiplier and take it in $W^{3,\infty}\tp{\Sigma_\ep}$ while implementing Corollary~\ref{coro on bore map III}.
    
    In this way we can reduce to studying the argument of $\tp{1 - \pd_1^2}^{-1}$ in~\eqref{this is how we plan on defining this operator, just so you know} in the norm $\ep^{1/2}\tnorm{\cdot}_{L^2\tp{\Sigma_\ep}} + \ep\tnorm{\cdot}_{H^{1/2}\tp{\Sigma_\ep}}$. For the $\mathfrak{f}_0$ term we can crudely go up with the $H^{1/2}\tp{\Sigma_\ep}$-norm and employ estimates~\eqref{residual forcing bounds} and~\eqref{lipschitz 0__} from Proposition~\ref{prop on the mapping propertieso of the residual forcing}. Therefore, we can reduce to studying the $u$-dependent terms of~\eqref{this is how we plan on defining this operator, just so you know}. For the vertical average term, we only need to note the bound
    \begin{equation}\label{_KIBBLE_}
        \ep^{1/2}\bnorm{\f{1}{\ep}\int_0^\ep u_1\tp{\cdot,y}\;\m{d}y}_{L^2\tp{\R}} + \ep\bnorm{\f{1}{\ep}\int_0^\ep u_1\tp{\cdot,y}\;\m{d}y}_{H^{1/2}\tp{\R}}\lesssim\tnorm{u}_{H^1\tp{\Omega_\ep}},
    \end{equation}
    which follows directly from Cauchy-Schwarz. For the remaining term we note that it expands as
    \begin{equation}\label{almost done proof reading}
        \pd_1\tp{u\cdot\mathcal{N}_{\ep\pmb{\zeta}\tp{\ep,\lambda,\eta}}}/\ep = \tp{\pd_1 u_2}/\ep - \pd_1\tp{\pmb{\zeta}\tp{\ep,\lambda,\eta}u_1}.
    \end{equation}
    The second term in this decomposition is straightforward to estimate using the second item of Lemma~\ref{lem on traces in thin domains}. For the first term on the right hand side of~\eqref{almost done proof reading}, we use the estimates
    \begin{equation}\label{yet another nash moser implicit function theorem_}
        \ep^{1/2}\bnorm{\f1\ep\pd_1u_2}_{L^2\tp{\Sigma_\ep}} + \ep\bnorm{\f1\ep\pd_1u_2}_{H^{1/2}\tp{\Sigma_\ep}}\lesssim\tnorm{u}_{H^2\tp{\Omega_\ep}},
    \end{equation}
    which are a consequence of the first item of Lemma~\ref{lem on traces in thin domains} and $\m{Tr}_{\Sigma_0}\pd_1u_2 = 0$. The implicit constants in~\eqref{_KIBBLE_} and~\eqref{yet another nash moser implicit function theorem_} are independent of $u$ and $\ep$.
\end{proof}

We are, at long last, ready to produce solution to the coupled ODE - PDE system from the fifth and sixth items of Theorem~\ref{thm on the shallow water and residual equations} via a contraction mapping argument. The Banach space in which we will run the argument is given by
\begin{equation}\label{the space in which we run the contraction mapping argument}
    \bf{W}_\ep = H^{5/2}\tp{\Sigma_\ep}\times H^{2}_{\m{tan}}\tp{\Omega_\ep;\R^2} \text{ with norm }
    \tnorm{\eta,u}_{\bf{W}_\ep} = \max\tcb{\tnorm{\eta}^{\tp{\ep}}_{H^{5/2}\tp{\Sigma_\ep}}, \tnorm{u}_{H^2\tp{\Omega_\ep}}}.
\end{equation}
Also, recall the definition of the mollification operators $\uppi_\Xi$ given before Proposition~\ref{prop on the specification of the coupling}. Notice that in the following result we are taking the smoothing rate $\Xi = \tabs{\log\ep}$. We are also finally linking the parameter $\lambda\in\Bar{B(0,\be)}\subset\tsb{\tp{C^{10}\cap W^{10,\infty}\cap H^{10}}\tp{\R}}^2$ back to the residual in the form promised by Proposition~\ref{prop on the specification of the coupling}.
\begin{propC}[Existence of a fixed point]\label{prop on the contraction}
    There exists $R_\star\in\tp{1,\infty}$ and $0<\epsilon_\star<\min\tcb{1/3,\ep_0\tp{R_\star}}$ for which the following hold for all $\ep\in\tp{0,\epsilon_\star}$.
    \begin{enumerate}
        \item The following maps are well-defined: the coupling function $\pmb{\uplambda}\tp{\ep;\cdot}:\Bar{B_{\bf{W}_\ep}\tp{0,R_{\star}}}\to\Bar{B(0,\be)}$ defined via
        \begin{equation}\label{lambda update map}
            \pmb{\uplambda}\tp{\ep;\eta,u} = \ep^2\bp{\uppi_{\tabs{\log\ep}}\eta,\uppi_{\tabs{\log\ep}}\f{1}{\ep}\int_0^\ep u_1\tp{\cdot,y}\;\m{d}y},
        \end{equation}
        the velocity update function $\pmb{\mathsf{u}}\tp{\ep;\cdot}:\Bar{B_{\bf{W}_\ep}\tp{0,R_\star}}\to\tcb{u\in H^2_{\m{tan}}\tp{\Omega_\ep;\R^2}\;:\;\tnorm{u}_{H^2\tp{\Omega_\ep}}\le\sqrt{R_\star/2}}$ defined via
        \begin{equation}\label{u update map}
            \pmb{\mathsf{u}}\tp{\ep;\eta,u} = \mathfrak{R}^{\tabs{\log\ep}}_1\tp{R_\star,\ep;\pmb{\uplambda}\tp{\ep;\eta,u},\eta,u},
        \end{equation}
        where $\mathfrak{R}_1$ is defined in Propositions~\ref{prop on velocity - pressure fixed point reformulation};     
        and the free surface update function $\pmb{\upeta}\tp{\ep;\cdot}:\Bar{B_{\bf{W}_\ep}\tp{0,R_\star}}\to\tcb{\eta\in H^{5/2}\tp{\Sigma_\ep}\;:\;\tnorm{\eta}^{\tp{\ep}}_{H^{5/2}\tp{\Sigma_\ep}}\le R_\star/2}$ defined via
        \begin{equation}\label{eta update map}
            \pmb{\upeta}\tp{\ep;\eta,u}=\mathfrak{R}_0\tp{R_\star,\ep;\pmb{\uplambda}\tp{\ep;\eta,u},\eta,\pmb{\mathsf{u}}\tp{\ep;\eta,u}},
        \end{equation}
        where $\mathfrak{R}_0$ is defined in Proposition~\ref{prop on free surface fixed point reformulation}. In particular, the update map $\pmb{\mathsf{F}}\tp{\ep;\cdot}:\Bar{B_{\bf{W}_\ep}(0,R_\star)}\to\Bar{B_{\bf{W}_\ep}\tp{0,R_\star/2}}$ defined via
        \begin{equation}\label{the update map}
            \pmb{\mathsf{F}}\tp{\ep;\eta,u} = \tp{\pmb{\upeta}\tp{\ep;\eta,u},\pmb{\mathsf{u}}\tp{\ep;\eta,u}}
        \end{equation}
        is well-defined. 
        
        \item There exists a unique $\tp{\eta,u}\in\Bar{B_{\bf{W}_\ep}(0,R_\star)}$ obeying the fixed point identity $\tp{\eta,u} = \pmb{\mathsf{F}}\tp{\ep;\eta,u}$.
        \item  Using the fixed point $\tp{\eta,u}\in\Bar{B_{\bf{W}_\ep}(0,R_\star)}$ from the previous item, define $\lambda = \pmb{\uplambda}\tp{\ep;\eta,u}\in\Bar{B(0,\be)}$ and $p = \mathfrak{R}_2^{\tabs{\log\ep}}\tp{R_\star,\ep;\lambda,\eta,u}\in H^1\tp{\Omega_\ep}$, with $\mathfrak{R}_2$ as in Proposition~\ref{prop on velocity - pressure fixed point reformulation}. Then $(\eta,u,p)$ along with the following evaluations of the bore map constituents~\eqref{ansatz, but now precise}
        \begin{multline}\label{cotton candy grapes}
            H = H(\lambda),\;U = U(\lambda),\;U_1 = U_1(\lambda),\;U_2 = U_2(\lambda),\;P = P(\lambda),\;P_1 = P_1\tp{\lambda},\;P_2 = P_2\tp{\lambda},\\\mathfrak{r}_1 = \bp{\tp{U - 4}\uppi_{\tabs{\log\ep}}\f{1}{\ep}\int_0^\ep u_1\tp{\cdot,y}\;\m{d}y}',\; \mathfrak{r}_3 = \ep\mathfrak{L}\uppi_{\tabs{\log\ep}}\eta,\;\mathfrak{r}_4 = \m{g}\uppi_{\tabs{\log\ep}}\eta
        \end{multline}
        satisfy the first and third through sixth items of Theorem~\ref{thm on the shallow water and residual equations};  in particular, these are a solution to the system of equations given by~\eqref{_SWE_0}, \eqref{_SWE_1}, \eqref{_SWE_2}, and~\eqref{first form of the residual PDEs} with $\inf\pmb{\zeta}>0$ (where $\pmb{\zeta} = H + \ep^2\eta$).
    \end{enumerate}
\end{propC}
\begin{proof}
    We begin with the construction of the maps~\eqref{lambda update map}, \eqref{u update map}, and~\eqref{eta update map}. First, we note that the following Bernstein-type estimate holds:
    \begin{equation}\label{crazy bernstein estimate}
        \sup_{0<\ep<1/3}\sup_{\tp{\eta,u}\in\bf{W}_\ep\setminus\tcb{0}}\f{\snorm{\ep^2\sp{\uppi_{\tabs{\log\ep}}\eta,\uppi_{\tabs{\log\ep}}\f{1}{\ep}\int_0^\ep u_1\tp{\cdot,y}\;\m{d}y}}_{\Lambda}}{\ep^{3/2}\tabs{\log\ep}^9\tnorm{\eta,u}_{\bf{W}_\ep}}<\infty.
    \end{equation}
    The exponent $9$ here is perhaps suboptimal, but sufficient to deduce the existence of a function $\ep_5:\tp{1,\infty}\to\tp{0,1/3}$ with the property that 
    \begin{equation}\label{application of the crude bernstein estimate}
        \sup_{\substack{1<R<\infty\\0<\ep<\ep_5\tp{R}}}\sup_{\substack{\tp{\eta,u}\in\bf{W}_\ep\\\tnorm{\eta,u}_{\bf{W}_\ep}\le R}}\bnorm{\ep^2\bp{\uppi_{\tabs{\log\ep}}\eta,\uppi_{\tabs{\log\ep}}\f{1}{\ep}\int_0^\ep u_1\tp{\cdot,y}\;\m{d}y}}_{\Lambda}\le\f{\be}{2},
    \end{equation}
     where $\beta >0$ is as in Definition~\ref{defn big steppah}.  Estimate~\eqref{application of the crude bernstein estimate} then permits us to define the coupling map $\pmb{\uplambda}\tp{\ep;\cdot}:\Bar{B_{\bf{W}_\ep}\tp{0,R}}\to\Bar{B(0,\be)}$ via the formula~\eqref{lambda update map} for any $1<R<\infty$ and any $0<\ep<\ep_5\tp{R}$.

    The next step is to unpack estimates~\eqref{R1 R2 bound} and~\eqref{the r0 bound} from Propositions~\ref{prop on velocity - pressure fixed point reformulation} and~\ref{prop on free surface fixed point reformulation}. If we let $\ep_6:\tp{1,\infty}\to\tp{0,1/3}$ be defined via $\ep_6 = \min\tcb{\ep_3,\ep_4,\ep_5}$, then the aforementioned estimates provide a constant $N\in\tp{1,\infty}$ with the property that for all $1<R<\infty$,  $0<\ep<\ep_6\tp{R}$, and $\tp{\lambda,\eta,u}\in F^R_\ep$ (where the latter set is defined in~\eqref{theses sets prevent illegal curry}) we have the estimates
    \begin{equation}\label{the unpacking of the estimates}
        \tnorm{\mathfrak{R}_1^{\tabs{\log\ep}}\tp{R,\ep;\lambda,\eta,u}}_{H^2\tp{\Omega_\ep}}\le N\tp{1 + R/\tabs{\log\ep}}\text{ and }\tnorm{\mathfrak{R}_0\tp{R,\ep;\lambda,\eta,u}}_{H^{5/2}\tp{\Sigma_\ep}}^{\tp{\ep}}\le N\tp{1 + \tnorm{u}_{H^2\tp{\Omega_\ep}}}.
    \end{equation}
    Estimate~\eqref{the unpacking of the estimates} suggests that we define $R_{\star} = 8N^2$ and $\tilde{\ep}_\star\in\tp{0,1/3}$ via $\tilde{\ep}_\star = \min\tcb{\ep_6\tp{R_\star},e^{-R_{\star}}}$. Then for all $0<\ep<\tilde{\ep}_\star$ and all $\tp{\lambda,\eta,u}\in F^{R_\star}_\ep$ we deduce the bounds
    \begin{equation}\label{bounds to mappa ball ina ball brutha}
        \tnorm{\mathfrak{R}_1^{\tabs{\log\ep}}\tp{R_\star,\ep;\lambda,\eta,u}}_{H^2\tp{\Omega_\ep}}\le\sqrt{R_\star/2} \text{ and }
        \tnorm{\mathfrak{R}_0\tp{R_\star,\ep;\lambda,\eta,\mathfrak{R}^{\tabs{\log\ep}}_1\tp{R_\star,\ep;\lambda,u,\eta}}}^{\tp{\ep}}_{H^{5/2}\tp{\Sigma_\ep}}\le R_\star/2.
    \end{equation}
    Through the combination of estimates~\eqref{application of the crude bernstein estimate} and~\eqref{bounds to mappa ball ina ball brutha} we deduce that~\eqref{u update map}, \eqref{eta update map}, and~\eqref{the update map} are well-defined functions mapping between the stated spaces for $0<\ep<\tilde{\ep}_\star$. This completes the proof of the first item.

    We next prove the second item. By unpacking estimates~\eqref{R1 R2 Lip} and~\eqref{crazy bernstein estimate} and using the linearity of $\pmb{\uplambda}(\ep;\cdot)$, we find a constant $M_1\in\tp{1,\infty}$ with the property that for all $0<\ep<\tilde{\ep}_\star$ and all $\tp{\eta,u},\tp{\tilde{\eta},\tilde{u}}\in\Bar{B_{\bf{W}_\ep}(0,R_\star)}$ we have the velocity update map Lipschitz estimate
    \begin{equation}\label{u update lippy}
        \tnorm{\pmb{\mathsf{u}}\tp{\ep;\eta,u} - \pmb{\mathsf{u}}\tp{\ep;\tilde{\eta},\tilde{u}}}_{H^2\tp{\Omega_\ep}}\le M_1\tp{\ep^{3/2}\tabs{\log\ep}^9 + 1/\tabs{\log\ep}}\tnorm{\eta - \tilde{\eta},u - \tilde{u}}_{\bf{W}_\ep}.
    \end{equation}
    We then combine the above considerations, estimate~\eqref{u update lippy}, and an unpacking of estimate~\eqref{the r0 Lip} in order to be  granted the existence of another constant $M_0\in\tp{1,\infty}$ with the property that for all $0<\ep<\tilde{\ep}_\star$ and all $\tp{\eta,u},\tp{\tilde{\eta},\tilde{u}}\in\Bar{B_{\bf{W}_\ep}\tp{0,R_{\star}}}$ we have the free surface update map Lipschitz estimate
    \begin{equation}\label{free surface update lippy}
        \tnorm{\pmb{\upeta}\tp{\ep;\eta,u} - \pmb{\upeta}\tp{\ep;\tilde{\eta},\tilde{u}}}^{\tp{\ep}}_{H^{5/2}\tp{\Sigma_\ep}}\le M_0\tp{\ep^{3/2}\tabs{\log\ep}^9 + 1/\tabs{\log\ep}}\tnorm{\eta - \tilde{\eta},u - \tilde{\eta}}_{\bf{W}_\ep}.
    \end{equation}

    Upon synthesizing estimates~\eqref{u update lippy} and~\eqref{free surface update lippy}, we find the existence of $0<\epsilon_\star<\tilde{\ep}_\star$ such that for all $0<\ep<\epsilon_\star$ the update map~\eqref{the update map} is a contraction on the complete metric space $\Bar{B_{\bf{W}_\ep}(0,R_\star)}$. The second item now follows from Banach's fixed point theorem.

    To prove the third item we must unpack the meaning of the fixed point found in the second item.  Examining the second component of~\eqref{the update map}, we see that the velocity part $u$ of the fixed point satisfies (by~\eqref{u update map})
    \begin{equation}\label{u fixed}
        u = \pmb{\mathsf{u}}\tp{\ep;\eta,u} = \mathfrak{R}^{\tabs{\log\ep}}_1\tp{R_\star,\ep;\pmb{\uplambda}\tp{\ep;\eta,u},\eta,u}.
    \end{equation}
    Extracting the first component of~\eqref{the update map} (given by~\eqref{eta update map}) and substituting~\eqref{u fixed}, we then find
    \begin{equation}\label{eta fixed}
        \eta = \pmb{\upeta}\tp{\ep;\eta,u} = \mathfrak{R}_0\tp{R_\star,\ep;\pmb{\uplambda}\tp{\ep;\eta,u},\eta,u}.
    \end{equation}
    We then define $\lambda = \pmb{\uplambda}\tp{\ep;\eta,u}\in\Bar{B(0,\be)}$, $p = \mathfrak{R}^{\tabs{\log\ep}}_2\tp{R_\star,\ep;\lambda,\eta,u}\in H^1\tp{\Omega_\ep}$, and $H, U, U_1, U_2, P, P_1, P_2\in W^{\infty,\infty}\tp{\R}$ via the evaluation of the bore maps~\eqref{ansatz, but now precise} at $\lambda$ (here we are implicitly using the smoothness assertion of the third item of Proposition~\ref{prop on the bore map}). Also set $\mathfrak{r}_1,\mathfrak{r}_3,\mathfrak{r}_4\in H^\infty\tp{\R}$ as in~\eqref{cotton candy grapes}, where $\mathfrak{L}$ is the linear operator defined in~\eqref{adversarial linear operator}, and $\pmb{\zeta} = H + \ep^2\eta$. Note that by Corollary~\ref{coro on bore map III}, it follows that $\inf_{\R}\pmb{\zeta}\ge\f12\m{H}_->0$.

    The first item of Proposition~\ref{prop on velocity - pressure fixed point reformulation} tells us that $L^\ep_{\pmb{\zeta}}\tp{u,p} = \tp{\varphi^1,0,\psi^1,0}$ where $\varphi^1$ and $\psi^1$ are defined as in~\eqref{this is the definition of the R XI operators}, $\lambda$ is as above, and $\Xi = \tabs{\log\ep}$. By using the expressions~\eqref{first driver full expression}, \eqref{d2 expression}, \eqref{d3 expression}, and~\eqref{d4 expression}, we then compute
    \begin{equation}
        \mathfrak{d}_1^\Xi\tp{\ep,\lambda,\eta,u} = \tp{X - \tp{4 + \ep^2\Bar{\upgamma}}e_1}\cdot\grad^{\mathcal{A}_{\pmb{\zeta}}}u_1 + u\cdot\grad^{\mathcal{A}_{\pmb{\zeta}}}X_1 - \mathfrak{r}_1,
    \end{equation}
    \begin{equation}
        \mathfrak{d}_2\tp{\ep,\lambda,\eta,u} = \tp{X - \tp{4 + \ep^2\Bar{\upgamma}}e_1}\cdot\grad^{\mathcal{A}_{\pmb{\zeta}}}u_2 + u\cdot\grad^{\mathcal{A}_{\pmb{\zeta}}}X_2,
    \end{equation}
    \begin{equation}
        \mathfrak{d}^\Xi_3\tp{\ep,\lambda,\eta} = \ep\mathfrak{L}\eta - \mathfrak{r}_3,
    \end{equation}
    and
    \begin{equation}
        \mathfrak{d}^\Xi_4\tp{\lambda} = \m{g}\eta - \mathfrak{r}_4,
    \end{equation}
    where $X = V + \ep^2 W$ and $V$, $W$, $Q$ are defined in terms of $\pmb{\zeta}$, $U$, $U_1$, $U_2$, $P$, $P_1$, $P_2$ and $\lambda$ as in~\eqref{any suggestions for how to make this stuff nicer}. Hence, from $L^\ep_{\pmb{\zeta}}\tp{u,p} = \tp{\varphi^1,0,\psi^1,0}$ (and the definition~\eqref{definition of the linear differential operator}) we deduce that $(u,p)$ solve all but the last equation in system~\eqref{first form of the residual PDEs} where $\mathfrak{f}_i$, for $i\in\tcb{1,2,3,4}$, are the appropriate evaluations of the maps~\eqref{I am sorry for how ugly this is}, \eqref{ugly expression of f2}, \eqref{country roads take me home}, and~\eqref{please sir may I have a label}, which coincide with the residual forcing terms~\eqref{residual forcing no 1}, \eqref{residual forcing no 2}, \eqref{residual forcing no 3}, and~\eqref{residual forcing no 4} considered in the fourth item of Theorem~\ref{thm on the shallow water and residual equations}.

    To see that $\eta$ solves the final equation in system~\eqref{first form of the residual PDEs}, we turn to the first item of Proposition~\ref{prop on free surface fixed point reformulation}, which provides the desired identity thanks to~\eqref{eta fixed}. Hence, we conclude that $\tp{\eta,u,p}\in H^{5/2}\tp{\Sigma_\ep}\times H^2\tp{\Omega_\ep;\R^2}\times H^1\tp{\Omega_\ep}$ is a solution to the residual PDEs~\eqref{first form of the residual PDEs}. On the other hand, the ODEs~\eqref{_SWE_0}, \eqref{_SWE_1}, and~\eqref{_SWE_2} are satisfied thanks to the fifth item of Proposition~\ref{prop on the bore map}.
\end{proof}

We can recast the previous result as granting us the existence of bore wave solutions to the free boundary Navier-Stokes system~\eqref{FBINSE, param. tuned and funny flux}. See Figure~\ref{the 4 bores plots} for a graphical representation of the free surface component of some of these traveling these bore waves.
\begin{thmC}[Existence of bore waves]\label{thm on existence of bore waves}
    Let $\upmu,\m{a}>0$, $\iota\in\tcb{-1,1}$, $\tp{\m{g},\m{A}}\in\mathfrak{C}_\iota$, $\upsigma\ge0$. There exists $\mathsf{H}\in W^{\infty,\infty}\tp{\R}$, $C,\del\in\R^+$, and $\epsilon_\star\in\tp{0,1/3}$ with the property that for all $\ep\in\tp{0,\epsilon_\star}$, there exists $\tp{\Bar{\eta},\Bar{u},\Bar{p}}\in H^{5/2}\tp{\Sigma_\ep}\times H^2\tp{\Omega_\ep;\R^2}\times H^1\tp{\Omega_\ep}$ such that the following hold.
    \begin{enumerate}
        \item We have $\tnorm{\Bar{\eta}}_{H^{5/2}\tp{\Sigma_\ep}} + \tnorm{\Bar{u}}_{\tp{L^\infty\cap H^2}\tp{\Omega_\ep}} + \tnorm{\Bar{p}}_{H^1\tp{\Omega_\ep}}\le C\ep$.
        \item The function $\mathsf{H}$ is a solution to the one-dimensional shallow water equation~\eqref{shallow water just H version} and we have $\inf\mathsf{H} = \m{H}_->0$, $\inf\tp{\mathsf{H} + \Bar{\eta}}\ge\f{1}{2}\m{H}_->0$, and $\lim_{x\to\pm\infty}\mathsf{H}\tp{\iota x} = \lim_{x\to\pm\infty}\tp{\mathsf{H} + \Bar{\eta}}\tp{\iota x} = \m{H}_{\pm}$.
        \item Upon  defining $\pmb{\zeta}$, $\pmb{u}$, and $\pmb{p}$ as in~\eqref{isolation of just the distinguished part of the solution},
        we have that $\tp{\pmb{\zeta},\pmb{u},\pmb{p}}\in H^{5/2}_{\loc}\tp{\Sigma_\ep}\times H^2_\loc\tp{\Omega_\ep;\R^2}\times H^1_\loc\tp{\Omega_\ep}$ is a strong solution to the free boundary Navier-Stokes system~\eqref{FBINSE, param. tuned and funny flux}, moreover $\inf\tp{4 - \pmb{u}_1}\ge \del$.
    \end{enumerate}
\end{thmC}
\begin{proof}
    Let $\epsilon_\star\in\tp{0,1/3}$ be the number granted by Proposition~\ref{prop on the contraction}, $0<\ep<\epsilon_\star$, $\tp{\eta,u,p}\in H^{5/2}\tp{\Sigma_\ep}\times H^2\tp{\Omega_\ep;\R^2}\times H^1\tp{\Omega_\ep}$ be as in the third item of Proposition~\ref{prop on the contraction}, $\lambda = \pmb{\uplambda}\tp{\ep;\eta,u}$ (for the map as in~\eqref{lambda update map}), and $H,U,U_1,U_2,P,P_1,P_2\in W^{\infty,\infty}\tp{\R}$ as in~\eqref{cotton candy grapes}.

    We combine the conclusion of the third item of Proposition~\ref{prop on the contraction} with Theorem~\ref{thm on the shallow water and residual equations} to deduce that the functions $\pmb{\zeta}$, $\pmb{u}$, and $\pmb{p}$ defined as in~\eqref{ansatz_1}, \eqref{ansatz_2}, and~\eqref{ansatz_3} are a solution to the free boundary Navier-Stokes system~\eqref{FBINSE, param. tuned and funny flux}. We then define $\Bar{\eta}$, $\Bar{u}$, and $\Bar{p}$ (which appear in~\eqref{isolation of just the distinguished part of the solution}) to be the difference between $\pmb{\zeta}$, $\pmb{u}$, and $\pmb{p}$ and their respective distinguished parts. In other words, $\Bar{\eta}$, $\Bar{u}$, and $\Bar{p}$ are defined such that~\eqref{isolation of just the distinguished part of the solution} is true.

    According to Proposition~\ref{prop on the bore map}, specifically~\eqref{you wouldnt really want to type this mess more than once_Dos}, we have decompositions of the type $Z = \mathsf{Z} + \mathfrak{Z}\tp{\lambda}$
    for the quantities $Z \in \{H,U,U_1,U_2,P,P_1,P_2\}$,  where $\mathsf{Z}$ is part of the distinguished bore solution; for instance, $Z = H$, $\mathsf{Z} = \mathsf{H}$, and $\mathfrak{Z}\tp{\lambda} = \mathfrak{h}\tp{\lambda}$.  Because of the choice of $\lambda$ in~\eqref{lambda update map}, the bounds~\eqref{crazy bernstein estimate}, and the Lipschitz bounds~\eqref{rocky mountain spotted fever_Dos} of the maps involved,  the non-distinguished part in each decomposition is small in the sense that we have the estimate 
    \begin{equation}\label{wean without power}
    \tnorm{\mathfrak{Z}\tp{\lambda}}_{W^{7,\infty}\cap H^7}\le C\ep.          
    \end{equation}
    
    Next, we aim to show that the bounds~\eqref{wean without power} imply the estimate claimed in the first item.  The free surface term is simplest: $\Bar{\eta} = \pmb{\zeta} - \mathsf{H}$, and hence
    \begin{equation}
        \tnorm{\bar{\eta}}_{H^{5/2}\tp{\Sigma_\ep}} = \tnorm{\mathfrak{h}\tp{\lambda} + \ep^2\eta}_{H^{5/2}\tp{\Sigma_\ep}}\le C\ep + \ep\tnorm{\eta}^{\tp{\ep}}_{H^{5/2}\tp{\Sigma_\ep}}\le C\ep.
    \end{equation}
    Next we handle the vertical velocity remainder, $\Bar{u}_2$, which has the expression
    \begin{equation}\label{expression for bar u2}
        \Bar{u}_2 = \pmb{u}_2 =-\bp{\tp{y\pmb{\zeta}}U' + \ep\f12\tp{y\pmb{\zeta}}^2U_1'+\f16\tp{y\pmb{\zeta}}^3U_2'} +\ep^2u_2.
    \end{equation}
    Corollary~\ref{coro on bore map III} and the third item of Lemma~\ref{lem on poincare and sobolev in thin domains} permit us to take the norm of the above expression in $\tp{L^\infty\cap H^2}\tp{\Omega_\ep}$ and verify that $\tnorm{\Bar{u}_2}_{L^2\tp{\Omega_\ep}}\le C\tp{\ep + \ep^3 + \ep^2\tnorm{u_2}_{\tp{L^\infty\cap H^2}\tp{\Omega_\ep}}}\le C\ep$.

    The horizontal velocity remainder, $\Bar{u}_1$, is somewhat more involved.  By rewriting the expression for $\pmb{u}_1$, as defined in the first item of Theorem~\ref{thm on the shallow water and residual equations}, via
    \begin{equation}\label{expression for bold u1}
        \pmb{u}_1  = \mathsf{U} + \ep \tp{y\mathsf{H}}\mathsf{U}_1 + \f12\tp{y\mathsf{H}}^2\mathsf{U}_2+\ep\tsb{\tp{y\pmb{\zeta}} - \tp{y\mathsf{H}}}\mathsf{u}_1 + \f12\tsb{\tp{y\pmb{\zeta}}^2 - \tp{y\mathsf{H}}^2}\mathsf{U}_2 +  \mathfrak{u}\tp{\lambda} + \ep\tp{y\pmb{\zeta}}\mathfrak{u}_1\tp{\lambda} + \f12\tp{y\pmb{\zeta}}^2\mathfrak{u}_2\tp{\lambda} + \ep^2u,
    \end{equation}
    we learn that
    \begin{equation}
        \bnorm{\pmb{u}_1 - \bp{\mathsf{U} + \ep\tp{y\mathsf{H}}\mathsf{U}_1 + \f12\tp{y\mathsf{H}}^2\mathsf{U}_2}}_{\tp{L^\infty\cap H^2}\tp{\Omega_\ep}}\le C\ep.
    \end{equation}
    On the other hand, by using identities~\eqref{hello there distinguished gentlemen}, we may further expand
    \begin{equation}\label{favorite va}
        \mathsf{U} + \ep\tp{y\mathsf{H}}\mathsf{U}_1 + \f12\tp{y\mathsf{H}}^2\mathsf{U}_2 = 4 - \f{\m{A}}{\m{H}} + \f{4\m{a}}{\upmu}\mathsf{H}^2\bp{\ep y - \f12y^2}  + \bsb{\f{\m{a}}{\upmu}\tp{4\mathsf{H} - \m{A} - 4\mathsf{H}^2}\bp{\ep y - \f{1}{2}y^2} + \f12\tp{y\mathsf{H}}^2\bp{\mathsf{U}'' + 4\f{\mathsf{H}'\mathsf{U}'}{\mathsf{H}}}}.
    \end{equation}
    The final term in square brackets in~\eqref{favorite va} has size at most $C\ep$ when measured in the norm $\tnorm{\cdot}_{\tp{L^\infty\cap H^2}\tp{\Omega_\ep}}$, thanks again to Corollary~\ref{coro on bore map III} and the bounds
    \begin{equation}\label{you get what you need}
        \tnorm{\ep y\tp{4\mathsf{H} = \m{A} - 4\mathsf{H}^2}}_{\tp{L^\infty\cap H^2}\tp{\Omega_\ep}} + \tnorm{y^2\tp{4\mathsf{H} - \m{A} - 4\mathsf{H}^2}}_{\tp{L^\infty\cap H^2}\tp{\Omega_\ep}}\le C\ep,
    \end{equation}
    which follow from~\eqref{sweq} and the first and second items of Proposition~\ref{prop on the bore map}. Synthesizing the above observations yields the bound
    \begin{equation}
        \tnorm{\Bar{u}_1}_{\tp{L^\infty\cap H^2}\tp{\Omega_\ep}} = \bnorm{\pmb{u}_1 - \bp{4 - \f{\m{A}}{\m{H}} + \f{4\m{a}}{\upmu}\mathsf{H}^2\bp{\ep y - \f12y^2}}}_{\tp{L^\infty\cap H^2}\tp{\Omega_\ep}}\le C\ep.
    \end{equation}

    We now estimate the pressure remainder $\Bar{p}$.  We use~\eqref{hello there distinguished gentlemen} to write
    \begin{equation}
        \Bar{p} = \pmb{p} - \bp{\m{g}\mathsf{H} - 2\upmu\m{A}\f{\mathsf{H}'}{\mathsf{H}^2}} = \mathfrak{p}\tp{\lambda} + \ep\tp{y\pmb{\zeta}}\tp{\mathsf{P}_1 + \mathsf{p}_1\tp{\lambda}} + \f12\tp{y\pmb{\zeta}}^2\tp{\mathsf{P}_2 + \mathfrak{p}_2\tp{\lambda}}.
    \end{equation}
    The estimate $\tnorm{\Bar{p}}_{H^1\tp{\Omega_\ep}}\le C\ep$ now follows by Proposition~\ref{prop on the bore map}.  This completes the proof of the first item.

    We now turn to the proof the second item.  These claims are a direct consequences of the first item of Proposition~\ref{prop on the bore map} and the calculation done at the beginning of Section~\ref{subsection on distinguished shallow water bore solutions}, which takes us from the system of ODEs~\eqref{The shallow water ODEs} to the ODE~\eqref{shallow water just H version}.

    Finally, we prove the third item.  We have already seen that $\tp{\pmb{\zeta},\pmb{u},\pmb{p}}$ belongs to the stated space and solves the desired equations.  In addition, since $\sup\tp{4 - \m{A}/\mathsf{H}}<4$ we may let $0<\del=\tp{1/2}\tp{\m{A}/\max\mathsf{H}}$ and enforce the bound $\inf\tp{4 - \pmb{u}_1}\ge\del$, by making $\epsilon_\star$ smaller if necessary.
\end{proof}
\subsection{Regularity}\label{subsection on regularity}

We now take the solutions to the free boundary Navier-Stokes system produced by Theorem~\ref{thm on existence of bore waves} and establish that they are, in fact, smooth. It is temping to think that because of the similarities between equations~\eqref{FBINSE, param. tuned and funny flux} and the elliptic system of Stokes equations, we should get  elliptic regularity for free and hence smoothness of solutions. However, while this is true in spirit, in practice some care is needed since the system of equations is quasilinear and the free surface unknown is not defined in the bulk.  Nevertheless, our presentation here follows the familiar themes of elliptic regularity and, as such, is abbreviated as much as possible.

The following definition introduces the ingredients for this regularity promotion. Note that in what follows there are some similarities with the objects of Theorem~\ref{thm on the shallow water and residual equations} and Definition~\ref{defn big steppah}.

\begin{defnC}[Objects for regularity promotion theory]\label{defn of objects for regularity promotion theory}
    Assume the hypotheses of Definition~\ref{defn big steppah} and let $\mathsf{H}\in W^{\infty,\infty}\tp{\R}$ and $\del,\epsilon_\star\in\tp{0,1/3}$ be the objects produced by Theorem~\ref{thm on existence of bore waves}. Define $\mathsf{U}, \mathsf{U}_1, \mathsf{U}_2, \mathsf{P}$, $\mathsf{P}_1, \mathsf{P}_2\in W^{\infty,\infty}\tp{\R}$ in terms of $\mathsf{H}$ as in~\eqref{hello there distinguished gentlemen}. Given any $n\in\N$ and $\ep\in\tp{0,\epsilon_\star}$ we define the Banach spaces of functions $\bf{U}^n_\ep$ and $\bf{V}^n_\ep$ via
    \begin{multline}\label{domain and codomain for regularity promotion}
        \bf{U}^n_\ep = H^{n + 5/2}\tp{\Sigma_\ep} \times H^{n+2}_{\m{tan}}\tp{\Omega_\ep;\R^2}\times H^{n+1}\tp{\Omega_\ep}\text{ and}\\
        \bf{V}^n_\ep = H^{n + 1/2}\tp{\Sigma_\ep}\times H^n\tp{\Omega_\ep;\R^2}\times H^{n+1}\tp{\Omega_\ep}\times H^{n + 1/2}\tp{\Sigma_\ep;\R^2}\times H^{n + 1/2}\tp{\Sigma_0}.
    \end{multline}
    The spaces of~\eqref{domain and codomain for regularity promotion} are endowed with the norms
    \begin{multline}\label{domain and codomain norm for the regularity promotion}
        \tnorm{\eta,u,p}_{\bf{U}^n_\ep} = \tnorm{\eta}^{\tp{\ep}}_{H^{5/2}\tp{\Sigma_\ep}} + \tnorm{u,p}_{\bf{X}^1_\ep} + n\tp{\tnorm{\eta}_{H^{n+5/2}\tp{\Sigma_\ep}} + \tnorm{u}_{H^{n+2}\tp{\Omega_\ep}} + \tnorm{p}_{H^{n+1}\tp{\Omega_\ep}}}\text{ and}\\
        \tnorm{\psi^0,\varphi^1,\varphi^2,\psi^1,\psi^2}_{\bf{V}^n_\ep} = \ep^{1/2}\tnorm{\psi^0}_{L^2\tp{\Sigma_\ep}} + \ep\tnorm{\psi^0}_{H^{1/2}\tp{\Sigma_\ep}} + \tnorm{\varphi^1,\varphi^2,\psi^1,\psi^2}_{\bf{Y}^1_\ep} \\+ n\tp{\tnorm{\psi^0}_{H^{n+1/2}\tp{\Sigma_\ep}} + \tnorm{\varphi^1}_{H^n\tp{\Omega_\ep}} + \tnorm{\varphi^2}_{H^{n + 1}\tp{\Omega_\ep}} + \tnorm{\psi^1}_{H^{n + 1/2}\tp{\Sigma_\ep}} + \tnorm{\psi^2}_{H^{n + 1/2}\tp{\Sigma_0}}},
    \end{multline}
    for all $\tp{\eta,u,p}\in\bf{U}^n_\ep$ and all $\tp{\psi^0,\varphi^1,\varphi^2,\psi^1,\psi^2}\in\bf{V}^n_\ep$.  We also define the closed subsets of $\mathscr{U}^n_\ep\subset\bf{U}^n_\ep$ given by
    \begin{equation}\label{some subsets for there}
        \mathscr{U}^n_\ep=\tcb{\tp{\eta,u,p}\in\bf{U}^n_\ep\;:\;\mathsf{H} + \ep^2\eta\ge\m{H}_-/4,\; 4 + \ep^2u_1 - \mathsf{U}\ge\del/2,\;\tnorm{\eta,u,p}_{\bf{U}^0_\ep}\le\tabs{\log\ep}^{10}}.
    \end{equation}

    Next,  we define a family of bounded linear operators acting between these function spaces and indexed by $\bf{u} = \tp{\tilde{\eta},\tilde{u},\tilde{p}}\in\mathscr{U}^n_\ep$.  Indeed, we set $G^\ep_{\bf{u}}:\bf{U}^n_\ep\to\bf{V}^n_\ep$ via $G^\ep_{\bf{u}}\tp{\eta,u,p} = \tp{\psi^0,\varphi^1,\varphi^2,\psi^1,\psi^2}$ where  
    \begin{multline}\label{linear operator for regularity promotion}
            \tp{\varphi^1,\varphi^2,\psi^1,\psi^2} = L^\ep_{\mathsf{H} + \ep^2\tilde{\eta}}\tp{u,p} - \tp{0,0, \upsigma\ep^2\tp{1 + \ep^2\tabs{\mathsf{H} + \ep^2\tilde{\eta}}^2}^{-3/2}\eta''\mathcal{N}_{\ep\tp{\mathsf{H} + \ep^2\tilde{\eta}}},0}, \text{ and }\\
            \psi^0 = \tp{4 + \ep^2\tp{\Bar{\upgamma} + \tilde{u}_1} - \mathsf{U}}\tp{1 - \pd_1^2}\eta - \f{1}{\ep}\pd_1u\cdot\mathcal{N}_{\ep\tp{\mathsf{H} + \ep^2\tilde{\eta}}}.
    \end{multline}
    where the action of $L$ is defined in~\eqref{definition of the linear differential operator}.

    Finally, we introduce the lower order term conglomerate map $F^\ep:\mathscr{U}^n_\ep\to\bf{V}^{n+1}_\ep$, whose components $F^\ep=\tp{F^\ep_0,\tp{F^\ep_1,F^\ep_2},0,\tp{F^\ep_3,F^\ep_4},0}$ are given by
    \begin{equation}\label{F0 component}
        F^\ep_0\tp{\eta,u,p} = \tp{1 - \pd_1^2}\tp{\Bar{\upgamma}\mathsf{H} - \tp{\mathsf{H} + \ep^2\eta}^2\mathsf{U}_1/2 - \tp{\mathsf{H} + \ep^3\eta}\mathsf{U}_2/6 - \Bar{\m{A}}} - \f{\mathsf{H} + \ep^2\eta}{\ep}\int_0^\ep u_1\tp{\cdot,y}\;\m{d}y  -\mathsf{H}''u_1 + 2\mathsf{U}'\eta' + \mathsf{U}''\eta,
    \end{equation}
    \begin{multline}\label{F1 component}
        F^\ep_1\tp{\eta,u,p} = \tsb{\tp{\mathsf{X}\tp{\mathsf{H} + \ep^2\eta} - \tp{4 + \ep^2\Bar{\upgamma}}e_1}\cdot\grad^{\mathcal{A}_{\mathsf{H} + \ep^2\eta}}u + u\cdot\grad^{\mathcal{A}_{\mathsf{H} + \ep^2\eta}}\mathsf{X}\tp{\mathsf{H} + \ep^2\eta}}\cdot e_1 - \Bar{\upgamma}\mathsf{U}'\\
        +\tp{\mathsf{V}\tp{\mathsf{H} + \ep^2\eta} - \tp{4 + \ep^2\Bar{\upgamma}}e_1}\cdot\grad^{\mathcal{A}_{\mathsf{H} + \ep^2\eta}}\mathsf{W}_1\tp{\mathsf{H} + \ep^2\eta} + \mathsf{W}\tp{\mathsf{H} + \ep^2\eta}\cdot\grad^{\mathcal{A}_{\mathsf{H} + \ep^2\eta}}\mathsf{X}_1\tp{\mathsf{H} +\ep^2\eta}\\+\ep^2u\cdot\grad^{\mathcal{A}_{\mathsf{H} + \ep^2\eta}}u_1 + \pd_1^{\mathcal{A}_{\mathsf{H} + \ep^2\eta}}\mathsf{Q}\tp{\mathsf{H} + \ep^2\eta} - \upmu\tp{\pd_1^{\mathcal{A}_{\mathsf{H} + \ep^2\eta}}}^2\mathsf{W}_1\tp{\mathsf{H} + \ep^2\eta},
    \end{multline}
    \begin{multline}\label{F2 component}
        F^\ep_2\tp{\eta,u,p} = \tsb{\tp{\mathsf{X}\tp{\mathsf{H} + \ep^2\eta} - \tp{4 + \ep^2\Bar{\upgamma}}e_1}\cdot\grad^{\mathcal{A}_{\mathsf{H} + \ep^2\eta}}u + u\cdot\grad^{\mathcal{A}_{\mathsf{H} + \ep^2\eta}}\mathsf{X}\tp{\mathsf{H} + \ep^2\eta}}\cdot e_2 + \Bar{\upgamma}y\tp{\mathsf{H} + \ep^2\eta}\mathsf{U}''\\+\tp{\mathsf{V}\tp{\mathsf{H} + \ep^2\eta} - \tp{4 + \ep^2\Bar{\upgamma}}e_1}\cdot\grad^{\mathcal{A}_{\mathsf{H} + \ep^2\eta}}\mathsf{W}_2\tp{\mathsf{H} + \ep^2\eta} + \mathsf{W}\tp{\mathsf{H} + \ep^2\eta}\cdot\grad^{\mathcal{A}_{\mathsf{H} + \ep^2\eta}}\mathsf{X}_2\tp{\mathsf{H} +\ep^2\eta}\\+\ep^2u\cdot\grad^{\mathcal{A}_{\mathsf{H} + \ep^2\eta}}u_2 - \upmu\tp{\pd_1^{\mathcal{A}_{\mathsf{H} + \ep^2\eta}}}^2\mathsf{W}_2\tp{\mathsf{H} + \ep^2\eta},
    \end{multline}
    \begin{multline}\label{F3 component}
        F^\ep_3\tp{\eta,u,p} = \ep\upsigma\tp{\mathsf{H} + \ep^2\eta}'\mathsf{H}''\tp{1 + \ep^2\tabs{\tp{\mathsf{H} + \ep^2\eta}'}^2}^{-3/2} + \ep\tp{\mathsf{P} - \m{g}\mathsf{H} -2\upmu\mathsf{U}'}\eta' + \ep\upmu\tp{\mathsf{U}_2 - \mathsf{U}''}\eta - \ep\m{g}\mathsf{H}'\eta\\
        +\ep\mathsf{Q}\tp{\mathsf{H} + \ep^2\eta}\tp{\mathsf{H} + \ep^2\eta}' - 2\ep\upmu\pd_1^{\mathcal{A}_{\mathsf{H} + \ep^2\eta}}\mathsf{W}_1\tp{\mathsf{H} + \ep^2\eta}\tp{\mathsf{H} + \ep^2\eta}' + \upmu\pd_1^{\mathcal{A}_{\mathsf{H} + \ep^2\eta}}\mathsf{W}_2\tp{\mathsf{H} + \ep^2\eta} - \ep^3\m{g}\eta\eta',
    \end{multline}
    and
    \begin{multline}\label{F4 component}
        F^\ep_4\tp{\eta,u,p} = -\upsigma\mathsf{H}''\tp{1 + \ep^2\tabs{\tp{\mathsf{H} + \ep^2\eta}'}^2}^{-3/2} + \m{g}\eta \\- \mathsf{Q}\tp{\mathsf{H} + \ep^2\eta} + \upmu\tp{\mathsf{H} + \ep^2\eta}\tp{\mathsf{H} + \ep^2\eta}'\mathsf{U}'' + \upmu\mathbb{D}^{\mathcal{A}_{\mathsf{H} + \ep^2\eta}}\mathsf{W}\tp{\mathsf{H} + \ep^2\eta}\mathcal{N}_{\ep\tp{\mathsf{H} + \ep^2\eta}}\cdot e_2,
    \end{multline}
    where the auxiliary functions $\mathsf{V}$, $\mathsf{W}$, $\mathsf{X}$, and $\mathsf{Q}$ are given by
    \begin{multline}
        \mathsf{V}\tp{\zeta} = \mathsf{U}e_1-y\zeta\mathsf{U}'e_2,\quad\mathsf{W}\tp{\zeta} = \bp{\f{y\zeta}{\ep}\mathsf{U}_1 = \f12\bp{\f{y\zeta}{\ep}}^2\mathsf{U}_2}e_1 - \ep\bp{\f{1}{2}\bp{\f{y\zeta}{\ep}}^2\mathsf{U}_1' + \f16\bp{\f{y\zeta}{\ep}}^3\mathsf{U}_2'}e_2,  \\
        \mathsf{X}\tp{\zeta} = \mathsf{V}\tp{\zeta} + \ep^2\mathsf{W}\tp{\zeta},
        \text{ and } \mathsf{Q}\tp{\zeta} = \f{y\zeta}{\ep}\mathsf{P}_1 + \f{1}{2}\bp{\f{y\zeta}{\ep}}^2\mathsf{P}_2
    \end{multline}
    for $\zeta = \mathsf{H} + \ep^2\eta$.
\end{defnC}

We remark that the precise choice of the norms~\eqref{domain and codomain norm for the regularity promotion} (and how they depend on $\ep$) is only really relevant in the case $n=0$. For $n\ge1$, any choice of equivalent norms will be fine for the proofs that follow.

Definition~\ref{defn of objects for regularity promotion theory} takes for granted that the  map $F^\ep :\mathscr{U}^n_\ep \to \bf{V}^{n+1}_\ep$ is well-defined;  this can be verified by elementary but tedious calculations, which we omit here for the sake of brevity.  We note, though, that Section~\ref{subsection on properties of residual forcing and nonlinearities} establishes a related and much sharper version of the $n=0$ case of the claimed mapping properties, and the same considerations  will work for general $n$.

The following lemma reveals how the solutions produced by Theorem~\ref{thm on existence of bore waves} interact with the objects of Definition~\ref{defn of objects for regularity promotion theory}.

\begin{lemC}[Reformulation for regularity promotion]\label{lem on reformulation for regularity promotion}
    Assume the hypotheses of Definition~\ref{defn of objects for regularity promotion theory}. For $0<\ep<\epsilon_\star$ let $\tp{\eta,u,p}\in\bf{U}^0_\ep$ and $\lambda\in\Bar{B(0,\be)}$ denote the fixed point granted by the second and third items of Proposition~\ref{prop on the contraction}. Define $\tp{h,v,q}\in\bf{U}^0_\ep$ via
    \begin{equation}\label{definition of the distinguished residuals}
        h = \eta + \ep^{-2}\mathfrak{h}\tp{\lambda},\; v = u + \ep^{-2}\tp{X\tp{\ep,\lambda,\eta} - \mathsf{X}\tp{\mathsf{H} + \ep^2 h}},\text{ and } q = p + Q\tp{\ep,\lambda,\eta} - \mathsf{Q}\tp{\mathsf{H} + \ep^2h} + \ep^{-2}\tp{\mathsf{P} - P\tp{\lambda}},
    \end{equation}
    where $\mathfrak{h}\tp{\lambda}$ is defined in Proposition~\ref{prop on the bore map}; the functions $\mathsf{X}$, $\mathsf{H}$, $\mathsf{Q}$, and $\mathsf{P}$ are from Definition~\ref{defn of objects for regularity promotion theory}; and $X$, $Q$, and $P$ are from Definition~\ref{defn big steppah}. 
    
    There exists $\epsilon_{\bigstar}\in\tp{0,\epsilon_\star}$ such that the following hold for all $0<\ep<\epsilon_\bigstar$.
    \begin{enumerate}
        \item We have the inclusion $\tp{h,v,q}\in\m{int}\mathscr{U}^0_\ep$, where $\m{int}$ denotes the topological interior.
        \item If $\tp{\pmb{\zeta},\pmb{u},\pmb{p}}\in H^{5/2}_\loc\tp{\Sigma_\ep}\times H^2_\loc\tp{\Omega_\ep;\R^2}\times H^{1}_\loc\tp{\Omega_\ep}$ denotes the solution to~\eqref{FBINSE, param. tuned and funny flux} produced by Theorem~\ref{thm on existence of bore waves}, then we have the identities
        \begin{equation}\label{expansion of the solution in terms of the distinguished residuals}
            \pmb{\zeta} = \mathsf{H} + \ep^2h,\;\pmb{u} = \mathsf{X}\tp{\mathsf{H} + \ep^2h} + \ep^2v,\text{ and }\pmb{p} = \mathsf{P} + \ep^2\mathsf{Q}\tp{\mathsf{H} + \ep^2 h} + \ep^2 q.
        \end{equation}
        In particular, if we know that $\tp{h,v,q}\in\bf{U}^n_\ep$ for some $n\in\N^+$, then we have the higher regularity inclusions $\pmb{\zeta}\in\tp{W^{\infty,\infty} + H^{n+5/2}}\tp{\Sigma_\ep}$, $\pmb{u}\in\tp{W^{\infty,\infty} + H^{n+2}}\tp{\Omega_\ep;\R^2}$, and $\pmb{p}\in\tp{W^{\infty,\infty} + H^{n+1}}\tp{\Omega_\ep}$.
        \item The triple $\tp{h,v,q}$ obeys the identity
        \begin{equation}\label{the fundamental identity for the regularity promotion}
            G^\ep_{\tp{h,v,q}}\tp{h,v,q} = - F^\ep\tp{h,v,q},
        \end{equation}
        where $G^\ep$ and $F^\ep$ are the family of linear operators and lower order terms, respectively, from Definition~\ref{defn of objects for regularity promotion theory}.
    \end{enumerate}
\end{lemC}
\begin{proof}
    We begin by establishing the second item, as it actually holds over the full range of $0<\ep<\epsilon_\star$. Thanks to the proof of Theorem~\ref{thm on existence of bore waves} we know that $\pmb{\zeta}$, $\pmb{u}$, and $\pmb{p}$ are defined in terms of $\eta$, $u$, $p$, and $\lambda$ via identities~\eqref{ansatz_1}, \eqref{ansatz_2}, and~\eqref{ansatz_3}. It is then immediate that $h$, $v$, and $q$ as defined in~\eqref{definition of the distinguished residuals} yield the identities in~\eqref{expansion of the solution in terms of the distinguished residuals}. The remainder of the second item follows readily.

    Next, we establish the first item.  Since $\pmb{\zeta} = \mathsf{H} + \ep^2h$ we know by Theorem~\ref{thm on existence of bore waves} that $\pmb{\zeta}\ge\m{H}_-/2>\m{H}_-/4$. By the proof of the same result, we know that $\inf\tp{4 - \mathsf{U}}\ge\del$ and so one can check that $4 + \ep^2 v_1 - \mathsf{U}>\del/2$ for $0<\ep<\epsilon_\bigstar$ small enough by direct computation from the definition~\eqref{definition of the distinguished residuals}.  The claimed upper bound $\tnorm{h,v,q}_{\bf{U}^0_\ep}<\tabs{\log\ep}^{10}$ also follows from direct computation from~\eqref{definition of the distinguished residuals}, using the fact that $\tnorm{\eta,u}_{\bf{W}_\ep}\le R_\star$,  the estimates on $\lambda$ provided by~\eqref{crazy bernstein estimate}, the estimates on $p$ provided by Proposition~\ref{prop on velocity - pressure fixed point reformulation}, and taking $\epsilon_\bigstar>0$ sufficiently small.

   We now turn our attention to the third item, which amounts to a computation and packing the resulting terms into the expressions $G^\ep_{\tp{h,v,q}}$ and $F^\ep$ appearing in~\eqref{the fundamental identity for the regularity promotion}.  The first step is to note that Definition~\ref{defn of objects for regularity promotion theory} requires that $\tp{\pmb{\zeta},\pmb{u},\pmb{p}}$ is a solution to~\eqref{FBINSE, param. tuned and funny flux} and that the tuple  $\tp{\mathsf{H},\mathsf{U},\mathsf{U}_1,\mathsf{U}_2,\mathsf{P},\mathsf{P}_1,\mathsf{P}_2}$ satisfies the ODEs from~\eqref{hello there distinguished gentlemen}.  Consequently, the first through  fifth items of Theorem~\ref{thm on the shallow water and residual equations} are satisfied under the replacements $\tp{\eta,u,p}\mapsto\tp{h,v,q}$ and $\tp{H,U,U_1,U_2,P,P_1,P_2}\mapsto\tp{\mathsf{H},\mathsf{U},\mathsf{U}_1,\mathsf{U}_2,\mathsf{P},\mathsf{P}_1,\mathsf{P}_2}$ with $\mathfrak{r}_1=\mathfrak{r}_3=\mathfrak{r}_4 = 0$.  This allows us to invoke its converse statement  to learn that the triple $\tp{h,v,q}$ is a solution to system~\eqref{first form of the residual PDEs} under the aforementioned replacements. To arrive at~\eqref{the fundamental identity for the regularity promotion}, we take the lower order terms on the left hand side of these equations (such as the residual advective derivative) and place them into $F^\ep$, and we also take the higher order capillarity terms hiding in $\mathfrak{f}_3$ and $\mathfrak{f}_4$ and place them into $G^\ep$.
\end{proof}

The second item of the previous lemma reduces establishing smoothness of $\tp{\pmb{\zeta},\pmb{u},\pmb{p}}$ to establishing that of $\tp{h,v,q}$. To take the next step in achieving the latter, we embark on a further study of the left hand side of the operator equation~\eqref{the fundamental identity for the regularity promotion} from the third item of the lemma.

\begin{propC}[Main engine for regularity promotion]\label{prop on properties of the highest order terms}
    Assume the hypotheses of Lemma~\ref{lem on reformulation for regularity promotion} and let the parameter $\epsilon_\bigstar\in\tp{0,1/3}$ be as in the lemma. There exist $C\in\R^+$ and $\epsilon_\bigstar^1\in\tp{0,\epsilon_\bigstar}$ such that the following hold for all $0<\ep<\epsilon^1_\bigstar$ and all $n\in\N$.
    \begin{enumerate}
        \item The map $E^\ep:\m{int}\mathscr{U}^n_\ep\to\bf{V}^n_\ep$, defined via $E^\ep\tp{\bf{u}} = G^\ep_{\bf{u}}\bf{u}$ for all $\bf{u} = \tp{\eta,u,p}\in\mathscr{U}^n_\ep$, is continuously differentiable.
        \item There exists a function $R^\ep:\m{int}\mathscr{U}^{n}_\ep\to\bf{V}^{n}_\ep$ such that for all $\bf{u}\in\m{int}\mathscr{U}^1_\ep$ we have the identity
        \begin{equation}\label{equation of the tangentially differentiated equation}
            \pd_1\tsb{E^\ep\tp{\bf{u}}} = DE^\ep\tp{\bf{u}}[\pd_1\bf{u}] + R^\ep\tp{\bf{u}}.
        \end{equation}
        \item If $\bf{u}\in\m{int}\mathscr{U}^n_\ep$ is such that $\pd_1\bf{u}\in\bf{U}^n_\ep$ and $E^\ep\tp{\bf{u}}\in\bf{V}^{n+1}_\ep$, then $\bf{u}\in\m{int}\mathscr{U}^{n+1}_\ep$.
        \item For all $\tilde{\bf{u}}\in\m{int}\mathscr{U}^n_\ep$ the derivative $DE^\ep\tp{\tilde{\bf{u}}}:\bf{U}^n_\ep\to\bf{V}^n_\ep$ is a linear isomorphism.
        \item If $\bf{u}\in\m{int}\mathscr{U}^0_\ep$ is such that $E^\ep\tp{\bf{u}}\in\bf{V}^{n}_\ep$, then $\bf{u}\in\m{int}\mathscr{U}^{n}_\ep$.
    \end{enumerate}
\end{propC}
\begin{proof}
    The first item follows from standard computations and estimates for products and compositions in Sobolev spaces, which we omit for brevity.

    For the second item we take $\bf{u} = \tp{\eta,u,p}\in\m{int}\mathscr{U}^1_\ep$ and define $R^\ep\tp{\bf{u}} = \pd_1\tsb{E^\ep\tp{\bf{u}}} - DE^\ep\tp{\bf{u}}\tsb{\pd_1\bf{u}}$. A straightforward computation, which we again mostly omit, will show the mapping property $R^\ep:\m{int}\mathscr{U}^n_\ep\to\bf{V}^n_\ep$ holds. To illustrate the flavor of the terms in $R^\ep$, consider just the $e_2$-momentum pressure contribution:
    \begin{equation}
        \pd_1\tsb{\pd_2^{\mathcal{A}_{\mathsf{H}+\ep^2\eta}}p} - \bp{\pd_2^{\mathcal{A}_{\mathsf{H} + \ep^2\eta}}\tsb{\pd_1p} - \ep^2\f{\tsb{\pd_1\eta}}{\tp{\mathsf{H} + \ep^2\eta}^2}\pd_2p} = -\f{\pd_1\mathsf{H}}{\tp{\mathsf{H} + \ep^2\eta}^2}\pd_2p.
    \end{equation}
    As in the right side above, generically the summands in $R^\ep$ occur when the derivative $\pd_1$ falls onto the smooth background terms.

    We now prove the third item. Let $\bf{u} = \tp{\eta,u,p}\in\mathscr{U}^n_\ep$ satisfy $\pd_1\bf{u}\in\bf{U}^n_\ep$ and $E^\ep\tp{\bf{u}}\in\bf{V}^{n+1}_\ep$. The former two of these inclusions directly imply higher regularity of the free surface:  $\eta\in H^{n+7/2}\tp{\Sigma_\ep}$. In turn, we find  from the definition~\eqref{linear operator for regularity promotion} that
    \begin{equation}\label{bulk improved regularity}
        \varphi^1=\grad^{\mathcal{A}_{\mathsf{H} + \ep^2\eta}}p - \upmu\grad^{\mathcal{A}_{\mathsf{H}+\ep^2\eta}}\cdot\mathbb{D}^{\mathcal{A}_{\mathsf{H}+\ep^2\eta}}u\in H^{n+1}\tp{\Omega_\ep;\R^2} 
        \text{ and }
        \varphi^2=\grad^{\mathcal{A}_{\mathsf{H} + \ep^2\eta}}\cdot u\in H^{n+2}\tp{\Omega_\ep}.
    \end{equation}
    
    Now it remains to show that $\pd_2 p\in H^{n+1}\tp{\Omega_\ep}$ and $\pd_2^2 u\in H^{n+1}\tp{\Omega_\ep;\R^2}$. For this we argue as in the later part of the third step of the proof of Proposition~\ref{prop on theory of strong solutions} and use identities~\eqref{bulk improved regularity} to establish tangential to normal regularity promotion. In fact, with $h = \mathsf{H} + \ep^2\eta$, we compute that identities~\eqref{solving for some other weirdo} and~\eqref{believe it or not, this equation does it referenced} remain true. The desired inclusions and the third item follow.

    Next,  consider the fourth item. We begin by establishing that when $0<\ep<\epsilon^1_\bigstar$ is sufficiently small the derivative $DE^\ep\tp{\tilde{\bf{u}}}:\bf{U}^0_\ep\to\bf{V}^0_\ep$ is an isomorphism whenever $\tilde{\bf{u}}\in\m{int}\mathscr{U}^0_\ep$. A direct computation shows that this derivative has the form 
    \begin{equation}\label{splitting of the derivative}
        DE^\ep\tp{\tilde{\bf{u}}}\bf{u} = M^\ep\bf{u} + \ep^2 N^\ep_{\tilde{\bf{u}}}\bf{u}
    \end{equation}
    where for $\tp{\eta,u,p} = \bf{u}$ we have
    \begin{equation}\label{split 2}
        M^\ep\bf{u} = \tp{\tp{4 - \mathsf{U}}\tp{1 - \pd_1^2}\eta - \tp{\pd_1u\cdot\mathcal{N}_{\ep\mathsf{H}}}/\ep,L^\ep_{\mathsf{H}}\tp{u,p}}\text{ and }N^\ep_{\tilde{\bf{u}}}\bf{u} = \ep^{-2}\tp{DE^\ep\tp{\tilde{\bf{u}}}\bf{u} - M^\ep\bf{u}}.
    \end{equation}
    The main part $M^\ep:\bf{U}^0_\ep\to\bf{V}^0_\ep$ is invertible due to Theorem~\ref{thm on synthesis of linear analysis}, its block structure, the lower bound $\inf\tp{4 - \mathsf{U}}>0$, and the invertibility of $\tp{1 - \pd_1^2}$. In fact, we deduce from the definition of the norm~\eqref{domain and codomain norm for the regularity promotion} (when $n=0$) and estimate~\eqref{estimates on the inverse stokes operator} that the corresponding operator norms satisfy
    \begin{equation}\label{estimate 1_}
        \sup_{0<\ep<\epsilon_\bigstar}\max\tcb{\tnorm{M^\ep}_{\mathcal{L}\tp{\bf{U}^0_\ep;\bf{V}^0_\ep}},\tnorm{\tp{M^\ep}^{-1}}_{\mathcal{L}\tp{\bf{V}^0_\ep;\bf{U}^0_\ep}}}<\infty.
    \end{equation}
    By using the inequality $\tnorm{\tilde{\bf{u}}}_{\bf{U}^0_\ep}\le\tabs{\log\ep}^{10}$ built into the definition~\ref{some subsets for there}, the remaining piece $N^\ep_{\tilde{\bf{u}}}$ can be tamed, under a conservative argument, with only a factor of $\ep^{3/2}$. More precisely, we compute (in the style of Section~\ref{subsection on properties of residual forcing and nonlinearities}) that we have
    \begin{equation}\label{estimate 2_}
        \sup_{0<\ep<\epsilon_\bigstar}\sup_{\tilde{\bf{u}}\in\m{int}\mathscr{U}^0_\ep}\ep^{3/2}\tnorm{N^\ep_{\tilde{\bf{u}}}}_{\mathcal{L}\tp{\bf{U}^0_\ep;\bf{V}^0_\ep}}<\infty.
    \end{equation}
    Combining the estimates~\eqref{estimate 1_} and~\eqref{estimate 2_} with the decomposition~\ref{splitting of the derivative} and recalling that the set of invertible operators is open, we deduce the existence of $0<\epsilon_\bigstar^1<\epsilon_\bigstar$ such that whenever $0<\ep<\epsilon^1_\bigstar$ and $\tilde{\bf{u}}\in\m{int}\mathscr{U}^0_\ep$ the operator $DE^\ep\tp{\tilde{\bf{u}}}:\bf{U}^0_\ep\to\bf{V}^0_\ep$ is an isomorphism.

    The `base case' regularity $n=0$ of the fourth item is now established. Notice that we immediately get that the maps $DE^\ep\tp{\tilde{\bf{u}}}:\bf{U}^n_\ep\to\bf{V}^n_\ep$ are injective for all $n$. Therefore, we only need to establish surjectivity in the higher regularity cases, which can be viewed as regularity promotion.    The key to this is the linear problem's analog of the tangential to normal regularity promotion mechanism, which we discuss next.  
    
    We claim that for any $n\in\N$, if $\tilde{\bf{u}}\in\m{int}\mathscr{U}^{n+1}_\ep$ and  $\bf{u},\pd_1\bf{u}\in\bf{U}^n_\ep$ with $DE^\ep\tp{\tilde{\bf{u}}}\bf{u}\in\bf{V}^{n+1}_\ep$, then we have the inclusion $\bf{u}\in\bf{U}^{n+1}_\ep$. If we write $\bf{u} = \tp{\eta,u,p}$ then the assumed tangential regularity inclusion $\pd_1\bf{u}\in\bf{U}^n_\ep$ forces $\eta\in H^{7/2 + n}\tp{\Sigma_\ep}$. We then return to the bulk (i.e. momentum and divergence equations) components of the decomposition~\ref{splitting of the derivative} and see that the $N^\ep_{\tilde{\bf{u}}}$ contribution depends only on $\eta$, and hence it can be thrown on the higher regularity right hand side of the equation. Doing this leaves us with an inclusion of the form~\eqref{bulk improved regularity}, but replacing the argument of the geometry matrices, namely $\mathsf{H} + \ep^2\eta$, with the term $\mathsf{H} + \ep^2\tilde{\eta}$ where $\tilde{\bf{u}} = \tp{\tilde{\eta},\tilde{u},\tilde{p}}$. We can then argue exactly as before to deduce that $\pd_2p\in H^{n+1}\tp{\Omega_\ep}$ and $\pd^2_2u\in H^{n+1}\tp{\Omega_\ep;\R^2}$, and hence $\tp{\eta,u,p} = \bf{u}\in\bf{U}^{n+1}_\ep$.  This completes the proof of the claim.

    The above claim reduces the higher regularity promotion (and thus surjectivity) question to  proving only tangential regularity promotion.  We aim to do this with an inductive argument; however, we need first deal with the second base case, $n=1$, before we can properly differentiate the equation. Suppose that $\bf{v}\in\bf{V}^1_\ep$, $\tilde{\bf{u}}\in\m{int}\mathscr{U}^1_\ep$ , and $\bf{u}\in\bf{U}^0_\ep$ are related by the equation $DE^\ep\tp{\tilde{\bf{u}}}\bf{u} = \bf{v}$. Instead of differentiating the equation in the $e_1$-direction, we use finite forward difference quotients of step-size $t\in\tp{0,1}$, denoted by $\updelta_t$. We schematically write
    \begin{equation}\label{the difference quotient boy}
        \updelta_t\tsb{DE^\ep\tp{\tilde{\bf{u}}}\bf{u}} = {_0}L^t_{\tilde{\bf{u}}}\updelta_t\bf{u} + {_1}L_{\tilde{\bf{u}}}^t\bf{u} = \updelta_t\bf{v}
    \end{equation}
    where ${_0}L^t_{\bf{\tilde{u}}}\approx DE^\ep\tp{\tilde{\bf{u}}}$ (in fact ${_0}L_{\tilde{\bf{u}}}^t\to DE^\ep\tp{\tilde{\bf{u}}}$ in the operator norm as $t\to0$) and ${_1}L_{\tilde{\bf{u}}}^t$ and $\updelta_t\bf{v}$ have the correct order in the sense that $\sup_{0<t<1}\tnorm{{_1}L_{\tilde{\bf{u}}}^t\bf{u}}_{\bf{V}^0_\ep}<\infty$ and $\sup_{0<t<1}\tnorm{\updelta_t\bf{v}}_{\bf{V}^0_\ep}<\infty$. The same argument that established the invertibility of $DE^{\ep}\tp{\tilde{\bf{u}}}:\bf{U}^0_\ep\to\bf{V}^0_\ep$ applies for the `translated' operator ${_0}L^t_{\tilde{\bf{u}}}$ for small enough $t>0$, and so from~\eqref{the difference quotient boy} we deduce that
    \begin{equation}\label{the other difference quotient boy}
        \updelta_t\bf{u} = \tp{{_0}L^t_{\tilde{\bf{u}}}}^{-1}\tp{\updelta_t\bf{v} - {_1}L_{\tilde{\bf{u}}}^t\bf{u}}\text{ and hence }\limsup_{t\to0}\tnorm{\updelta_t\bf{u}}_{\bf{U}^0_\ep}<\infty.
    \end{equation}
    The above implies that $\pd_1\bf{u}\in\bf{U}^0_\ep$. Feeding this information into the linear problem's tangential to normal regularity promotion mechanism yields then $\bf{u}\in\bf{U}^1_\ep$. Thus, the $n=1$ case of the fourth item is now established.

    The cases $n\ge 2$ can now be established inductively, as we have enough regularity to differentiate the equation. Indeed, if we suppose that $\bf{v}\in\bf{V}^{n}_\ep$, $\tilde{\bf{u}}\in\m{int}\mathscr{U}^{n}_\ep$, and $\bf{u}\in\bf{U}^{n-1}_\ep$ satisfy $DE^\ep\tp{\tilde{\bf{u}}}\bf{u} = \bf{v}$, then an application of $\pd_1$ to this identity produces the more favorable form of~\eqref{the difference quotient boy}:
    \begin{equation}
        \pd_1\tsb{DE^\ep\tp{\tilde{\bf{u}}}\bf{u}} = DE^\ep\tp{\tilde{\bf{u}}}\pd_1\bf{u} + {_1}L_{\tilde{\bf{u}}}^0\bf{u} = \pd_1\bf{v}.
    \end{equation}
    The induction hypothesis is that $DE^\ep\tp{\tilde{\bf{u}}}:\bf{U}^{n-1}_\ep\to\bf{V}^{n-1}_\ep$ is an isomorphism, and we have that ${_1}L^0_{\tilde{\bf{u}}}\bf{u},\pd_1\bf{v}\in\bf{V}^{n-1}_\ep$, so by arguing as in~\eqref{the other difference quotient boy}, we find that $\pd_1\bf{u}\in\bf{U}^{n-1}_\ep$. The tangential-to-normal engine kicks in again and we deduce further that $\bf{u}\in\bf{U}^n_\ep$. This completes the induction argument and the proof of the  fourth item as well.

    Finally, we prove the fifth item. The case $n=0$ is trivial, while the case $n=1$ suffers from the same `base case' differentiability woes as above.  Suppose that $\bf{u}\in\m{int}\mathscr{U}^0_\ep$, $\bf{v}\in\bf{V}^1_\ep$ are related by the equation $E^\ep\tp{\bf{u}} = \bf{v}$.   We apply, for $t\in\tp{0,1}$, the difference quotient operator $\updelta_t$ to this identity and derive the following equation, which is similar to~\eqref{equation of the tangentially differentiated equation}:
    \begin{equation}\label{identity identity where is my identity help i cant find it}
        \updelta_t\tsb{E^\ep\tp{\bf{u}}} = {_0}L^t_{\bf{u}}\updelta_t\bf{u} + R^\ep_t\tp{\bf{u}} = \updelta_t\bf{v},
    \end{equation}
    where again ${_0}L^t_{\bf{u}}\to DE^\ep\tp{\bf{u}}$ as $t\to 0$ in $\mathcal{L}\tp{\bf{U}^0_\ep;\bf{V}^0_\ep}$ and we have the estimates $\sup_{0<t<1}\tnorm{R^\ep_t\tp{\bf{u}}}_{\bf{V}^0_\ep}<\infty$ and $\sup_{0<t<1}\tnorm{\updelta_t\bf{v}}_{\bf{V}^0_\ep}<\infty$. Therefore, we may apply $\tp{{_0}L^t_{\bf{u}}}^{-1}$ to identity~\eqref{identity identity where is my identity help i cant find it} for $t$ sufficiently small and deduce that $\limsup_{t\to0}\tnorm{\updelta_t\bf{u}}_{\bf{U}^0_\ep}<\infty$. This implies that $\pd_1\bf{u}\in\bf{U}^0_\ep$, and feeding this into the third item then yields $\bf{u}\in\m{int}\mathscr{U}^{1}_\ep$.  The base case $n=1$ of the fifth item is now proved.

    The remaining cases of $n\ge 2$ for the fifth item now follow from a simple induction argument, as we have enough regularity to differentiate the equation with $\pd_1$ and use the second and third items. The strategy is very similar to that used in the proof of the fourth item, so we omit the remaining details.  This completes the proof of the fifth item.
\end{proof}

The results of this subsection are combined to give the following theorem.
\begin{thmC}[Smoothness of bore waves]\label{thm on smoothness of bore waves} Assume the hypotheses of Theorem~\ref{thm on existence of bore waves} and Definition~\ref{defn of objects for regularity promotion theory} while letting $0<\ep<\epsilon^1_\bigstar$ (where the latter smallness threshold is from Proposition~\ref{prop on properties of the highest order terms}). The remainders and solutions produced by the aforementioned theorem for such $\ep$ obey the higher regularity inclusions
\begin{equation}\label{smoothness of bore waves inclusions}
    \tp{\Bar{\eta},\Bar{u},\Bar{p}}\in H^\infty\tp{\Sigma_\ep}\times H^\infty\tp{\Omega_\ep;\R^2}\times H^\infty\tp{\Omega_\ep}\text{ and }\tp{\pmb{\zeta},\pmb{u},\pmb{p}}\in W^{\infty,\infty}\tp{\Sigma_\ep}\times W^{\infty,\infty}\tp{\Omega_\ep;\R^2}\times W^{\infty,\infty}\tp{\Omega_\ep}.
\end{equation}
\end{thmC}
\begin{proof}
    Let $\tp{h,v,q}\in\m{int}\mathscr{U}^0_\ep$ be the variables instantiated by Lemma~\ref{lem on reformulation for regularity promotion}. Due to the third item of this result, we know that $E^\ep\tp{h,v,q} = -F^\ep\tp{h,v,q}$ where $F^\ep:\m{int}\mathscr{U}_\ep^n\to\bf{V}^{n+1}_\ep$ (for every $n\in\N$) is the function from Definition~\ref{defn of objects for regularity promotion theory}. The operator $E^\ep$ obeys the regularity promotion property of the fifth item of Proposition~\ref{prop on properties of the highest order terms}. Therefore, an induction argument applies and we find that $\tp{h,v,q}\in H^\infty\tp{\Sigma_\ep}\times H^\infty\tp{\Omega_\ep;\R^2}\times H^\infty\tp{\Omega_\ep}$. In turn, thanks to the identity~\eqref{expansion of the solution in terms of the distinguished residuals}, we find that the right inclusion of~\eqref{smoothness of bore waves inclusions} must hold.

    The validity of the left inclusion in~\eqref{smoothness of bore waves inclusions} follows by writing
    \begin{multline}
        \mathsf{H} + \Bar{\eta} = \pmb{\zeta} = \mathsf{H} + \ep^2h,\quad\bp{4 - \f{\m{A}}{\mathsf{H}} + \f{4\m{a}}{\upmu}\mathsf{H}^2\bp{\ep y - \f12y^2}}e_1 + \Bar{u} = \pmb{u} = \mathsf{X}\tp{\mathsf{H} + \ep^2 h} + \ep^2v,\\
        \text{ and } 
        \m{g}\mathsf{H} - 2\upmu\m{A}\f{\mathsf{H}'}{\mathsf{H}^2} + \Bar{p} = \pmb{p} = \mathsf{P} + \ep^2\mathsf{Q}\tp{\mathsf{H} + \ep^2 h} + \ep^2q,
    \end{multline}
    and solving for $\Bar{\eta}$, $\Bar{u}$, and $\Bar{p}$.
\end{proof}

\subsection{Conclusions}\label{subsection on further properties}
This subsection is devoted to the proofs of the main theorem and corollaries of Section~\ref{subsection on statement of the main theorems and discussion}.

\begin{proof}[Proof of Theorem~\ref{first main theorem}]
    This result nearly follows by combining Theorems~\ref{thm on existence of bore waves} and~\ref{thm on smoothness of bore waves} and taking $\epsilon_\star = \epsilon^1_\bigstar$. We only need to check that the solutions produced by these results satisfy the requirements to be traveling bore waves in the sense of Definition~\ref{defn of bore waves}. The first item of this definition is manifestly satisfied by Theorem~\ref{thm on existence of bore waves} and the equivalence between systems~\eqref{flattened free boundary Navier-Stokes system} and~\eqref{FBINSE, param. tuned and funny flux}. The second item is satisfied by $\hat{\m{A}} = \m{A} = \ep^2\Bar{\m{A}}$ and the definition of $\Bar{\upgamma}$, $\Bar{\m{A}}$ in the parameter tuning identities~\eqref{the parameter tuning identities 440 hz}. For the third item, we can use the identity~\eqref{useful_limit_2_Dos} to check that the limits~\eqref{end states of the bore solutions} hold. Indeed, we compute initially that
    \begin{equation}\label{_THE_ABOVE_}
        \lim_{x\to\pm\infty}\pmb{u}_1\tp{\iota x,y} = \bp{4 - \f{\m{A}}{\m{H}_{\pm}}} + \f{4\m{a}}{\upmu}\m{H}^2_{\pm}\bp{\ep y - \f12 y^2}.
    \end{equation}
    Now we substitute the identity $4 - \m{A}/\m{H}_\pm = 4\m{H}_{\pm}$, which is a consequence of~\eqref{sweq}, into~\eqref{_THE_ABOVE_} to find that the limit claimed for $\pmb{u}_1$ in~\eqref{end states of the bore solutions} holds. The remaining limits are clear.
\end{proof}

\begin{proof}[Proof of Corollary~\ref{coro on the existence of large bores}]
Fix $\al\in(0,1)$ and $\iota\in\tcb{-1,1}$. Thanks to~\eqref{sweq}, we know that $\lim_{\m{A}\to 1}\tp{\m{H}_+(\m{A}) - \m{H}_-\tp{\m{A}}}=1$, so for some $\m{A}\in\tp{0,1}$ we have $\tp{\m{H}_+(\m{A}) - \m{H}_-\tp{\m{A}}}>1-\al$. The existence of a $\m{g}$ such that $\tp{\m{g},\m{A}}\in\mathfrak{C}_\iota$ is assured by~\eqref{non-degenerate I} from Lemma~\ref{lem  on sufficient conditions for a sign}. We are now in a position to invoke the third item of  Theorem~\ref{first main theorem}  to obtain $\pmb{\zeta}$, $\pmb{u}$, and $\pmb{p}$ that are a smooth bore wave solution to system~\eqref{flattened free boundary Navier-Stokes system} in the sense of Definition~\ref{defn of bore waves}. The formulas for the solution~\eqref{isolation of just the distinguished part of the solution}, paired with the end state identities~\eqref{the following limits are satisfied} and~\eqref{sweq}, permit us to readily compute the jumps~\eqref{the dimensionless hydraulic jumps}.
\end{proof}

\begin{proof}[Proof of Corollary~\ref{coro on justification of the shallow water limit}]
    The computation done at the beginning of Section~\ref{subsection on distinguished shallow water bore solutions} shows that the equations~\eqref{the one dimensional shallow water equations UH form} are satisfied when~\eqref{shallow water just H version} holds with $\tp{4 - \mathsf{U}}\mathsf{H} = \m{A}$, $\mathsf{H}>0$, and $\mathsf{U}<4$. Thus, the result follows by Theorem~\ref{first main theorem} and using the estimate from the first item with the expressions~\eqref{isolation of just the distinguished part of the solution} to bound the differences~\eqref{justification of the shallow water limit}.
\end{proof}

\begin{proof}[Proof of Corollary~\ref{coro on dimensional and eulerian bores}]
    We only prove the case $\iota = -1$, as the case $\iota = 1$ follows by similar arguments. Since $0<\gam<2ga/\kappa$ we have that $\m{g} = 16ga/\kappa\gamma>8$. Then the third item of Lemma~\ref{lem  on sufficient conditions for a sign} gives the existence of $\m{A}\in\tp{0,1}$ such that $\tp{\m{g},\m{A}}\in\mathfrak{C}_{-1}$. Once we define the nondimensional parameters as in~\eqref{nondimensional parameters}, Theorem~\ref{first main theorem} applies and, for sufficiently small $\ep>0$, we obtain a nondimensional smooth solution $\tp{\pmb{\zeta},\pmb{u},\pmb{p}}$ to system~\eqref{FBINSE, param. tuned and funny flux}.
    
    Recall that the flattening diffeomorphism $\mathfrak{F}_{\pmb{\zeta}}:\Bar{\Omega_\ep}\to\Bar{\Omega_{\ep\pmb{\zeta}}}$ is defined in Section~\ref{subsection on flattening, etc}. Now define $\pmb{v} = \pmb{u}\circ\mathfrak{F}_{\pmb{\zeta}}^{-1}$ and $\pmb{q} = \pmb{p}\circ\mathfrak{F}_{\pmb{\zeta}}^{-1}$ with $\mathfrak{F}_{\pmb{\zeta}}^{-1}\tp{x,z} = \tp{x,z/\pmb{\zeta}\tp{x}}$ being the inverse flattening diffeomorphism $\Bar{\Omega_{\ep\pmb{\zeta}}}\to\Bar{\Omega_\ep}$. Then $\pmb{v}\in W^{\infty,\infty}\tp{\Omega_{\ep\pmb{\zeta}};\R^2}$ and $\pmb{q}\in W^{\infty,\infty}\tp{\Omega_{\pmb{\zeta}}}$, and the triple $\tp{\pmb{\zeta},\pmb{v},\pmb{q}}$ is a solution to system~\eqref{parent Navier-Stokes system nondimensionalized}.
    
    As in Section~\ref{subsection on flattening, etc}, we define the length scale $\m{L} = a\gam/\kappa$ and the velocity scale $\m{U} = \gamma/4$ and then determine the dimensional variables $\zeta$, $v$, and $q$ as in~\eqref{nondimensional unknowns}. These are the sought after solution to the Eulerian formulation, system~\eqref{parent free boundary navier-stokes system}. The limits~\eqref{the dimensional limits} and~\eqref{the dimensional limits 2} are a direct consequence of~\eqref{the following limits are satisfied}, \eqref{isolation of just the distinguished part of the solution}, and dimensional unpacking. Finally, the expression for the relative velocity flux $\Phi$ given by~\eqref{full relative velocity flux of the bore wave solutions} can by proved by directly inserting either the left or right equilibrium height into equation~\eqref{intro_bore_equation} and using the parameter tuning identities~\eqref{the parameter tuning identities 440 hz} to simplify.
\end{proof}

\bibliographystyle{abbrv}
\bibliography{bib.bib}
\end{document}